\documentclass[twoside,11pt]{article}

\usepackage{jmlr2e}

\usepackage{helvet}

\usepackage{amsmath, amsfonts, amssymb} % math equations, symbols
\usepackage[english]{babel}
\usepackage[cmyk]{xcolor}      % color content
\usepackage{graphicx}   % import figures
\usepackage{url}        % hyperlinks
\usepackage{bm}         % bold type for equations
\usepackage{multirow}
\usepackage{booktabs}
\usepackage{epsfig}
\usepackage{algorithm}
\usepackage{algorithmic}
\usepackage{latexsym, mathrsfs, subfigure, stmaryrd}
\usepackage[scr=boondoxo,scrscaled=1.05]{mathalfa}

\usepackage{hyperref} %bookmarks

\usepackage{appendix}

\newtheorem{assumption_in_section}[theorem]{Assumption}
\newtheorem{assumption}{Assumption}

\numberwithin{equation}{section}

\usepackage{multirow}

\usepackage[normalem]{ulem}

\newcommand{\mathbbR}{\mathbb{R}}

\renewcommand*\cite[1]{\citep{#1}}

\usepackage{lastpage}
\jmlrheading{23}{2022}{1-\pageref{LastPage}}{1/21; Revised 5/22}{9/22}{21-0000}{Hao Yu, Yixiao Guo and Pingbing Ming}

\ShortHeadings{Error Estimate of DL Method for Schrödinger Eigenvalue Problem}{Yu, Guo and Ming}
\firstpageno{1}

\begin{document}

\title{Generalization Error Estimates of Machine Learning Methods for Solving High Dimensional Schrödinger Eigenvalue Problems}

\author{\name Hao Yu \email yuhao@amss.ac.cn \\
\addr LSEC, Institute of Computational Mathematics and Scientific/Engineering Computing, Academy of Mathematics and Systems Science, Chinese Academy of Sciences \\
Beijing 100190, China \\
School of Mathematical Sciences, University of Chinese Academy of Sciences \\
Beijing 100049, China\\
\AND
\name Yixiao Guo \email guoyixiao@lsec.cc.ac.cn \\
\addr LSEC, Institute of Computational Mathematics and Scientific/Engineering Computing, Academy of Mathematics and Systems Science, Chinese Academy of Sciences \\
Beijing 100190, China \\
School of Mathematical Sciences, University of Chinese Academy of Sciences \\
Beijing 100049, China \\
\AND
\name Pingbing Ming \email mpb@lsec.cc.ac.cn \\
\addr LSEC, Institute of Computational Mathematics and Scientific/Engineering Computing, Academy of Mathematics and Systems Science, Chinese Academy of Sciences \\
Beijing 100190, China \\
School of Mathematical Sciences, University of Chinese Academy of Sciences \\
Beijing 100049, China 
}

\editor{My editor}

\maketitle

\begin{abstract}%   <- trailing '%' for backward compatibility of .sty file
{We propose a machine learning method for computing eigenvalues and eigenfunctions of the Schrödinger operator on a $d$-dimensional hypercube with Dirichlet boundary conditions. The cut-off function technique is employed to construct trial functions that precisely satisfy the homogeneous boundary conditions. This approach eliminates the error caused by the standard boundary penalty method, improves the overall accuracy of the method, as demonstrated by the typical numerical examples. Under the assumption that the eigenfunctions belong to a spectral Barron space, we derive an explicit convergence rate of the generalization error of the proposed method, which does not suffer from the curse of dimensionality. We verify the assumption by proving a new regularity shift result for the eigenfunctions when the potential function belongs to an appropriate spectral Barron space. Moreover, we extend the generalization error bound to the normalized penalty method, which is widely used in practice.}
\end{abstract}

\begin{keywords}
Generalization error, Schrödinger operator, Dirichlet boundary condition, Spectral Barron space, Curse of dimensionality
\end{keywords}

%\newpage

\section{Introduction}

The high-dimensional Schrödinger eigenvalue problem plays a crucial role in various fields, such as computational chemistry, condensed matter physics, and quantum computing~\cite{lubich2008quantum,nielsen2010quantum}. Though classical numerical methods have achieved great success in solving low-dimensional PDE and eigenvalue problems, a major challenge persists: the curse of dimensionality (CoD), where computational costs grow exponentially with the dimension. Recently, machine learning has emerged as a promising approach to mitigate the CoD. Significant progress has been made in applying deep neural networks to solve PDEs~\cite{EweinanYB2017DRM, raissi2019physics, han2017deep, sirignano2018dgm, han2018solving, al2018solving, jagtap2020extended, jagtap2020conservative, zang2020weak, finzi2023stable, hu2023augmented, mouli2024using, raissi2024forward}  
and Schrödinger eigenvalue problems~\cite{cai2018approximating, carleo2017solving, choo2020fermionic, gao2017efficient, eigenvalue_han2020solving, han2019solvingusingDNN, hermann2020deep, li2022semigroup} to just name a few. 

Many theoretical attempts have been made to understand how machine learning methods for solving PDEs can overcome CoD and explain the empirical success.
In~\cite{luo2020two, xu2020finite, DRMlu2021priori, siegel2023greedy}, the authors established a priori generalization error bounds for solving elliptic PDEs by two-layer neural networks, demonstrating convergence rates independent of dimensionality. In~\cite{grohs2023proof, berner2020analysis, hutzenthaler2022overcoming}, the authors disclose that neural networks can efficiently approximate high-dimensional functions in the case of numerical approximations of Black-Scholes PDEs and high-dimensional nonlinear heat equations.  Gonon and Schwab showed that deep ReLU neural networks~\cite{gonon2023deep} and random feature neural networks~\cite{gonon2023random} can overcome CoD for partial integrodifferential equations and Black-Scholes type PDEs, respectively.
%Deep ReLU neural networks overcome the curse of dimensionality for partial integrodifferential equations\cite{gonon2023deep}
%Random feature neural networks learn Black-Scholes type PDEs without curse of dimensionality\cite{gonon2023random} 
Despite these advances, theoretical analysis for learning high-dimensional eigenvalue problems is still scarce, apart from the notable exceptions~\cite{lu2021priori, Jixia2024DeepRM4Eigenvalue}. This study aims to develop a high-precision machine learning method for Schrödinger eigenvalue problems and analyze its generalization error bound.

%In recent years, deep learning has achieved great success in solving PDEs.  
One of the main challenges of employing neural networks to solve PDEs and eigenvalue problems is how to accurately deal with the essential boundary conditions properly. A common way to address this is to incorporate a boundary penalty term to the loss function ~\cite{EweinanYB2017DRM, Jixia2024DeepRM4Eigenvalue}.
However, recent studies~\cite{wang2021understanding, lyu2020enforcing, chen2020comparison, dwivedi2020physics, krishnapriyan2021characterizing, SUKUMAR2022114333} have demonstrated that inexact imposing boundary conditions can negatively impact network training and accuracy.

The work in~\cite{xu2020finite,Jixia2024DeepRM4Eigenvalue} established that the error caused by the boundary penalty is inversely proportional to the penalty factor. To alleviate this issue, many works use a trial function that contains two terms~\cite{lagaris1998artificial, lagaris2000neural, berg2018unified}. The first component is a function satisfying the Dirichlet boundary condition, while the second component is the product of the neural network output and a function that vanishes on the boundary. The second function can be manually constructed or using approximated distance functions (ADFs)~\cite{mcfall2009artificial, sheng2021pfnn, SUKUMAR2022114333}. In our approach, we adopt this approach that are the product of neural network outputs and the cut-off functions, which {\em exactly} satisfy the boundary conditions. This method aims to enhance accuracy by ensuring that boundary conditions are perfectly met, thus avoiding the issues associated with approximating the boundary conditions.

Approximating the high-dimensional functions in Sobolev spaces or Hölder spaces poses significant challenge due to the CoD, regardless of the traditional numerical methods~\cite{xu2020finite} or neural networks~\cite{YAROTSKY2017103, yarotsky2018optimal}. To achieve dimension free approximation rates, the researchers have turned to the Barron function classes~\cite{barron1993universal}, which offers a promising alternative, such spaces have been recently elaborated on~\cite{siegel2022high, weinan2022some, DRMlu2021priori, LiaoMing:2023}. 

We introduce a new Barron function space on the hypercube $\Omega=(0,1)^d$ by employing a sine series expansion, which is dubbed as the sine spectral Barron space $\mathfrak{B}^{s}(\Omega)$. Functions within this space inherently satisfies the homogeneous Dirichlet boundary condition. This space may be regarded as a variant of the cosine spectral Barron space $\mathfrak{C}^{s}(\Omega)$, which was initially introduced in~\cite{DRMlu2021priori, lu2021priori}. We study the properties of $\mathfrak{B}^{s}(\Omega)$ and prove certain novel approximation results for functions lying in $\mathfrak{B}^{s}(\Omega)$, the rate of convergence is also dimension free.

The regularity properties of PDEs within Barron spaces have been investigated recently. In~\cite{weinan2022some}, the regularity results for the screened Poisson equation and various time-dependent equations in the Barron space have been proved through integral representations. Similarly, ~\cite{DRMlu2021priori} established a solution theory for the Poisson and the Schr\"odinger equations on the hypercube with homogeneous Neumann boundary condition. This was further extended in~\cite{lu2021priori} to include regularity estimates for the ground state of the Schrödinger operator. 
Additionally, \cite{CZLJ2023RegularSSonRd} extended the regularity theory for the static Schr\"odinger equations on $\mathbb{R}^d$ within the spectral Barron space. 

We use a loss function in the form of Rayleigh quotient for high numerical accuracy. However, since the denominator is the square of  $L^2$ norm of the trial functions, the loss function is not Lipschitz, which brings new difficulties to the derivation of generalization bounds. To address this issue, we exploit the concentration inequalities for ratio type empirical processes~\cite{RatioLimitTheorems4EP, Gine2006ConcentrationIA} and bounds for expected values of sup-norms of empirical processes~\cite{Gin2001OnCO}. The study of ratio type empirical processes has a long history that goes back to the 1970s-1980s when certain classical function classes $\{\mathbf{1}_{(-\infty,t]} : t \in \mathbb{R}\}$ have been explored in great detail~\cite{Wellner1978LimitTF} and Alexander extended this theory to ratio type empirical processes indexed by VC classes of sets~\cite{Alexander1987RatesOG, Alexander1987TheCL} in the late 1980s. Thereafter, there has been a great deal of work on the development of ratio type inequalities, primarily, in more specialized contexts of nonparametric statistics~\cite{Talagrand2018SomeAO, Geer2000ApplicationsOE} and learning theory~\cite{Bartlett1999AnIF, Koltchinskii2000RademacherProBound, Bartlett2002LocalizedRC, Bousquet2002ConcentrationIA, Bousquet2002SomeLocM, Koltchinskii2002EmpiricalMD, Panchenko2002SomeEO, Panchenko2003SYMMETRIZATIONAT, Bartlett2006EmpiricalM, Bartlett2006LocalRCOrIneq, Massart2007ConcentrationIA}, etc. These inequalities have become the main ingredients in determining asymptotically sharp convergence rates in regression, classification and other nonparametric problems and they proved to be crucial in bounding the generalization error of learning algorithms based on empirical risk minimization. 
\subsection{Our contributions}
We propose a machine learning method for solving high dimensional Schrödinger eigenvalue problems on the $d$-dimensional hypercube with Dirichlet boundary conditions and establish the generalization error bounds. The thorough error analysis provides deeper insights and guidance for the design and application of such algorithm. Our contributions are summarized as follows.

\textbullet~We introduce a sine spectral Barron space $\mathfrak{B}^{s}(\Omega)$ on $\Omega = (0, 1)^d$, which is suitable for analyzing PDEs and eigenvalue problems on the hypercube with Dirichlet boundary conditions.  We develop novel regularity estimates showing that all eigenfunctions lie in the sine spectral Barron space provided that the potential function belongs to an appropriate spectral Barron space (see Theorem \ref{Thm: Regularity of eigenfunctions}). Similar arguments may be used for other boundary value problems.

\textbullet~We introduce the cut-off function technique to construct trial functions that exactly satisfy the essential boundary conditions. We show that functions in the sine spectral Barron space may be well approximated in the $H^1$-norm using two-layer ReLU or Softplus networks multiplied by a specific cut-off function (see Theorem~\ref{Thm: u H1 approximation by varphi ReLU networks} and Theorem~\ref{Thm: u H1 approximation by varphi Softplus networks}). The approximation rate is $O(m^{-1/2})$, which is independent of dimension. The prefactors in the error bound linearly depend on the spectral Barron norm.

\textbullet~We introduce concentration inequalities for ratio-type suprema to handle the Rayleigh quotient and derive an oracle inequality for the empirical loss. To bound the statistical error, we propose an exponential inequality to bound the normalized complexity measure in ratio-type concentration inequalities.

\textbullet~We prove a prior generalization error estimates for our machine learning method applied to Schrödinger eigenvalue problems (see Theorem \ref{Main generalization theorem}). We emphasize that the error bound is valid for the higher-order eigenstates, not just the ground state; and the convergence rate $\tilde{O}(n^{-1/4})$ is independent of dimension. 
The prefactors in the error estimates depend explicitly on all relevant parameters, including the order of eigenstates and dimension, with these dependencies being polynomial and of lower degree.

\textbullet~As extensions, we demonstrate the effectiveness of incorporating a normalization penalty term by proving the solutions of the normalization penalty method\cite{EweinanYB2017DRM} are away from zero with high probability (Theorem~\ref{theorem: solutions obtained by penalty method are away from 0}). We prove the generalization error bounds for the normalization penalty method (Corollary~\ref{corollary: Main generalization theorem for the penalty method}), which reflects the advantages and disadvantages of our method and the normalization penalty method. We also characterize the accumulative error in practice (Theorem~\ref{Main generalization theorem when cumulative error is added}) and prove sharp accumulative rate of generalization errors (Proposition~\ref{proposition: Quadratic growth of cumulative error}).

\textbullet~We test our method in various scenarios, including problems posed on high-dimensional domain, irregular domains and problems with singular potentials. The numerical results demonstrate that our method performs well in these cases, achieving relative error of the eigenvalues as good as $\mathcal{O}(10^{-3})$ or even better for up to the first $30$ eigenvalues. The results also highlight that enforcing boundary conditions exactly significantly improves numerical accuracy compared to the boundary penalty method.
\subsection{Related works}
\cite{lu2021priori} solved the Schrödinger eigenvalue problems with Neumann boundary condition with neural networks. They defined a spectral Barron space $\mathfrak{C}^{s}(\Omega)$ via the cosine function expansion and prove that functions in $\mathfrak{C}^{s}(\Omega)$ may be well approximated in the $H^1$-norm using two layer ReLU or Softplus networks. Using the Krein-Rutman theorem~\cite{krein1948linear}, they proved the existence of ground state in $\mathfrak{C}^{s+2}(\Omega)$ if the potential functions lie in $\mathfrak{C}^{s}(\Omega)$ with $s \geq 0$. In the present work, we use two layer ReLU or Softplus networks multiplied by a cut-off function to approximate functions in $\mathfrak{B}^{s}(\Omega)$. 
We improve the approximation result in~\cite{lu2021priori}. The term $\ln m$ is removed from the error bound for Softplus network. Additionally, we prove that all eigenfunctions lie in $\mathfrak{B}^{s+2}(\Omega)$ if the potential function belongs to $\mathfrak{C}^{s}(\Omega)$ with $s \geq 0$. Our approach may also be applied to the Neumann boundary condition, and recovers the regularity of the eigenfunctions proved in~\cite{lu2021priori}.
%leading to that all eigenfunctions lie in $\mathfrak{C}^{s+2}(\Omega)$.

\cite{Jixia2024DeepRM4Eigenvalue} also employed neural networks to solve the Schrödinger eigenvalue problems posed on a bounded $C^{m}$ domain $\Omega$ with Dirichlet boundary condition. They have assumed that the potential function $V \in C^{m-1}(\Omega)$ with $m>\max \left\{2, d/2-2\right\}$, and utilize a loss function that contains a boundary penalty, a normalization penalty and an orthogonal penalty. They proved a convergence rate $\mathcal{O}(n^{-1/16})$ for the generalization error. By using the Rayleigh quotient and trial functions with exact imposition of boundary conditions, our method only needs to consider the orthogonal penalty, which leads to better numerical accuracy. We also prove a better accumulation rate of the generalization errors.

The main limitation of our work lies in the following two parts. The first is the \textbf{Assumption 1}, in which we assume that the potential function $V$ is bounded from below and above. Such assumption definitely exclude the commonly used singular potential. It should be pointed out that the method equally applies to the singular function, e.g., the inverse square potential in the numerical tests. Another limitation is the regularity assumption on $V$, for which we require it belongs to $\mathfrak{C}^s$ with $s\ge 1$, this is quite smooth as demonstrated in~\cite{lu2021priori}.
\subsection{Notations}
Let $H^{1}(\Omega)$  be the standard Sobolev space \cite{bookSobolevSpaces} with the norm  $\|\cdot\|_{H^{1}(\Omega)}$, while  $H_{0}^{1}(\Omega)$  is the closure of $C_{0}^{\infty}(\Omega)$  in $H^{1}(\Omega)$. For a function set $\mathcal{F}$, we denote $\mathcal{F}_{>r} := \{ f \in \mathcal{F} : \|f\|_{L^{2}(\Omega)} > r \}$.
When $C$ is used to denote some absolute constant, its value might not necessarily be the same in each occurrence. 
%\yixiao{boundary condition or boundary conditions? Barron space or Barron spaces? Relu or Relu function?}
\section{Basic settings}
Let $\Omega = (0, 1)^d$  be the unit hypercube.
Consider the Dirichlet eigenvalue problem for the Schrödinger operator
\begin{equation}
\begin{cases} \label{eq1}
\mathcal{H}u(x) := -\Delta u(x) + V(x) u(x) 
= \lambda u(x), & x \in \Omega, \\
u(x) = 0, & x \in \partial \Omega,
\end{cases}
\end{equation}
where $\mathcal{H}$ is the Schrödinger operator and $V$ is a potential function. We are interested in computing the eigenvalues and the corresponding eigenfunctions.

Throughout this paper, we make the following assumption on the potential function $V$.
\begin{assumption}\hypertarget{Assumption 1}{} \label{Assumption: V is bounded above and below}
There exist positive constants  $V_{\min}$  and  $V_{\max}$  such that  $V_{\min} \leq  V(x) \leq V_{\max}$  for every  $x \in \Omega$.
\end{assumption}
Without loss of generality, we may assume that $V_{\min}$ is positive, since adding a constant to $V$ does not alter the eigenfunctions.

Given Assumption \ref{Assumption: V is bounded above and below} and the fact that $\Omega$ is a connected open bounded domain, $(-\Delta + V)^{-1}$ is a compact self-adjoint operator on $L^2(\Omega)$. 
Consequently, $\mathcal{H}$  always has a purely discrete spectrum  $\left\{\lambda_{j}\right\}_{j=1}^{\infty}$ with $\infty$ as its only point of accumulation, and the eigenfunctions of $\mathcal{H}$ form an orthonormal basis on $L^2(\Omega)$~\cite{Evans2010PartialDE}.
Here, we list all the eigenvalues in an ascending order, with multiplicities, as  $0 < \lambda_{1}\leq \lambda_{2} \leq \lambda_{3} \leq \cdots  \uparrow \infty$, and denote by $\{\psi_{j}\}_{j = 1}^{k-1}$ the first $k-1$ normalized orthogonal eigenfunctions.

For Problem~\eqref{eq1}, we intend to find the eigenvalues and their corresponding eigenfunctions in $H_{0}^{1}(\Omega)$.
For any positive integer $k$, the $k$-th smallest eigenvalue may be determined using the Rayleigh quotient
\begin{equation*}
\lambda_{k} = \min_{E} \max_{u \in E \backslash\{0\}} \frac{\langle u, \mathcal{H} u\rangle_{H^{1} \times H^{-1}} }{\|u\|_{L^{2}(\Omega)}^{2}},
\end{equation*}
where the minimum is taken over all  $k$-dimensional subspace $E \subset H_{0}^{1}(\Omega)$,  and $\langle \cdot, \cdot\rangle_{H^{1} \times H^{-1}}$ denotes the dual product on $H^{1}(\Omega) \times H^{-1}(\Omega)$.
Given the first  $k-1$ eigenpairs $\left(\lambda_{1}, \psi_{1}\right)$, $\cdots$, $\left(\lambda_{k-1}, \psi_{k-1}\right)$, the subsequent eigenvalue $\lambda_{k}$ may be characterized as
\begin{equation} \label{characterization of lambdak with orthogonal constraints}
\lambda_{k} = \min_{u \in E^{(k-1)}} \frac{\langle u, \mathcal{H} u\rangle_{H^{1} \times H^{-1}}}{\|u\|_{L^{2}(\Omega)}^{2}},
\end{equation}
where
$E^{(k-1)}=\left\{u \in H_{0}^{1}(\Omega) \backslash\{0\}: u \perp \psi_{i}, 1 \leq i \leq k-1\right\}$.

It is intuitive to seek an approximate solution to Problem (\ref{characterization of lambdak with orthogonal constraints}) within a
hypothesis class $\mathcal{F} \subset H^{1}_{0}(\Omega)$ that is parameterized by neural networks. 
%We construct a loss function by incorporating orthogonal penalty terms to handle the orthogonal constraints.
We formulate a loss function that integrates orthogonal penalty terms to enforce the orthogonal constraints.
Specifically, the loss function for computing the  $k$-th eigenfunction is
\begin{equation}
\begin{aligned} \label{the first definition of the loss function Lk(u)}
L_{k}(u) & = \frac{ \langle u, \mathcal{H} u\rangle_{H^{1} \times H^{-1}} }{\|u\|_{L^{2}(\Omega)}^{2}}+\beta \sum_{j=1}^{k-1} \frac{\langle u, \psi_{j}\rangle^{2} }{\|u\|_{L^{2}(\Omega)}^{2}},
\end{aligned}
\end{equation}
where $\langle \cdot, \cdot\rangle$ denotes the inner product on $L^{2}(\Omega)$, and $\beta$ is a penalty parameter that should be greater than $\lambda_{k} - \lambda_{1}$.
To compute the ground state, i.e., $k = 1$, the loss function becomes the Rayleigh quotient and the orthogonal penalty term is no longer necessary.
In practice, the Monte Carlo method is employed to compute the high dimensional integral in the loss function and an approximate solution is then obtained through empirical loss (or risk) minimization.
Denote by $\mathcal{P}_{\Omega}$   the uniform probability distribution on  $\Omega$  and  $X$, $X_{1}$, $X_{2}$, $\cdots$ are $i.i.d.$ random variables according to  $\mathcal{P}_{\Omega}$.
The population loss  $L_{k}(u)$  is expressed  as
\begin{equation}
\begin{aligned} \label{population loss}
L_{k}(u) & = \frac{\mathcal{E}_{V}(u) +  \mathcal{E}_{P}(u)}{\mathcal{E}_{2}(u)} := \frac{\mathbf{E} \left[|\nabla u(X)|^{2} + V(X)|u(X)|^{2}\right] + \beta \sum_{j=1}^{k-1}  \left( \mathbf{E} \left[ u(X) \psi_j(X)\right]\right)^2}{\mathbf{E}\left[|u(X)|^{2}\right]} .
\end{aligned}
\end{equation}
The population loss $L_{k}(u)$ may be approximated by the empirical loss
\begin{equation}
\begin{aligned} \label{empirical loss}
L_{k, n}(u)=\frac{\mathcal{E}_{n, V}(u)+\mathcal{E}_{n, P}(u)}{\mathcal{E}_{n, 2}(u)},
\end{aligned}
\end{equation}
where
\begin{equation*}
\begin{aligned} \label{}
\mathcal{E}_{n, V} & := \frac{1}{n} \sum_{i=1}^{n}\left(\left|\nabla u\left(X_{i}\right)\right|^{2} + V\left(X_{i}\right)\left|u\left(X_{i}\right)\right|^{2}\right), \\
\mathcal{E}_{n, P} & := \beta \sum_{j=1}^{k-1}  \left(  \frac{1}{n} \sum_{i=1}^{n} u(X_{i}) \psi_j(X_{i}) \right)^{2}, \quad \quad
\mathcal{E}_{n, 2}  :=\frac{1}{n} \sum_{i=1}^{n}\left|u\left(X_{i}\right)\right|^{2}.
\end{aligned}
\end{equation*}
Let  $\widehat{u}_{n}$  be a minimizer of $L_{k, n}(u)$  within  $\mathcal{F}$, i.e.,
\(
\widehat{u}_{n}=\mathop{\arg\min}_{u \in \mathcal{F}} L_{k, n}(u),
\)
which is regarded as an approximation of the $k$-th eigenfunction. Then, we can approximate $\lambda_{k}$ by 
\begin{equation*}
\begin{aligned} \label{}
\widehat{\lambda}_{k, n} = \frac{ \langle \widehat{u}_{n}, \mathcal{H} \widehat{u}_{n}\rangle}{\langle \widehat{u}_{n},  \widehat{u}_{n}\rangle}.
\end{aligned}
\end{equation*}

Let $U_{k}$ be the true solution space for the $k$-th eigenfunction in $H^{1}_{0}(\Omega)$, i.e.,
\begin{equation} \label{subspace Uk, solution space}
U_{k} := \operatorname{span}\left\{\psi_{1}, \psi_{2}, \ldots, \psi_{k-1}\right\}^{\perp} \cap \operatorname{ker}\left( \mathcal{H} - \lambda_{k} I\right)  \subset H^{1}_{0}(\Omega).
\end{equation}
Any non-zero function is a minimizer of $L_{k}(u)$ if and only if it lies in $U_{k}$. Our goal is to derive quantitative estimates for  $|\widehat{\lambda}_{k, n} - \lambda_{k}|$  and the offset of the direction of $\widehat{u}_{n}$ from the subspace $U_{k}$, which is commonly referred to as the generalization error.

The following proposition shows that $|\widehat{\lambda}_{k, n} - \lambda_{k}|$ may be bounded by the energy excess $L_{k}(\widehat{u}_{n}) - \lambda_{k}$.
\begin{proposition} \label{proposition: error between Rayleigh quotient and lambdak bounded by energy excess}
Under Assumption \ref{Assumption: V is bounded above and below},  for any  nonzero $u \in H^{1}(\Omega)$ and $\beta > \lambda_{k} - \lambda_{1}$, %there holds
%Assume that $\mathcal{H}$ satisfies Assumption \ref{Assumption: V is bounded above and below}. Then for any  nonzero $u \in H^{1}(\Omega)$ and $\beta > \lambda_{k} - \lambda_{1}$,
\begin{equation*}
\begin{aligned}
\left| \frac{\langle u, \mathcal{H} u\rangle}{\langle u,  u\rangle} - \lambda_{k}\right| \leqslant \max\left\{\frac{\lambda_{k} - \lambda_{1}}{\beta+\lambda_{1}-\lambda_{k}}, 1\right\} \left( L_{k}(u)-\lambda_{k} \right) .
\end{aligned}
\end{equation*}
\end{proposition}

To quantify the offset of a direction $u$ from the subspace $U_{k}$, we define  $P^{\perp}$ as the orthogonal projection operator from $L^2(\Omega)$ to $U_{k}^{\perp}$, the orthogonal complement of $U_{k}$. 
Let $\lambda_{k^{\prime}}$ be the first eigenvalue of $\mathcal{H}$ that is strictly greater than $\lambda_{k}$, i.e., $k^{\prime} \geq k+1,\ \lambda_{k^{\prime}} > \lambda_{k}$ and $\lambda_{k^{\prime}-1} = \lambda_{k}$. 
The following proposition shows that $\|P^{\perp}\widehat{u}_{n}\|_{H^1(\Omega)}$ may also be bounded by  the energy excess $L_{k}(\widehat{u}_{n}) - \lambda_{k}$.
\begin{proposition}
\label{Generalized proposition 2.1}
Under Assumption \ref{Assumption: V is bounded above and below},  for any  $u \in H^{1}(\Omega)$ and $\beta > \lambda_{k} - \lambda_{1}$,
%Assume that $\mathcal{H}$ satisfies Assumption  \ref{Assumption: V is bounded above and below}. Then for any  $u \in H^{1}(\Omega)$ and $\beta > \lambda_{k} - \lambda_{1}$,
\begin{subequations}
\begin{align}
\left\|P^{\perp} u\right\|^2_{L^{2}(\Omega)} & \leqslant \frac{L_{k}(u)-\lambda_{k}}{\min \left\{\beta+\lambda_{1}-\lambda_{k}, \lambda_{k^{\prime}}-\lambda_{k}\right\}}\|u\|_{L^{2}(\Omega)}^{2},  \label{Eq: stable estimate for u}  \\ 
\left\|\nabla\left(P^{\perp} u\right)\right\|^2_{L^{2}(\Omega)} & \leqslant \left( L_{k}(u)-\lambda_{k} \right) \left(\frac{\lambda_{k} - V_{\min}}{\min \left\{\beta+\lambda_{1}-\lambda_{k}, \lambda_{k^{\prime}}-\lambda_{k}\right\}}+1\right)  \|u\|_{L^{2}(\Omega)}^{2}.  \label{Eq: stable estimate for nabla u} 
\end{align}
\end{subequations}
\end{proposition}
%\iffalse
%Proposition \ref{proposition: error between Rayleigh quotient and lambdak bounded by energy excess} reveals that $\beta+\lambda_{1}-\lambda_{k}$ is a metric for assessing the stability of the approximating eigenvalue. %obtained by minimizing $L_{k}(u)$
%While, Proposition \ref{Generalized proposition 2.1} demonstrates that both the gap $\lambda_{k^{\prime}} - \lambda_{k}$ and the factor $\beta+\lambda_{1}-\lambda_{k}$ influence the stability of the approximate eigenfunction. 
%Therefore, it is imperative to choose $\beta$ to be sufficiently large for stability.
%In addition, Proposition \ref{proposition: error between Rayleigh quotient and lambdak bounded by energy excess} and Proposition \ref{Generalized proposition 2.1} imply that to estimate error $\left| \widehat{\lambda}_{k, n} - \lambda_{k}\right|$ and $\left\|P^{\perp}\widehat{u}_{n}\right\|_{H^{1}(\Omega)}^{2}$, we only need to bound the energy excess $L_{k}(\widehat{u}_{n}) - \lambda_{k}$.
%We postpone the proof of Proposition \ref{proposition: error between Rayleigh quotient and lambdak bounded by energy excess} and Proposition \ref{Generalized proposition 2.1} to Appendix \ref{section: Stability estimate of the k-th eigenfunction}.
%\fi

Proposition \ref{proposition: error between Rayleigh quotient and lambdak bounded by energy excess} suggests that $\beta + \lambda_{1} - \lambda_{k}$ serves as a metric for evaluating the stability of the approximating eigenvalues, while Proposition \ref{Generalized proposition 2.1} demonstrates that both the gap $\lambda_{k^{\prime}} - \lambda_{k}$ and the factor $\beta + \lambda_{1} - \lambda_{k}$ influence the stability of the approximating eigenfunctions. We postpone the proof of Proposition \ref{proposition: error between Rayleigh quotient and lambdak bounded by energy excess} and Proposition \ref{Generalized proposition 2.1} to Appendix \ref{section: Stability estimate of the k-th eigenfunction}.

\section{Main Results}
%We will use the notation  $a \lesssim b$  to mean there exists a universal constant  $C$  such that  $a \leq  C b$,  $a \gtrsim b$  to mean  $b  \lesssim a$, and  $a \simeq b$  to mean  $a \lesssim b$  and  $a \gtrsim b$. 
%We will use the notation  $a \lesssim b$  to mean there exists an absolute constant  $C$  such that  $a \leq  C b$,  $a \gtrsim b$  to mean  $b  \lesssim a$, and  $a \simeq b$  to mean  $a \lesssim b$  and  $a \gtrsim b$. 
We are interested in establishing quantitative generalization error that is free of CoD.
To this end, we assume that the eigenfunctions lie in a function space, which should be smaller than the typical Sobolev space, where functions can be approximated by neural networks without CoD.
Specifically, we work with the Sine spectral Barron space defined below. Consider the set of sine functions
$$\mathfrak{S}=\left\{\Phi_{k}\right\}_{k \in \mathbb{N}_{+}^{d}} := \left\{\prod_{i=1}^{d} \sin \left(\pi k_{i} x_{i}\right) \mid k = (k_{1}, k_{2}, ... , k_{d}),  k_{i} \in \mathbb{N}_{+}\right\} .$$ 
Let  $\{\hat{u}(k)\}_{k \in \mathbb{N}_{+}^{d}}$  be the Fourier coefficients of $u \in L^{1}(\Omega)$ with respect to the basis  $\left\{\Phi_{k}\right\}_{k \in \mathbb{N}_{+}^{d}}$. For  $s \geq 0$, the sine spectral Barron space  $\mathfrak{B}^{s}(\Omega)$ is defined by
\begin{equation} \label{definition of sine spectral Barron space}
	\begin{aligned}
		\mathfrak{B}^{s}(\Omega) := \left\{u \in L^{1}(\Omega): \|u\|_{\mathfrak{B}^{s}(\Omega)}<\infty \right\},
	\end{aligned}
\end{equation}
which is equipped with the spectral Barron norm
\begin{equation} \label{definition of sine spectral Barron norm}
	\begin{aligned}
		\|u\|_{\mathfrak{B}^{s}(\Omega)} = \sum_{k \in \mathbb{N}_{+}^{d}}\left(1+\pi^{s}|k|_{1}^{s}\right)|\hat{u}(k)| ,
	\end{aligned}
\end{equation}
where $|k|_{1}$ is the  $\ell^{1}$-norm of a vector  $k$.  
$\mathfrak{B}^{s}(\Omega)$  is a Banach space since it can be viewed as a weighted  $\ell^{1}$  space  $\ell_{W_{s}}^{1}\left(\mathbb{N}_{+}^{d}\right)$  of the sine coefficients defined on the lattice  $\mathbb{N}_{+}^{d}$  with the weight  $W_{s}(k) = 1+\pi^{s}|k|_{1}^{s} $. 
Moreover, it is straightforward to verify that the functions in  $\mathfrak{B}^{s}(\Omega)$   are continuous and precisely satisfy the homogeneous Dirichlet boundary condition. In what follows, we prove new approximation results for functions in $\mathfrak{B}^{s}(\Omega)$.

%In \cite{SUKUMAR2022114333},  the solution structure for the homogeneous  Dirichlet boundary condition is given by 
The solution is represented as $u = \varphi  v$, where $\varphi: \mathbb{R}^d \to \mathbb{R}$ is an approximate distance function and $v$ is a neural network function. 
Subsequently, we employ a two layer neural network to construct  $v = v_{nn}(x;\theta)$ and take $\varphi$ as 
\begin{equation*} \label{sin d-dim cutoff func}
	\begin{aligned}
		\varphi(x) = \left[\sum_{i=1}^{d} \frac{1}{\sin \left(\pi x_{i}\right)}\right]^{-1} = \frac{\prod_{i=1}^{d} \sin \left(\pi x_{i}\right)}{\sum_{i=1}^{d} \prod_{\substack{1 \leqslant j \leqslant d \\
					j \neq i}} \sin \left(\pi x_{j}\right)}, \qquad x \in \Omega.
	\end{aligned}
\end{equation*}
Moreover, one can continuously extend $\varphi(x)$ to be $0$ on $\partial \Omega$. %\textcolor{red}{(Move to introduction?)}
%(\textcolor{red}{Here, do we need to supplement some properties of $\varphi(x)$ as a first-order ADF?})
%Notice that any trial function of this form $\varphi\tilde{v}_{nn}(x;\theta)$ will exactly satisfy the essential boundary condition.

Our approximation result indicates that functions in $\mathfrak{B}^{s}(\Omega)$ may be well approximated by $\varphi v_{nn}(x;\theta)$ with respect to $H^{1}$-norm when using either the ReLU or Softplus activation functions.
%the activation function is either ReLU or Softplus. 
More precisely, for a given activation function $\phi$, the number of hidden neurons $m$ and a positive constant $B$, we consider 
\begin{equation*} \label{hypothesis spaces of two layer neural networks, general activation function}
	\begin{aligned}
		\mathcal{F}_{\operatorname{\phi}, m}(B) := \left\{c+\sum_{i=1}^{m} \gamma_{i} \phi\left(w_{i} \cdot x - t_{i}\right)\mid |c| \leq  B, \left|w_{i}\right|_{1}=1, \left|t_{i}\right| \leq 1, \sum_{i=1}^{m}\left|\gamma_{i}\right| \leq 4 B\right\}.
	\end{aligned}
\end{equation*}

The first approximation result concerns the approximation of the hypothesis space $\varphi \mathcal{F}_{\operatorname{ReLU}, m}(B)$\footnote{In this paper, let $\varphi \mathcal{F} := \{ \varphi v : v \in \mathcal{F}\}$, where $\mathcal{F}$ is a function set.} with the activation function $\operatorname{ReLU}(x) = \max\{x, 0\}$. 
\begin{theorem}
	\label{Thm: u H1 approximation by varphi ReLU networks}
	For $u \in \mathfrak{B}^{s}(\Omega)$ with $s\geq 3$, there exists  $v_{m} \in \mathcal{F}_{\mathrm{ReLU}, m}(B)$ with $B = \left(1+\frac{2 d}{\pi  s}\right) \cdot$ $\|u\|_{\mathfrak{B}^{s}(\Omega)}$ such that
	\begin{equation*} \label{}
		\begin{aligned}
			\left\|u-\varphi v_{m}\right\|_{H^{1}(\Omega)} \leq \frac{28 B}{\sqrt{m}}.
		\end{aligned}
	\end{equation*}
\end{theorem}

Next, we consider the approximation by the hypothesis space $\varphi \mathcal{F}_{\mathrm{SP}_{\tau}, m}(B)$ with the Softplus activation
function
\[
\operatorname{SP}_{\tau}(z):=\tau^{-1}\operatorname{SP}(\tau z)=\tau^{-1}\ln \left(1+e^{\tau z}\right),
\]
where $\tau>0$  is a scaling parameter. The rescaled Softplus function may be viewed as a smooth approximation of the ReLU function 
%since $\mathrm{SP}_{\tau} \rightarrow \operatorname{ReLU}$  pointwise as  $\tau \rightarrow \infty$  
\citep[see][Lemma 4.6]{DRMlu2021priori}.
%(see \cite[Lemma 4.6]{DRMlu2021priori}). %(see Lemma \ref{lemma: ReLU-Softplus difference})
This approximation is particularly useful for bounding
%which will be exploited to bound 
the complexities of function classes that involve the derivatives of the neural network functions.
%The reason why we need networks with the Softplus activation is that we need networks with $C^{2}$ activation function to \textcolor{red}{bound the complexities of function classes involving derivatives of network functions} more easily. The fact that two-layer networks with the ReLU activation only admit first order weak derivatives makes it difficult to bound the complexities of function classes involving derivatives of network functions, which is a crucial ingredient for getting a generalization bound.  (\textcolor{red}{The reason for using ReLU and Softplus networks needs to be supplemented here.})
%
The following theorem demonstrates that  $\varphi \mathcal{F}_{\mathrm{SP}_{\tau}, m}(B)$ admits a similar approximation bound as $\varphi \mathcal{F}_{\operatorname{ReLU}, m}(B)$.
\begin{theorem}\label{Thm: u H1 approximation by varphi Softplus networks}
For $u \in \mathfrak{B}^{s}(\Omega)$ with $s\geq 3$,  there exists  $v_{m} \in \mathcal{F}_{\mathrm{SP}_{\tau}, m}\left(B\right)$  with  $B = \left(1+\frac{2 d}{\pi  s}\right)\cdot $ $\|u\|_{\mathfrak{B}^{s}(\Omega)}$ 
	and %the activation function  $\mathrm{SP}_{\tau}$ taking 
	$\tau=9\sqrt{m}$ such that
	$$\left\|u-\varphi v_{m}\right\|_{H^{1}(\Omega)} \leq \frac{64B}{\sqrt{m}}.$$
\end{theorem}

Motivated by Theorem \ref{Thm: u H1 approximation by varphi ReLU networks} and Theorem \ref{Thm: u H1 approximation by varphi Softplus networks}, we make the following regularity assumption.
\begin{assumption} \label{Assumption: exist u* in both mathfrakB^s for s geq 3 and U_k}
	There exists a normalized eigenfunction $u^{*} \in \mathfrak{B}^{s}(\Omega)$ for some  $s \geq 3$ lying in the subspace $U_{k}$ defined in (\ref{subspace Uk, solution space}). 
\end{assumption}

Denote $\bar{\mu}_{1} = 0$ and $\bar{\mu}_{k} = \max_{1 \leq j \leq k-1} \|\psi_{j}\|_{L^{\infty}(\Omega)}$ for $k\geq 2$.
%For a function set $\mathcal{F}$, denote $\mathcal{F}_{>r} := \{ f \in \mathcal{F} : \|f\|_{L^{2}(\Omega)}  > r \}$.
We establish a-priori generalization error estimate of our machine learning method for solving Schrödinger eigenvalue problems, which is the main result of this work.
\begin{theorem}  \label{Main generalization theorem}
Under Assumptions \ref{Assumption: V is bounded above and below} and \ref{Assumption: exist u* in both mathfrakB^s for s geq 3 and U_k}, 
let $\mathcal{F} = \varphi \mathcal{F}_{\mathrm{SP}_{\tau}, m}(B)$
with $B = \left(1+\frac{2 d}{\pi s}\right) \|u^{*}\|_{\mathfrak{B}^{s}(\Omega)}$  and   $\tau=9\sqrt{m}$. 
	%Denote $\bar{\mu}_{k} = \max_{1 \leq j \leq k-1} \|\psi_{j}\|_{L^{\infty}(\Omega)}$  and  $\varphi(x)$ is defined in (\ref{sin d-dim cutoff func}).
For $r\in (0,1/2)$ and let $u_{n}^{m}$  be a minimizer of the empirical loss $L_{k,n}$  within $\mathcal{F}_{> r}$. Given  $\delta \in\left(0, 1/3\right)$, assume that $n$ and $m$ are large enough so that $64B/\sqrt{m} \leq 1/2$ and
\begin{subequations} \label{ineqs: conditions for Main generalization theorem}
\begin{align}
			& \Upsilon_{1}(n,m,B,r,\delta) := \frac{B}{r}\sqrt{\frac{1 + V_{\max }}{n}}  \left( \sqrt{m \ln \frac{ B \left(1 + \sqrt{m}/d\right) \left(1 + V_{\max} \right)}{r d}}  +  \sqrt{\frac{\ln(1/\delta)}{d}}\right) \leq C, \label{ineq: assumption in Main generalization theorem to ensure xi1+xi2<1/2}\\
			& \Upsilon_{2}(n,m,k,B,\bar{\mu}_{k},r,\delta) := \sqrt{\frac{k \bar{\mu}_{k} B}{n  r}}\left[ \sqrt{m \ln \left(\frac{\bar{\mu}_{k} B}{r d}\right)} + \sqrt{\frac{\ln( k/\delta)}{d}} \right] \leq 1, \label{Simplifying condition for xi3 term}
		\end{align}    
\end{subequations}
where  $C$ is an absolute constant. Then with probability at least $1-3 \delta$, 
\begin{equation} \label{ineq: error bound in Main generalization theorem, full parameter}
		\begin{aligned}
			L_{k}\left(u_{n}^{m}\right)-\lambda_{k} & \leqslant  C \big[ \lambda_{k} \Upsilon_{1}(n,m,B,r,\delta) +  \beta \Upsilon_{2}(n,m,k,B,\bar{\mu}_{k},r,\delta)  +  \left(V_{\max}  + \beta + \lambda_{k}\right) B/\sqrt{m}\big], %\\
			%& =:  C \Upsilon(n,m,k,B,\bar{\mu}_{k},\lambda_{k},\beta,r,\delta)
		\end{aligned}
	\end{equation}
where $C$ is an absolute constant. In particular, with the choice of $m = O(\sqrt{n/k})$ and $n$ large enough, there exists  $\tilde{C} >0$  such that with probability at least $1-3 \delta$,
\[
L_{k}\left(u_{n}^{m}\right)-\lambda_{k} \leq \tilde{C} \left[ \left(\frac{k}{n}\right)^{1/4} \left(\sqrt{\frac{\ln (n/k)}{k}} + 1\right) + \sqrt{\frac{k\ln \left(k/\delta\right)}{n}}\, \right].
\]
\end{theorem}

To prove Theorem \ref{Main generalization theorem}, we introduce concentration inequalities for ratio-type suprema to deal with Rayleigh quotient and derive an oracle inequality for the empirical loss in \S \ref{section: Oracle inequality for the generalization error}, which decomposes the generalization error into the sum of the approximation error and the statistical error. Combining the approximation results in \S \ref{section: Approximation theory for sine spectral Barron functions} and the statistical error estimates in \S \ref{section: Statistical error}, we prove Theorem \ref{Main generalization theorem}.

In~\eqref{ineq: error bound in Main generalization theorem, full parameter}, $\lambda_{k} \Upsilon_{1}$ and $\beta \Upsilon_{2}$ correspond to the statistical errors of the Rayleigh quotient and the orthogonal penalty term, respectively, and $\left(V_{\max}+ \beta + \lambda_{k}\right) B/\sqrt{m}$ stands for the approximation error. The error bound can be applied to the eigenvalue problem of any order. The convergence rate $\tilde{O}(n^{-1/4})$ is independent of the dimension. 
The dependence of $\tilde{C}$ on all parameters is explicit and at most a polynomial of low degree. We have established a high probability bound for the generalization error, which immediately implies an expectation bound. 
\begin{remark}
Under Assumption \ref{Assumption: V is bounded above and below},  there exist constants $c_1$, $c_2$ and $c_3$ such that for $k=1,2,\cdots$, the asymptotic distribution of the eigenvalues
\[
	c_1 d k^{2/d} + V_{\min} \leq \lambda_{k} \leq c_2 d k^{2/d} + V_{\max},
\] 
and $L^{\infty}(\Omega)$ estimate of the normalized eigenfunctions
\[
	\|\psi_{k}\|_{L^{\infty}(\Omega)} \leq \left(c_3 k^{2/d} +  \frac{e\left(V_{\max}- V_{\min}\right)}{\pi d}\right)^{d / 4}.
\] 
Thus, $\bar{\mu}_{k} < \infty$ for all $k$. We refer to Appendix \ref{appendix section: Simple properties of eigenvalues and eigenfunctions} for a proof.% of the above estimates for $\lambda_k$ and $\psi_k$.
\end{remark}

\begin{remark}
Condition~\eqref{Simplifying condition for xi3 term} may be removed provided that $\Upsilon_{2}$ in (\ref{ineq: error bound in Main generalization theorem, full parameter}) is replaced by $\Upsilon_{2}(1+\Upsilon_{2})^{3}$, i.e., some high order terms are in demand. 
\end{remark}

Substituting Theorem \ref{Main generalization theorem} into Proposition \ref{proposition: error between Rayleigh quotient and lambdak bounded by energy excess} and Proposition \ref{Generalized proposition 2.1}, we obtain
\begin{corollary}\label{Corollary: the generalization error in terms of eigenvalues and $H^{1}$-norm of eigenfunctions}
Under the same assumptions of Theorem \ref{Main generalization theorem}, with $m =\mathcal{O}(\sqrt{n/k})$, there exist constants $\tilde{C}_{1}$ depending only on $B$, $d$, $V_{\max}$, $\lambda_{k}$, $\bar{\mu}_{k}$, $\beta$, $r^{-1}$, $\left(\beta+\lambda_{1}-\lambda_{k}\right)^{-1}$ polynomially and  $\tilde{C}_{2}$ depending only on $B$, $d$, $V_{\max}$, $\lambda_{k}$, $\bar{\mu}_{k}$, $\beta$, $r^{-1}$,   $\left(\beta+\lambda_{1}-\lambda_{k}\right)^{-1}$, $\left(\lambda_{k^{\prime}}-\lambda_{k}\right)^{-1}$ polynomially such that with probability at least  $1-3 \delta$,
\begin{align*}
			& \left| \frac{\langle u_{n}^{m}, \mathcal{H} u_{n}^{m}\rangle}{\langle u_{n}^{m},  u_{n}^{m}\rangle} - \lambda_{k}\right|\leq \tilde{C}_{1} \left[  \left(\frac{k}{n}\right)^{1/4} \left(\sqrt{\frac{\ln (n/k)}{k}} + 1\right) + \sqrt{\frac{k\ln \left(k/\delta\right)}{n}} \right] ,\\
			& \left\|P^{\perp} u_{n}^{m}\right\|_{H^{1}(\Omega)}^{2}  \leq \tilde{C}_{2} \left[  \left(\frac{k}{n}\right)^{1/4} \left(\sqrt{\frac{\ln (n/k)}{k}} + 1\right) + \sqrt{\frac{k\ln \left(k/\delta\right)}{n}} \right].
\end{align*}
\end{corollary}

Next, we prove a regularity result of the eigenfunctions in Barron type spaces, which justifies Assumption \ref{Assumption: exist u* in both mathfrakB^s for s geq 3 and U_k}. We firstly recall the spectral Barron space defined in~\cite{DRMlu2021priori, lu2021priori}. Consider the set of cosine functions
$$\left\{\Psi_{k}\right\}_{k \in \mathbb{N}_{0}^{d}}:=\left\{\prod_{i=1}^{d} \cos \left(\pi k_{i} x_{i}\right) \mid k_{i} \in \mathbb{N}_{0}\right\} .$$
Let $\{\check{w}(k)\}_{k \in \mathbb{N}_{0}^{d}}$  be the Fourier coefficients of a function  $w \in L^{1}(\Omega)$ against the basis $\left\{\Psi_{k}\right\}_{k \in \mathbb{N}_{0}^{d}}$. For $s \geq 0$, the cosine spectral Barron space $\mathfrak{C}^s(\Omega)$ is defined by
\begin{equation} \label{eq: definition of cosine spectral Barron space}
	\mathfrak{C}^{s}(\Omega):=\left\{w \in L^{1}(\Omega): \|w\|_{\mathfrak{C}^{s}(\Omega)} <\infty\right\},
\end{equation}
which is equipped with the spectral Barron norm
\[
\|w\|_{\mathfrak{C}^{s}(\Omega)}=\sum_{k \in \mathbb{N}_{0}^{d}}\left(1+\pi^{s}|k|_{1}^{s}\right)|\check{w}(k)|.
\]
\begin{theorem} \label{Thm: Regularity of eigenfunctions}
Assume that  $V \in \mathfrak{C}^{s}(\Omega)$  with  $s \geq 0$.  Then any eigenfunction of Problem (\ref{eq1}) lies in $ \mathfrak{B}^{s+2}(\Omega)$.
\end{theorem}

Theorem \ref{Thm: Regularity of eigenfunctions} shows that Assumption \ref{Assumption: exist u* in both mathfrakB^s for s geq 3 and U_k} is true if $V \in \mathfrak{C}^{1}(\Omega)$. 
To prove Theorem \ref{Thm: Regularity of eigenfunctions}, we show that the inverse of the Schrödinger operator $\mathcal{H}^{-1}: \mathfrak{B}^{s}(\Omega) \rightarrow \mathfrak{B}^{s+2}(\Omega)$ is bounded, and then prove regularity estimates for the eigenfunctions through a bootstrap argument. See \S \ref{Section: Solution theory in spectral Barron Spaces} for a complete proof. Our proof may be applied to more general situations such as other boundary conditions and less smoother assumptions on $V$.

In Theorem \ref{Main generalization theorem}, we require that the functions for problem solving have a lower bound $r \in (0, 1/2)$ of $L^{2}$-norm.
Such an assumption is reasonable. In practice, too small $L^{2}$ norm leads to numerical instability when calculating the Rayleigh quotient. From a theoretical perspective, the $L^{2}$ normalization $u/\|u\|_{L^{2}(\Omega)}$ of all functions $u$ in $\varphi \mathcal{F}_{\mathrm{SP}_{\tau}, m}$ is not uniformly $L^{\infty}$ bounded, which may lead to a blow up of statistical error bounds. It follows from our results that the statistical error bound depends on $r^{-1}$ linearly. In practice, we treat the minimization of $L_{k, n}$ in $\mathcal{F}_{> r}$ as solving an optimization problem in $\mathcal{F}$ with constraint $\|u\|_{L^{2}(\Omega)} > r$.
In numerical experiments, due to the scaling invariance of $L_{k, n}$, we have observed that the $L^{2}$-norm of the network function always increases oscilatorily as one minimizes $L_{k, n}$ using stochastic gradient descent methods. The gradient descent methods have an implicit regularization effect on the $L^{2}$-norm of the solution when minimizing a scaling invariant loss function. This allows us to obtain solutions that satisfy the constraint $\|u\|_{L^{2}(\Omega)}>r$ automatically and generalizes well without specifically adding normalization penalties to enforce the $L^{2}$-norm 
away from zero.
	
Adding a normalization penalty term is also a natural way to solve the constrained optimization problem and drive the $L^2$-norm of the solution away from zero. In \cite{EweinanYB2017DRM, Jixia2024DeepRM4Eigenvalue}, the authors add a normalization penalty term $\gamma\left(\mathcal{E}_{2}(u)-1\right)^{2}$ in the loss function to make numerical solution closer to the normalized one. We shall analyze the normalization penalty method in \S \ref{subsection: Extensions to the penalty method}.
\subsection{Extensions to the normalization penalty method} \label{subsection: Extensions to the penalty method}
The normalization penalty method in~\cite{EweinanYB2017DRM} employs the population loss
\begin{equation*}
	\begin{aligned}
		\mathscr{L}_{k}(u) & := L_{k}(u) + \gamma\left(\mathcal{E}_{2}(u)-1\right)^{2} 
		%=\frac{\mathcal{E}_{V}(u)+\mathcal{E}_{P}(u)}{\mathcal{E}_{2}(u)}+\gamma\left(\mathcal{E}_{2}(u)-1\right)^{2}
	\end{aligned}
\end{equation*}
and the corresponding empirical loss
\begin{equation*}
	\begin{aligned}
		\mathscr{L}_{k, n}(u) & := L_{k, n}(u)+\gamma\left(\mathcal{E}_{n, 2}(u)-1\right)^{2},
		%=\frac{\mathcal{E}_{n, V}(u)+\mathcal{E}_{n, P}(u)}{\mathcal{E}_{n, 2}(u)}+\gamma\left(\mathcal{E}_{n, 2}(u)-1\right)^{2},
	\end{aligned}
\end{equation*}
where $\gamma>0$ is a penalty parameter. Let $\mathscr{u}_{n}$ be a minimizer of $\mathscr{L}_{k, n}(u)$ within $\mathcal{F}$. Firstly, we show that when $\gamma$ is chosen properly large, $\|\mathscr{u}_{n}\|_{L^{2}(\Omega)} \geq 1/2$ with high probability. 
\begin{theorem}\label{theorem: solutions obtained by penalty method are away from 0}
Under Assumptions \ref{Assumption: V is bounded above and below} and \ref{Assumption: exist u* in both mathfrakB^s for s geq 3 and U_k},  
	%Denote $\bar{\mu}_{k} = \max_{1 \leq j \leq k-1} \|\psi_{j}\|_{L^{\infty}(\Omega)}$  and  $\varphi(x)$ is defined in (\ref{sin d-dim cutoff func}).
let $\gamma \geq 4 \lambda_{k}$ and $\mathscr{u}_{n} = \mathop{\arg\min}_{u \in \mathcal{F}} \mathscr{L}_{k, n}(u)$ % be a minimizer of the empirical loss  $\mathscr{L}_{k, n}(u)$  within  
where $\mathcal{F} = \varphi \mathcal{F}_{\mathrm{SP}_{\tau}, m}(B)$
with $B = \left(1+\frac{2 d}{\pi s}\right) \|u^{*}\|_{\mathfrak{B}^{s}(\Omega)}$  and  $\tau=9\sqrt{m}$. 
Given  $\delta \in\left(0, 1/4\right)$, assume that  $n$  and  $m$  are large enough so that $C \left(1 + V_{\max } + \beta/\gamma\right) B /\sqrt{m} \leq  1 $ and
\begin{subequations} \label{ineqs: conditions for theorem: solutions obtained by penalty method are away from 0}
\begin{align}
			&  \frac{C B}{d}\left(\frac{B}{d}+1\right) \sqrt{ \frac{d \left(1+\ln B\right) m +  \ln (1 / \delta)}{n}} \leq 1, \label{ineq: condition in theorem: solutions obtained by penalty method are away from 0, ensure R_1 < 1/16} \\
			& C \left(1 + V_{\max } + \beta/\gamma\right) B^{2}\sqrt{\frac{\ln (1 / \delta)}{n}} + \frac{C \beta \bar{\mu}_{k} B}{\gamma d}  \sqrt{\frac{k\ln (k / \delta)}{n}} \leq 1,   \label{ineq: condition in theorem: solutions obtained by penalty method are away from 0, ensure xi_5, R_2, R_4 small} %\\
			%& C \left(1 + V_{\max }\right) B /\sqrt{m} \leq  1 ,  \label{ineq: condition in theorem: solutions obtained by penalty method are away from 0, ensure R_3, R_5 small}
\end{align}
\end{subequations}
where  $C$ is an absolute constant. Then with probability at least  $1-4 \delta$, 
\[
\mathcal{E}_{2}(\mathscr{u}_{n}) \geq 1/4.
\]
\end{theorem}

\begin{remark}
The assumption $\gamma \geq 4 \lambda_{k}$  in Theorem \ref{theorem: solutions obtained by penalty method are away from 0} may be relaxed.
Proceeding along the same line of our proof, one can prove that $\|\mathscr{u}_{n}\|_{L^{2}(\Omega)} \geq r$ with high probability when $n, m$ are large enough, as long as $\gamma > \lambda_{k}/(1-r^{2})^{2}$.
\end{remark}

The conditions~\eqref{ineqs: conditions for theorem: solutions obtained by penalty method are away from 0} are weaker than~\eqref{ineqs: conditions for Main generalization theorem} to certain degree because the left hand sides of
(\ref{ineq: assumption in Main generalization theorem to ensure xi1+xi2<1/2}), (\ref{Simplifying condition for xi3 term}) are $O(\sqrt{m\ln m /n})$, $O(\sqrt{m/n})$ while the left hand sides of
(\ref{ineq: condition in theorem: solutions obtained by penalty method are away from 0, ensure R_1 < 1/16}), (\ref{ineq: condition in theorem: solutions obtained by penalty method are away from 0, ensure xi_5, R_2, R_4 small}) are $O(\sqrt{m/n})$, $O(1/\sqrt{n})$. 

As a direct application of our method, we obtain the generalization error estimate of the normalization penalty method.
 Denote by $\Upsilon(n,m,k,B,\bar{\mu}_{k},\beta,r,\delta)$ the error bound in the right hand side of~\eqref{ineq: error bound in Main generalization theorem, full parameter}.
\begin{corollary} \label{corollary: Main generalization theorem for the penalty method}
Under the same assumptions of Theorem \ref{theorem: solutions obtained by penalty method are away from 0}, for $\delta \in (0, 1/7)$, assume further that  $n$  and  $m$  are large enough so that (\ref{ineq: assumption in Main generalization theorem to ensure xi1+xi2<1/2}) and (\ref{Simplifying condition for xi3 term}) with $r = 0.49$ hold.
Then, there is an absolute constant $C$ such that with probability at least  $1-7 \delta$,
\begin{equation} \label{ineq: error bound in corollary: Main generalization theorem for the penalty method}
\begin{aligned}
			\ L_{k}(\mathscr{u}_{n})  -  \lambda_{k} 
			& \le\Upsilon(n,m,k,B,\bar{\mu}_{k},\beta,1,\delta) \\
			&\quad+  C\gamma \left[ \! \left(\frac{B^{2}}{d^{2}} \! + \! 1\right)  \! \sqrt{ \frac{d \left(1 \! + \! \ln  \! B\right) m \!  +  \!  \ln (1/\delta)}{n}}  \!  +  \! \frac{B^{2}}{m}\right] .
			%+  C\gamma B^{2}\left(\left(\frac{1}{d^{2}}+\frac{1}{B^{2}}\right) \sqrt{ \frac{d \left(1+\ln B\right) m +  \ln (1 / \delta)}{n}}  + \frac{1}{m}\right) .
			%+  C\gamma B^{2}\left(\left(d^{-2}+B^{-2}\right) \sqrt{ \frac{d \left(1+\ln B\right) m +  \ln (1 / \delta)}{n}}  + \frac{1}{m}\right) .
			%+ C \gamma B\left[  \left(\frac{B}{d}+1\right) \sqrt{ \frac{\left(1+\ln B\right) m}{d n}} + \frac{B}{d^{2}}  \sqrt{\frac{ \ln (1 / \delta)}{n}} + \frac{1}{\sqrt{m}}\right].
\end{aligned}
\end{equation}
\end{corollary}

The second term in (\ref{ineq: error bound in corollary: Main generalization theorem for the penalty method}) corresponds to the statistical and approximation error for the normalization penalty term.
For the normalization penalty method, there is a trade-off: a large $\gamma$ guarantees $\mathscr{u}_{n}$ away from $0$, but results in a larger generalization error.% and difficulties in training.
To prove Theorem \ref{theorem: solutions obtained by penalty method are away from 0}, we shall derive a new oracle inequality to bound $\left|\mathcal{E}_{2}(\mathscr{u}_{n})-1\right| $ with the aid of Hoeffding's inequality and generalization bound via the Rademacher complexity. We use Dudley’s theorem and bounds for the covering numbers obtained in \S \ref{section: Statistical error} to bound the Rademacher complexity.

The poof of Corollary \ref{corollary: Main generalization theorem for the penalty method} relies on the fact $\|\mathscr{u}_{n}\|_{L^{2}(\Omega)} \geq 1/2$ and so all the estimates used to prove Theorem \ref{Main generalization theorem} with $r = 0.49$ are fully applicable. 
See Appendix \ref{appendix section: About the penalty method} for the proof of Theorem \ref{theorem: solutions obtained by penalty method are away from 0} and Corollary \ref{corollary: Main generalization theorem for the penalty method}.
%We 
\begin{remark}[About the choice of hyperparameters] 
	%In our method, we require $\beta > \lambda_{k}-\lambda_{1}$. 
	By Proposition \ref{proposition: error between Rayleigh quotient and lambdak bounded by energy excess} and Proposition \ref{Generalized proposition 2.1}, $\beta - \lambda_{k} + \lambda_{1} > 0$ should not be excessively small. 
	For the normalization penalty method, $\gamma$ should be larger than constant times $\lambda_{k}$.
	It follows from (\ref{ineq: error bound in Main generalization theorem, full parameter}) and (\ref{ineq: error bound in corollary: Main generalization theorem for the penalty method}),
	a reasonable choice of $\beta$ and $\gamma$ are $\mathcal{O}(\lambda_k)$  with approriate constants.
	%By Theorem \ref{theorem: solutions obtained by penalty method are away from 0}, to guarantee that solutions obtained by the penalty method are away from zero, we require $\gamma \geq c (\lambda_{k} + \beta)$.
	%Considering the generalization error bound in Corollary \ref{corollary: Main generalization theorem for the penalty method}, we think taking $\beta = \Theta(\lambda_{k})$, $\gamma = \Theta(\lambda_{k}+\beta)$ is a good choice.
\end{remark}

%\subsubsection{\textbf{Analysis of error accumulation}}
%\textbf{About the cumulative error.}
\subsection{Analysis of the error accumulation}
In practice, since the exact eigenfunctions are unknown, we replace the exact eigenfunctions in the orthogonal penalty term with the approximate eigenfunctions obtained in the previous $k-1$ steps, which introduces the cumulative error. 
%In other words, the accuracy of the approximate eigenfunction obtained in the first $k-1$ steps  affect the accuracy of the solution in the $k$-th step. %which is analyzed below.
%Theorem \ref{Main generalization theorem when cumulative error is added} delineates the impact of the accuracy of the approximated eigenfunctions acquired in the initial $k - 1$ steps on the precision of the solution in the $k$-th step.
Theorem \ref{Main generalization theorem when cumulative error is added} below delineates the influence of using approximate eigenfunctions %obtained in the previous steps
on the generalization bound of the $k$-th step, and Proposition \ref{proposition: Quadratic growth of cumulative error} gives the accumulation rate of generalization errors.
Denote by $\mathfrak{u}_{\theta j}$ the $j$-th  approximate eigenfunction parameterized by the neural network.
%In the actual calculation, 
We use the $L^{2}$ normalization of $\mathfrak{u}_{\theta j}$ as a natural approximation of $\psi_{j}$ and denote $\bar{\nu}_{k} = \max_{1 \leq j \leq k-1} \|\mathfrak{u}_{\theta j}\|_{L^{\infty}(\Omega)}/\|\mathfrak{u}_{\theta j}\|_{L^{2}(\Omega)}$.
Consider the loss function %used in practice
\begin{equation*} 
	\begin{aligned}
		\widetilde{L}_{k}(u) = \frac{\mathcal{E}_{V}(u)}{\mathcal{E}_{2}(u)} + \beta_{k} \sum_{j=1}^{k-1} \frac{\mathscr{P}_{j}(u) }{\mathcal{E}_{2}(u) \mathcal{E}_{2}(\mathfrak{u}_{\theta j})}
	\end{aligned}
\end{equation*}
and the corresponding empirical loss
\begin{equation*} 
	\begin{aligned}
		\widetilde{L}_{k, n}(u) = \frac{\mathcal{E}_{n, V}(u)}{\mathcal{E}_{n, 2}(u)} + \beta_{k} \sum_{j=1}^{k-1} \frac{\mathscr{P}_{n, j}(u)}{\mathcal{E}_{n, 2}(u) \mathcal{E}_{n, 2}\left(\mathfrak{u}_{\theta j}\right)},
	\end{aligned}
\end{equation*}
where
$\mathscr{P}_{j}(u):=\left\langle u, \mathfrak{u}_{\theta j}\right\rangle^{2}$ and  $\mathscr{P}_{n, j}(u) := \big(n^{-1} \sum_{i=1}^{n} u(X_{i}) \mathfrak{u}_{\theta j}(X_{i})\big)^{2}.$ 
%Denote by $\mathfrak{u}_{\theta k}$ a minimizer of $\widetilde{L}_{k, n}(u)$ within $\mathcal{F}_{>r}$, i.e.,  $$\mathfrak{u}_{\theta k} = \mathop{\arg\min}_{u \in \mathcal{F}_{>r}} \widetilde{L}_{k, n}(u).$$ For error analysis,
%Here, we take $\psi_{j}$ to be the normalization of the orthogonal projection of $\mathfrak{u}_{\theta j}$ to subspace $U_{j}$, i.e.,  $$\psi_{j} = \frac{P_{j} \mathfrak{u}_{\theta j}}{\|P_{j} \mathfrak{u}_{\theta j}\|_{L^{2}(\Omega)}}$$
%Denote $\bar{\nu}_{k} = \max_{1 \leq j \leq k-1} \|\mathfrak{u}_{\theta j}\|_{L^{\infty}(\Omega)}/\|\mathfrak{u}_{\theta j}\|_{L^{2}(\Omega)}$.
\begin{theorem}  \label{Main generalization theorem when cumulative error is added}
	Under Assumptions \ref{Assumption: V is bounded above and below} and \ref{Assumption: exist u* in both mathfrakB^s for s geq 3 and U_k}, 
	let $\mathcal{F} = \varphi \mathcal{F}_{\mathrm{SP}_{\tau}, m}(B)$
	with $B \! = \! \left(1 \! + \! \frac{2 d}{\pi s}\right) \! \|u^{*}\|_{\mathfrak{B}^{s}(\Omega)}$  and   $\tau=9\sqrt{m}$.  
	Let $r \in (0, 1/2)$, $\|\mathfrak{u}_{\theta j}\|_{L^{2}(\Omega)} \geq r$ for all $j$  and $\mathfrak{u}_{\theta k} = \mathop{\arg\min}_{u \in \mathcal{F}_{> r}} \widetilde{L}_{k,n}(u)$. 
	%Denote $\bar{\nu}_{k} = \max_{1 \leq j \leq k-1} \|\mathfrak{u}_{\theta j}\|_{L^{\infty}(\Omega)}/\|\mathfrak{u}_{\theta j}\|_{L^{2}(\Omega)}$.
	%Let $\mathfrak{u}_{\theta k}$  be a minimizer of %the empirical loss  
	%$\widetilde{L}_{k,n}$  within $\mathcal{F}_{> r}$. 
	Given  $\delta \in\left(0, 1/3\right)$, assume that  $n$  and  $m$  are large enough so that $64B/\sqrt{m} \leq 1/2$, (\ref{ineq: assumption in Main generalization theorem to ensure xi1+xi2<1/2}) and (\ref{Simplifying condition for xi3 term}) with $\bar{\mu}_{k}$ replaced by $\bar{\nu}_{k}$ hold true. Then, with probability at least  $1-3 \delta$,
	\begin{equation} \label{ineq: error bound in Main generalization theorem when cumulative error is added}
		\begin{aligned}
			L_{k}\left(\mathfrak{u}_{\theta k}\right)-\lambda_{k}  
			&  \leqslant \Upsilon(n,m,k,B,\bar{\nu}_{k},\beta_{k},r,\delta) + 8 \beta_{k} \sum_{j=1}^{k-1} \sqrt{\frac{L_{j}(\mathfrak{u}_{\theta j})-\lambda_{j}}{\min \left\{\beta_{j}+\lambda_{1}-\lambda_{j}, \lambda_{j^{\prime}}-\lambda_{j}\right\}} }  ,
		\end{aligned}
	\end{equation}
	where $C$ in $\Upsilon$ is an absolute constant.
	%are some absolute constants that may be different from line to line.
\end{theorem} 

The first term on the right side of (\ref{ineq: error bound in Main generalization theorem when cumulative error is added})  is an extra generalization error at step $k$ and the second term represents the cumulative error. 
$\bar{\nu}_{k}$ plays a similar role as $\bar{\mu}_{k}$ in Theorem \ref{Main generalization theorem}.
Proposition \ref{proposition: Quadratic growth of cumulative error} shows that the generalization error grows quadratically as the order of the eigenfunction. % increases. gives a characterization
\begin{proposition} \label{proposition: Quadratic growth of cumulative error}
Assume that for all $k \geq 1$, 
	\begin{equation*} \label{Error accumulation relation in proposition: Quadratic growth of cumulative error}
		\begin{aligned}
			L_{k}\left(\mathfrak{u}_{\theta k}\right)-\lambda_{k}  
			& \leqslant \Delta_{k} + 8 \beta_{k} \sum_{j=1}^{k-1} \sqrt{\frac{L_{j}(\mathfrak{u}_{\theta j})-\lambda_{j}}{\min \left\{\beta_{j}+\lambda_{1}-\lambda_{j}, \lambda_{j^{\prime}}-\lambda_{j}\right\}} }.
		\end{aligned}
	\end{equation*}
	Let $\tau_{k} = \max _{1 \leq j \leq k} \Delta_{j}/\beta_{j}$, $\rho_{0} = 0$ and $\rho_{k} = \max _{1 \leq j \leq k} 4 \sqrt{\beta_{j} / \min \left\{\beta_{j}+\lambda_{1}-\lambda_{j}, \lambda_{j^{\prime}}-\lambda_{j}\right\}} .$ 
	Then, 
	\begin{equation} \label{eq: conclusion of prop: Quadratic growth of cumulative error, recursive relation}
		L_{k}\left(\mathfrak{u}_{\theta k}\right)-\lambda_{k} \leq \beta_{k} \Big((k-1) \rho_{k-1} + \sqrt{\tau_{k}} \Big)^{2}.
	\end{equation}
\end{proposition}

When $k = 1$, the bound (\ref{eq: conclusion of prop: Quadratic growth of cumulative error, recursive relation}) changes to $L_{k}\left(\mathfrak{u}_{\theta k}\right)-\lambda_{k}  \leqslant \Delta_{k}$ and there is no cumulative error.
%For simplicity, denote $  \Upsilon(n,m,k,B,\bar{\nu}_{k},\lambda_{k},\beta_{k},r,\delta)$ by $\Delta_{k}$ for all $k$.
To prove Theorem \ref{Main generalization theorem when cumulative error is added}, we shall derive a uniform bound for $\left|\widetilde{L}_{k}(u) - L_{k}(u)\right|$ and a new oracle inequality to handle the different penalty term. The Proposition \ref{proposition: Quadratic growth of cumulative error} is proved by induction. By a similar argument, one can prove that the quadratic growth rate of error accumulative with respect to $k$ is sharp.
We postpone the poof of Theorem \ref{Main generalization theorem when cumulative error is added} and Proposition \ref{proposition: Quadratic growth of cumulative error} to Appendix \ref{appendix section: About the cumulative error}.
%(\textcolor{red}{Sharpness of Proposition} \ref{proposition: Quadratic growth of cumulative error}).
Similar to Corollary \ref{Corollary: the generalization error in terms of eigenvalues and $H^{1}$-norm of eigenfunctions}, the error bounds in Corollary \ref{corollary: Main generalization theorem for the penalty method}, Theorem \ref{Main generalization theorem when cumulative error is added} and Proposition \ref{proposition: Quadratic growth of cumulative error} may also be translated into error bounds with respect to the eigenvalues and the $H^{1}$-norm of eigenfunctions. We will not delve into details for simplicity.

The remainder of the paper is organized as follows.
We firstly show the numerical results in \S \ref{section: Numerical results}. 
In \S \ref{section: Oracle inequality for the generalization error}, we derive oracle inequality for the generalization error. 
We prove the approximation results in \S \ref{section: Approximation theory for sine spectral Barron functions}, the statistical error estimate in \S \ref{section: Statistical error}, Theorem \ref{Main generalization theorem} in \S \ref{Proof of the main generalization theorem} and regularity estimate in \S \ref{Section: Solution theory in spectral Barron Spaces}.
Stability estimates and the proofs for certain technical results are postponed to Appendix \ref{section: Stability estimate of the k-th eigenfunction}-\ref{Appendix Section: Solution theory in spectral Barron Spaces}. We prove the results about the penalty method in Appendix \ref{appendix section: About the penalty method}, analyze of the accumulative error in Appendix \ref{appendix section: About the cumulative error}, and prove the a priori bounds for the eigenmodes in Appendix~\ref{appendix section: Simple properties of eigenvalues and eigenfunctions}.

\section{Numerical results}
\label{section: Numerical results}
\subsection{Deep Learning Scheme}
We firstly show the neural network architecture designed to solve the eigenvalue problem \eqref{eq1}; see Figure \ref{architecture}.
The left part of our model consists of a fully connected neural network with several hidden layers, each of equal width. We denote the number of hidden layers as $l$ and the width of each layer as $m$. In our experiments, we take $l=3$ and $m=40$. Thus, each network has approximately $3,500$ parameters. The activation function is $\sigma = \tanh$, and we vectorize $\sigma(x)$ as $\tilde{\sigma}(x)$, i.e., $ \tilde{\sigma}(x) = (\sigma(x_1), \sigma(x_2), \dots, \sigma(x_m)).$
\begin{figure}%[H]
\begin{center}		
	\includegraphics[width=0.6\columnwidth]{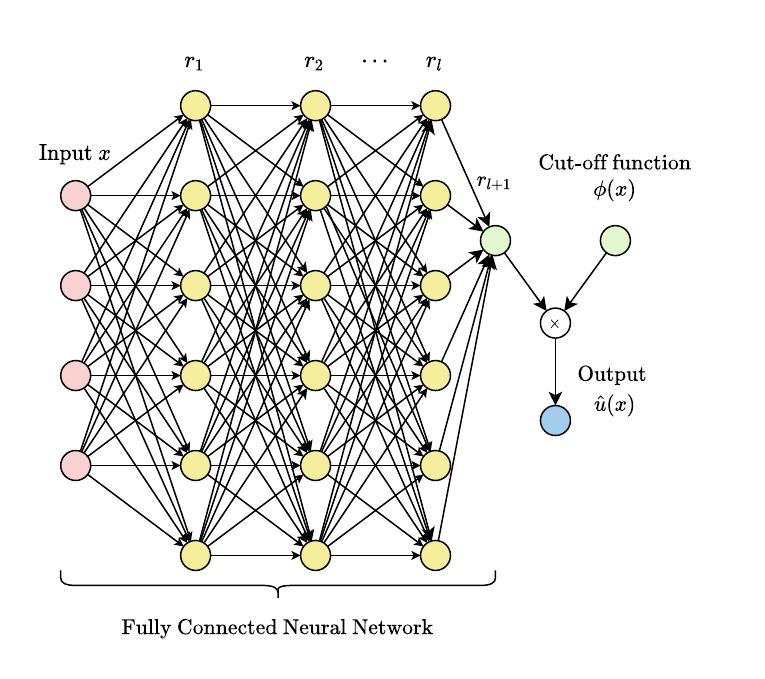}
	\caption{The neural network architecuture combined with cut-off funtcions.}
	\label{architecture}
\end{center}
\end{figure}
%\yixiao{If the figure does not appear, please check whether the "Normal" compile mode is selected instead of "Fast" compile mode in the triangle to the right of the "Recompile" button}

The input layer maps the coordinates of the sampling points, from $\mathbbR^d$ to $\mathbbR^m$. The output of the first layer is given by
$r_1 = \tilde{\sigma}(W_1 x + b_1)$,
where $W_1 \in \mathbbR^{m \times d}$ and $b_1 \in \mathbbR^{m}$. 
Subsequent hidden layers also contain similar transformations, mapping values from $\mathbbR^m$ to $\mathbbR^m$, and the output of the $i$-th layer is represented as
\begin{equation*}
	r_i = \tilde{\sigma}(W_ir_{i-1} + b_i), \quad 2 \leq i \leq l,
\end{equation*}
where $W_i \in \mathbbR^{m \times m}$ and $b_i \in \mathbbR^m$.
The final layer on the left part is an output node, yielding a value given by
\(
	r_{l+1} = W_{l+1}r_l + b_{l+1},
\)
where $W_{l+1} \in \mathbbR^{m}$, $b_{l+1} \in \mathbbR$.
The final output function $u(x)$ is obtained by multiplying the output of the left part network with a cut-off function
\[
	u(x) = r_{l+1}(x) \phi(x).
\]
As mentioned earlier, the cut-off function plays a key role in the architecture by ensuring the output function satisfies the Dirichlet boundary condition. Similar architectures appear in \cite{architecture1_khoo2019solving, architecture2_guo2024deep, architecture3_holliday2023solving}.

The complete set of parameters in our architecture is defined as
\begin{equation*}
	\theta{:}= \{W_1, \dots, W_{l+1}, b_1, \dots, b_{l+1}\}.
\end{equation*}
In each epoch, the loss function is computed using \eqref{empirical loss}, and parameters are updated using the adaptive moment estimation (ADAM) optimizer. This process is iterated over multiple epochs until the loss function decreases to a sufficiently small value, indicating approximate eigenmodes have been achieved. Initially, the learning rate is set to $5\times10^{-3}$, and $1,000$ points are used for calculating the empirical loss. To balance accuracy and efficiency, the learning rate is reduced to a quarter of its current value, and the number of sampling points is doubled every $20,000$ epochs. Each eigenmode is derived through parameter updates over $120,000$ epochs. Besides, when solving for the $k$-th eigenvalue $(k \geq 2)$, we take $\beta = 4 \lambda_{k-1}$. This choice is sufficiently large to discover the subsequent correct eigenvalue in these cases.
\subsection{Regular potential}
In the first test, we use the potential function
\begin{equation}
	V(x_1, \cdots, x_d) = \frac{1}{d} \sum_{i=1}^d \cos(\pi x_i+\pi),
\label{potential_1}
\end{equation}
and test our method in the $d$-dimensional cube $\Omega = (-1,1)^d$.
Since this potential function is essentially decoupled, we compute the reference eigenmodes using the spectral method, as described in \cite{eigenvalue_han2020solving}.

We test four different cut-off functions,
\begin{equation*}
	\phi_a = \prod_{i=1}^{d} (1-x_i^2), \quad 
	\phi_b = \bigg(\sum_{i=1}^{d} \frac{1}{1-x_i^2}\bigg)^{-1}, \quad
	\phi_c = \prod_{i=1}^{d} \cos(\frac{\pi}{2} x_i), \quad 
	\phi_d = \bigg(\sum_{i=1}^{d} \frac{1}{\cos(\frac{\pi}{2} x_i)}\bigg)^{-1},
\end{equation*}
and the results are summarized in Table \ref{table1-1}. All results indicate that our method provides satisfactory solutions. Each of the four cut-off functions produced solutions with errors less than $6\times 10^{-4}$ for the first eigenvalue and less than $5\times 10^{-3}$ for at least the first $30$ eigenvalues. Notably, the performance varied slightly across different cut-off functions, with specific functions, such as $\phi_c$, achieving errors of less than $1 \times 10^{-3}$ for the first $30$ eigenvalues.

We also compare our results with the standard penalty method, which imposes the boundary conditions in a soft manner. The loss function we used is
\begin{equation}
	\frac{ \langle u, \mathcal{H} u\rangle_{H^{1} \times H^{-1}} }{\|u\|_{L^{2}(\Omega)}^{2}}+\beta \sum_{j=1}^{k-1} \frac{\langle u, \psi_{j}\rangle^{2} }{\|u\|_{L^{2}(\Omega)}^{2}} + \gamma \frac{\|u\|_{L^{2}(\partial \Omega)}^{2}}{\|u\|_{L^{2}(\Omega)}^{2}}.
	\label{loss_penalty}
\end{equation}
It includes an additional boundary penalty term compared to \eqref{the first definition of the loss function Lk(u)}.
As mentioned before, a small $\gamma$ will introduce considerable model error. Conversely, using a large $\gamma$ can enhance calculation accuracy, but training will become more difficult and inefficient due to the rough loss landscape.
To ensure a fair comparison, we use the boundary penalty method with different hyperparameters $\gamma$ to show its optimal performance.

For comparison, we let the cut-off function be the identity function in the network architecture and keep the rest of the configurations the same as in the previous test. Table~\ref{table1-2} demonstrates that the boundary penalty method performs much worse than our method, with errors one or more orders of magnitude greater than ours. Even with the optimal choice, $\gamma = 500$, the accuracy of numerical solutions deteriorate rapidly, and the error for the tenth eigenvalue exceeds $1 \times 10^{-2}$. While taking other hyperparameter $\gamma$, the calculation error is greater than $1 \times 10^{-2}$ even for the first eigenvalue.
%\unsure{Furthermore, we test our method with a regularization penalty. The results suggest that the regularization effect of the stochastic gradient descent method is adequate, and adding a normalization penalty term may deteriorate the accuracy of the numerical solution.}
%
\begin{table}%[H]
	\footnotesize
	\caption{Estimates of the eigenvalues with potential function \eqref{potential_1} and $d=5$, using different cut-off functions.}
	\begin{tabular}{ccccccccc}
		\hline
		&            & k=1      & k=2      & k=3      & k=5      & k=10     & k=15     & k=30     \\ \hline
		& Exact      & 11.8345  & 19.3369  & 19.3369  & 19.3369  & 26.8392  & 26.8392  & 34.3416  \\ \hline
		\multirow{2}{*}{$\phi_a$} & Result       & 11.8379  & 19.3338  & 19.3394  & 19.3512  & 26.8583  & 26.8621  & 34.3886  \\
		& Rel. error & 2.87$\times 10^{-4}$ & 1.60$\times 10^{-4}$ & 1.29$\times 10^{-4}$ & 7.40$\times 10^{-4}$ & 7.12$\times 10^{-4}$ & 8.53$\times 10^{-4}$ & 1.37$\times 10^{-3}$ \\ \hline
		\multirow{2}{*}{$\phi_b$} & Result       & 11.8396  & 19.3588  & 19.3700  & 19.3784  & 26.9059  & 26.9130  & 34.5095  \\
		& Rel. error & 4.31$\times 10^{-4}$ & 1.13$\times 10^{-3}$ & 1.71$\times 10^{-3}$ & 2.15$\times 10^{-3}$ & 2.49$\times 10^{-3}$ & 2.75$\times 10^{-3}$ & 4.89$\times 10^{-3}$ \\ \hline
		\multirow{2}{*}{$\phi_c$} & Result       & 11.8343  & 19.3358  & 19.3382  & 19.3428  & 26.8474  & 26.8553  & 34.3702  \\
		& Rel. error & 1.69$\times 10^{-5}$ & 5.69$\times 10^{-5}$ & 6.72$\times 10^{-5}$ & 3.05$\times 10^{-4}$ & 3.06$\times 10^{-4}$ & 6.00$\times 10^{-4}$ & 8.33$\times 10^{-4}$ \\ \hline
		\multirow{2}{*}{$\phi_d$} & Result       & 11.8413  & 19.3533  & 19.3692  & 19.3714  & 26.8972  & 26.9095  & 34.4887  \\
		& Rel. error & 5.75$\times 10^{-4}$ & 8.48$\times 10^{-4}$ & 1.67$\times 10^{-3}$ & 1.78$\times 10^{-3}$ & 2.16$\times 10^{-3}$ & 2.62$\times 10^{-3}$ & 4.28$\times 10^{-3}$ \\ \hline
	\end{tabular}
	\label{table1-1}
\end{table}
\begin{table}%[H]
\footnotesize
\caption{Estimates of the eigenvalues with potential function \eqref{potential_1} and $d=5$, using boundary penalty method with different parameters $\gamma$.}
\begin{tabular}{ccccccccc}
\hline
&            & k=1      & k=2      & k=3      & k=5      & k=10     & k=15     & k=30     \\ \hline
$\gamma$               & Exact      & 11.8345  & 19.3369  & 19.3369  & 19.3369  & 26.8392  & 26.8392  & 34.3416  \\ \hline
\multirow{2}{*}{100}   & Result       & 11.3854  & 18.6194  & 18.6198  & 18.6323  & 25.9356  & 25.9711  & 33.3787  \\
& Rel. error & 3.79$\times 10^{-2}$ & 3.71$\times 10^{-2}$ & 3.71$\times 10^{-2}$ & 3.64$\times 10^{-2}$ & 3.37$\times 10^{-2}$ & 3.23$\times 10^{-2}$ & 2.80$\times 10^{-2}$ \\ \hline
\multirow{2}{*}{500}   & Result       & 11.8023  & 19.3761  & 19.3781  & 19.4148  & 27.2571  & 27.3496  & 35.4504  \\
& Rel. error & 2.72$\times 10^{-3}$ & 2.03$\times 10^{-3}$ & 2.13$\times 10^{-3}$ & 4.03$\times 10^{-3}$ & 1.56$\times 10^{-2}$ & 1.90$\times 10^{-2}$ & 3.23$\times 10^{-2}$ \\ \hline
\multirow{2}{*}{2000}  & Result       & 11.9934  & 19.8330  & 19.9228  & 20.0627  & 28.3052  & 28.6824  & 38.7467  \\
& Rel. error & 1.34$\times 10^{-2}$ & 2.57$\times 10^{-2}$ & 3.03$\times 10^{-2}$ & 3.75$\times 10^{-2}$ & 5.46$\times 10^{-2}$ & 6.87$\times 10^{-2}$ & 1.28$\times 10^{-1}$ \\ \hline
\multirow{2}{*}{10000} & Result       & 12.4805  & 21.0279  & 21.2185  & 21.8146  & 30.5426  & 33.0487  & 45.7008  \\
& Rel. error & 5.46$\times 10^{-2}$ & 8.74$\times 10^{-2}$ & 9.73$\times 10^{-2}$ & 1.28$\times 10^{-1}$ & 1.38$\times 10^{-1}$ & 2.31$\times 10^{-1}$ & 3.31$\times 10^{-1}$ \\ \hline
	\end{tabular}
\label{table1-2}
\end{table}

Next, we set $d = 10$ and evaluate our method in a higher dimensional scenario. In this case, the smallest eigenvalue is unique, and the second to the eleventh eigenvalues are equal, followed by $45$ equal eigenvalues.
Therefore, we choose to only calculate the first $15$ eigenvalues. 
The cut-off functions $\phi_a$ and $\phi_c$ are employed, while also using the boundary penalty method with proper hyperparameter $\gamma$.
As shown in Table~\ref{table2}, our method demonstrates superior performance compared to the boundary penalty method. 
Although the performance of our method with different cut-off functions varies slightly, all calculations yield errors less than $1 \times 10^{-2}$ for the first $15$ eigenvalues, which is significantly less than that of the boundary penalty method. 
Notably, the cut-off function $\phi_c$ seems to be most suitable for $d=10$.
% using cut-off function $\phi_c$, we achieve errors less than $1 \times 10^{-3}$ for the first $3$ eigenvalues and errors less than $6 \times 10^{-3}$ for all 15 eigenvalues in such a high dimensional problem.
%
\begin{table}[H]
	\footnotesize
\caption{Estimates of the eigenvalues with potential function \eqref{potential_1} and $d=10$.} 
\begin{tabular}{ccccccccc}
\hline
&            & k=1      & k=2      & k=3      & k=5      & k=11     & k=12     & k=15     \\ \hline
& Exact      & 24.1728  & 31.6250   & 31.6250   & 31.6250   & 31.6250   & 39.0772  & 39.0772  \\ \hline
cut-off  & Result     & 24.2677  & 31.7994  & 31.8178  & 31.8693  & 31.9868  & 39.4044  & 39.4476  \\
($\phi_a$)& Rel. error & 3.93$\times 10^{-3}$ & 5.51$\times 10^{-3}$ & 6.10$\times 10^{-3}$ & 7.72$\times 10^{-3}$ & 1.14$\times 10^{-2}$ & 8.37$\times 10^{-3}$ & 9.48$\times 10^{-3}$ \\ \hline
cut-off  & Result     & 24.1895  & 31.6329  & 31.6541  & 31.6633  & 31.711   & 39.2598  & 39.2898  \\
($\phi_c$)& Rel. error & 6.91$\times 10^{-4}$ & 2.50$\times 10^{-4}$ & 9.20$\times 10^{-4}$ & 1.21$\times 10^{-3}$ & 2.72$\times 10^{-3}$ & 4.67$\times 10^{-3}$ & 5.44$\times 10^{-3}$ \\ \hline
penalty  & Result     & 20.6123  & 26.4801  & 26.5147  & 26.6332  & 26.7601  & 32.7747  & 32.9277  \\
($\gamma = 20$)& Rel. error & 1.47$\times 10^{-1}$ & 1.63$\times 10^{-1}$ & 1.62$\times 10^{-1}$ & 1.58$\times 10^{-1}$ & 1.54$\times 10^{-1}$ & 1.61$\times 10^{-1}$ & 1.87$\times 10^{-1}$ \\ \hline
penalty & Result     & 25.7217  & 34.8861  & 34.9384  & 35.2846  & 37.4258  & 44.3579  & 45.8481  \\
($\gamma = 100$)& Rel. error & 6.41$\times 10^{-2}$ & 1.03$\times 10^{-1}$ & 1.05$\times 10^{-1}$ & 1.16$\times 10^{-1}$ & 1.83$\times 10^{-1}$ & 1.35$\times 10^{-1}$ & 1.73$\times 10^{-1}$ \\ \hline
penalty & Result     & 31.7256  & 43.7200  & 43.9436  & 44.4301  & 49.7948  & 51.1777  & 54.2697  \\
($\gamma = 500$)& Rel. error & 3.12$\times 10^{-1}$ & 3.82$\times 10^{-1}$ & 3.90$\times 10^{-1}$ & 4.05$\times 10^{-1}$ & 5.75$\times 10^{-1}$ & 3.10$\times 10^{-1}$ & 3.89$\times 10^{-1}$ \\ \hline
\end{tabular}
\label{table2}
\end{table}
\subsection{Inverse square potential}
In the last test, we use the inverse square potential
\begin{equation}
V(x,y,z) = \frac{c^2}{x^2+y^2+z^2},
\label{potential_2}
\end{equation} and solve the Schrödinger eigenvalue problem with $d=3$. We compare our result with the spectral method based on separation of variable \cite{eigenvalue_sui2020efficient} to demonstrate the effectiveness of our approach.

We let the domain be a unit ball, $c=1/3$ and take the cut-off function
\begin{equation*}
	\phi(x,y,z) = 1-(x^2+y^2+z^2).
\end{equation*}
As presented in Table \ref{table3-1}, the relative difference between these two methods for the first five eigenvalues is less than $4 \times 10^{-4}$. For the tenth eigenvalue, the highest eigenvalue provided in \cite{eigenvalue_sui2020efficient}, the relative difference remains around $1 \times 10^{-3}$.

Additionally, we test our method in a 3-dimensional ring, i.e., $\Omega = \{(x,y,z) \in \mathbbR^3: \frac{1}{2} \leq \sqrt{x^2+y^2+z^2} \leq 1\}$ and take $c=1/2$. The cut-off function we used here is
\begin{equation}
\phi(x,y,z) = \bigg(1-(x^2+y^2+z^2)\bigg) \times \bigg((x^2+y^2+z^2) - \frac{1}{4}\bigg).
\end{equation}
Table~\ref{table3-2} demonstrates the effectiveness of the method, which keeps the differences of the first nine eigenvalue less than $2 \times 10^{-3}$.
\begin{table}%[H]
\caption{Estimates of the eigenvalues with potential function \eqref{potential_2} in a unit ball.}
\begin{tabular}{cccccc}
	\hline
	& k=1      & k=2      & k=3      & k=5      & k=10     \\ \hline
	Our        & 10.7873  & 20.6167  & 20.6184  & 33.5391  & 41.4362  \\
	\cite{eigenvalue_sui2020efficient}       & 10.7836  & 20.6206  & 20.6206  & 33.5352  & 41.3859  \\
	Rel. diff. & 3.43$\times 10^{-4}$ & 1.89$\times 10^{-4}$ & 1.07$\times 10^{-4}$ & 1.16$\times 10^{-4}$ & 1.22$\times 10^{-3}$ \\ \hline
\end{tabular}
\label{table3-1}
\end{table}
\begin{table}%[H]
\caption{Estimates of the eigenvalues with potential function \eqref{potential_2} in a $3-$dimensional ring.}
\begin{tabular}{cccccc}
\hline
& k=1      & k=2      & k=3      & k=5      & k=9      \\ \hline
Our        & 40.0149  & 43.7195  & 43.7281  & 51.1062  & 51.1355  \\
\cite{eigenvalue_sui2020efficient}        & 39.9433  & 43.6545  & 43.6545  & 51.0341  & 51.0341  \\
Rel. diff. & 1.79$\times 10^{-3}$ & 1.49$\times 10^{-3}$ & 1.69$\times 10^{-3}$ & 1.41$\times 10^{-3}$ & 1.99$\times 10^{-3}$ \\ \hline
\end{tabular}
\label{table3-2}
\end{table}

\section{Oracle inequality for the generalization error}  \label{section: Oracle inequality for the generalization error}
In this part, we introduce an oracle inequality for the empirical loss.
As a preparation, we firstly introduce concentration inequalities for ratio-type suprema of empirical processes. 

\subsection{Concentration inequalities for normalized empirical processes} \label{Concentration inequalities for normalized empirical processes}
Let  $\mathscr{F}$  be a class of real valued measurable functions\footnote{In order to avoid measurability problems, we shall assume that the supremum over the class $\mathscr{F}$ or over any of the subclasses we consider is in fact a countable supremum. In this case we say that the class $\mathscr{F}$ is measurable.} taking values in  $[0, 1]$. 
Let  $X$, $X_{1}$, $X_{2}$, $...$ be $i.i.d.$ random variables with distribution  $\mathcal{P}$. 
We denote by  $\mathcal{P}_{n} := n^{-1} \sum_{i=1}^{n} \delta_{X_{i}} $  the empirical measure based on the sample  $\left(X_{1}, \ldots, X_{n}\right)$. 
Let $\mathcal{P} f = \mathbf{E}f(X)$ and $\operatorname{Var}_{\mathcal{P}}(f) = \mathcal{P} f^{2}-(\mathcal{P} f)^{2}$.
Suppose that  $\sigma_{\mathcal{P}}(f)$  is defined such that %in such a way that
$$\operatorname{Var}_{\mathcal{P}}(f) \leq \sigma_{\mathcal{P}}^{2}(f) \leq 1, \quad  f \in \mathscr{F}.$$
In particular,  $\sigma_{\mathcal{P}}(f)$  may be the standard deviation itself or equal to  $\sqrt{\mathcal{P} f}$ because $f$  takes values in  $[0, 1]$.
Here, we present concentration inequalities for the supremum of the normalized empirical process
$$\sup _{f \in \mathscr{F}, \sigma_{\mathcal{P}}(f)>r } \frac{\left|\mathcal{P}_{n} f - \mathcal{P} f\right|}{\sigma_{\mathcal{P}}^{2}(f)}$$
for some properly chosen cutoff $r\in (0, 1)$. Define the random variable
$$\left\|\mathcal{P}_{n}-\mathcal{P}\right\|_{\mathscr{F}} := \sup _{f \in \mathscr{F}} \left|\mathcal{P}_{n} f-\mathcal{P} f\right| =  \sup _{f \in \mathscr{F}} \left|\frac{1}{n} \sum_{i=1}^{n} f\left(X_{i}\right)-\mathbf{E}f(X)\right|,$$
which measures the absolute deviation between the sample average  $\mathcal{P}_{n} f$  and the population average  $\mathcal{P}f$, uniformly over the class  $\mathscr{F}$.  For  $0<r<s$,  we define
$$\mathscr{F}(r):=\left\{f \in \mathscr{F}: \sigma_{\mathcal{P}}(f) \leq r\right\}
\quad \text{and} \quad 
\mathscr{F}(r, s]:=\mathscr{F}(s) \backslash \mathscr{F}(r) .$$ For  $q > 1$  and  $r < s \leq r q^{l}$  with $l \in \mathbb{N}$, let $\rho_j:= r q^{j}$ and define\footnote{When there is no ambiguity, we omit the superscript $\mathscr{F}$ of $K$.}
\begin{equation} \label{the quantity En,q(r, s]}
	\begin{aligned} 
		K^{\mathscr{F}}_{n, q}(r, s] :=
		%\sup _{\rho \in(r, s]} \frac{\chi_{n, q}(\rho)}{\rho^{2}} = 
		\max_{1\leq j\leq l} \frac{\mathbf{E}\left\|\mathcal{P}_{n}-\mathcal{P}\right\|_{\mathscr{F}\left(\rho_{j-1}, \rho_{j}\right]}}{\rho_{j-1}^{2}}.
	\end{aligned}
\end{equation}

We recall a concentration inequality proved in \citep[Lemma 2]{RatioLimitTheorems4EP} when $\sigma_{\mathcal{P}}(f) = \sqrt{\mathcal{P} f}$.
\begin{lemma}{\citep[Lemma 2]{RatioLimitTheorems4EP}}
	For  $t>0$,
	\begin{equation*} \label{cite RatioLimitTheorems4EP Lemma 2}
		\begin{aligned} 
			\mathbf{P}\left\{\sup _{f \in \mathscr{F}(r, s]}\left|\frac{\mathcal{P}_{n} f}{\mathcal{P} f} - 1\right| \geq K_{n, q}(r, s]+\sqrt{\frac{2t}{n r^{2}}\left(q^{2}+2 K_{n, q}(r, s]\right)}+\frac{t}{3 n r^{2}}\right\} 
			\leq \frac{q^{2}}{q^{2}-1} \frac{q}{t} e^{-t / q}.
		\end{aligned}
	\end{equation*}
\end{lemma}
Proceeding along the same line that leads to \citep[Lemma 2]{RatioLimitTheorems4EP}, we may extend the above result to the more general $\sigma_{\mathcal{P}}(f)$. The proof is quite straightforward, and we omit the details and leave it to interested readers.
\begin{lemma} \label{Ratio Limit Theorems for EP, cited}
	For  $t>0$,
	\begin{equation*} \label{}
		\begin{aligned} 
			\mathbf{P}\left\{\sup _{f \in \mathscr{F}(r, s]}\frac{\left|\mathcal{P}_{n} f - \mathcal{P} f\right|}{\sigma_{\mathcal{P}}^{2}(f)} \geq K_{n, q}(r, s]+\sqrt{\frac{2t}{n r^{2}}\left(q^{2}+2 K_{n, q}(r, s]\right)}+\frac{t}{3 n r^{2}}\right\} 
			\leq \frac{q^{2}}{q^{2}-1} \frac{q}{t} e^{-t / q}.
		\end{aligned}
	\end{equation*}
\end{lemma}
Define $\mathcal{K}_{n}(\mathscr{F}, r) := K^{\mathscr{F}}_{n, \sqrt{2}}(r, 1]$. It follows from Lemma \ref{Ratio Limit Theorems for EP, cited} that
\begin{lemma} \label{Ratio Limit Theorems for EP, high probability version deduced by ourselves}
	For $0 < \delta < 2/e$, with probability at least $1 - \delta$,
	\begin{equation*} \label{}
		\begin{aligned} 
			\sup _{f \in \mathscr{F}, \sigma_{\mathcal{P}}(f) > r}\frac{\left|\mathcal{P}_{n} f - \mathcal{P} f\right|}{\sigma_{\mathcal{P}}^{2}(f)} < 2 \mathcal{K}_{n}(\mathscr{F}, r) + \frac{5}{2}\sqrt{\frac{ \ln(2/\delta)}{n r^{2}}} +\frac{2 \ln(2/\delta)}{ n r^{2}}.
		\end{aligned}
	\end{equation*}
\end{lemma}

\begin{proof} Recall that $\sigma_{\mathcal{P}}(f) \leq 1$. 
	Taking $s = 1$, $q = \sqrt{2}$ and $t = \sqrt{2} \ln(2/\delta)$ in Lemma \ref{Ratio Limit Theorems for EP, cited}, we obtain
	$\frac{q^{2}}{q^{2}-1} \frac{q}{t} e^{-t / q} = \delta / \ln(2/\delta) < \delta,$
	and with probability at least $1 - \delta$, 
	\begin{equation*} \label{}
		\begin{aligned} 
			\sup _{f \in \mathscr{F}(r, 1]}\frac{\left|\mathcal{P}_{n} f - \mathcal{P} f\right|}{\sigma_{\mathcal{P}}^{2}(f)} & < %\mathcal{K}_{n}(\mathscr{F}, r) + \sqrt{2 \frac{t}{n r^{2}}\left(2 + 2 \mathcal{K}_{n}(\mathscr{F}, r) \right)}+\frac{t}{3 n r^{2}}\\ & \leq 
			\mathcal{K}_{n}(\mathscr{F}, r) + 2\sqrt{\frac{t}{n r^{2}} \mathcal{K}_{n}(\mathscr{F}, r)} + 2\sqrt{\frac{t}{n r^{2}}} +\frac{t}{3 n r^{2}}\\
			& \leq 2 \mathcal{K}_{n}(\mathscr{F}, r) + 2\sqrt{\frac{\sqrt{2} \ln(2/\delta)}{n r^{2}}} +\frac{4 \sqrt{2} \ln(2/\delta)}{3 n r^{2}}\\
			& \leq 2 \mathcal{K}_{n}(\mathscr{F}, r) + \frac{5}{2}\sqrt{\frac{ \ln(2/\delta)}{n r^{2}}} +\frac{2 \ln(2/\delta)}{ n r^{2}},
		\end{aligned}
	\end{equation*}
	which completes the proof.
\end{proof}
\subsection{Oracle inequality for the generalization error}
Recall the population loss $L_{k}(u)$ and the empirical loss $L_{k, n}(u)$ defined in (\ref{population loss}) and (\ref{empirical loss}) respectively.
Let $\mathcal{F}$ be some hypothesis class and $0 < r < 1/2$.
Consider the minimization of $L_{k, n}(u)$ in $\mathcal{F}_{>r} = \{ u \in \mathcal{F} : \mathcal{E}_{2}(u)  > r^{2}\}$ with a minimizer $u_{n} = \mathop{\arg\min}_{u \in \mathcal{F}_{>r}} L_{k, n}(u) .$
%$$u_{n} = \mathop{\arg\min}_{u \in \mathcal{F}, \mathcal{E}_{2}(u) > r^{2}} L_{k, n}(u).$$
We aim to bound the energy excess  $L_{k}(u_{n}) - \lambda_{k}$. For any $u_{\mathcal{F}} \in \mathcal{F}_{>r}$, we write
\begin{equation} \label{ineq: decomposition of L_k(u_n)-lambda_k}
	\begin{aligned}
		L_{k}\left(u_{n}\right)-\lambda_{k}
		& = L_{k}\left(u_{n}\right)-L_{k, n}\left(u_{n}\right) + L_{k, n}\left(u_{n}\right)-L_{k, n}\left(u_{\mathcal{F}}\right) \\
		& \quad + L_{k, n}\left(u_{\mathcal{F}}\right)-L_{k}\left(u_{\mathcal{F}}\right)+L_{k}\left(u_{\mathcal{F}}\right)-\lambda_{k}.
	\end{aligned}
\end{equation}
Note that  $L_{k, n}\left(u_{n}\right)-L_{k, n}\left(u_{\mathcal{F}}\right) \leq 0$  because  $u_{n}$  is the minimizer of  $L_{k, n}\left(u\right)$. Therefore,
\begin{equation} \label{decompose generalization error}
	\begin{aligned}
		L_{k}\left(u_{n}\right)-\lambda_{k}
		& \leq  \Big(L_{k}(u_{n})-L_{k, n}(u_{n})\Big) + \Big(L_{k, n}\left(u_{\mathcal{F}}\right)-L_{k}\left(u_{\mathcal{F}}\right)\Big) + \Big(L_{k}\left(u_{\mathcal{F}}\right)-\lambda_{k}\Big)\\
		&  =: T_1+T_2+T_3,
	\end{aligned}
\end{equation}
where $T_1$  is the statistical error arising from the random approximation of the integrals,  $T_2$  is the Monte Carlo error and  $T_3$  is the approximation error due to restricting the minimization of  $L_{k}(u)$  from  $H_{0}^{1}(\Omega)$  to  $\mathcal{F}_{>r}$. 

\textbf{Bounding  $T_{1}$:}
\begin{equation*} \label{}
	\begin{aligned}
		T_{1} & \leq \left| \frac{\mathcal{E}_{n, V}\left(u_{n}\right)}{\mathcal{E}_{n, 2}\left(u_{n}\right)} - \frac{\mathcal{E}_{V}\left(u_{n}\right)}{\mathcal{E}_{2}\left(u_{n}\right)} \right| + \left|\frac{\mathcal{E}_{n, P}\left(u_{n}\right)}{\mathcal{E}_{n, 2}\left(u_{n}\right)} - \frac{\mathcal{E}_{P}\left(u_{n}\right)}{\mathcal{E}_{2}\left(u_{n}\right)}\right|\\
		& \leq \left|\frac{\mathcal{E}_{n, V}\left(u_{n}\right)}{\mathcal{E}_{V}\left(u_{n}\right)} \! \frac{\mathcal{E}_{2}\left(u_{n}\right)}{\mathcal{E}_{n, 2}\left(u_{n}\right)}  \! - \!  1\right| \frac{\mathcal{E}_{V}\left(u_{n}\right)}{\mathcal{E}_{2}\left(u_{n}\right)} + 
		\frac{\mathcal{E}_{P}\left(u_{n}\right) \left|\mathcal{E}_{2}\left(u_{n}\right) \! - \! \mathcal{E}_{n, 2}\left(u_{n}\right)\right|}{\mathcal{E}_{2}\left(u_{n}\right) \mathcal{E}_{n, 2}\left(u_{n}\right)} + \frac{\left|\mathcal{E}_{n, P}\left(u_{n}\right) \! - \! \mathcal{E}_{P}\left(u_{n}\right)\right|}{\mathcal{E}_{n, 2}\left(u_{n}\right)}\\
		& =: T_{11}+T_{12}+T_{13}.
	\end{aligned}
\end{equation*}

To bound  $T_{11}$, $T_{12}$  and  $T_{13}$, we define the following sets of functions
\begin{equation} \label{function classes defined for bounding generalization error}
	\begin{aligned}
		& \mathcal{G}_{1} := \left\{g: g=u^{2} \text { where } u \in \mathcal{F}\right\}, \\
		& \mathcal{G}_{2} := \left\{g: g=|\nabla u|^{2}+V|u|^{2} \text { where } u \in \mathcal{F}\right\}, \\
		& \mathcal{F}_{j} := \left\{g: g = u \psi_{j}  \text { where } u \in \mathcal{F}\right\}, \quad \text{for }  j = 1, 2, \ldots, k-1.
	\end{aligned}
\end{equation}
We assume that the set  $\mathcal{F}$  satisfies  $\sup _{u \in \mathcal{F}}\|u\|_{L^{\infty}(\Omega)} \leq M_{\mathcal{F}}<\infty$  so that  $\sup _{g \in \mathcal{G}_{1}}  \|g\|_{L^{\infty}(\Omega)} \leq   M_{\mathcal{F}}^{2}$. Assume further that  $\sup _{g \in \mathcal{G}_{2}} \|g\|_{L^{\infty}(\Omega)} \leq M_{\mathcal{G}_{2}}<\infty$ and $\|\psi_{j}\|_{L^{\infty}(\Omega)} \leq \mu_{j}$ for each $j$. %$j = 1, 2, \ldots, k-1.$ 
So $\sup _{g \in \mathcal{F}_{j}}\|g\|_{L^{\infty}(\Omega)} \leq \mu_{j} M_{\mathcal{F}}.$ 
In what follows, we shall derive high probability bounds for $T_{11}$, $T_{12}$ and $T_{13}$ by Lemma \ref{Ratio Limit Theorems for EP, high probability version deduced by ourselves}.
To this end, we rescale the function classes $\mathcal{G}_{1}$, $\mathcal{G}_{2}$ and $\mathcal{F}_{j}$ so that their elements take values in $[0, 1]$. 

We firstly derive a high probability bound for  $\mathcal{E}_{n, 2}\left(u_{n}\right)/\mathcal{E}_{2}\left(u_{n}\right)$. 
For the rescaled set $\mathcal{G}_{1}/M_{\mathcal{F}}^{2}$\footnote{In this paper, $a\mathscr{F} + b := \{a f + b : f \in \mathscr{F}\}$ where $\mathscr{F}$ is a set of functions and $a, b \in \mathbb{R}$ are some constants.}, we take $\sigma_{\mathcal{P}}(f) = \sqrt{\mathcal{P}f}$ and
define for  $n \in \mathbb{N}$  and  $0< \delta < 1/3$  the constant
\begin{equation} \label{xi1(n, r1, delta)}
	\xi_{1}(n, r, \delta):= 2 \mathcal{K}_n(\mathcal{G}_{1}/M_{\mathcal{F}}^{2}, r/M_{\mathcal{F}}) + \frac{5 M_{\mathcal{F}}}{2}\sqrt{\frac{ \ln(2/\delta)}{n r^{2}}} + \frac{2 M_{\mathcal{F}}^{2} \ln(2/\delta)}{ n r^{2}}.
\end{equation}
We also define the event
\begin{equation*} 
	A_{1}(n, r, \delta):= \left\{\sup _{u \in \mathcal{F}, \mathcal{E}_{2}\left(u\right) > r^{2}}  \left|\frac{\mathcal{E}_{n, 2}\left(u\right)}{\mathcal{E}_{2}\left(u\right)} - 1\right|  <  \xi_{1}(n, r, \delta)\right\} .
\end{equation*}
Applying Lemma \ref{Ratio Limit Theorems for EP, high probability version deduced by ourselves} to  $\mathcal{G}_{1}/M_{\mathcal{F}}^{2}$,  we get
\begin{equation} \label{bound A1(n, r1, delta)}
	\mathbf{P}\left[A_{1}(n, r, \delta)\right] \geq 1-\delta.
\end{equation}
Recall that  $\mathcal{E}_{2}(u_{n}) > r^{2}$. So if $\xi_{1}(n, r, \delta) < 1$, then, on the event $A_{1}(n, r, \delta)$,  there holds
\begin{equation} \label{bound T12}
	T_{12} \leq \frac{\xi_{1}(n, r, \delta)}{1 - \xi_{1}(n, r, \delta)} \frac{\mathcal{E}_{P}\left(u_{n}\right)}{\mathcal{E}_{2}\left(u_{n}\right)}.
\end{equation}

To bound  $T_{11}$, it follows from $\mathcal{E}_{2}(u) > r^{2}$ that $\mathcal{E}_{V}(u) > \lambda_{1}r^{2}$.
For the rescaled set $\mathcal{G}_{2}/M_{\mathcal{G}_{2}}$, we take $\sigma_{\mathcal{P}}(f) = \sqrt{\mathcal{P}f}$ and
define the constant
\begin{equation} \label{xi2(n, r2, delta)}
	\xi_{2}(n, r, \delta):= 2 \mathcal{K}_{n}(\mathcal{G}_{2}/M_{\mathcal{G}_{2}}, r\sqrt{\lambda_{1}/M_{\mathcal{G}_{2}}}) + \frac{5}{2}\sqrt{\frac{M_{\mathcal{G}_{2}} \ln(2/\delta)}{n \lambda_{1} r^{2}}} + \frac{2 M_{\mathcal{G}_{2}} \ln(2/\delta)}{n \lambda_{1} r^{2}}.
\end{equation}
Define the event
\begin{equation*} 
	A_{2}(n, r, \delta) := \left\{\sup _{u \in \mathcal{F}, \mathcal{E}_{V}\left(u\right) > \lambda_{1}r^{2}} \left|\frac{\mathcal{E}_{n, V}\left(u\right)}{\mathcal{E}_{V}\left(u\right)} - 1\right|  <  \xi_{2}(n, r, \delta)\right\} .
\end{equation*}
Applying Lemma \ref{Ratio Limit Theorems for EP, high probability version deduced by ourselves} to  $\mathcal{G}_{2}/M_{\mathcal{G}_{2}}$, we obtain
\begin{equation} \label{bound A2(n, r2, delta)}
	\mathbf{P}\left[A_{2}(n, r, \delta)\right] \geq 1-\delta .
\end{equation}
Therefore, if $\xi_{1}(n, r, \delta) < 1$, then, on the event $A_{1}(n, r, \delta) \cap A_{2}(n, r, \delta)$,  there holds
\begin{equation} \label{bound T11}
	\begin{aligned}
		T_{11} \leq \left(\frac{1+\xi_{2}(n, r, \delta)}{1-\xi_{1}(n, r, \delta)}-1\right) \frac{\mathcal{E}_{V}\left(u_{n}\right)}{\mathcal{E}_{2}\left(u_{n}\right)}
		= \frac{\xi_{1}(n, r, \delta)+\xi_{2}(n, r, \delta)}{1-\xi_{1}(n, r, \delta)} \cdot \frac{\mathcal{E}_{V}\left(u_{n}\right)}{\mathcal{E}_{2}\left(u_{n}\right)}.
	\end{aligned}
\end{equation}

To bound  $T_{13}$, we rescale $\mathcal{F}_{j}$ to 
$\mathcal{F}_{j}/\left(2\mu_{j} M_{\mathcal{F}}\right) + 1/2$. For any $g \in \mathcal{F}_{j}$, $g = u \psi_{j}$,
\begin{equation*} 
	\begin{aligned}
		\operatorname{Var}\left(\frac{g}{2 \mu_{j} M_{\mathcal{F}}} +\frac{1}{2}\right) 
		%= \frac{1}{4}\operatorname{Var}\left(\frac{g}{\mu_{j} M_{\mathcal{F}}}\right) 
		\leq \frac{1}{4} \mathbf{E}\left|\frac{u \psi_{j}}{\mu_{j} M_{\mathcal{F}}}\right|^{2} \leq \frac{1}{4} \mathbf{E}\left|\frac{u \psi_{j}}{\mu_{j} M_{\mathcal{F}}}\right| 
		\leq  \frac{\|u\|_{L^{2}(\Omega)}}{4 \mu_{j} M_{\mathcal{F}}},
	\end{aligned}
\end{equation*}
where we have used $\left|u \psi_{j}\right| \leq  \mu_{j} M_{\mathcal{F}}$ in the second inequality and the fact $\|\psi_{j}\|_{L^{2}(\Omega)} = 1$ in the last inequality. 
Therefore, we may take 
\begin{equation} \label{def sigma_P^2(f) for mathcalF_j/(2mu_j M_mathcalF) + 1/2}
	\sigma_{\mathcal{P}}^{2}(f) = \frac{\|u\|_{L^{2}(\Omega)}}{4 \mu_{j} M_{\mathcal{F}}}\quad \text{for any}\  f := \frac{u \psi_{j}}{2 \mu_{j} M_{\mathcal{F}}} +\frac{1}{2} \in \frac{\mathcal{F}_{j}}{2\mu_{j} M_{\mathcal{F}}} + \frac{1}{2}.
\end{equation}
Note that $\mathcal{E}_{2}(u) > r^{2}$ is precisely $\|u\|_{L^{2}(\Omega)} > r$. 
For all $1 \leq j \leq k-1$, we define 
\begin{equation*}  \label{xi3(n, r3, delta)}
	\begin{gathered} 
		\xi_{3, j}(n, r, \delta):= 2 \mathcal{K}_{n}\left(\frac{\mathcal{F}_{j}}{2\mu_{j} M_{\mathcal{F}}} + \frac{1}{2},  \sqrt{\frac{r}{4 \mu_{j} M_{\mathcal{F}}}}\right) + 5\sqrt{\frac{ \mu_{j} M_{\mathcal{F}} \ln(2 k/\delta)}{n r}} + \frac{8 \mu_{j} M_{\mathcal{F}}\ln(2 k/\delta)}{ n r},   
	\end{gathered}
\end{equation*}
and $\xi_{3}(n, r, \delta) := \max _{1 \leq j \leq k-1} \xi_{3, j}(n, r, \delta).$ 
Notice that for $f \in \mathcal{F}_{j}/\left(2\mu_{j} M_{\mathcal{F}}\right) + 1/2$,
$$\sup _{\sigma_{\mathcal{P}}(f) > \sqrt{r/4 \mu_{j} M_{\mathcal{F}}} } \frac{\left|\mathcal{P}_{n} f - \mathcal{P} f\right|}{\sigma_{\mathcal{P}}^{2}(f)} = \sup _{u \in \mathcal{F}, \|u\|_{L^{2}(\Omega)}>r } \frac{2 \left|\mathcal{P}_{n}\left(u \psi_{j}\right)  - \left\langle u, \psi_{j}\right\rangle\right|}{\|u\|_{L^{2}(\Omega)}}.$$
For each  $1 \leq j \leq k-1$, we define the events 
\begin{equation*} 
	\begin{gathered}
		A_{3, j}(n, r, \delta) := \left\{\sup _{u \in \mathcal{F}, \|u\|_{L^{2}(\Omega)}> r} \frac{2 \left|\mathcal{P}_{n}\left(u \psi_{j}\right)  - \left\langle u, \psi_{j}\right\rangle\right|}{\|u\|_{L^{2}(\Omega)}}  <  \xi_{3, j}(n, r, \delta)\right\} ,
	\end{gathered}
\end{equation*}
and $A_{3}(n, r, \delta) := \bigcap_{j=1}^{k-1} A_{3, j}(n, r, \delta) .$
Applying Lemma \ref{Ratio Limit Theorems for EP, high probability version deduced by ourselves} to  $\mathcal{F}_{j}/\left(2\mu_{j} M_{\mathcal{F}}\right) + 1/2$, we get $\mathbf{P}\left[A_{3, j}(n, r, \delta)\right] \geqslant 1-\delta/k.$
Hence,
\begin{equation} \label{bound A3(n, r3, delta)}
	\mathbf{P}\left[A_{3}(n, r, \delta) \right] \geqslant 1-\frac{\delta}{k}\left(k-1\right) \geqslant 1-\delta.
\end{equation}
Using  $a^{2} - b^{2} = (a - b)^{2} + 2b (a - b)$ for $a, b \in \mathbb{R}$, on event $A_{3}(n, r, \delta)$, we obtian 
\begin{equation} \label{ineq: bound |mathcalE_n, P(u_n) -mathcalE_P(u_n)|/beta mathcalE_2(u_n), references for similar formulas}
	\begin{aligned}
		\frac{\left|\mathcal{E}_{n, P}\left(u_{n}\right)  \! - \! \mathcal{E}_{P}\left(u_{n}\right)\right|}{\beta \mathcal{E}_{2}\left(u_{n}\right)} & \leq \sum_{j=1}^{k-1}\left[\frac{\left|\mathcal{P}_{n}\left(u_{n} \psi_{j}\right)-\mathcal{P}\left(u_{n} \psi_{j}\right)\right|^{2}}{\left\|u_{n}\right\|_{L^{2}(\Omega)}^{2}}+\frac{2 \left|\left\langle u_{n} , \psi_{j}\right\rangle\right|}{\left\|u_{n}\right\|_{L^{2}(\Omega)}} \frac{\left| \mathcal{P}_{n}\left(u_{n} \psi_{j}\right)-\mathcal{P}\left(u_{n} \psi_{j}\right)\right|}{\left\|u_{n}\right\|_{L^{2}(\Omega)}}\right] \\
		& \leq \sum_{j=1}^{k-1}\left[\frac{\xi_{3, j}^{2}}{4}  + \xi_{3, j}  \frac{\left|\left\langle u_{n}, \psi_{j}\right\rangle\right|}{\left\|u_{n}\right\|_{L^{2}(\Omega)}} \right] \\
		& \leq \sum_{j=1}^{k-1} \frac{\xi_{3, j}^{2}}{4} + \left(\sum_{j=1}^{k-1} \xi_{3, j}^{2}\right)^{1/2} \left(\sum_{j=1}^{k-1} \frac{\left|\left\langle u_{n}, \psi_{j}\right\rangle\right|^{2}}{\left\|u_{n}\right\|_{L^{2}(\Omega)}^{2}}\right)^{1/2}  \\
		& \leq \frac{k}{4} \xi_{3}(n, r, \delta)^{2} + \sqrt{k} \xi_{3}(n, r, \delta) .
	\end{aligned}
\end{equation}
%where the third inequality follows from Cauchy-Schwartz inequality and the last inequality follows from the definition of $\xi_{3}(n, r, \delta)$. 
Hence, on event $A_{1}(n, r, \delta) \cap A_{3}(n, r, \delta)$, if  $\xi_{1}(n, r, \delta) < 1$, then
\begin{equation} \label{bound T13}
	T_{13} = \frac{\mathcal{E}_{2}\left(u_{n}\right)}{\mathcal{E}_{n, 2}\left(u_{n}\right)} \cdot \frac{\left|\mathcal{E}_{n, P}\left(u_{n}\right)-\mathcal{E}_{P}\left(u_{n}\right)\right|}{\mathcal{E}_{2}\left(u_{n}\right)}  \leq \frac{\beta}{1-\xi_{1}}\left(\frac{k}{4} \xi_{3}^{2}+\sqrt{k} \xi_{3}\right).
\end{equation}
It follows from (\ref{bound T12}), (\ref{bound T11}) and (\ref{bound T13}) that if  $\xi_{1}(n, r, \delta)<1$, then, within event $\bigcap_{i=1}^{3} A_{i}(n, r, \delta)$, we obtain
\begin{equation} \label{bound T1}
	\begin{aligned}
		T_{1} & \leq \frac{\xi_{1}+\xi_{2}}{1-\xi_{1}} \cdot \frac{\mathcal{E}_{V}\left(u_{n}\right)}{\mathcal{E}_{2}\left(u_{n}\right)} + \frac{\xi_{1}}{1 - \xi_{1}} \frac{\mathcal{E}_{P}\left(u_{n}\right)}{\mathcal{E}_{2}\left(u_{n}\right)} + \frac{\beta}{1-\xi_{1}}\left(\frac{k}{4} \xi_{3}^{2}+\sqrt{k} \xi_{3}\right)\\
		& \leq \frac{\xi_{1}+\xi_{2}}{1-\xi_{1}} L_{k}\left(u_{n}\right) + \frac{\beta}{1-\xi_{1}}\left(\frac{k}{4} \xi_{3}^{2}+\sqrt{k} \xi_{3}\right),
	\end{aligned}
\end{equation}
while it follows from (\ref{bound A1(n, r1, delta)}), (\ref{bound A2(n, r2, delta)}) and (\ref{bound A3(n, r3, delta)}) that 
\begin{equation} \label{bound probability of intersection of A(n, r, delta)}
	\begin{aligned}
		\mathbf{P}\left[A_{1}(n, r, \delta) \cap A_{2}(n, r, \delta) \cap A_{3}(n, r, \delta) \right] \geqslant 1-3\delta.
	\end{aligned}
\end{equation}

\textbf{Bounding  $T_{2}$.} Similar to the process of bounding  $T_{1}$, since $\|u_{\mathcal{F}}\|_{L^{2}(\Omega)} > r$, we know that if  $\xi_{1}(n, r, \delta)<1$, then, within event $\bigcap_{i=1}^{3} A_{i}(n, r, \delta)$, we get
\begin{equation} \label{bound T2}
	\begin{aligned}
		T_{2} %& \leq \frac{\xi_{1}+\xi_{2}}{1-\xi_{1}} \cdot \frac{\mathcal{E}_{V}\left(u_{\mathcal{F}}\right)}{\mathcal{E}_{2}\left(u_{\mathcal{F}}\right)} + \frac{\xi_{1}}{1 - \xi_{1}} \frac{\mathcal{E}_{P}\left(u_{\mathcal{F}}\right)}{\mathcal{E}_{2}\left(u_{\mathcal{F}}\right)} + \frac{\beta}{1-\xi_{1}}\left(\frac{k}{4} \xi_{3}^{2}+\sqrt{k} \xi_{3}\right)\\
		& \leq \frac{\xi_{1}+\xi_{2}}{1-\xi_{1}} L_{k}\left(u_{\mathcal{F}}\right) + \frac{\beta}{1-\xi_{1}}\left(\frac{k}{4} \xi_{3}^{2}+\sqrt{k} \xi_{3}\right).
	\end{aligned}
\end{equation}
A combination of the bounds (\ref{decompose generalization error}), (\ref{bound T1})  and (\ref{bound T2}) leads to
\begin{equation*} \label{}
	\begin{aligned}
		L_{k}\left(u_{n}\right)-\lambda_{k}
		& \leq  \frac{\xi_{1}+\xi_{2}}{1-\xi_{1}} \left(L_{k}\left(u_{n}\right) - \lambda_{k}\right) + \frac{1 +\xi_{2}}{1-\xi_{1}} \left(L_{k}\left(u_{\mathcal{F}}\right)-\lambda_{k}\right) + 2\lambda_{k} \frac{\xi_{1}+\xi_{2}}{1-\xi_{1}} \\
        & \quad + \frac{2\beta}{1-\xi_{1}}\left(\frac{k}{4} \xi_{3}^{2}+\sqrt{k} \xi_{3}\right).
		%& \leq  \frac{\left(\xi_{1}+\xi_{2}\right) \left(L_{k}\left(u_{n}\right) - \lambda_{k}\right)  + \left(1 +\xi_{2}\right) \left(L_{k}\left(u_{\mathcal{F}}\right)-\lambda_{k}\right)  +  2\lambda_{k} \left(\xi_{1}+\xi_{2}\right) + 2\beta \left(k \xi_{3}^{2}/4 + \sqrt{k} \xi_{3}\right) }{1-\xi_{1}} .
	\end{aligned}
\end{equation*} 
If $2\xi_{1}+\xi_{2}<1$, then
\begin{equation*} \label{}
	\begin{aligned}
		L_{k}\left(u_{n}\right)-\lambda_{k}
		& \leq  \frac{\left(1 +\xi_{2}\right) \left(L_{k}\left(u_{\mathcal{F}}\right)-\lambda_{k}\right)  + 2\lambda_{k}\left(\xi_{1}+\xi_{2}\right)   + \beta\left(k \xi_{3}^{2}/2 + 2\sqrt{k} \xi_{3}\right)}{1-2\xi_{1}-\xi_{2}}.
	\end{aligned}
\end{equation*}
Combining the above estimate with (\ref{bound probability of intersection of A(n, r, delta)}), we obtain 
% the following oracle inequality. 
\begin{theorem}
	\label{Thm: oracle inequality for the generalization error}
	Let  $u_{n} = \mathop{\arg\min}_{u \in \mathcal{F}_{>r}} L_{k, n}(u),$  $\delta \in\left(0, 1/3\right)$  %be fixed, 
	and $\{\xi_{i}(n, r, \delta)\}_{i = 1}^{3}$  be defined in (\ref{xi1(n, r1, delta)}), (\ref{xi2(n, r2, delta)}), (\ref{xi3(n, r3, delta)}), respectively.
	Assume that $2\xi_{1}+\xi_{2} \leq 1/2$ and  $u_{\mathcal{F}}\in \mathcal{F}_{>r}.$ 
	Then, with probability at least  $1-3 \delta$,
	\begin{equation*} \label{}
		\begin{aligned}
			L_{k}\left(u_{n}\right)-\lambda_{k}
			& \leq  4\lambda_{k}\left(\xi_{1}+\xi_{2}\right)   + \beta\left(k \xi_{3}^{2} + 4 \sqrt{k} \xi_{3}\right)  + 3 \left(L_{k}\left(u_{\mathcal{F}}\right)-\lambda_{k}\right).
		\end{aligned}
	\end{equation*}
\end{theorem}

\section{Approximation theorem for sine spectral Barron functions} \label{section: Approximation theory for sine spectral Barron functions}
In this section, we study the properties of the sine spectral Barron functions on the hypercube as well as the neural network approximation. 

\subsection{Preliminaries}
We start by stating some preliminary results about the sine functions $\mathfrak{S} = \left\{\Phi_{k}\right\}_{k \in \mathbb{N}_{+}^{d}}$. It is clear that the set  $\mathfrak{S}$  forms an orthogonal basis of  $L^{2}(\Omega)$  and  $H_{0}^{1}(\Omega)$.
Given  $u \in L^{2}(\Omega)$, let  $\{\hat{u}(k)\}_{k \in \mathbb{N}_{+}^{d}}$  be the Fourier coefficients of  $u$ against the basis  $\left\{\Phi_{k}\right\}_{k \in \mathbb{N}_{+}^{d}}$. Then,
\begin{equation*}
u(x) = \sum_{k \in \mathbb{N}_{+}^{d}} \hat{u}(k) \Phi_{k}(x) \quad \text{and}\quad  \|u\|_{L^{2}(\Omega)}^{2}=\sum_{k \in \mathbb{N}_{+}^{d}} 2^{-d}|\hat{u}(k)|^{2},
\end{equation*}
where we have used  $\left\langle\Phi_{k}, \Phi_{k}\right\rangle_{L^{2}(\Omega)}=2^{-d}$. 
A straightforward calculation yields that for  $u \in H_{0}^{1}(\Omega)$,
\begin{equation*}
\|u\|_{H^{1}(\Omega)}^{2}=\sum_{k \in \mathbb{N}_{+}^{d}} 2^{-d}\left(1+\pi^{2}|k|^{2}\right)|\hat{u}(k)|^{2}.
\end{equation*}
\begin{lemma} \label{lemma: embedding results for sine spectral Barron space}
The following embedding hold:\\
\quad (1)  $\mathfrak{B}^{0}(\Omega) \hookrightarrow L^{\infty}(\Omega)$;\\
\quad (2)  $\mathfrak{B}^{2}(\Omega) \hookrightarrow H_{0}^{1}(\Omega)$.
\end{lemma}

We postpone the proof of Lemma \ref{lemma: embedding results for sine spectral Barron space} to appendix \ref{appendix proof of Lemma: embedding results for sine spectral Barron space}.

\subsection{Sine Spectral Barron Space and the Neural Network Approximation}
\iffalse Recall that for  $s \geq 0$, the sine spectral Barron space  $\mathfrak{B}^{s}(\Omega)$  is given by (\ref{definition of sine spectral Barron space}) with the norm (\ref{ definition of sine spectral Barron norm}). The solution structure $\varphi v_{\theta}$ is employed to approximate functions in  $\mathfrak{B}^{s}(\Omega)$. Functions in  $\mathfrak{B}^{s}(\Omega)$ fulfill the zero boundary condition.
For $u \in \mathfrak{B}^{s}(\Omega)$, we rewrite it in the form $u = \varphi v$.  Suppose that if the neural network function $v_{\theta}$ has a good approximation to $v$, then its multiplication by the cut off function $\varphi$ will be a good approximation to $u$. From this perspective, we firstly study the properties of $v = u/\varphi$ and then derive an approximation result for $u$ by neural networks with trigonometric activation function. \fi

The main results in this part are summarized in the following two propositions.
\begin{proposition}
\label{proposition: v cos-sin expansion}
Assume that $u \in \mathfrak{B}^{s+1}(\Omega)$ for some $s\geq 0$. Then, $u/\varphi$  admits the representation
\begin{equation} \label{u(x)/varphi(x) expansion form respect to cos-sin functions}
\begin{aligned}
\frac{u(x)}{\varphi(x)}=\sum_{(k, i) \in \Gamma} \hat{v}(k, i) \cos \left(k_{i} \pi x_{i}\right) \prod_{\substack{1 \leq j \leq d \\ j \neq i}} \sin \left(k_{j} \pi x_{j}\right) , 
\end{aligned}
\end{equation}
where 
%$\Gamma$ is an index set given by
$\Gamma=\left\{(k, i) \mid k \in \mathbb{N}_{0}^{d},\  1 \leqslant i \leqslant d,\  \left(k+e_{i}\right) \in \mathbb{N}_{+}^{d}\right\}$
and $e_{i}$ is the $i$-th canonical basis. 

Moreover, the coefficients $\hat{v}(k, i)$ satisfy 
\begin{equation} \label{u(x)/varphi(x) expansion, coefficients norm bound}
\begin{aligned}
\sum_{(k, i) \in \Gamma} \left(1+\pi^{s}|k|_{1}^{s}\right) |\hat{v}(k, i)| \leqslant \left(1+\frac{2 d}{\pi (s + 1)}\right) \|u\|_{\mathfrak{B}^{s+1}(\Omega)} . 
%\quad \textcolor{red}{(\text{To be modified.)}}
\end{aligned}
\end{equation}
\end{proposition}

Roughly speaking, the above result indicates that $u/\varphi$ lies in a spectral Barron space of one order lower than $u$.
Using this proposition, we may prove a preliminary $H^1$ approximation result for functions in the sine spectral Barron space. Denote by $\operatorname{conv}(G)$ the convex hull of a set  $G$, and denote by $\overline{G}$ the $H^{1}$-closure of $G$. Let
$$\Gamma_{1} = \left\{ k \in \mathbb{Z}^{d} : \exists 1 \leqslant i \leqslant d,  \left((|k_{1}|, |k_{2}|, \cdots, |k_{d}|), i\right) \in \Gamma \right\}.$$
Note that $k\in \Gamma_{1}$ if and only if $k$ has at most one zero component.
\begin{proposition}
\label{proposition: u H1 approximation by cos-sin}
For $s \geq 0$, define 
\begin{equation*} \label{}
\begin{aligned}
\mathcal{F}_{s}(B) := \left\{\frac{\gamma}{1+\pi^{s}|k|_{1}^{s}} f(\pi(k \cdot x+b)) : |\gamma| \leqslant B, b \in\{0,1\},   k \in \Gamma_{1}  \right\},
\end{aligned}
\end{equation*}
where $f(x)=\cos x$ if $d$ is odd and $f(x)=\sin x$ if $d$ is even. Then, for any $u  \in \mathfrak{B}^{s+1}(\Omega)$ with $s\geq 1$,  $u \in \overline{\operatorname{conv}(\varphi\mathcal{F}_{s}(B_{u}))}$ with 
\begin{equation*} \label{}
\begin{aligned}
B_{u}=\left(1+\frac{2 d}{\pi(s+1)}\right) \|u\|_{\mathfrak{B}^{s+1}(\Omega)},
\end{aligned}
\end{equation*}
and there exists $v_m$ which is a convex combination of $m$ functions in $\mathcal{F}_{s}(B_{u})$ 
such that 
\begin{equation*} \label{}
\begin{aligned}
\left\|u - \varphi v_{m}\right\|_{H^{1}(\Omega)} \leqslant \sqrt{\frac{6} {m}}B_{u}.
\end{aligned}
\end{equation*}
When $d > 1$, the constant $\sqrt{6}$ may be replaced by $2$.
\end{proposition}

We postpone the proof of Proposition \ref{proposition: v cos-sin expansion} and Proposition \ref{proposition: u H1 approximation by cos-sin} to Appendix \ref{appendix subsection: Sine Spectral Barron Space and Neural Network Approximation}.
%
%In what follows,  
We cite a famous result for nonlinear approximation known as Maurey's method, which is exploited to prove the approximation bounds.
\begin{lemma}{\cite{MaureyPisier1981RemarquesSU, barron1993universal}}
\label{Lemma: Maurey method, Universal approximation bounds for superpositions of a sigmoidal function}
If  $\bar{f}$  is in the closure of the convex hull of a set  $G$  in a Hilbert space, with  $\|g\| \leq b$  for each  $g \in G$, then for every  $m \geq 1$, there is an  $f_{m}$  in the convex hull of  $m$  points in  $G$  such that
$$\left\|\bar{f}-f_{m}\right\|^{2} \leq \frac{b^{2}}{m} .$$
\end{lemma}
\subsection{Reduction to ReLU and Softplus Activation Functions}  \label{subsection: Reduction to ReLU and Softplus Activation Functions}
We have found that if $u \in \mathfrak{B}^{s+1}$ with $s\geq1$, $u$ lies in the bounded set $\overline{\operatorname{conv}(\varphi\mathcal{F}_{s}(B_{u}))} \subset H^1(\Omega)$ with $B_{u}=\left(1+\frac{2 d}{\pi(s+1)}\right) \|u\|_{\mathfrak{B}^{s+1}}$.
Define the function classes
\begin{equation} \label{function class F ReLU(B) and F mathrmSPtau(B)}
\begin{aligned}
& \mathcal{F}_{\mathrm{ReLU}}(B):=\left\{c+\gamma \operatorname{ReLU}(w \cdot x-t) : |c| \leq B, |w|_{1}=1, |t| \leq 1, |\gamma| \leq 4 B \right\},\\%\qquad\text{and}
& \mathcal{F}_{\mathrm{SP}_{\tau}}(B):=\left\{c+\gamma \mathrm{SP}_{\tau}(w \cdot x-t):|c| \leq  B, |w|_{1}=1, |t| \leq 1, |\gamma| \leq 4 B\right\} .
\end{aligned}
\end{equation}
In this subsection, we aim to illustrate that when $s\geq2$, each function in $\varphi\mathcal{F}_{s}(B)$ lies in  $\overline{\operatorname{conv}(\varphi\mathcal{F}_{\operatorname{ReLU}}(B))}$ and $\overline{\operatorname{conv}(\varphi \mathcal{F}_{\mathrm{SP}_{\tau}}(B))}$, %a set composed of ReLU (and then, Softplus) network functions multiplied by cut off function $\varphi$
and yield
% Using Lemma \ref{Lemma: Maurey method, Universal approximation bounds for superpositions of a sigmoidal function}, we obtain 
the approximation bounds in Theorem \ref{Thm: u H1 approximation by varphi ReLU networks} and Theorem \ref{Thm: u H1 approximation by varphi Softplus networks} with Lemma \ref{Lemma: Maurey method, Universal approximation bounds for superpositions of a sigmoidal function}

By Lemma \ref{lemma: Magnification of the H1-norm caused by the action of varphi}, to prove $\varphi\mathcal{F}_{s}(B)$ lies in    $\overline{\operatorname{conv}(\varphi\mathcal{F}_{\operatorname{ReLU}}(B))}$ and $\overline{\operatorname{conv}(\varphi \mathcal{F}_{\mathrm{SP}_{\tau}}(B))}$, one only needs to show $\mathcal{F}_{s}(B)$  lies in 
$\overline{\operatorname{conv}(\mathcal{F}_{\operatorname{ReLU}}(B))}$ and $\overline{\operatorname{conv}(\mathcal{F}_{\mathrm{SP}_{\tau}}(B))}$, respectively.
%In fact, when $s\geq2$, functions in $\mathcal{F}_{s}(B)$ can be directly approximated by convex combinations of two-layer ReLU and Softplus networks.
Notice that every function in $\mathcal{F}_{s}(B)$ is a composition of a function $g$ defined on $[-1, 1]$ by
\begin{equation} \label{the part of function in Fs(B) approximated by the network, g(z)}
g(z) = \begin{cases}
\frac{\gamma}{1+\pi^{s}|k|_{1}^{s}} \cos(\pi(|k|_1 z +b)), & d \text{ is odd},\\
\frac{\gamma}{1+\pi^{s}|k|_{1}^{s}} \sin(\pi(|k|_1 z +b)), & d \text{ is even},
\end{cases}
\end{equation}
where $k \in \Gamma_{1}$, $|\gamma| \leq B$, $b \in \{0, 1\}$, and a linear function $z = w \cdot x$ with $w = k/|k|_1$ or $z = x_1$ in case $k = 0$. When $s\geq 2$, it is clear that $g \in C^2([-1, 1])$ and $g$ satisfies  $$\|g^{(r)}\|_{L^{\infty}([-1,1])}\leq|\gamma|\leq B, \quad \text{ for } r = 0, 1, 2.$$
The uniform boundness of $\|g\|_{W^{2,\infty}([-1,1])}$ for all $k$ is key to the proof. In addition, we observe that $g^{\prime}(0) = 0$ if $d$ is odd and $g^{\prime}(\frac{1}{2|k|_1}) = 0$ if $d$ is even. Proceeding along the same line as  \citep[Section 4.3]{DRMlu2021priori}, we prove that functions in $\mathcal{F}_{s}(B)$ can be well approximated by two-layer ReLU networks. 
Lemma \ref{lemma: ReLU for both sine and cosine case} generalizes \citep[Lemma 4.5]{DRMlu2021priori} to handle both sine and cosine cases.
\begin{lemma} \label{lemma: ReLU for both sine and cosine case}
Let  $g \in C^{2}([-1,1])$  with  $\left\|g^{(r)}\right\|_{L^{\infty}([-1,1])} \leq B$  for  $r=0,1,2$. 
Assume that  $g^{\prime}(\rho)=0$ for some $\rho \in [0,1/2]$. 
Let  $\left\{z_{j}\right\}_{j=0}^{2 m}$  be a partition of  $[-1,1]$  with  $z_{0}=-1$, $z_{m}=\rho$, $z_{2 m}=1$  and  $z_{j+1}-z_{j}=h_{1}=(\rho+1) / m$  for each  $j=0, \cdots, m-1$; $z_{j+1}-z_{j}=h_{2}=(1-\rho) / m$  for each  $j=m, \cdots, 2m-1$. Then there exists a two-layer ReLU network 
\begin{equation} \label{sin piecewise linear interpolation}
\begin{aligned}
g_{m}(z)=c+\sum_{i=1}^{2 m} a_{i} \operatorname{ReLU}\left(\epsilon_{i} z-b_{i}\right), \quad z \in[-1,1]
\end{aligned}
\end{equation}
with  $c = g(\rho)$, $b_{i} \in[-1,1]$  and  $\epsilon_{i} \in\{ \pm 1\}, i=1, \cdots, 2 m$, such that
\begin{equation} \label{sin piecewise linear interpolation, W1,inf error}
\begin{aligned}
\left\|g-g_{m}\right\|_{W^{1, \infty}([-1,1])} \leq \frac{2 B}{m} .
\end{aligned}
\end{equation}
Moreover, we have $|c| \leq B $, $\left|a_{i}\right| \leq 2 B h_{1}$  if  $i<m$,  $\left|a_{m}\right| \leq B h_{1}$,  $\left|a_{m+1}\right| \leq B h_{2}$  and $\left|a_{i}\right| \leq 2 B h_{2}$  if  $i>m+1$  so that  $\sum_{i=1}^{2 m} |a_{i}| \leq 4B$.
\end{lemma}

Given Proposition \ref{proposition: u H1 approximation by cos-sin} and Lemma \ref{lemma: ReLU for both sine and cosine case}, we are ready to prove Theorem \ref{Thm: u H1 approximation by varphi ReLU networks}.
\begin{proof}[Proof of Theorem \ref{Thm: u H1 approximation by varphi ReLU networks}]
%Thanks to the proof of
By Proposition \ref{proposition: u H1 approximation by cos-sin}, for any $u \in \mathfrak{B}^{s}(\Omega)$ with $s\geq 3$, $u \in \overline{\operatorname{conv}(\varphi\mathcal{F}_{s-1}(B_{u}))}$ with $B_{u}=\left(1+\frac{2 d}{\pi s}\right) \|u\|_{\mathfrak{B}^{s}(\Omega)}$. Since $s-1\geq 2$, a combination of Lemma \ref{lemma: ReLU for both sine and cosine case} and Lemma \ref{lemma: Magnification of the H1-norm caused by the action of varphi} yields that the set  $\varphi\mathcal{F}_{s-1}(B_{u})$  lies in the  $H^{1}$-closure of   $\operatorname{conv}(\varphi\mathcal{F}_{\mathrm{ReLU}}(B_{u}))$. 
Hence, $u\in \overline{\operatorname{conv}(\varphi\mathcal{F}_{\mathrm{ReLU}}(B_{u}))}$. 
Notice that when $|w|_{1} = 1$, for constants $a$, $b \geq 0$,
\begin{equation} \label{simple fact about integral of w cdot x}
\begin{aligned}
\int_{\Omega}\left(a + b w \cdot x\right)^{2} d x
& = a^{2} + a b \sum_{i=1}^{d} w_{i} + b^{2} \left(\frac{1}{3}\sum_{i=1}^{d}w_{i}^2 + \frac{1}{4} \sum_{i\neq j} w_{i}w_{j} \right)\\
& \leq a^{2} + a b |w|_{1} +  b^{2} \left(\frac{1}{4}|w|_{1}^2 + \frac{1}{12} |w|_{2}^2 \right)\\
& \leq a^{2} + a b + b^{2}/3.
\end{aligned}
\end{equation}
By (\ref{simple fact about integral of w cdot x}), we check that for each $v \in \mathcal{F}_{\mathrm{ReLU}}(B)$,
\begin{equation*}
\begin{aligned}
\|v\|_{L^{2}(\Omega)}^{2}
& \leq \int_{\Omega}\left[B+4B (w \cdot x + 1)\right]^{2} d x 
%= B^{2} \int_{\Omega}\left(5 + 4 w \cdot x \right)^{2} d x 
\leq 51 B^{2}.
\end{aligned}
\end{equation*}
Since $|v(x)| \leq (5 + 4 w \cdot x) B$, $\|\nabla v\|_{L^{\infty}(\Omega)} \leq |\gamma| \leq 4 B$,  by Lemma \ref{preliminary bounds for cutoff function varphi(x)} and the inequality (\ref{simple fact about integral of w cdot x}),  \begin{equation*}
\begin{aligned}
\|\nabla(\varphi v)\|_{L^{2}(\Omega)}^{2} 
& \leq \int_{\Omega}\left(\left|\nabla\varphi\right| |v| + \varphi \left|\nabla v\right| \right)^{2} d x
\leq B^{2} \int_{\Omega}\left[ \pi (5 + 4 w \cdot x) + 4 \right]^{2} d x 
\leq 689 B^{2}.
\end{aligned}
\end{equation*}
Hence,  $\|\varphi v\|_{H^{1}(\Omega)}^{2} 
%= \left\|\varphi v\right\|_{L^{2}(\Omega)}^{2} + \|\nabla(\varphi v)\|_{L^{2}(\Omega)}^{2}
%\leq  51 B^{2}/d^2 + 689B^{2}
\leq 740 B^{2}.$ Therefore, the $H^{1}$-norm of each function in $\varphi\mathcal{F}_{\mathrm{ReLU}}(B)$ can be bounded by $28B$. 
Theorem \ref{Thm: u H1 approximation by varphi ReLU networks} follows immediately from Lemma \ref{Lemma: Maurey method, Universal approximation bounds for superpositions of a sigmoidal function}.
\end{proof}
Lemma \ref{lemma: Softplus for sine and cosine case, modified} shows that functions in $\mathcal{F}_{s}(B)$ are also well approximated by two-layer Softplus networks.

\begin{lemma} \label{lemma: Softplus for sine and cosine case, modified}
Under the same assumption of Lemma \ref{lemma: ReLU for both sine and cosine case},
there exists a two-layer neural network  $g_{\tau, m}$  of the form
\begin{equation} \label{gtau,m Softplus 1}
\begin{aligned}
g_{\tau, m}(z)=c+\sum_{i=1}^{2 m} a_{i} \mathrm{SP}_{\tau}\left(\epsilon_{i} z-b_{i}\right),\quad z \in[-1,1]
\end{aligned}
\end{equation}
with $\tau>0$, $c = g(\rho)$, $b_{i} \in[-1,1]$  and  $\epsilon_{i} \in\{ \pm 1\}$, $i=1, \cdots, 2 m $ such that
\begin{equation} \label{approximation error by two-layer Softplus networks for functions in mathcalFs(B)}
\begin{aligned}
\left\|g-g_{\tau, m}\right\|_{W^{1, \infty}([-1,1])} \leq \frac{4 B}{\tau} \left(1+\frac{1}{\tau}\right).
\end{aligned}
\end{equation}
Moreover, the bounds for $|c|$, $\left|a_{i}\right|$  and  $\sum_{i=1}^{2 m} |a_{i}|$ all hold true as in Lemma \ref{lemma: ReLU for both sine and cosine case}.
\end{lemma}
Now we are ready to prove Theorem \ref{Thm: u H1 approximation by varphi Softplus networks}.  It follows from (\ref{ReLU-Softplus difference 
Wi,inf bound}) that 
\begin{equation} \label{convex hull Softplus H1 bound}
\begin{aligned}
\sup _{f \in \mathcal{F}_{\mathrm{SP}_{\tau}(B)}}\|f\|_{H^{1}(\Omega)} \leq  B+4 B\left\|\mathrm{SP}_{\tau}\right\|_{W^{1, \infty}([-2,2])} \leq 13 B+ 4 B/ \tau .
\end{aligned}
\end{equation}

\begin{proof}[Proof of Theorem \ref{Thm: u H1 approximation by varphi Softplus networks}]
According to Proposition \ref{proposition: u H1 approximation by cos-sin}, for any $u \in \mathfrak{B}^{s}(\Omega)$ with $s\geq 3$, $u\in$ $\overline{\operatorname{conv}(\varphi\mathcal{F}_{s-1}(B))}$ with $B=\left(1+\frac{2 d}{\pi s}\right) \|u\|_{\mathfrak{B}^{s}(\Omega)}$. 
Note that each function in $\mathcal{F}_{s-1}(B)$ with $s\geq 3$ is a composition of the multivariate linear function  $z=w \cdot x$  with  $|w|_{1}=1$  and the univariate function  $g(z)$  defined in (\ref{the part of function in Fs(B) approximated by the network, g(z)}) such that  $g^{\prime}(\rho)=0$ for some $\rho\in [0, 1/2]$ and  $\left\|g^{(r)}\right\|_{L^{\infty}([-1,1])} \leq B$  for  $r=0,1,2$. 
By Lemma \ref{lemma: Softplus for sine and cosine case, modified}, such  $g$  may be approximated by  $g_{\tau, m}$  which lies in the convex hull of the set of functions
$\left\{c+\gamma \mathrm{SP}_{\tau}(\epsilon z-b)  \! :  \! |c|  \! \leq \!  B, \epsilon  \! \in  \! \{ \pm 1\},|b|  \! \leq  \! 1,\right.$ $\left.\gamma \leq 4 B\right\} .$
Moreover,  $\left\|g-g_{\tau, m}\right\|_{W^{1, \infty}([-1,1])} \leq 4 B (1+\tau)/\tau^{2}$. As a consequence, we have 
$$\left\|g(w \cdot x)-g_{\tau, m}(w \cdot x)\right\|_{H^{1}(\Omega)} \leq\left\|g-g_{\tau, m}\right\|_{W^{1, \infty}([-1,1])} \leq \frac{4 B}{\tau} \left(1+\frac{1}{\tau}\right).$$
By Lemma \ref{lemma: Magnification of the H1-norm caused by the action of varphi}, there exists a function  $v_{\tau}$  in the convex hull of  $\mathcal{F}_{\mathrm{SP}_{\tau}}(B)$  such that
$$\left\|u-\varphi v_{\tau}\right\|_{H^{1}(\Omega)} \leq \frac{4\sqrt{21} B}{\tau} \left(1+\frac{1}{\tau}\right).$$ 
Thanks to Lemma \ref{Lemma: Maurey method, Universal approximation bounds for superpositions of a sigmoidal function} and the bound (\ref{convex hull Softplus H1 bound}), there exists  $v_{m} \in \mathcal{F}_{\mathrm{SP}_{\tau}, m}(B) $, which is a convex combination of  $m$  functions in  $\mathcal{F}_{\mathrm{SP}_{\tau}}(B)$  such that
$$\left\|\varphi v_{\tau}-\varphi v_{m}\right\|_{H^{1}(\Omega)} \leq \sqrt{21} \left\| v_{\tau} - v_{m}\right\|_{H^{1}(\Omega)} \leq\sqrt{\dfrac{21}m}B\left(\frac{4}{\tau}+13\right),$$
where the first inequality follows from Lemma \ref{lemma: Magnification of the H1-norm caused by the action of varphi}.

Combining the last two inequalities and setting  $\tau = 9\sqrt{m}$, we obtain~\eqref{approximation error by two-layer Softplus networks for functions in mathcalFs(B)}.
\end{proof}
\subsection{Bounding the approximation error}
\iffalse
Consider that $u^{*}$ is any normalized eigenfunction lying in the subspace  $U_{k}$ defined in (\ref{subspace Uk, solution space}). By definition of $L_{k}(u)$, we have $L_{k}(u^{*}) = \lambda_{k}$.
Suppose additionally that the function $u^{*}\in \mathfrak{B}^{s}$ for some $s \geq 3$.
The following theorem bounds the approximation error  $$\inf_{u \in \mathcal{F}}\left(L_{k}(u)-\lambda_{k}\right) $$  when  $\mathcal{F} = \varphi \mathcal{F}_{\mathrm{SP}_{\tau}, m} .$ \fi
The following theorem bounds the approximation error in (\ref{decompose generalization error}) when  $\mathcal{F} = \varphi \mathcal{F}_{\mathrm{SP}_{\tau}, m} .$
We postpone the proof of Theorem \ref{Thm: bounding the approximation error Lk(u in F)-lambdak} to Appendix \ref{Appendix proof of Thm: bounding the approximation error Lk(u in F)-lambdak}.
\begin{theorem} \label{Thm: bounding the approximation error Lk(u in F)-lambdak}
Under Assumptions \ref{Assumption: V is bounded above and below} and \ref{Assumption: exist u* in both mathfrakB^s for s geq 3 and U_k}, 
let $B_{u^{*}} = \left(1+\frac{2 d}{\pi s}\right) \|u^{*}\|_{\mathfrak{B}^{s}}$  and  $v_{m} \in   \mathcal{F}_{\mathrm{SP}_{\tau}, m}\left(B_{u^{*}}\right)$  be defined in Theorem \ref{Thm: u H1 approximation by varphi Softplus networks}.
Assume in addition that $\eta\left(B_{u^{*}}, m\right) := 64B_{u^{*}}/\sqrt{m} \leq 1/2.$ Then,
\begin{equation*} \label{Ineq: conclusion in Thm: bounding the approximation error Lk(u in F)-lambdak}
\begin{aligned}
%\inf_{u \in \varphi \mathcal{F}_{\mathrm{SP}_{\tau}, m}}\left(L_{k}(u)-\lambda_{k}\right) & \leq 
L_{k}(\varphi v_{m})-L_{k}(u^{*})  
\leq  \left(3\max \left\{1, V_{\max }\right\}  + 7\lambda_{k} + 5 \beta\right) \eta\left(B_{u^{*}}, m\right) .
\end{aligned}
\end{equation*}
\end{theorem}

\section{Statistical error}  \label{section: Statistical error}
According to Theorem \ref{Thm: oracle inequality for the generalization error}, in order to bound the statistical error, we need to control 
%$$E_{n}(\mathcal{G}_{1}/M_{\mathcal{F}}^{2}, \hat{r}/M_{\mathcal{F}}),\ \ E_{n}(\mathcal{G}_{2}/M_{\mathcal{G}_{2}}, \hat{r}\sqrt{\lambda_{1}/M_{\mathcal{G}_{2}}}) \ \text{  and }\  E_{n}(\frac{\mathcal{F}_{j}}{\left(2\mu_{j} M_{\mathcal{F}}\right)} + \frac{1}{2},  \sqrt{\hat{r}/4 \mu_{j} M_{\mathcal{F}}}) , \ \ 1 \leq j \leq k-1.$$
\begin{equation} \label{quantities to be bounded for statistical error}
\mathcal{K}_{n}(\frac{\mathcal{G}_{1}}{M_{\mathcal{F}}^{2}}, \frac{r}{M_{\mathcal{F}}}),\ \ \mathcal{K}_{n}(\frac{\mathcal{G}_{2}}{M_{\mathcal{G}_{2}}}, r\sqrt{\frac{\lambda_{1}}{M_{\mathcal{G}_{2}}}}) \ \text{  and  }\  \mathcal{K}_{n}(\frac{\mathcal{F}_{j}}{2\mu_{j} M_{\mathcal{F}}} + \frac{1}{2},  \sqrt{\frac{r}{4 \mu_{j} M_{\mathcal{F}}}}) , \ \ 1 \leq j \leq k-1.
\end{equation}
%Recall the quantity (\ref{the quantity En,q(r, s]}) defined in subsection \ref{Concentration inequalities for normalized empirical processes}. In this section, we firstly derive inequalities which bound  $E_{n, q}(r, s]$  with respect to covering numbers. Secondly, we bound the covering numbers of the properly rescaled function classes. Finally, the bounds for statistical error may be summarized.
To this end, we firstly bound the covering numbers of the properly rescaled function classes in \S \ref{subsection: Bounding the covering numbers}.
Secondly, we derive inequalities to bound the quantity  $K_{n, q}(r, s]$ defined in (\ref{the quantity En,q(r, s]}) with respect to covering numbers in \S \ref{subsection: Estimates of the expectation of suprema of empirical processes}.
Then, we obtain the bounds for the quantities in (\ref{quantities to be bounded for statistical error}).

\subsection{Bounding the covering numbers} \label{subsection: Bounding the covering numbers}
%In this subsection, we control the covering numbers of the properly scaled function classes.
For fixed positive constants  $C$, $\Gamma$, $W$  and  $T$, we consider the set of two-layer neural networks
\begin{equation} \label{widetilde mathcal Fm}
\begin{aligned}
\widetilde{\mathcal{F}}_{m}=\left\{v_{\theta}(x)=c+\sum_{i=1}^{m} \gamma_{i} \phi\left(w_{i} \cdot x+t_{i}\right): x \in \Omega,  \left| c\right|\leq C, \sum_{i=1}^{m} \left|\gamma_{i}\right|  \leq \Gamma,\left|w_{i}\right|_{1} \leq W,\left|t_{i}\right| \leq T \right\} ,
\end{aligned}
\end{equation}
where  $\phi$  is the activation function,  $\theta=\left(c,\left\{\gamma_{i}\right\}_{i=1}^{m},\left\{w_{i}\right\}_{i=1}^{m},\left\{t_{i}\right\}_{i=1}^{m}\right)$  denotes collectively the parameters of the two-layer neural network. Denote  the parameter space $$\Theta=\Theta_{c} \times \Theta_{\gamma} \times \Theta_{w} \times \Theta_{t}=[-C, C] \times B_{1}^{m}(\Gamma) \times   \left(B_{1}^{d}(W)\right)^{m} \times[-T, T]^{m}.$$   
We consider the set  $\Theta$  endowed with the metric  $\rho$  defined for  $\theta=(c, \gamma, w, t),\ \theta^{\prime}=\left(c^{\prime}, \gamma^{\prime}, w^{\prime}, t^{\prime}\right)$ in $\Theta$  by
\begin{equation} \label{metric rho defined on parameter space Theta}
\begin{aligned}
\rho_{\Theta}\left(\theta, \theta^{\prime}\right)=\max \left\{\left|c-c^{\prime}\right|,\left|\gamma-\gamma^{\prime}\right|_{1}, \max _{i}\left|w_{i}-w_{i}^{\prime}\right|_{1},\left\|t-t^{\prime}\right\|_{\infty}\right\}.
\end{aligned}
\end{equation}
Assume that  $\phi$  satisfies the following assumption, which is valid for the Softplus activation function.
\begin{assumption_in_section} \label{assumption for activation function} \hypertarget{Assumption 3}{}
$\phi \in C^{2}(\mathbb{R})$  and   $\phi$  (resp.  $\phi^{\prime}$, the derivative of  $\phi$) is $L$-Lipschitz (resp. is  $L^{\prime}$-Lipschitz) for some  $L$, $L^{\prime}>0$. Moreover, there exist positive constants  $\phi_{\max }$  and  $\phi_{\max }^{\prime}$  such that
$$\sup _{w \in \Theta_{w}, t \in \Theta_{t}, x \in \Omega}|\phi(w \cdot x+t)| \leq \phi_{\max } \quad \text{and} \quad \sup _{w \in \Theta_{w}, t \in \Theta_{t}, x \in \Omega}\left|\phi^{\prime}(w \cdot x+t)\right| \leq \phi_{\max }^{\prime} .$$
\end{assumption_in_section}

\begin{example} 
Let $\Theta_{w} = \left(B_{1}^{d}(1)\right)^{m}$ and $\Theta_{t} = [-1, 1]^{m}$.
It is clear that  $\left\|\mathrm{SP}_{\tau}^{\prime}\right\|_{L^{\infty}(\mathbb{R})} \leq 1$  and  %$\left\|\mathrm{SP}_{\tau}^{\prime \prime}\right\|_{L^{\infty}(\mathbb{R})} \leq \tau=9\sqrt{m} $,
$\left\|\mathrm{SP}_{\tau}^{\prime \prime}\right\|_{L^{\infty}(\mathbb{R})} \leq \tau$, hence  $\mathrm{SP}_{\tau}$  satisfies Assumption \ref{assumption for activation function} with
\begin{equation} \label{constants in assumption for Softplus activation function}
L = \phi_{\max }^{\prime}=1, \quad L^{\prime}=\tau, \quad  \phi_{\max } \leq 2 + 1/\tau .
\end{equation}
\end{example}
\begin{example}
The activation function $tanh$ satisfies Assumption \ref{assumption for activation function} with
$L = \phi_{\max }^{\prime}=1,\ L^{\prime} = 4\sqrt{3}/9  \text{ and }  \phi_{\max } = 1.$
\end{example}

Let  $(E, \rho)$  be a metric space with metric  $\rho $. A  $\delta $-cover of a set  $A \subset E$  with respect to  $\rho$  is a collection of points  $\left\{x_{1}, \cdots, x_{n}\right\} \subset A$  such that for every  $x \in A $, there exists  $i \in\{1, \cdots, n\}$  such that  $\rho\left(x, x_{i}\right) \leq \delta $. The  $\delta $-covering number  $\mathcal{N}(\delta, A, \rho)$  is the cardinality of the smallest  $\delta$  cover of the set  $A$  with respect to the metric  $\rho $. Equivalently, the  $\delta $-covering number  $\mathcal{N}(\delta, A, \rho)$  is the minimal number of balls  $B_{\rho}(x, \delta)$  of radius  $\delta$  needed to cover the set $A$.

Let $Q$ be any probability measure on $\Omega$ and $\|g\|_{*} = \sup_{x \in \Omega} \left|g(x)\right|$. Define
\begin{equation*} \label{}
\begin{aligned}
& \mathcal{G}_{m}^{1} := \left\{g: \Omega \rightarrow \mathbb{R}\ \left|\ g= \varphi^{2} v_{\theta}^{2} \text { where } v_{\theta} \in \widetilde{\mathcal{F}}_{m}\right.\right\} ,\\
& \mathcal{G}_{m}^{2} := \left\{g: \Omega \rightarrow \mathbb{R}\ \left|\ g=|\nabla \left(\varphi v_{\theta}\right)|^{2} + V\varphi^{2} v_{\theta}^{2} \text { where } v_{\theta} \in \widetilde{\mathcal{F}}_{m}\right.\right\} ,\\
& \mathcal{G}_{m}^{3} := \left\{g: \Omega \rightarrow \mathbb{R}\ \left|\ g= \varphi \psi v_{\theta} \text { where } v_{\theta} \in \widetilde{\mathcal{F}}_{m}\right.\right\}.
\end{aligned}
\end{equation*}

Thanks to Assumption \ref{assumption for activation function}, 
\begin{align}
& \max_{\theta \in \Theta}\left\| v_{\theta} \right\|_{*}  \leq \left|c\right| + \sum_{i=1}^{m}\left|\gamma_{i}\right| 
\sup _{w_{i} \in \Theta_{w}, t_{i} \in \Theta_{t}, x \in \Omega} \left|\phi\left(w_{i} \cdot x+t_{i}\right)\right| \leq  C + \Gamma  \phi_{\max } ,  \label{maximum of Linf norm for vtheta}  \\
& \max_{\theta \in \Theta}\left\|\left|\nabla v_{\theta}\right|\right\|_{*}  \leq \sum_{i=1}^{m}\left|\gamma_{i}\right|\left|w_{i}\right| \sup _{w_{i} \in \Theta_{w}, t_{i} \in \Theta_{t}, x \in \Omega} \left|\phi^{\prime}\left(w_{i} \cdot x+t_{i}\right)\right| \leq \Gamma W \phi_{\max }^{\prime} . \notag %\label{maximum of Linf norm for nabla vtheta}
\end{align}
Since $|\nabla \left(\varphi v_{\theta}\right)| \leq \left|\varphi\right| \left|\nabla v_{\theta}\right| + \left| v_{\theta}\right| \left|\nabla \varphi\right|$ and $0 \leq V \leq V_{\max}$, 
\begin{equation} \label{maximum of Linf norm for functions in Gm2}
\begin{aligned}
\sup _{g \in \mathcal{G}_{m}^{2}} \|g\|_{*} \leq \left[\Gamma W \phi_{\max}^{\prime}/d + \pi \left(C + \Gamma\phi_{\max}\right) \right]^{2} + V_{\max} \left(C+\Gamma \phi_{\max }\right)^{2} /d^{2}.
\end{aligned}
\end{equation}
The next proposition provides upper bounds for
$\mathcal{N}\left(\delta, \mathcal{G}_{m}^{i}/M_{i},\|\cdot\|_{L^{2}(Q)}\right)$, $i = 1,2,3$, where $M_{i}$ are some scaling parameters.
Given $C$, $\Gamma$,  $W$  and  $T$  in (\ref{widetilde mathcal Fm}), as in \cite{DRMlu2021priori}, we define
\begin{equation*} \label{mathcal M used for covering number}
\begin{aligned} 
\mathcal{M}(\delta, \Lambda, m, d):=  \frac{2 C \Lambda}{\delta} \left(\frac{3 \Gamma \Lambda}{\delta}\right)^{m} \left(\frac{3 W \Lambda}{\delta}\right)^{d m} \left(\frac{3 T \Lambda}{\delta}\right)^{m} .
\end{aligned}
\end{equation*}
\begin{proposition} \label{proposition: covering number for Gm1, Gm2 and Gm3}   \label{covering number for Gm1}  \label{covering number for Gm2}   \label{covering number for Gm3}
If  $0 \leq V \leq V_{\max}$ and %the activation function
$\phi$  satisfies Assumption \ref{assumption for activation function}. Then 
\[\mathcal{N}\left(\delta, \mathcal{G}_{m}^{i}/M_{i},\|\cdot\|_{L^{2}(Q)}\right) \leq \mathcal{M}\left(\delta, \Lambda_{i}/M_{i}, m, d\right)
\]
for $i = 1, 2, 3$, where
\begin{subequations}  \label{}
\begin{align*}
\Lambda_{1} &  := 2\left(C+\Gamma \phi_{\max }\right)\left(1+\phi_{\max }+2 L \Gamma\right) /d^{2}, \\   %\label{magnification Lambda1 for Gm1}
\Lambda_{2} & := 2\left[\Gamma W \phi_{\max }^{\prime}/d + \pi \left(C \! + \!  \Gamma\phi_{\max }\right)\right] \left[\Big((W \!  + \! \Gamma) \phi_{\max }^{\prime}+2 \Gamma W L^{\prime}\Big)/d + \pi\left(1 \! + \! \phi_{\max } \! + \! 2 L \Gamma\right) \right]\\
& \quad \ + 2V_{\max}\left(C+\Gamma \phi_{\max }\right)\left(1+\phi_{\max }+2 L \Gamma\right), \\  %\label{magnification Lambda1 for Gm2} 
\Lambda_{3} & := \|\psi\|_{L^{2}(Q)}\left(1+\phi_{\max }+2 L \Gamma\right)/d.   %\label{magnification Lambda3 for Gm3}
\end{align*}    
\end{subequations}
\end{proposition}
The proof is postponed to Appendix \ref{Appendix proof of Proposition: covering number for Gm1, Gm2 and Gm3}.
\subsection{Estimates of the expectation of suprema of empirical processes} \label{subsection: Estimates of the expectation of suprema of empirical processes}
%To begin with, we introduce \cite[Proposition 2.1]{Gin2001OnCO}.
Let $\{\epsilon_i\}_{i = 1}^{\infty}$ be independent Rademacher variables\footnote{A Rademacher variable $\epsilon$ is one that satisfies $\mathbf{P}(\epsilon = 1) = \mathbf{P}(\epsilon = -1) = 1/2.$} independent from  $\{X_i\}_{i = 1}^{\infty}$,  and let $F \geq \sup_{f \in \mathcal{F}}  \left|f\right|$  be a measurable envelope of the function class $\mathscr{F}$. Here, we call $\mathscr{F}$ a VC class if there exist some finite  $A \geq 3 \sqrt{e}$  and  $v \geq 1$ such that for all probability measures $Q$ and  $0 < \tau < 1$,
$\mathcal{N}\left(\tau \|F\|_{L^{2}(Q)}, \mathscr{F}, \|\cdot\|_{L^{2}(Q)}\right) \leq \left(A/\tau\right)^{v}.$
We firstly recall the following useful lemma.
\begin{lemma}{{\citep[Proposition 2.1]{Gin2001OnCO}}} \label{Proposition 2.1 in cite Gin2001OnCO, original version}
Let  $\mathscr{F}$  be a measurable uniformly bounded VC class. 
Let $U \geq \sup_{f \in \mathscr{F}}\|f\|_{L^{\infty}}$ and $\sigma^{2} \geq \sup_{f \in \mathscr{F}} E_{\mathcal{P}} f^{2}$  be such that  $0<\sigma \leq U$. 
Then there exists a universal constant $C$ such that for all  $n \in \mathbb{N} ,$
$$\mathbf{E}\left\|\sum_{i=1}^{n} \epsilon_{i} f\left(X_{i}\right)\right\|_{\mathscr{F}} \leq C\left[v U \ln \frac{A U}{\sigma} + \sigma \sqrt{v n \ln \frac{A U}{\sigma}}\right].$$
%where  $C$  is a universal constant.
\end{lemma}

%Combining this and taking $U = \sup_{f \in \mathcal{F}}  \left\|f\right\|_{L^{\infty}(\Omega)} = 1$ in  in \cite[Proposition 2.1]{Gin2001OnCO}, we have:
To bound $\mathbf{E}\left\|\mathcal{P}_{n}-\mathcal{P}\right\|_{\mathscr{F}}$ with $\mathscr{F}$ a subset of a rescaled class $\mathcal{G}_{m}^{i}/M_{i}$ ($1\leq i\leq 3$) in \S \ref{subsection: Bounding the covering numbers}, we need the following lemma, whose proof is deferred to Appendix \ref{appendix proof of Lemma: Prop 2.1 from Gin2001OnCO, U=2, centered}.

%Based on the standard symmetrization and the above verification, by standard symmetrization and taking $U = \sup_{f \in \bar{\mathcal{F}}}  \left\|f\right\|_{L^{\infty}} = 2$, 
% yields: the following result.
\begin{lemma} \label{Prop 2.1 from Gin2001OnCO, U=2, centered}
Let  $\mathscr{F}$  be a class of real valued measurable functions taking values in  $[-1, 1]$. 
Let  $F \leq 1$  be a measurable envelope of $\mathscr{F}$ and $\sup_{f \in \mathscr{F}} \operatorname{Var}_{\mathcal{P}} f \leq \sigma^{2} \leq 1$.
Assume that for all  $0 < \tau < 1$, there exists a universal net $\left\{f_{i}\right\}_{i = 1}^{M}$ for all probability measures $Q$ such that  $M \leq \left(A/\tau\right)^{v}$ and for any $f \in \mathscr{F}$, 
$$\min_{1 \leq  i \leq  M} \left\| f - f_{i} \right\|_{L^{2}(Q)} \leq \tau\|F\|_{L^{2}(Q)}.$$
Then, there exists a universal $C$ such that for all  $n \in \mathbb{N}$,
$$\mathbf{E} \left\|\mathcal{P}_{n}-\mathcal{P}\right\|_{\mathscr{F}} \leq C\left( \frac{v}{n} \ln \frac{A}{\sigma} +  \sigma \sqrt{\frac{v}{n} \ln \frac{A}{\sigma}}\right).$$
%where  $C$ is a universal constant.
\end{lemma}
\begin{remark}
%If $\mathscr{F}$ is taken to be a subset of $\mathcal{G}_{m}^{i}/M_{i}$ ($1\leq i\leq 3$),
If $\mathscr{F} \subset \mathcal{G}_{m}^{i}/M_{i}$ ($1\leq i\leq 3$),
the universal nets exist due to Proposition \ref{Proposition: covering number of parameter space Theta} and the procedure by which we control the covering numbers.
The universal nets correspond to the nets for the 
parameter space $\Theta$.
\end{remark}

Recall $K^{\mathscr{F}}_{n, q}(r, s]$ defined in (\ref{the quantity En,q(r, s]}) and $\mathcal{K}_{n}(\mathscr{F}, r) = K^{\mathscr{F}}_{n, \sqrt{2}}(r, 1]$. 
Corollary \ref{estimates of En,q(r, s], U = 1} is a direct consequence of Lemma \ref{Prop 2.1 from Gin2001OnCO, U=2, centered}. 
We postpone its proof to Appendix \ref{appendix proof of Corollary: estimates of En,q(r, s], U = 1}.
\begin{corollary} \label{estimates of En,q(r, s], U = 1}
Let  $\mathscr{F}$  satisfy the assumptions in Lemma \ref{Prop 2.1 from Gin2001OnCO, U=2, centered} and all functions in  $\mathscr{F}$  take values in  $[0, 1]$.
Then, for all  $n \in \mathbb{N}$,
\begin{equation*} \label{}
\begin{aligned} 
\mathcal{K}_{n}(\mathscr{F}, r) & \leq C \left( \frac{v}{n r^{2}} \ln \frac{A}{r} +   \sqrt{\frac{v}{n r^{2}} \ln \frac{A}{r}}\right) ,
\end{aligned}
\end{equation*}
where $C \geq 1$ is an absolute constant. In particular, if $\frac{v}{n r^{2}} \ln \frac{A}{r} \leq 1,$  then  $$\mathcal{K}_{n}(\mathscr{F}, r) \leq  2C \sqrt{\frac{v}{n r^{2}} \ln \frac{A}{r}} .$$
\end{corollary}

\subsection{Bounding \texorpdfstring{$\mathcal{K}_{n}$}{} in the statistical error}
We take $\mathcal{F} = \varphi \widetilde{\mathcal{F}}_{m}$  and $\widetilde{\mathcal{F}}_{m} = \mathcal{F}_{\mathrm{SP}_{\tau}, m}(B)$ with $\tau = 9\sqrt{m}$, and consider
\begin{equation*} \label{function classes defined for bounding generalization error, 2}
\begin{aligned}
& \mathcal{G}_{1} = \mathcal{G}_{\mathrm{SP}_{\tau}, m, 1}(B) := \left\{g: g=\varphi^2 v^{2} \text { where } v \in \mathcal{F}_{\mathrm{SP}_{\tau}, m}(B)\right\}, \\
& \mathcal{G}_{2} = \mathcal{G}_{\mathrm{SP}_{\tau}, m, 2}(B) := \left\{g: g=|\nabla \left(\varphi v\right)|^{2}+V|\varphi v|^{2} \text { where } v \in \mathcal{F}_{\mathrm{SP}_{\tau}, m}(B)\right\}, \\
& \mathcal{F}_{j} = \mathcal{F}_{\mathrm{SP}_{\tau}, m, j}(B) := \left\{g: g = \varphi \psi_{j}  v  \text { where } v \in \mathcal{F}_{\mathrm{SP}_{\tau}, m}(B)\right\}, \quad \text{for }  j = 1, 2, \ldots, k-1.
\end{aligned}
\end{equation*}
Note that  $\mathcal{F}_{\mathrm{SP}_{\tau}, m}(B)$ coincides with the set  $\widetilde{\mathcal{F}}_{m} $  defined in (\ref{widetilde mathcal Fm}) with 
\begin{equation} \label{concrete representation of constants C Gamma W T}
C = B,\ \Gamma = 4B,\ W = 1,\ T = 1.
\end{equation}
%Also, $M_{\mathcal{F}}^{2}$, $M_{\mathcal{G}_{2}}$, $\mu_{j} M_{\mathcal{F}}$ coincide with $M_{1}$, $M_{2}$, $M_{3, j}$ respectively.
By (\ref{maximum of Linf norm for vtheta}) and (\ref{maximum of Linf norm for functions in Gm2}), with (\ref{constants in assumption for Softplus activation function}) and (\ref{concrete representation of constants C Gamma W T}),  we take $M_{\mathcal{F}}$, $M_{\mathcal{G}_{2}}$ as
\begin{equation} \label{upper bound of maximum value of hypothesis class with Softplus activation}
\begin{aligned}
\sup _{g \in \mathcal{F}} \|g\|_{L^{\infty}(\Omega)} 
\leq 9.5B/d =: M_{\mathcal{F}}, \quad
\sup _{g \in \mathcal{G}_{2}} \|g\|_{L^{\infty}(\Omega)} \leq 34^{2}B^{2} + V_{\max } \left(9.5B/d\right)^{2} =: M_{\mathcal{G}_{2}} .
\end{aligned}
\end{equation}
\iffalse
\begin{equation} \label{upper bound of maximum value of hypothesis class with Softplus activation}
\begin{aligned}
\sup _{g \in \mathcal{F}} \|g\|_{L^{\infty}(\Omega)} & %\leq \left(C+\Gamma \phi_{\max }\right)/d 
\leq 9.5B/d =: M_{\mathcal{F}} ,\\
\sup _{g \in \mathcal{G}_{2}} \|g\|_{L^{\infty}(\Omega)} %& \leq \left[\Gamma W \phi_{\max}^{\prime}/d + \pi \left(C + \Gamma\phi_{\max}\right) \right]^{2} +  V_{\max }\left(C+\Gamma \phi_{\max }\right)^{2}/d^{2}\\ 
&  \leq 34^{2}B^{2} + V_{\max } \left(9.5B/d\right)^{2} =: M_{\mathcal{G}_{2}} .
\end{aligned}
\end{equation}
\fi
%and set rescaling parameters $M_{1} = M_{\mathcal{F}}^{2}$, $M_{2} = M_{\mathcal{G}_{2}}$, $M_{3, j} = \mu_{j} M_{\mathcal{F}}$. 

Applying Lemma \ref{Prop 2.1 from Gin2001OnCO, U=2, centered} or Corollary \ref{estimates of En,q(r, s], U = 1} to certain rescaled function classes, we obtain the estimates for $\mathcal{K}_{n}$.

%We summarize the statistical error bounds as follows.
\begin{theorem} \label{Thm:  bounds for all spaces needed in oracle inequality for the generalization error}
Assume that  $0 \leq V \leq V_{\max }$  and  $\|\psi_{j}\|_{L^{\infty}(\Omega)} \leq \mu_{j}$ for $1\leq j \leq k-1$. 
Consider the sets $\mathcal{F} = \varphi \mathcal{F}_{\mathrm{SP}_{\tau}, m}(B)$, $\mathcal{G}_{1} = \mathcal{G}_{\mathrm{SP}_{\tau}, m, 1}(B)$, $\mathcal{G}_{2} = \mathcal{G}_{\mathrm{SP}_{\tau}, m, 2}(B)$  and  $\mathcal{F}_{j} = \mathcal{F}_{\mathrm{SP}_{\tau}, m, j}(B)$   with  $\tau=9\sqrt{m}$ and $B \geq 1$. Assume that  $n$  is large enough such that 
\begin{equation} \label{Global simplifying condition for statistical error}
C_{0} \frac{m B^{2} \left(1 + V_{\max } \right)}{n r^{2}} \ln \frac{ B \left(1 + \sqrt{m}/d\right) \left(1 + V_{\max}\right)}{r d } \leq 1,
\end{equation}
where $C_{0}$ is an absolute constant.
Then,  there exists an absolute constant
$C$ such that
\begin{subequations}
\begin{align}
& \mathcal{K}_{n}\left(\mathcal{G}_{1}/M_{\mathcal{F}}^{2}, r/M_{\mathcal{F}}\right) \leq  C   \sqrt{ \frac{mB^{2}}{n d r^{2}} \ln \frac{B}{r d}},  
\label{bound for mathcalK_n(mathcalG_1/M_mathcalF^2)}  \\
& \mathcal{K}_{n}\left(\mathcal{G}_{2}/M_{\mathcal{G}_{2}}, r\sqrt{\lambda_{1}/M_{\mathcal{G}_{2}}}\right)   \leq C   \sqrt{\frac{m B^{2} \left(1 + V_{\max } \right)}{n r^{2}} \ln \frac{ B \left(1 + \sqrt{m}/d\right) \left(1 + V_{\max}\right)}{r d}},   
\label{bound for mathcalK_n(mathcalG_2/M_mathcalG_2)}  \\
& \mathcal{K}_{n}\left(\frac{\mathcal{F}_{j}}{2\mu_{j} M_{\mathcal{F}}} + \frac{1}{2},  \sqrt{\frac{r}{4 \mu_{j} M_{\mathcal{F}}}}\right)   \leq  C\left[ \frac{m\mu_{j} B}{n  r} \ln \left(\frac{\mu_{j} B}{r d}\right) +  \sqrt{\frac{m\mu_{j} B}{n r} \ln \left(\frac{\mu_{j} B}{r d}\right)}\right].
\label{bound for mathcalK_n(mathcalF_j/2mu_j M_mathcalF + 1/2)}  
\end{align}
\end{subequations}
\end{theorem}
The proof is postponed to Appendix \ref{Appendix proof of Thm:  bounds for all spaces needed in oracle inequality for the generalization error}.

\section{Proof of the main generalization theorem}
\label{Proof of the main generalization theorem}
Combining Theorem \ref{Thm: oracle inequality for the generalization error}, Theorem \ref{Thm: bounding the approximation error Lk(u in F)-lambdak} and Theorem \ref{Thm:  bounds for all spaces needed in oracle inequality for the generalization error}, we shall prove Theorem \ref{Main generalization theorem}.

\begin{proof}[Proof of Theorem \ref{Main generalization theorem}]
%Combine Theorem \ref{Thm: oracle inequality for the generalization error}, Theorem \ref{Thm: bounding the approximation error Lk(u in F)-lambdak} and Theorem \ref{Thm:  bounds for all spaces needed in oracle inequality for the generalization error}. 
Thanks to Theorem \ref{Thm: oracle inequality for the generalization error}, taking $\mathcal{F}_{> r} = \varphi \mathcal{F}_{\mathrm{SP}_{\tau}, m}(B) \cap \{ \|u\|_{L^{2}(\Omega)} > r \}$  with $B =  \left(1+\frac{2 d}{\pi s}\right) \|u^{*}\|_{\mathfrak{B}^{s}(\Omega)}$  and  $\tau=9\sqrt{m}$, if $2\xi_{1}+\xi_{2}<1/2$, for any $u_{\mathcal{F}} \in \mathcal{F}_{> r},$
\begin{equation} \label{restate oracle ineq for generalization error in pf of Main generalization thm}
\begin{aligned}
L_{k}\left(u_{n}^{m}\right)-\lambda_{k}
& \leq  4\lambda_{k}\left(\xi_{1}+\xi_{2}\right)   + \beta\left(k \xi_{3}^{2} + 4 \sqrt{k} \xi_{3}\right)  + 3 \left(L_{k}\left(u_{\mathcal{F}}\right)-\lambda_{k}\right).
\end{aligned}
\end{equation}
By Theorem \ref{Thm: u H1 approximation by varphi Softplus networks}, there exists $u_{\mathcal{F}} \in \varphi \mathcal{F}_{\mathrm{SP}_{\tau}, m}(B)$ such that 
$\left\|u^{*}-u_{\mathcal{F}}\right\|_{H^{1}(\Omega)} \leq 64B/\sqrt{m}.$ 
Since $64B/\sqrt{m} \leq 1/2$ and $0 < r < 1/2$, $\|u_{\mathcal{F}}\|_{L^{2}(\Omega)} \geq \|u^{*}\|_{L^{2}(\Omega)} - 64B/\sqrt{m} > r.$
Thus, $u_{\mathcal{F}} \in \mathcal{F}_{> r}.$
By Theorem \ref{Thm: bounding the approximation error Lk(u in F)-lambdak}, 
\begin{equation} \label{restate approximation error in pf of Main generalization thm}
\begin{aligned}
L_{k}\left(u_{\mathcal{F}}\right)-\lambda_{k} \leq 64 \left(3\max \left\{1, V_{\max }\right\}  + 7\lambda_{k} + 5 \beta\right) B/\sqrt{m} .
\end{aligned}
\end{equation}
Substituting the bounds in Theorem \ref{Thm:  bounds for all spaces needed in oracle inequality for the generalization error} into the expression of $\{\xi_{i}(n, r, \delta)\}_{i = 1}^{3}$ gives that if (\ref{Global simplifying condition for statistical error}) holds, then
\begin{equation} \label{statistical error bound for xi1, xi2 and xi3}
\begin{aligned}
\xi_{1}(n, r, \delta) & \leq  \frac{C B}{ r d \sqrt{n}} \left( \sqrt{ m d \ln \frac{ B}{rd}} + \sqrt{ \ln(1/\delta)} \right), \\
\xi_{2}(n, r, \delta) & \leq C\Upsilon_{1}(n,m,B,r,\delta),\\
%\xi_{2}(n, r, \delta) & \leq  \frac{C B}{r}\sqrt{\frac{1 + V_{\max}}{n}}  \left[ \sqrt{m \ln \frac{ B \left(1 + \sqrt{m}/d\right) \left(1 + V_{\max} \right)}{r d}}  +  \sqrt{\frac{\ln(1/\delta)}{d}}\right],\\
\xi_{3}(n, r, \delta) & \leq  C \Upsilon_{2}(n,m,k,B,\bar{\mu}_{k},r,\delta)/\sqrt{k},
%\xi_{3}(n, r, \delta) & \leq  C \sqrt{\frac{\bar{\mu}_{k} B}{n  r}}\left[ \sqrt{m \ln \left(\frac{\bar{\mu}_{k} B}{r d}\right)} + \sqrt{\frac{\ln(k/\delta)}{d}} \right],
\end{aligned}
\end{equation}
\iffalse
%%%%
\begin{equation*} \label{}
\begin{aligned}
\xi_{3, j}(n, r, \delta) &  \leq  \frac{C \mu_{j} B}{n  r} \left[m \ln \left(B\sqrt{\frac{\mu_{j}}{r d}}\right) + \frac{\ln( k/\delta)}{d} \right] + C \sqrt{\frac{\mu_{j} B}{n  r}}\left[ \sqrt{m \ln \left(B\sqrt{\frac{\mu_{j}}{r d}}\right)} + \sqrt{\frac{\ln( k/\delta)}{d}} \right],
\end{aligned}
\end{equation*}
\fi
where we have used (\ref{upper bound of maximum value of hypothesis class with Softplus activation}), $\lambda_{1}\geq d\pi^{2}$ and (\ref{Simplifying condition for xi3 term}). 
%$$\sqrt{\frac{k \mu_{\max,k} B}{n  r}}\left[ \sqrt{m \ln \left(B\sqrt{\frac{\mu_{\max,k}}{r d}}\right)} + \sqrt{\frac{\ln( k/\delta)}{d}} \right] \leq 1.$$
Note that the bound for $\xi_{1}$ is smaller than the bound for $\xi_{2}$. Hence, there exists an absolute constant $C$ such that  %$$\frac{C_{a} B}{r}\sqrt{\frac{1 + V_{\max}}{n}}  \left[ \sqrt{m \ln \frac{ B \left(1 + \sqrt{m}/d\right) \left(1 + V_{\max} \right)}{r d}}  +  \sqrt{\frac{\ln(1/\delta)}{d}}\right] \leq 1,$$ which 
(\ref{ineq: assumption in Main generalization theorem to ensure xi1+xi2<1/2}) guarantees both (\ref{Global simplifying condition for statistical error}) and  $2\xi_{1} + \xi_{2} \leq 1/2$.
%Finally, 
A combination of (\ref{restate oracle ineq for generalization error in pf of Main generalization thm}), (\ref{restate approximation error in pf of Main generalization thm}) and (\ref{statistical error bound for xi1, xi2 and xi3}) completes the proof.
\end{proof}

\section{Solution theory in spectral Barron Spaces} \label{Section: Solution theory in spectral Barron Spaces}
In this section we aim to prove the regularity of the eigenfunctions 
stated in Theorem \ref{Thm: Regularity of eigenfunctions}. 

Without loss of generality, we assume that $V\ge 0$.
%Since functions in $\mathfrak{C}^{s}(\Omega)$ are bounded and %one can always add a constant to $V$ wuthout changing the %eigenfunctions, it suﬀices to prove for $V \geq 0$.
%In \S \ref{subsection: Boundness of inverse of Schrödinger operator in spectral Barron spaces}, We firstly show that the inverse of the Schrödinger operator $\mathcal{H}^{-1}: \mathfrak{B}^{s}(\Omega) \rightarrow \mathfrak{B}^{s+2}(\Omega)$ is bounded in Theorem \ref{Thm: compactness of inverse of Schrödinger Operator} and Corollary \ref{corollary: inverse of Schrödinger operator is compact}. 
%Secondly, we prove regularity estimates for the eigenfunctions through a bootstrap argument. %in \S \ref{subsection: Regularity of the eigenfunctions}.
%\subsection{Compactness of \texorpdfstring{$\mathcal{H}^{-1}$}{} in spectral Barron spaces} \label{subsection: Boundness of inverse of Schrödinger operator in spectral Barron spaces}
%\subsection{Boundedness of \texorpdfstring{$\mathcal{H}^{-1}: \mathfrak{B}^{s}(\Omega) \rightarrow \mathfrak{B}^{s+2}(\Omega)$}{}} \label{subsection: Boundness of inverse of Schrödinger operator in spectral Barron spaces}
Consider the static Schrödinger equation with Dirichlet boundary condition
\begin{equation}
\begin{cases} \label{static Schrödinger equation}
\mathcal{H}u(x) = -\Delta u(x) + V(x) u(x) = f(x), & x \in \Omega , \\
u(x) = 0, & x \in \partial\Omega,
\end{cases}
\end{equation}
where $f \in L^2(\Omega)$.
To show the boundedness of $\mathcal{H}^{-1}$, we prove an estimate for the solution of (\ref{static Schrödinger equation}).
\begin{theorem} \label{Thm: compactness of inverse of Schrödinger Operator}
Assume that  $f \in \mathfrak{B}^{s}(\Omega)$  and   $V \in \mathfrak{C}^{s}(\Omega) $  with  $s \geq 0$  and  $V(x) \geq 0$ for every $x \in \Omega$. Then 
%the static Schrödinger problem 
(\ref{static Schrödinger equation}) has a unique solution  $u \in \mathfrak{B}^{s+2}(\Omega)$. Moreover, there exists  $C>0$  depending on  $V$  and  $d$  such that
$$\|u\|_{\mathfrak{B}^{s+2}(\Omega)} \leq C(V, d) \|f\|_{\mathfrak{B}^{s}(\Omega)} .$$
\end{theorem}
\begin{corollary} \label{corollary: inverse of Schrödinger operator is compact}
Assume that  $V \in \mathfrak{C}^{s}(\Omega) $ with $V(x) \geq  0$ for every $x \in \Omega$. Let   $\mathcal{S}:=\mathcal{H}^{-1}$  be the inverse of the Schrödinger operator. Then the operator  $\mathcal{S}: \mathfrak{B}^{s}(\Omega) \rightarrow \mathfrak{B}^{s+2}(\Omega)$  is bounded and $\mathcal{S}$ is a compact operator on $\mathfrak{B}^{s}(\Omega)$.
\end{corollary}
The proof of the above two results may be found in Appendix \ref{Appendix proof of Thm: compactness of inverse of Schrödinger Operator}.
%\subsection{Regularity of the eigenfunctions} \label{subsection: Regularity of the eigenfunctions}
In what follows, we shall use the notations 
%about the series expansions 
defined in Appendix \ref{section: Some useful facts on sine and cosine series}.
To prove Theorem \ref{Thm: Regularity of eigenfunctions}, we start with $s=0$.
\begin{proposition} \label{Proposition: Regularity of eigenfunctions, case s=0}
If  $V \in \mathfrak{C}^{0}(\Omega)$, then any eigenfunction of Problem (\ref{eq1}) lies in $\mathfrak{B}^{2}(\Omega)$.
\end{proposition}
\begin{proof}
By the definition of $\mathfrak{C}^{s}(\Omega)$ and Lemma \ref{tildeVk represented by checkVk}, $V \in \mathfrak{C}^{0}(\Omega)$ if and only if $\tilde{V}_{e} \in \ell^{1}\left(\mathbb{Z}^{d}\right)$. Specifically,
\begin{equation*} \label{}
\begin{aligned}
\|\tilde{V}_{e}\|_{\ell^{1}\left(\mathbb{Z}^{d}\right)} & 
%= \sum_{k \in \mathbb{Z}^{d}} |\tilde{V}_{e}(k)| 
%= \sum_{k \in \mathbb{Z}^{d}} |\tilde{V}_{e}(|k|)| 
= \sum_{k \in \mathbb{N}_{0}^{d}}  2^{\sum_{i=1}^{d} \mathbf{1}_{k_{i} \neq 0}} |\tilde{V}_{e}(k)|
%& = \sum_{k \in \mathbb{N}_{0}^{d}}  \beta_{k}^{-1} 2^{\mathbf{1}_{k \neq 0}} |\tilde{V}_{e}(k)| 
= \sum_{k \in \mathbb{N}_{0}^{d}}   2^{\mathbf{1}_{k \neq 0}} |\check{V}(k)|  
\leq %2\sum_{k \in \mathbb{N}_{0}^{d}} |\check{V}(k)| = 
\|V\|_{\mathfrak{C}^{0}(\Omega)},\\
\|V\|_{\mathfrak{C}^{0}(\Omega)} & %= \sum_{k \in \mathbb{N}_{0}^{d}}2|\check{V}(k)| 
= \sum_{k \in \mathbb{N}_{0}^{d}}2 \beta_{k}^{-1} |\tilde{V}_{e}(k)| 
%= \sum_{k \in \mathbb{N}_{0}^{d}}  2^{\sum_{i=1}^{d} \mathbf{1}_{k_{i} \neq 0} - \mathbf{1}_{k \neq 0} + 1} |\tilde{V}_{e}(k)|\\
\leq %2\sum_{k \in \mathbb{Z}^{d}} |\tilde{V}_{e}(k)| = 
2\|\tilde{V}_{e}\|_{\ell^{1}\left(\mathbb{Z}^{d}\right)},
\end{aligned} 
\end{equation*}
where we have used $\beta_{k} = 2^{\mathbf{1}_{k \neq 0}-\sum_{i=1}^{d} \mathbf{1}_{k_{i} \neq 0}} \in [2^{1-d}, 1]$.
Since $\tilde{V}_{e} \in \ell^{1}\left(\mathbb{Z}^{d}\right)$, Young's inequality implies that for $p \in [1, 2]$, if   $\left\{a_{k}\right\}_{k \in \mathbb{Z}^{d}} \in \ell^{p}\left(\mathbb{Z}^{d}\right)$, then  $(\tilde{V}_{e} * a) \in \ell^{p}\left(\mathbb{Z}^{d}\right)$  and 
\begin{equation} \label{Young's inequality for convolution}
\begin{aligned}
\|\tilde{V}_{e} * a\|_{\ell^{p}} \leq \|\tilde{V}_{e}\|_{\ell^{1}}\|a\|_{\ell^{p}} .
\end{aligned} 
\end{equation}
For any eigenfunction $\psi \in H^{1}_{0}(\Omega)$ and its eigenvalue $\lambda$ such that $\mathcal{H}\psi = \lambda\psi$, let $\psi_{o}$ be the odd extension of $\psi$ on $[-1,1]^d$, which may be regarded as an eigenfunction of the Schrödinger operator $\widetilde{\mathcal{H}} := -\Delta + V_{e}$ with periodic boundary conditions: $\widetilde{\mathcal{H}}\psi_{o} = -\Delta \psi_{o} + V_{e} \psi_{o} = \lambda \psi_{o}.$

Next, we claim that if $\tilde{\psi}_{o} \in \ell^{p}$ with some $p\in [1, 2]$, then $(\pi^2 |k|_{2}^2 + 1)\tilde{\psi}_{o} \in \ell^{p}$. Notice that 
\begin{equation*} \label{}
\begin{aligned} 
\tilde{\psi}_{o} & = \mathscr{F}\left[ (-\Delta + I)^{-1} (\widetilde{\mathcal{H}} + I - V_{e}) \psi_{o} \right] \\
& =  \left(\pi^2 |k|_{2}^2 + 1\right)^{-1} \left(\lambda + 1\right)\tilde{\psi}_{o} - \left(\pi^2|k|_{2}^{2}+1\right)^{-1}(\tilde{\psi}_{o} * \tilde{V}_{e}),
\end{aligned} 
\end{equation*}
where $\mathscr{F}$  represents the Fourier transform. 
Since  $\tilde{V}_{e} \in \ell^{1}\left(\mathbb{Z}^{d}\right), \tilde{\psi}_{o} \in \ell^{p}\left(\mathbb{Z}^{d}\right)$, we conclude  $(\lambda+1) \tilde{\psi}_{o} \in \ell^{p}\left(\mathbb{Z}^{d}\right)$ and $(\tilde{V}_{e} * \tilde{\psi}_{o}) \in \ell^{p}\left(\mathbb{Z}^{d}\right)$ from (\ref{Young's inequality for convolution}). Hence  $$\left(\lambda + 1\right)\tilde{\psi}_{o} - (\tilde{\psi}_{o} * \tilde{V}_{e}) = \left(\pi^2|k|_{2}^{2}+1\right) \tilde{\psi}_{o} \in l^{p}\left(\mathbb{Z}^{d}\right).$$
This proves the claim.

Now we complete the proof through a bootstrap argument. Since $\psi \in L^2(\Omega)$, then $\psi_{o} \in L^2(\widetilde{\Omega})$ and its  Fourier transform $\tilde{\psi}_{o}  \in \ell^{2}\left(\mathbb{Z}^{d}\right).$ 
It follows from the above claim that $\left(\pi^2|k|_{2}^{2}+1\right) \tilde{\psi}_{o} \in \ell^{2}\left(\mathbb{Z}^{d}\right).$  
For all $r > d/2$, $\left(\pi^2|k|_{2}^{2}+1\right)^{-1} \in \ell^{r}\left(\mathbb{Z}^{d}\right).$ By Hölder's inequality, as long as $q^{-1} < 2/d + 1/2$ and $q \geq 1$, there exists $r > d/2$ such that $q^{-1} = r^{-1} + 2^{-1}$ and $$   \|\tilde{\psi}_{o}\|_{\ell^{q}\left(Z^{d}\right)} \leq \left\|\left(\pi^2|k|_{2}^{2}+1\right)^{-1}\right\|_{\ell^{r}\left(Z^{d}\right)}\left\|\left(\pi^2|k|_{2}^{2}+1\right) \tilde{\psi}_{o}\right\|_{\ell^{2}\left(Z^{d}\right)} < \infty.$$ Thus $\tilde{\psi}_{o} \in \ell^{q}\left(\mathbb{Z}^{d}\right).$ By repeating this argument $j$ times we see that $\tilde{\psi}_{o} \in \ell^{q}\left(\mathbb{Z}^{d}\right)$ as long as $q^{-1} < 2j/d + 1/2$ and $q \geq 1$. 
Choosing $j$ properly, we conclude that $\tilde{\psi}_{o} \in \ell^{1}\left(\mathbb{Z}^{d}\right)$. Repeating the claim again, we have $\left(\pi^2|k|_{2}^{2}+1\right)\tilde{\psi}_{o} \in \ell^{1}\left(\mathbb{Z}^{d}\right).$ 
Using Lemma \ref{tildeuk represented by hatuk}, we get 
\begin{equation*} \label{}
\begin{aligned} 
\|\psi\|_{\mathfrak{B}^{2}(\Omega)} & 
%= \sum_{k \in \mathbb{N}_{+}^{d}}\left(1+\pi^{2}|k|_{1}^{2}\right)|\hat{\psi}(k)| 
= \sum_{k \in \mathbb{N}_{+}^{d}}\left(1+\pi^{2}|k|_{1}^{2}\right) 2^d |\tilde{\psi}_{o}(k)| 
= \sum_{k \in \mathbb{Z}^{d}}\left(1+\pi^{2}|k|_{1}^{2}\right) |\tilde{\psi}_{o}(k)| .
%\leq \sum_{k \in \mathbb{Z}^{d}}\left(1 + d \pi^{2}|k|^{2}\right) |\tilde{\psi}_{o}(k)| 
%\leq d \|\left(1 +  \pi^{2}|k|_{2}^{2}\right)\tilde{\psi}_{o}\|_{\ell^{1}\left(Z^{d}\right)}.
\end{aligned} 
\end{equation*}
Hence, $\|\psi\|_{\mathfrak{B}^{2}(\Omega)} \leq d \|\left(1 +  \pi^{2}|k|_{2}^{2}\right)\tilde{\psi}_{o}\|_{\ell^{1}\left(Z^{d}\right)}$ and then $\psi \in \mathfrak{B}^{2}(\Omega)$. 
\end{proof}

With the aid of Proposition \ref{Proposition: Regularity of eigenfunctions, case s=0} and Corollary \ref{corollary: inverse of Schrödinger operator is compact}, we are ready to prove %the regularity result stated in 
Theorem \ref{Thm: Regularity of eigenfunctions}.
%Since adding a constant to $V$ does not change the eigenfunctions,  without loss of generality we assume $V\geq 0$.

\begin{proof}[Proof of Theorem \ref{Thm: Regularity of eigenfunctions}]
Note that $\mathfrak{B}^{r}(\Omega) \hookrightarrow \mathfrak{B}^{s}(\Omega)$ and $\mathfrak{C}^{r}(\Omega) \hookrightarrow \mathfrak{C}^{s}(\Omega)$ for $0\leq r \leq s$.  
Take any eigenmode $(\lambda, \psi)$ of Problem (\ref{eq1}) such that $\mathcal{H}\psi = \lambda\psi$. Since $V\in \mathfrak{C}^{s}(\Omega)$ with $s\geq 0$,  $V \in \mathfrak{C}^{0}(\Omega)$ and thus $\psi \in \mathfrak{B}^{2}(\Omega)$ according to Proposition \ref{Proposition: Regularity of eigenfunctions, case s=0}. 
For any $0\leq r \leq s$, $V \in \mathfrak{C}^{r}(\Omega)$, it follows from Corollary \ref{corollary: inverse of Schrödinger operator is compact} that  $\mathcal{S}: \mathfrak{B}^{r}(\Omega) \rightarrow \mathfrak{B}^{r+2}(\Omega)$  is bounded. Notice that $\psi = \lambda \mathcal{S} \psi$. Hence $\psi \in \mathfrak{B}^{2}(\Omega)$ implies $\psi \in \mathfrak{B}^{\min (s+2, 4)}(\Omega)$. By repeating this argument $j$ times we see that $\psi \in \mathfrak{B}^{\min (s+2, 2j+2)}(\Omega)$. When $j$ is large enough so that $2j+2 > s+2$, we obtain $\psi\in\mathfrak{B}^{s+2}(\Omega)$, which completes the proof.
\end{proof}

% Acknowledgements and Disclosure of Funding should go at the end, before appendices and references

\acks{The work of Ming was supported by the National Natural Science Foundation of China under the Grant No. 12371438.}

% Manual newpage inserted to improve layout of sample file - not
% needed in general before appendices/bibliography.

\newpage

\appendix

\section{Stability estimate of the \texorpdfstring{$k$}{}-th eigenfunction} \label{section: Stability estimate of the k-th eigenfunction}
In this part, we prove the stability estimates Proposition \ref{proposition: error between Rayleigh quotient and lambdak bounded by energy excess} and Proposition \ref{Generalized proposition 2.1}.
Define $P$ as the orthogonal projection operator from $L^2(\Omega)$ to $U_{k}$ and $P^{\perp}$ as the orthogonal projection  from $L^2(\Omega)$ to $U_{k}^{\perp}$. 
Recall that $\{\psi_{j}\}_{j = 1}^{k-1}$ are the first $k-1$ normalized orthogonal eigenfunctions.
For $u \in H^{1}(\Omega)$, we write %consider the orthogonal decomposition 
\begin{equation*}
u = P u + P^{\perp} u = P u + w + z, 
\end{equation*}
where
\begin{equation*} 
w := \sum_{j=1}^{k-1}\left\langle u, \psi_{j}\right\rangle \psi_{j}, \quad z := P^{\perp} u - w.
\end{equation*}
Note that $w$ is the orthogonal projection of $u$ onto the subspace $W_{k} = \operatorname{span}\left\{\psi_{1}, \psi_{2}, \ldots, \psi_{k-1}\right\}$ and $z$ is the orthogonal projection of $u$ onto the subspace $Z_{k} = W_{k}^{\perp} \cap \operatorname{ker}\left( \mathcal{H} - \lambda_{k} I\right)^{\perp}.$
\iffalse
\begin{equation*}
Z_{k} = \operatorname{span}\left\{\psi_{1}, \psi_{2}, \ldots, \psi_{k-1}\right\}^{\perp} \cap \operatorname{ker}\left( \mathcal{H} - \lambda_{k} I\right)^{\perp}.
\end{equation*}
\fi
Recall that $\lambda_{k^{\prime}}$ is the first eigenvalue of $\mathcal{H}$ that is strictly greater than $\lambda_{k}$. 
For any $z \in Z_{k}$,
\begin{equation}\label{ineq for z in Zk}
\langle z, \mathcal{H} z\rangle \geqslant \lambda_{k^{\prime}}\|z\|_{L^{2}(\Omega)}^{2} \geqslant \lambda_{k}\|z\|_{L^{2}(\Omega)}^{2}.
\end{equation}
Notice that $\mathcal{H}$ leaves the three orthogonal subspaces $W_{k},$ $U_{k}$, $Z_{k}$ invariant. Therefore,
\begin{equation} \label{langle u, mathcalH u rangle, orthogonal decomposition}
\begin{aligned}
\langle u, \mathcal{H} u\rangle & =  \langle w, \mathcal{H} w\rangle + \langle P u, \mathcal{H} P u\rangle + \langle z, \mathcal{H} z\rangle  \\
& = \sum_{j=1}^{k-1}  \lambda_{j} \left\langle u, \psi_{j}\right\rangle^{2} + \lambda_{k}\langle P u,  P u\rangle + \langle z, \mathcal{H} z\rangle,
\end{aligned}
\end{equation}
Using the definition (\ref{the first definition of the loss function Lk(u)}), the decomposition  (\ref{langle u, mathcalH u rangle, orthogonal decomposition}) and
\begin{equation} \label{langle u, u rangle, orthogonal decomposition}
\begin{aligned}
\|u\|_{L^{2}(\Omega)}^{2} 
%= \langle u, u\rangle 
= \langle w, w\rangle + \langle P u, P u\rangle + \langle z, z\rangle,
\end{aligned}
\end{equation}
 we obtain 
\begin{equation} \label{(Lk(u)-lambdak)|u|2, orthogonal decomposition}
\begin{aligned}
\left(L_{k}(u)-\lambda_{k}\right)\|u\|_{L^{2}(\Omega)}^{2} & = \langle u, \mathcal{H} u\rangle + \beta \sum_{j=1}^{k-1}\left\langle u, \psi_{j}\right\rangle^{2} - \lambda_{k} \langle  u,  u\rangle\\
& = \sum_{j=1}^{k-1} \left( \beta + \lambda_{j} - \lambda_{k}\right)\left\langle u, \psi_{j}\right\rangle^{2} + \langle z, (\mathcal{H}-\lambda_{k}) z\rangle.
\end{aligned}
\end{equation}
This identity is key to prove Proposition \ref{proposition: error between Rayleigh quotient and lambdak bounded by energy excess} and Proposition \ref{Generalized proposition 2.1}.
%This identity is the starting point of the proof.

\begin{proof}[Proof of Proposition \ref{proposition: error between Rayleigh quotient and lambdak bounded by energy excess}]
%Since $\left\{\lambda_{j}\right\}_{j=1}^{\infty}$ are arranged in an increasing order, 
Note that 
$0\leq\lambda_{k} - \lambda_{j} \leq \lambda_{k} - \lambda_{1}$ for each $1\leq j \leq k-1$. 
It follows from  (\ref{langle u, mathcalH u rangle, orthogonal decomposition}), (\ref{langle u, u rangle, orthogonal decomposition}) and (\ref{ineq for z in Zk}) that
\begin{equation*}
\begin{aligned}
\left| \langle u, \mathcal{H} u\rangle -\lambda_{k}\|u\|_{L^{2}(\Omega)}^{2}\right| & = \left| \langle w, (\mathcal{H}-\lambda_{k}) w\rangle + \langle P u, (\mathcal{H}-\lambda_{k}) P u\rangle + \langle z, (\mathcal{H}-\lambda_{k}) z\rangle \right| \\
& \leqslant \sum_{j=1}^{k-1} \left( \lambda_{k} - \lambda_{j}\right)\left\langle u, \psi_{j}\right\rangle^{2} + \langle z, (\mathcal{H}-\lambda_{k}) z\rangle\\
& \leqslant \max\left\{\frac{\lambda_{k} - \lambda_{1}}{\beta+\lambda_{1}-\lambda_{k}}, 1\right\} \left(L_{k}(u)-\lambda_{k}\right)\|u\|_{L^{2}(\Omega)}^{2},
\end{aligned}
\end{equation*}
where we have used
% the relation
(\ref{(Lk(u)-lambdak)|u|2, orthogonal decomposition}) and $\beta > \lambda_{k} - \lambda_{1}$ in the last line.
%For  nonzero $u \in H^{1}(\Omega)$,  Proposition \ref{proposition: error between Rayleigh quotient and lambdak bounded by energy excess} follows from dividing both sides of the above inequality by $\|u\|_{L^{2}(\Omega)}^{2}$.
This gives Proposition \ref{proposition: error between Rayleigh quotient and lambdak bounded by energy excess}.
\end{proof}

\begin{proof}[Proof of Proposition \ref{Generalized proposition 2.1}]
%Recall that $\{\psi_{i}\}_{i = 1}^{k-1}$ are the first $k-1$ normalized orthonormal eigenfunctions, which implies $\|w\|_{L^{2}(\Omega)}^{2} = \sum_{j=1}^{k-1}\left\langle u, \psi_{j}\right\rangle^{2}.$ 
It follows from  (\ref{ineq for z in Zk}) that  $\langle z, \left(\mathcal{H}-\lambda_{k}\right) z\rangle \geqslant \left(\lambda_{k^{\prime}}-\lambda_{k}\right) \|z\|_{L^{2}(\Omega)}^{2}$.
Using (\ref{(Lk(u)-lambdak)|u|2, orthogonal decomposition}) and the facts $0\leq\lambda_{k} - \lambda_{j} \leq \lambda_{k} - \lambda_{1}$ for all $1\leq j \leq k-1$, we obtain
\begin{equation} \label{lower bound of (Lk(u)-lambdak)|u|2 respect to Pperpu}
\begin{aligned}
\left(L_{k}(u)-\lambda_{k}\right)\|u\|_{L^{2}(\Omega)}^{2} & \geqslant \left( \beta + \lambda_{1} - \lambda_{k}\right) \|w\|_{L^{2}(\Omega)}^{2} + \left(\lambda_{k^{\prime}}-\lambda_{k}\right) \|z\|_{L^{2}(\Omega)}^{2}\\
& \geqslant \min \left\{\beta+\lambda_{1}-\lambda_{k}, \lambda_{k^{\prime}}-\lambda_{k}\right\}\left\|P^{\perp} u\right\|^{2}_{L^{2}(\Omega)},
\end{aligned}
\end{equation}
where we have used $\left\|P^{\perp} u\right\|^{2}_{L^{2}(\Omega)} = \|w\|_{L^{2}(\Omega)}^{2} + \|z\|_{L^{2}(\Omega)}^{2}$. 
Since $\beta+\lambda_{1}-\lambda_{k}$ and $\lambda_{k^{\prime}}-\lambda_{k}$ are both strictly greater than zero, the inequality (\ref{lower bound of (Lk(u)-lambdak)|u|2 respect to Pperpu}) implies the estimate \ref{Eq: stable estimate for u}.

To obtain the bound on $\left\|\nabla P^{\perp} u\right\|^{2}$, notice that
\begin{equation*} \label{}
\begin{aligned}
\left(L_{k}(u)-\lambda_{k}\right)\|u\|_{L^{2}(\Omega)}^{2} & = \langle P u, (\mathcal{H}-\lambda_{k}) P u\rangle + \left\langle P^{\perp} u, (\mathcal{H}-\lambda_{k}) P^{\perp} u\right\rangle + \beta \|w\|_{L^{2}(\Omega)}^{2}\\
& \geqslant \left\langle P^{\perp} u, (\mathcal{H}-\lambda_{k}) P^{\perp} u\right\rangle\\
& = \int_{\Omega}\left|\nabla P^{\perp} u\right|^{2} d x+\int_{\Omega}\left(V-\lambda_{k}\right) \left|P^{\perp} u\right|^{2} d x,
\end{aligned}
\end{equation*}
where we have used $\langle P u, (\mathcal{H}-\lambda_{k}) P u\rangle = 0$.
% and $\beta \|w\|_{L^{2}(\Omega)}^{2} \geqslant 0$.
Rearranging the terms, we arrive at
\begin{equation*}
\begin{aligned}
\left\|\nabla P^{\perp} u\right\|_{L^{2}(\Omega)}^{2} & \leqslant \left(L_{k}(u)-\lambda_{k}\right)\|u\|_{L^{2}(\Omega)}^{2} - \int_{\Omega}\left(V-\lambda_{k}\right) \left|P^{\perp} u\right|^{2} d x \\
& \leqslant \left(L_{k}(u)-\lambda_{k}\right)\|u\|^{2}_{L^{2}(\Omega)} + \left(\lambda_{k}-V_{\min}\right) \left\|P^{\perp} u\right\|^{2}_{L^{2}(\Omega)}.
\end{aligned}
\end{equation*}
Substituting (\ref{Eq: stable estimate for u}) into the above inequality, we finally obtain (\ref{Eq: stable estimate for nabla u}).
% , which completes the proof.
\end{proof}

\section{Missing proof in section \ref{section: Approximation theory for sine spectral Barron functions}}

\subsection{Preliminaries} 
\iffalse
\begin{proof}[Proof of Lemma \ref{lemma: The set mathfrak(S) forms an orthogonal basis}] 
First that  $\mathfrak{S}$  forms an orthogonal basis of  $L^{2}(\Omega)$  follows directly from the Parseval's theorem applied to the Fourier expansion of the odd extension of a function  $u$  from  $L^{2}(\Omega)$. 
To see  $\mathfrak{S}$  is an orthogonal basis of  $H_{0}^{1}(\Omega) $, since  $\mathfrak{S}$  is an orthogonal set of  $H_{0}^{1}(\Omega)$, it suffices to show that if  $u \in H_{0}^{1}(\Omega)$  satisfying
$$\left(u, \Phi_{k}\right)_{H^{1}(\Omega)}=0$$
for all  $k \in \mathbb{N}_{+}^{d}$, then  $u=0$. In fact, the last display above yields that
\begin{equation*}
\begin{aligned}
0 & =\int_{\Omega} u \cdot \Phi_{k} d x+\int_{\Omega} \nabla u \cdot \nabla \Phi_{k} d x \\
& =\int_{\Omega} u \cdot\left(\Phi_{k}-\Delta \Phi_{k}\right) d x \\
& =\left(1+\pi^{2}|k|^{2}\right) \int_{\Omega} u \cdot \Phi_{k} d x
\end{aligned}
\end{equation*}
where for the second identity we have used the Green's formula and the fact that $u$ vanishes on the boundary of  $\Omega$. 
Therefore we have obtained that  $\left(u, \Phi_{k}\right)_{L^{2}}=0$  for any  $k \in \mathbb{N}_{+}^{d}$, which implies that  $u=0$  since  $\mathfrak{S}$  is an orthogonal basis of  $L^{2}(\Omega)$.
\end{proof}
\fi

\begin{proof}[Proof of Lemma \ref{lemma: embedding results for sine spectral Barron space}] \label{appendix proof of Lemma: embedding results for sine spectral Barron space}
(1) For  $u \in \mathfrak{B}^{0}(\Omega) $, using the fact  $\left\|\Phi_{k}\right\|_{L^{\infty}(\Omega)} \leq 1$,  we have 
$$\|u\|_{L^{\infty}(\Omega)}=\left\|\sum_{k \in \mathbb{N}_{+}^{d}} \hat{u}(k) \Phi_{k}\right\|_{L^{\infty}(\Omega)} \leq \sum_{k \in \mathbb{N}_{+}^{d}}|\hat{u}(k)| = \frac{1}{2}\|u\|_{\mathfrak{B}^{0}(\Omega)}.$$
Moreover, since  $u \in \mathfrak{B}^{s}(\Omega)$ with $s\geq 0$ have summable sine coefficients, the sum of sine expansion converges uniformly, which implies that $u\in C(\overline{\Omega})$ and $u$ vanishes on the boundary of $\Omega$.

(2) If  $u \in \mathfrak{B}^{2}(\Omega)$, then  $\|u\|_{\mathfrak{B}^{2}(\Omega)} = \sum_{k \in \mathbb{N}_{+}^{d}}\left(1+\pi^{2}|k|_{1}^{2}\right)|\hat{u}(k)| <\infty .$ 
This particularly implies  $|\hat{u}(k)| \leq \|u\|_{\mathfrak{B}^{2}(\Omega)}$  for each  $k \in \mathbb{N}_{+}^{d}$.  
From the Cauchy-Schwarz inequality, we have 
\begin{equation*}
\begin{aligned}
\|u\|_{H^{1}(\Omega)}^{2} & = \sum_{k \in \mathbb{N}_{+}^{d}} 2^{-d}\left(1+\pi^{2}|k|^{2}\right)|\hat{u}(k)|^{2} \\
& \leq 2^{-d}\|u\|_{\mathfrak{B}^{2}(\Omega)} \sum_{k \in \mathbb{N}_{+}^{d}}\left(1+\pi^{2} |k|_{1}^{2}\right)|\hat{u}(k)| \\
& \leq 2^{-d}\|u\|_{\mathfrak{B}^{2}(\Omega)}^{2} .
\end{aligned}
\end{equation*}
Hence, $u\in H_{0}^{1}(\Omega)$ follows from that $u$ lies in $H^{1}(\Omega)$ and vanishes on $\partial\Omega$. 
\end{proof}
%\textcolor{red}{The results in the above lemma (2) can be strengthened. Is it necessary?}

\subsection{Upper bounds for the cut-off function \texorpdfstring{$\varphi$}{}}
The following lemma gives upper bounds for the cut-off function $\varphi$ and its gradient.
\begin{lemma} \label{preliminary bounds for cutoff function varphi(x)}
For all $x\in\Omega$, it holds that $0 < \varphi(x) < 1/d$ and $|\nabla \varphi(x)| < \pi$.
\end{lemma}
\begin{proof}
For any  $x = (x_1, x_2,\cdots, x_d)^{T}\in\Omega$  and  $1 \leqslant i \leqslant d$, $0 < \sin \left(\pi x_{i}\right) < 1$. It is obvious that $0 < \varphi(x) < 1/d$. 
A direct calculation yields that 
\iffalse
for each $1 \leqslant i \leqslant d$,
\begin{equation*} \label{nabla varphi(x)}
\begin{aligned}
\frac{\partial}{\partial x_{i}} \varphi(x) & 
%= \frac{\pi \cos \left(\pi x_{i}\right)}{\sin ^{2}\left(\pi x_{i}\right)} \varphi^{2}(x) 
= \pi \cos \left(\pi x_{i}\right)\left[\frac{\prod_{j \neq i} \sin \left(\pi x_{j}\right)}{\sum_{l=1}^{d} \prod_{j \neq l} \sin \left(\pi x_{j}\right)}\right]^{2},
\end{aligned}
\end{equation*}
and 
\fi
\begin{equation*} \label{|nabla varphi(x)|2}
\begin{aligned}
|\nabla \varphi(x)|^{2} =\pi^{2} \frac{\sum_{l=1}^{d} \cos ^{2}\left(\pi x_{l}\right)\left[\prod_{j \neq l} \sin \left(\pi x_{j}\right)\right]^{4}}{\left[\sum_{l=1}^{d} \prod_{j \neq l} \sin \left(\pi x_{j}\right)\right]^{4}}.
\end{aligned}
\end{equation*}
Since for every $1 \leqslant l \leqslant d$ and  $x \in(0,1)^{d}$, 
$\prod_{j \neq l} \sin \left(\pi x_{j}\right)>0$ and $\cos ^{2}\left(\pi x_{l}\right) \in[0,1),$
we have
\begin{equation*}
\begin{aligned}
\sum_{l=1}^{d} \cos ^{2}\left(\pi x_{l}\right)\left[\prod_{j \neq l} \sin \left(\pi x_{j}\right)\right]^{4}
< \sum_{l=1}^{d}\left[\prod_{j \neq l} \sin \left(\pi x_{j}\right)\right]^{4} \leq \left[\sum_{l=1}^{d} \prod_{j \neq l} \sin \left(\pi x_{j}\right)\right]^{4} ,
\end{aligned}
\end{equation*}
which implies that $|\nabla \varphi(x)|^{2}<\pi^{2}$ for all  $x \in (0,1)^{d}$. 
%Hence, $|\nabla \varphi(x)| < \pi$ and $\|\nabla \varphi\|_{L^{\infty}} \leq \pi$.
\end{proof}

\begin{lemma} \label{lemma: Magnification of the H1-norm caused by the action of varphi}
For any $h\in H^{1}(\Omega)$, $\left\|\varphi h\right\|_{H^{1}(\Omega)} \leq \sqrt{21} \left\|h\right\|_{H^{1}(\Omega)}.$ 
\iffalse
\begin{equation*} %\label{Magnification of the H1-norm caused by the action of varphi}
\begin{aligned}
\left\|\varphi h\right\|_{H^{1}(\Omega)}
%\leq \sqrt{2\pi^{2} + 1/d^{2}} \left\|h\right\|_{H^{1}(\Omega)} 
\leq \sqrt{21} \left\|h\right\|_{H^{1}(\Omega)}.
\end{aligned}
\end{equation*}
\fi
Particularly, if $\{h_{j}\}_{j=1}^{\infty}$ converges to $h$ in $H^{1}(\Omega)$, then $\{ \varphi h_{j}\}_{j=1}^{\infty}$  converges to $\varphi h$ in $H^{1}(\Omega)$.
\end{lemma}
\begin{proof}
By Lemma \ref{preliminary bounds for cutoff function varphi(x)}, for $h\in H^{1}(\Omega)$, a direct calculation yields $\|\varphi h\|_{L^{2}(\Omega)} \leq \|h\|_{L^{2}(\Omega)}/d$ 
%$$\int_{\Omega} \varphi^{2}(x) h^{2}(x) d x \leq  \frac{1}{d^{2}}\int_{\Omega} h^{2}(x) d x,$$
and
\begin{equation*} \label{}
\begin{aligned}
\int_{\Omega}|\nabla(\varphi(x) h(x))|^{2} d x 
\leq  2 \int_{\Omega}\left( h^{2}|\nabla \varphi|^{2}+\varphi^{2}|\nabla h|^{2} \right)d x 
\leq  2 \int_{\Omega}\left( \pi^{2}h^{2} + \frac{1}{d^{2}}|\nabla h|^{2} \right)d x.
\end{aligned}
\end{equation*}
Notice that $2\pi^{2}+1 < 21$ and $d\geq 1$.  Thus, we obtain
\begin{equation*} \label{}
\begin{aligned}
\left\|\varphi h\right\|_{H^{1}(\Omega)}^{2} 
\leq \left(2\pi^{2} + \frac{1}{d^{2}}\right) \left\|h\right\|_{L^{2}(\Omega)}^{2} + \frac{2}{d^{2}} \left\|\nabla h\right\|_{L^{2}(\Omega)}^{2}
%\leq \left(2\pi^{2} + \frac{1}{d^{2}}\right)\left\|h\right\|_{H^{1}(\Omega)}^{2},
\leq 21 \left\|h\right\|_{H^{1}(\Omega)}^{2}.
\end{aligned}
\end{equation*}
Particularly, if $\{h_{j}\}_{j=1}^{\infty}$ converges to $h$ in $H^{1}(\Omega)$, then $\left\|\varphi h_{j} - \varphi h\right\|_{H^{1}(\Omega)}  \leq  \sqrt{21}  \left\| h_{j} - h\right\|_{H^{1}(\Omega)} $ $  \to 0.$
%$$\left\|\varphi h_{j} - \varphi h\right\|_{H^{1}(\Omega)}  \leq  \sqrt{21}  \left\| h_{j} - h\right\|_{H^{1}(\Omega)}   \to 0.$$
\end{proof}

\subsection{Sine Spectral Barron Space and Neural Network Approximation} \label{appendix subsection: Sine Spectral Barron Space and Neural Network Approximation}

To prove Proposition \ref{proposition: v cos-sin expansion}, we need the following elementary facts. 
\begin{lemma}
\label{Lemma: sin(kx)/sin(x) expansion}
%For $m \in \mathbb{N}_{+}$, $\frac{\sin (m \pi x)}{\sin (\pi x)}$  admits the expansion  with respect to  $\left\{cos(l \pi x)\right\}_{l=0}^{\infty}$ in  $(0, 1)$ % as follows:
The following expansion holds for $m \in \mathbb{N}_{+}$ and $x \in (0, 1)$
\begin{equation*}
\frac{\sin \left(m \pi x\right)}{\sin (\pi x)} = \left\{
\begin{aligned}
& 1 + \sum_{l=1}^{\left(m-1\right)/2} 2 \cos \left(2 l \pi x\right), &  m \text{ is odd}, \\
& \sum_{l=1}^{m/2} 2 \cos \left((2 l-1) \pi x\right), &  m \text { is even. }
\end{aligned} \right.
\end{equation*}
\iffalse
\begin{equation*}
\frac{\sin \left(m \pi x\right)}{\sin (\pi x)} = \left\{
\begin{aligned}
& 1 + \sum_{j=1}^{p} 2 \cos (2 j \pi x), & \text{for } m = 2 p + 1 \text{ and } p \in \mathbb{N}, \\
& \sum_{j=1}^{p} 2 \cos \left((2 j-1) \pi x\right), & \text{for } m = 2 p \text{ and } p \in \mathbb{N}_{+}.
\end{aligned} \right.
\end{equation*}

\begin{equation*}
\frac{\sin \left(m \pi x\right)}{\sin (\pi x)} = 
\begin{cases} 
1 + \sum_{j=1}^{p} 2 \cos (2 j \pi x), & \text{for } m = 2 p + 1 \text{ and } p \in \mathbb{N},\\
\sum_{j=1}^{p} 2 \cos \left((2 j-1) \pi x\right), & \text{for } m = 2 p \text{ and } p \in \mathbb{N}_{+},
\end{cases}
\end{equation*}

(1) If $m = 2 p + 1$  where $p \in \mathbb{N}$, then
\begin{equation*} \label{sin((2k+1)pix) expansion}
\begin{aligned}
\frac{\sin \left(m \pi x\right)}{\sin (\pi x)} = 1 + \sum_{j=1}^{p} 2 \cos (2 j \pi x).
\end{aligned}
\end{equation*}

(2) If $m = 2 p$  where $p \in \mathbb{N}_{+}$, then
\begin{equation*} \label{sin(2kpix) expansion}
\begin{aligned}
\frac{\sin (m \pi x)}{\sin (\pi x)} =  \sum_{j=1}^{p} 2 \cos \left((2 j-1) \pi x\right).
\end{aligned}
\end{equation*}
\fi
\end{lemma}

The proof is straightforward, and we omit the proof.
% The proof is based on the identity $2 \sin \alpha \cos \beta = \sin (\alpha+\beta) - \sin (\beta-\alpha)$ for any $\alpha, \beta \in \mathbb{R}$ and we leave it to interested readers. 
%We are ready to prove Proposition \ref{proposition: v cos-sin expansion}.

\begin{proof}[Proof of Proposition \ref{proposition: v cos-sin expansion}]
For $u \in \mathfrak{B}^{s+1}$, dividing both sides of  $u = \sum_{k \in \mathbb{N}_{+}^{d}} \hat{u}(k)\Phi_{k}$  by  $\varphi$  yields
\begin{equation} \label{proof in dD, 1}
\begin{aligned}
\frac{u(x)}{\varphi(x)}=\sum_{k \in \mathbb{N}_{+}^{d}} \hat{u}(k)\left(\sum_{i=1}^{d} \frac{\sin \left(k_i \pi x_{i}\right)}{\sin \left(\pi x_{i}\right)} \prod_{j \neq i} \sin \left(k_{j} \pi x_{j}\right)\right).
\end{aligned}
\end{equation}  
The sum in (\ref{proof in dD, 1}) is absolutely convergent for all $x \in \Omega$
because  $\left|\frac{\sin(k_i \pi x)}{\sin(\pi x)}\right| \leqslant k_i$ and 
\begin{equation*} \label{}
\begin{aligned}
\sum_{k \in \mathbb{N}_{+}^{d}} \left|\hat{u}(k)\right| \left(\sum_{i=1}^{d} \frac{\left|\sin \left(k_i \pi x_{i}\right)\right|}{\sin \left(\pi x_{i}\right)}  \prod_{j \neq i} \left|\sin \left(k_{j} \pi x_{j}\right) \right| \right) 
\leqslant \sum_{k \in \mathbb{N}_{+}^{d}} |\hat{u}(k)| \left(\sum_{i=1}^{d} k_i \right)
\leqslant \|u\|_{\mathfrak{B}^{1}(\Omega)}.
\end{aligned}
\end{equation*} 
Expanding $\frac{\sin(k_i \pi x_i)}{\sin(\pi x_i)}$ with $k_i \in \mathbb{N}_{+}$ by Lemma \ref{Lemma: sin(kx)/sin(x) expansion}
\iffalse , we obtain
\begin{equation*} \label{}
\frac{\sin \left(k_{i} \pi x_{i}\right)}{\sin \left(\pi x_{i}\right)}=\left\{\begin{aligned}
1+\sum_{l=1}^{\left(k_{i}-1\right)/2} 2 \cos \left(2 l \pi x_{i}\right), & \quad k_{i} \text { is odd, } \\
\ \sum_{l=1}^{k_{i}/2} 2 \cos \left((2 l-1) \pi x_{i}\right), & \quad k_{i} \text { is even. }
\end{aligned}\right.
\end{equation*} 
\fi
and plugging the expansion into (\ref{proof in dD, 1}) yield that $u/\varphi$  has an expansion of the form (\ref{u(x)/varphi(x) expansion form respect to cos-sin functions})
where an index $(k, i) \in \Gamma$ if and only if  $k \in \mathbb{N}_{0}^{d}$  has at most one zero component at position $k_{i}$.
Moreover, the expansion coefficients  $\hat{v}(k, i)$  may be expressed by $\hat{u}(k)$:
\begin{equation} \label{v(k,i) expressed by u(k)}
\begin{aligned}
& \hat{v}(k, i) = \left(1 + \mathbf{1}_{ \{k_{i} \geqslant 1 \}}\right)\sum_{l=0}^{\infty} \hat{u}\left(k + (2l + 1) e_{i}\right)   \quad \text{ for }  (k, i) \in \Gamma,
\end{aligned}
\end{equation}
where $e_{i}$ is the $i$-th cannonical basis.
In what follows, we prove (\ref{u(x)/varphi(x) expansion, coefficients norm bound}).

By (\ref{v(k,i) expressed by u(k)}) and merging the terms containing $\hat{u}(k)$  for each  $k \in \mathbb{N}_{+}^{d}$  respectively, we obtain 
\begin{equation} \label{proof in dD, 3}
\begin{aligned}
\sum_{(k, i) \in \Gamma}\left(1+\pi^{s}|k|_{1}^{s}\right) |\hat{v}(k, i)| \leqslant \sum_{i=1}^{d} \left( \sum_{k \in \mathbb{N}_{+}^{d}, k_{i}\text{ is odd}} A_{1} |\hat{u}(k)|  
+  \sum_{k \in \mathbb{N}_{+}^{d},  k_{i}\text{ is even}} A_{2} |\hat{u}(k)|   \right),
\end{aligned}
\end{equation}
where
\begin{equation*} 
\begin{aligned}
& A_{1} =  1+\pi^{s}\left|k - k_{i} e_{i}\right|_{1}^{s} + \mathbf{1}_{ \{k_{i} \geqslant 3 \}}\cdot 2 \sum_{l=0}^{(k_{i}-3)/2}\left[1+\pi^{s} \left| k-(2l+1) e_{i}\right|_{1}^{s}\right] ,  \\
& A_{2} =  2 \sum_{l=1}^{k_{i}/2}\left[1+\pi^{s} \left|k-(2l-1) e_{i}\right|_{1}^{s}\right].
\end{aligned}
\end{equation*}
Notice that $s\geq 0$, $t^{s}$ is a nondecreasing function for $t \geq 0$.  When  $k \in \mathbb{N}_{+}^{d}$  and  $k_{i}$ is odd, 
\begin{equation*} \label{proof in dD, 4}
\begin{aligned}
A_{1} & \leqslant  k_{i} + 2 \pi^{s} \sum_{l=0}^{(k_{i}-1)/2} \left(|k|_{1} - 2l-1\right)^{s} \leqslant   k_{i} + 2 \pi^{s} \int_{|k|_{1}-k_{i}}^{|k|_{1}} t^{s} dt .
\end{aligned}
\end{equation*}
Similarly, when  $k \in \mathbb{N}_{+}^{d}$  and  $k_{i}$ is even, 
\begin{equation*} \label{proof in dD, 5}
\begin{aligned}
A_{2}  & =  k_{i} + 2 \pi^{s} \sum_{l=1}^{k_{i}/2}\left(|k|_{1}- 2l +1 \right)^{s}
\leqslant   k_{i} + 2 \pi^{s} \int_{|k|_{1}-k_{i}}^{|k|_{1}} t^{s} d t .
\end{aligned}
\end{equation*}
Hence, we conclude that $A_{1}$, $A_{2}$ are both bounded by $$k_{i} + \frac{2 \pi^{s}}{s+1} \left(|k|_{1}^{s+1}-\left(|k|_{1}-k_{i}\right)^{s+1}\right).$$
Thus, substituting the above bound and exchanging the order of summation, we bound (\ref{proof in dD, 3}) by
\begin{equation} \label{proof in dD, 7}
\begin{aligned}
&  \sum_{i=1}^{d} \left\{ \sum_{k \in \mathbb{N}_{+}^{d}} |\hat{u}(k)| \left[ k_{i} + \frac{2 \pi^{s}}{s+1} \left(|k|_{1}^{s+1}-\left(|k|_{1}-k_{i}\right)^{s+1}\right) \right] \right\}\\
& = \sum_{k \in \mathbb{N}_{+}^{d}} |\hat{u}(k)| \left\{|k|_{1} + \frac{2 \pi^{s}}{s+1} \left[d |k|_{1}^{s+1} - \sum_{i=1}^{d}\left(|k|_{1}-k_{i}\right)^{s+1}\right]\right\}\\
& \leqslant \sum_{k \in \mathbb{N}_{+}^{d}} |\hat{u}(k)| \left\{|k|_{1} + \frac{2 \pi^{s}}{s+1} \left[d |k|_{1}^{s+1} - d\left(\frac{d-1}{d}|k|_{1}\right)^{s+1} \right]\right\}\\
& = \sum_{k \in \mathbb{N}_{+}^{d}} |\hat{u}(k)| \left\{ |k|_{1} + \frac{2 d}{\pi(s+1)}\left[1-\left(\frac{d-1}{d}\right)^{s+1}\right] \pi^{s+1}|k|_{1}^{s+1} \right\},
\end{aligned}
\end{equation}
where we have used Jensen's inequality and the fact that $t^{s+1}$ is convex, 
\begin{equation*} \label{}
\begin{aligned}
\frac{1}{d} \sum_{i=1}^{d}\left(|k|_{1}-k_{i}\right)^{s+1} 
\geqslant \left[\frac{1}{d} \sum_{i=1}^{d}\left(|k|_{1}-k_{i}\right)\right]^{s+1} 
=\left(\frac{d-1}{d}|k|_{1}\right)^{s+1}.
\end{aligned}
\end{equation*}
%Recall the definition of $\|u\|_{\mathfrak{B}^{s+1}(\Omega)}$ and notice that the last line of (\ref{proof in dD, 7}) can be bounded by 
Notice that the last line of (\ref{proof in dD, 7}) may be bounded by 
\begin{equation*} \label{}
\begin{aligned}
\left(1+\frac{2 d}{\pi(s+1)}\right) \sum_{k \in \mathbb{N}_{+}^{d}} |\hat{u}(k)|   \left(1+\pi^{s+1}|k|_{1}^{s+1}\right) 
= \left(1+\frac{2 d}{\pi(s+1)}\right)\|u\|_{\mathfrak{B}^{s+1}(\Omega)}.
\end{aligned}
\end{equation*}
The estimate (\ref{u(x)/varphi(x) expansion, coefficients norm bound}) follows from  (\ref{proof in dD, 3}), (\ref{proof in dD, 7}) and the above bound.
\end{proof}

The following lemma is useful to show that $u/\varphi$ lies in the convex hull of $\mathcal{F}_{s}(B)$. 
\begin{lemma}{{\citep[Lemma 4.2]{DRMlu2021priori}}}
\label{cos decomposition}
For any  $\theta=\left(\theta_{1}, \theta_{2}, \cdots, \theta_{d}\right)^{T} \in \mathbb{R}^{d}$,
$$\prod_{i=1}^{d} \cos \theta_{i} =\frac{1}{2^{d}} \sum_{\xi \in \{1,-1\}^{d}} \cos (\xi \cdot \theta).$$
%where    and  $\Xi= .$
\end{lemma}

%Now, we are ready to prove Proposition \ref{proposition: u H1 approximation by cos-sin}. 
\begin{proof}[Proof of Proposition \ref{proposition: u H1 approximation by cos-sin}]
\textbf{Step 1:} Show that $u$ lies in  $\overline{\operatorname{conv}(\varphi\mathcal{F}_{s}(B_{u}))}$ with $B_{u}=\left(1+\frac{2 d}{\pi(s+1)}\right) \|u\|_{\mathfrak{B}^{s+1}(\Omega)}$.
By Lemma \ref{cos decomposition} and $\sin \theta_{j} = \cos \left(\theta_{j}-\pi/2\right)$, let $\theta = \left(k_{1} \pi x_{1},\right.$ $ \left. k_{2} \pi x_{2}, \ldots, k_{d} \pi x_{d}\right)^{\top}$, then
\begin{equation} \label{cos(k_i pi x_i) prod sin(k_j pi x_j) expansion}
\begin{aligned}
\cos \left(k_{i} \pi x_{i}\right) \prod_{j \neq i} \sin \left(k_{j} \pi x_{j}\right) = \frac{1}{2^{d}} \sum_{\xi \in\{1,-1\}^d} \cos \left(\pi k_{\xi} \cdot x-\frac{\pi}{2}\sum_{j \neq i} \xi_{j}\right),
\end{aligned}
\end{equation}
where $k_{\xi}=\left(k_{1} \xi_{1}, k_{2} \xi_{2}, \ldots, k_{d} \xi_{d}\right).$
Since $u \in \mathfrak{B}^{s+1}(\Omega)$ with $s\geq 1$, plugging (\ref{cos(k_i pi x_i) prod sin(k_j pi x_j) expansion}) into the expansion of $u/\varphi$ in Proposition \ref{proposition: v cos-sin expansion} yields
\begin{equation*} \label{}
\begin{aligned}
\frac{u(x)}{\varphi(x)} & =\sum_{(k, i) \in \Gamma} \hat{v}(k, i) \cdot \frac{1}{2^{d}} \sum_{\xi \in\{1,-1\}^d} \cos \left(\pi k_{\xi}  \! \cdot  \! x-\frac{\pi}{2}\sum_{j \neq i} \xi_{j}\right) \\
& =\sum_{(k, i) \in \Gamma} \frac{|\hat{v}(k, i)|\left(1 \! + \! \pi^{s}|k|_{1}^{s}\right)}{Z_{v}} 
\frac{Z_{v}}{1 \! + \! \pi^{s}|k|_{1}^{s}} 
\frac{1}{2^{d}} \sum_{\xi \in\{1,-1\}^d} \operatorname{sign}(\hat{v}(k, i)) \cos \left( \! \pi k_{\xi} \!  \cdot  \! x-\frac{\pi}{2}\sum_{j \neq i} \xi_{j}  \! \right),
\end{aligned}
\end{equation*}
where $Z_{v}$ is a constant to be specified.
We define a probability measure on  $\Gamma$  by
$$\mu(d(k, i))=\frac{|\hat{v}(k, i)|\left(1+\pi^{s}|k|_{1}^{s}\right)}{Z_{v}} \delta(d(k, i))$$
with
$$Z_{v}=\sum_{(k, i) \in \Gamma}|\hat{v}(k, i)|\left(1+\pi^{s}|k|_{1}^{s}\right) \leqslant\left(1+\frac{2 d}{\pi (s+1)}\right) \|u\|_{\mathfrak{B}^{s+1}(\Omega)}.$$
Let  $\operatorname{sign}(\hat{v}(k, i))=(-1)^{\theta_{k, i}}$ with $\theta_{k, i}\in\{0, 1\}$, and
\begin{equation*} \label{}
\begin{aligned}
g(x, k, i)=\frac{Z_{v}}{1+\pi^{s}|k|^{s}} \cdot \frac{1}{2^{d}} \sum_{\xi \in\{1,-1\}^d} \cos \left(\pi \left(k_{\xi} \cdot x-\frac{1}{2} \sum_{j \neq i} \xi_{j}+\theta_{k, i}\right)\right) \varphi(x).
\end{aligned}
\end{equation*}
Then, for any  $x \in \Omega,$ %u(x)  =\int_{\Gamma} g(x, k, i) \mu(d(k, i)) =\mathbf{E}_{\mu} g(x, k, i). 
\begin{equation} \label{u(x) is expressed as the expectation of the function g(x, k, i)}
\begin{aligned}
u(x) = \sum_{(k, i) \in \Gamma}  g(x, k, i) \mu(d(k, i)) =\mathbf{E}_{\mu} g(x, k, i).
\end{aligned}
\end{equation}
For $\xi \in\{1,-1\}^d$, $\sum_{j \neq i} \xi_{j}$ is even when $d$ is odd, and $\sum_{j \neq i} \xi_{j}$ is odd when $d$ is even. 
% From the periodicity of the trigonometric functions and $\cos(\alpha-\pi/2) = \sin(\alpha)$, we know 
It is clear that $g(x, k, i)$ is a convex combination of $2^{d}$ elements in $\varphi\mathcal{F}_{s}(B_{u})$ with $B_{u} = \left(1+\frac{2 d}{\pi (s+1)}\right) \|u\|_{\mathfrak{B}^{s+1}(\Omega)}$. 
Moreover, thanks to (\ref{u(x) is expressed as the expectation of the function g(x, k, i)}) and the uniform boundness of $\|g(\cdot, k, i)\|_{H^{1}(\Omega)}$ derived from the following step, $u$ lies in  $\overline{\operatorname{conv}(\varphi\mathcal{F}_{s}(B_{u}))}$. 

\textbf{Step 2:} Check that $\varphi\mathcal{F}_{s}(B)$ with $s\geq1$ is a bounded set in $H^1(\Omega)$. 
By Lemma \ref{preliminary bounds for cutoff function varphi(x)}, whether $f$ is a sine function or a cosine function, %there hold that 
$$\int_{\Omega} \varphi^{2}(x) f^{2}(\pi(k \cdot x+b)) d x \leqslant \int_{\Omega} \varphi^{2}(x) d x \leqslant 1 / d^{2},$$
and
\begin{equation*} \label{}
\begin{aligned}
\int_{\Omega}|\nabla(\varphi(x) f(\pi(k \cdot x+b)))|^{2} d x
&\leqslant  2 \int_{\Omega}\left[ f^{2}(\pi(k \cdot x+b))|\nabla \varphi|^{2}+\varphi^{2}|\nabla f(\pi(k \cdot x+b))|^{2} \right] d x \\
&\leqslant 2 \int_{\Omega}\left(|\nabla \varphi|^{2}+\pi^{2}|k|^{2} \varphi^{2} \right)d x \leqslant 2 \pi^{2}+2 \pi^{2}|k|^{2} / d^{2}.
\end{aligned}
\end{equation*}
Hence, for any $w \in \varphi\mathcal{F}_{s}(B)$, $w(x)=\frac{\gamma}{1+\pi^{s}|k|^{s}} \varphi(x) f(\pi(k \cdot x+b))$,
\begin{equation*} \label{}
\begin{aligned}
\|w\|_{H^{1}(\Omega)}^{2} \leqslant \frac{B^{2}}{\left(1+\pi^{s}|k|_{1}^{s}\right)^{2}}\left(\frac{1}{d^{2}}+2 \pi^{2}+\frac{2 \pi^{2}|k|^{2}}{d^{2}}\right).
\end{aligned}
\end{equation*}
When  $|k|_{1} \geq 1$, since  $s \geq 1$ and $1/d^{2} + 2 \pi^{2} + 2 \pi^{2}|k|^{2} / d^{2} \leqslant 4+8 \pi|k|_{1}+4 \pi^{2}|k|_{1}^{2} \leqslant 4\left(1+\pi^{s}|k|_{1}^{s}\right)^{2} ,$ 
\iffalse
\begin{equation*} \label{}
\begin{aligned}
1/d^{2} + 2 \pi^{2} + 2 \pi^{2}|k|^{2} / d^{2} \leqslant 4+8 \pi|k|_{1}+4 \pi^{2}|k|_{1}^{2} \leqslant 4\left(1+\pi^{s}|k|_{1}^{s}\right)^{2} ,
\end{aligned}
\end{equation*}
\fi
then $\|w\|_{H^{1}(\Omega)}^{2}\leqslant 4 B^{2}$. 
Since  $k \in \Gamma_{1}$, $k=0$ may only appear in the one dimensional case. At that time,  $w(x) = \gamma \varphi(x) = \gamma \sin(\pi x)$  and  $\|w\|_{H^{1}(\Omega)}^{2}  \leqslant \gamma^{2}\left(1+\pi^{2}\right)/2  \leqslant 6 B^{2}.$ 
Thus, $\|w\|_{H^{1}(\Omega)} \leq \sqrt{6} B$ for any $w \in \varphi\mathcal{F}_{s}(B)$ with $s\geq1.$
 
\textbf{Step 3:} %Obtain the approximation rate.
%Combining the facts that $u$ lies in $\overline{\operatorname{conv}(\varphi\mathcal{F}_{s}(B_{u}))}$  with  $B_{u} = \left(1+\frac{2 d}{\pi (s+1)}\right) \|u\|_{\mathfrak{B}^{s+1}}$ and any $w \in \varphi\mathcal{F}_{s}(B)$ satisfies $\|w\|_{H^{1}(\Omega)} \leq \sqrt{6} B$,  Proposition \ref{proposition: u H1 approximation by cos-sin} follows directly from Lemma \ref{Lemma: Maurey method, Universal approximation bounds for superpositions of a sigmoidal function}.
Using  Lemma \ref{Lemma: Maurey method, Universal approximation bounds for superpositions of a sigmoidal function} along with  the facts that $u$ lies in $\overline{\operatorname{conv}(\varphi\mathcal{F}_{s}(B_{u}))}$  with  $B_{u} = \left(1+\frac{2 d}{\pi (s+1)}\right) \|u\|_{\mathfrak{B}^{s+1}}$ and $\|w\|_{H^{1}(\Omega)} \leq \sqrt{6} B$ for any $w \in \varphi\mathcal{F}_{s}(B)$,  we obtain Proposition \ref{proposition: u H1 approximation by cos-sin}.
\end{proof}

\subsection{Reduction to ReLU and Softplus Activation Functions} 
\begin{proof}[Proof of Lemma \ref{lemma: ReLU for both sine and cosine case}]
Let  $g_{m}$  be the piecewise linear interpolant of  $g$  with respect to the grid  $\left\{z_{j}\right\}_{j=0}^{2 m} $. %, i.e. $$g_{m}(z)=g\left(z_{j+1}\right) \frac{z-z_{j}}{z_{j+1}-z_{j}}+g\left(z_{j}\right) \frac{z_{j+1}-z}{z_{j+1}-z_{j}}, \quad\text { if } z \in\left[z_{j}, z_{j+1}\right].$$
Let $h = \max(h_{1},h_{2})$. Then, $h \leq (1+\rho)/m \leq 3/(2m)$,  According to \cite{Ascher2011FirstCourseinNM}, 
$\left\|g-g_{m}\right\|_{L^{\infty}([-1,1])} \leq h^{2} \left\|g^{\prime \prime}\right\|_{L^{\infty}([-1,1])} / 8.$
Consider  $z \in\left[z_{j}, z_{j+1}\right]$  for some  $0 \leq j \leq 2 m-1$. By the mean value theorem, there exist  $\xi, \eta \in\left(z_{j}, z_{j+1}\right)$  such that $ \left(g(z_{j+1})-g(z_{j})\right) / (z_{j+1}-z_{j}) = g^{\prime}(\xi)$  and %$g^{\prime}(z)-g^{\prime}(\xi)  =g^{\prime \prime}(\eta)(z-\xi)$. Hence,
\begin{equation*} \label{}
\begin{aligned}
\left|g^{\prime}(z)-\frac{g\left(z_{j+1}\right)-g\left(z_{i}\right)}{z_{j+1}-z_{j}}\right| & =\left|g^{\prime}(z)-g^{\prime}(\xi)\right| 
=\left|g^{\prime \prime}(\eta)\right||z-\xi| ,
%& \leq h\left\|g^{\prime \prime}\right\|_{L^{\infty}([-1,1])} ,
\end{aligned}
\end{equation*}
which implies that $\left\|g^{\prime}-g_{m}^{\prime}\right\|_{L^{\infty}([-1,1])} \leq h\left\|g^{\prime \prime}\right\|_{L^{\infty}([-1,1])} .$
Thus, $$\left\|g-g_{m}\right\|_{W^{1, \infty}([-1,1])} \leq \frac{h^{2}}{8}\left\|g^{\prime \prime}\right\|_{L^{\infty}([-1,1])} + h\left\|g^{\prime \prime}\right\|_{L^{\infty}([-1,1])} 
\leq \frac{19}{16}Bh 
\leq \frac{2 B}{m} .$$
This proves (\ref{sin piecewise linear interpolation, W1,inf error}).

Next, it is easy to verify that   $g_{m}$  can be rewritten as a two-layer ReLU neural network 
\begin{equation} \label{gm rewritten as ReLU network}
\begin{aligned}
g_{m}(z)=c+\sum_{i=1}^{m} a_{i} \operatorname{ReLU}\left(z_{i}-z\right)+\sum_{i=m+1}^{2 m} a_{i} \operatorname{ReLU}\left(z-z_{i-1}\right), \quad z \in[-1,1],
\end{aligned}
\end{equation}
where  $c=g\left(z_{m}\right)=g(\rho)$  and the parameters  $a_{i}$  defined by
\begin{equation*}
a_{i}=\left\{\begin{array}{ll}
\left(g\left(z_{m+1}\right)-g\left(z_{m}\right)\right)/h_{2}, & \text { if } i=m+1, \\
\left(g\left(z_{m-1}\right)-g\left(z_{m}\right)\right)/h_{1}, & \text { if } i=m, \\
\left(g\left(z_{i}\right)-2 g\left(z_{i-1}\right)+g\left(z_{i-2}\right)\right)/h_{2}, & \text { if } i>m+1, \\
\left(g\left(z_{i-1}\right)-2 g\left(z_{i}\right)+g\left(z_{i+1}\right)\right)/h_{1}, & \text { if } i<m .
\end{array}\right.
\end{equation*}
\iffalse
\begin{equation*}
a_{i}=\left\{\begin{array}{ll}
\frac{g\left(z_{m+1}\right)-g\left(z_{m}\right)}{h_{2}}, & \text { if } i=m+1, \\
\frac{g\left(z_{m-1}\right)-g\left(z_{m}\right)}{h_{1}}, & \text { if } i=m, \\
\frac{g\left(z_{i}\right)-2 g\left(z_{i-1}\right)+g\left(z_{i-2}\right)}{h_{2}}, & \text { if } i>m+1, \\
\frac{g\left(z_{i-1}\right)-2 g\left(z_{i}\right)+g\left(z_{i+1}\right)}{h_{1}}, & \text { if } i<m .
\end{array}\right.
\end{equation*}
\fi
Furthermore, by the mean value theorem, there exists  $\xi_{1}, \xi_{2} \in\left(z_{m}, z_{m+1}\right)$  such that  $\left|a_{m+1}\right|=\left|g^{\prime}\left(\xi_{1}\right)\right|=\left|g^{\prime}\left(\xi_{1}\right)-g^{\prime}(\rho)\right|=\left|g^{\prime \prime}\left(\xi_{2}\right) \xi_{1}\right| \leq B h_{2} $. 
In a similar manner we obtain that  $\left|a_{m}\right| \leq B h_{1}$, $\left|a_{i}\right| \leq 2 B h_{2}$  if  $i>m+1$  and  $\left|a_{i}\right| \leq 2 B h_{1}$  if  $i<m$. Thus,  $\sum_{i=1}^{2 m} |a_{i}| \leq m\cdot 2Bh_{1}+ m\cdot 2Bh_{2} = 4B$. 

Finally, by setting  $\epsilon_{i}=-1$, $b_{i}=-z_{i}$  for  $i=1, \cdots, m$  and  $\epsilon_{i}=1,$  $b_{i}=z_{i-1}$  for  $i= m+1, \cdots, 2 m $, we obtain the desired form (\ref{sin piecewise linear interpolation}) of  $g_{m}$. This completes the proof.
\end{proof}

Next, we prove the approximation results 
%using two-layer networks 
with the Softplus activation. 
To this end, we recall a lemma from \cite{DRMlu2021priori} which shows that ReLU may be well approximated by $\mathrm{SP}_{\tau}$ for $\tau \gg 1$.

\begin{lemma}{{\citep[Lemma 4.6]{DRMlu2021priori}}}
\label{lemma: ReLU-Softplus difference}
The following inequalities hold:
\begin{subequations}
\begin{align}
& \left|\operatorname{ReLU}(z)-\mathrm{SP}_{\tau}(z)\right| \leq \frac{1}{\tau} e^{-\tau|z|}, \quad\forall z \in[-2,2], \label{ReLU-Softplus difference bound}\\
& \left|\operatorname{ReLU}^{\prime}(z)-\mathrm{SP}_{\tau}^{\prime}(z)\right| \leq e^{-\tau|z|}, \quad \forall z \in[-2,0) \cup(0,2], \label{ReLU-Softplus derivative difference bound}\\
& \left\|\mathrm{SP}_{\tau}\right\|_{W^{1, \infty}([-2,2])} \leq 3+\frac{1}{\tau} . \label{ReLU-Softplus difference 
Wi,inf bound}
\end{align}
\end{subequations}
\end{lemma}

\begin{proof}[Proof of Lemma \ref{lemma: Softplus for sine and cosine case, modified}]
Thanks to Lemma \ref{lemma: ReLU for both sine and cosine case}, there exists  $g_{m}$  of the form (\ref{gm rewritten as ReLU network})
\iffalse
\begin{equation*} 
\begin{aligned} 
g_{m}(z)=c+\sum_{i=1}^{m} a_{i} \operatorname{ReLU}\left(z_{i}-z\right)+\sum_{i=m+1}^{2 m} a_{i} \operatorname{ReLU}\left(z-z_{i-1}\right), \ \  z \in[-1,1]
\end{aligned}
\end{equation*}
\fi
such that  $\left\|g-g_{m}\right\|_{W^{1, \infty}([-1,1])} \leq 2 B / m$. 
Moreover, the coefficients  $a_{i}$  satisfies that $\left|a_{m}\right| \leq B h_{1}$, $\left|a_{m+1}\right| \leq B h_{2}$ and $\left|a_{i}\right| \leq 2 B h_{1}$  if  $i<m$,   $\left|a_{i}\right| \leq 2 B h_{2}$  if  $i>m+1$,  so that  $\sum_{i=1}^{2 m} \left|a_{i}\right| \leq 4 B$. Now let  $g_{\tau, m}$  be the function obtained by replacing the activation  $\operatorname{ReLU}$  in  $g_{m}$  by  $\mathrm{SP}_{\tau} $, i.e.
\begin{equation} \label{gtau,m Softplus 2}
\begin{aligned}
g_{\tau, m}(z)=c+\sum_{i=1}^{m} a_{i} 
\mathrm{SP}_{\tau}\left(z_{i}-z\right)+\sum_{i=m+1}^{2 m} a_{i} \mathrm{SP}_{\tau}\left(z-z_{i-1}\right), \quad z \in[-1,1] .
\end{aligned}
\end{equation}
By the bound (\ref{ReLU-Softplus difference bound}) and the upper bounds of $|a_{i}|$, the difference $\left|g_{m}(z)-g_{\tau, m}(z)\right|$ may be bounded by
\begin{equation*} \label{}
\begin{aligned}
& \sum_{i=1}^{m} \left|a_{i}\right|\left|\operatorname{ReLU}\left(z_{i}-z\right)-\mathrm{SP}_{\tau}\left(z_{i}-z\right)\right| + \sum_{i=m+1}^{2 m}\left|a_{i}\right|\left|\operatorname{ReLU}\left(z-z_{i-1}\right)-\mathrm{SP}_{\tau}\left(z-z_{i-1}\right)\right|\\
& \leq \sum_{i=1}^{m-1} \frac{2Bh_{1}}{\tau} e^{-\tau|z-z_{i}|} + \frac{Bh_{1}}{\tau} e^{-\tau|z-z_{m}|} + \frac{Bh_{2}}{\tau} e^{-\tau|z-z_{m}|} + \sum_{i=m+2}^{2m} \frac{2Bh_{2}}{\tau} e^{-\tau|z-z_{i-1}|}\\
& \leq \frac{2B}{m\tau} + \frac{2B}{\tau}\left( \sum_{i=1}^{m-1} h_{1} e^{-\tau|z-z_{i}|} + \sum_{i=m+1}^{2m-1} h_{2} e^{-\tau|z-z_{i}|} \right) =: \frac{2B}{m\tau} + \frac{2B}{\tau}  I,
\end{aligned}
\end{equation*}
where we have used $e^{-\tau|z-z_{m}|}\leq 1$ and $h_{1}+h_{2} = 2/m$ in the last inequality.
Suppose that  $z \in\left(z_{j}, z_{j+1}\right)$  for some fixed  $0\leq j \leq 2m-1$. 
When $j = 0$, the term $i = j+1 = 1$ in the sum $I$ may be bounded by $h = \max(h_{1},h_{2}) \leq 3/(2m)$. The other terms can be bounded by 
\begin{equation*} \label{}
\begin{aligned}
\int_{z_{j+1}}^{z_{2m-1}} e^{-\tau (x - z)} d x
%= \frac{1}{\tau} \left(e^{-\tau (z_{j+1} - z)} - e^{-\tau (z_{2m-1} - z)} \right)
\leq \frac{1}{\tau} \left(1 - e^{-\tau (z_{2m-1} - z_{j+1})} \right)
\leq \frac{1}{\tau} \left(1 - e^{- 2\tau} \right).
\end{aligned}
\end{equation*}
Thus, $I \leq h + (1 - e^{- 2\tau} )/\tau$. Similar bound holds when $j = 2m-1$. When $1\leq j \leq 2m-2$, the term $i = j$ and $i = j+1$ in the sum $I$ is bounded by $h$, respectively. The other terms may be bounded by
\begin{equation*} \label{}
\begin{aligned}
\int_{z_{1}}^{z_{j}} e^{-\tau (z - x)} d x + \int_{z_{j+1}}^{z_{2m-1}} e^{-\tau (x - z)} d x
%= \frac{1}{\tau} \left(e^{-\tau (z_{j+1} - z)} - e^{-\tau (z_{2m-1} - z)} \right)
& \leq \frac{1}{\tau} \left(2 - e^{-\tau (z_{j} - z_{1})} - e^{-\tau (z_{2m-1} - z_{j+1})} \right)\\
& \leq \frac{2}{\tau} \left(1 - e^{- \tau(z_{2m-1} - z_{1})/2} \right) \leq \frac{2}{\tau} \left(1 - e^{- \tau} \right),
\end{aligned}
\end{equation*}
where we have used Jensen's inequality in the second inequality.
Since $2(1 - e^{- \tau}) > 1 - e^{- 2\tau}$ for all $\tau > 0$, we summarize that $I \leq 2h + 2(1 - e^{- \tau})/\tau$ for all $0\leq j \leq 2m-1$. Hence,
$$\left\|g_{m}-g_{\tau, m}\right\|_{L^{\infty}([-1,1])} \leq \frac{2 B}{\tau}\left( \frac{1}{m} + 2h +  \frac{2(1 - e^{- \tau})}{\tau}\right) \leq \frac{4 B}{\tau}\left( \frac{2}{m} + \frac{1 - e^{- \tau}}{\tau}\right).$$
Using (\ref{ReLU-Softplus derivative difference bound}), proceeding along the same line that leads to the above estimate, we obtain
%Thanks to the bound (\ref{ReLU-Softplus derivative difference bound}), we can similarly estimate the difference of the derivatives and obtain
$$\left\|g_{m}^{\prime}-g_{\tau, m}^{\prime}\right\|_{L^{\infty}([-1,1])} \leq 4 B\left( \frac{2}{m} + \frac{1 - e^{- \tau}}{\tau}\right).$$
A combination of the above estimates and  $\left\|g-g_{m}\right\|_{W^{1, \infty}([-1,1])} \leq 2 B / m$  yields 
\begin{equation*} \label{}
\begin{aligned}
\left\|g-g_{\tau, m}\right\|_{W^{1, \infty}([-1,1])} & \leq\left\|g-g_{m}\right\|_{W^{1, \infty}([-1,1])}+\left\|g_{m}-g_{\tau, m}\right\|_{W^{1, \infty}([-1,1])} \\
& \leq \frac{2 B}{m} + 4 B \left(1+\frac{1}{\tau}\right) \left( \frac{2}{m} + \frac{1 - e^{- \tau}}{\tau}\right)\\
& \leq 4 B \left(1+\frac{1}{\tau}\right) \left( \frac{3}{m} + \frac{1 - e^{- \tau}}{\tau}\right).
\end{aligned}
\end{equation*}
Since $\tau>0$, choose $m \in\mathbb{N}_{+}$  such that $m \geq \max\{3\tau e^{\tau}, 2\}$. Then, we obtain (\ref{approximation error by two-layer Softplus networks for functions in mathcalFs(B)}).
Finally, we rewrite (\ref{gtau,m Softplus 2}) in the form (\ref{gtau,m Softplus 1}), which completes the proof.
\end{proof}

\subsection{Bounding the approximation error}
\begin{proof}[Proof of Theorem \ref{Thm: bounding the approximation error Lk(u in F)-lambdak}] \label{Appendix proof of Thm: bounding the approximation error Lk(u in F)-lambdak}
For convenience, we denote $\varphi v_{m}$ by $u_{m}$.
Since $u^{*}\in U_{k}$,   $L_{k}\left(u^{*}\right) =\lambda_{k} $. 
 Now observe that
\begin{equation} \label{tildeum-u* energy excess decomp}
\begin{aligned}
L_{k}\left(u_{m}\right)-\lambda_{k} %& = L_{k}\left(u_{m}\right)-L_{k}\left(u^{*}\right) \\
= & \frac{\mathcal{E}_{V}\left(u_{m}\right)-\mathcal{E}_{V}\left(u^{*}\right)}{\mathcal{E}_{2}\left(u_{m}\right)} + \frac{\mathcal{E}_{P}\left(u_{m}\right)-\mathcal{E}_{P}\left(u^{*}\right)}{\mathcal{E}_{2}\left(u_{m}\right)} + \\ 
& \frac{\mathcal{E}_{2}\left(u^{*}\right)-\mathcal{E}_{2}\left(u_{m}\right)}{\mathcal{E}_{2}\left(u_{m}\right)} L_{k}\left(u^{*}\right) .
\end{aligned}
\end{equation}
Since $u^{*}\in \mathfrak{B}^{s}(\Omega)$ for some $s\geq 3$, by Theorem \ref{Thm: u H1 approximation by varphi Softplus networks}, we have $\left\|u^{*}- u_{m}\right\|_{H^{1}(\Omega)} \leq \eta\left(B_{u^{*}}, m\right) \leq 1/2$
with $B_{u^{*}} = \left(1+\frac{2 d}{\pi s}\right)\|u^{*}\|_{\mathfrak{B}^{s}(\Omega)} .$ 
Combining this with $ \left\|u^{*}\right\|_{L^{2}(\Omega)}^{2} = 1 $ yields
\begin{equation*} \label{lower and upper bound for L2 norm of um}
1/2  \leq  1 - \left\|u^{*}- u_{m}\right\|_{H^{1}(\Omega)} \leq \left\|u_{m}\right\|_{L^{2}(\Omega)} \leq 1 + \left\|u^{*}- u_{m}\right\|_{H^{1}(\Omega)} \leq  3/2,
\end{equation*}
and 
\begin{equation*} \label{bound for |mathcalE_2(u_m) - mathcalE_2(u*)|}
\begin{aligned}
\left|\mathcal{E}_{2}\left(u_{m}\right)-\mathcal{E}_{2}\left(u^{*}\right)\right| & =\left(\left\|u^{*}\right\|_{L^{2}(\Omega)}+\left\|u_{m}\right\|_{L^{2}(\Omega)}\right)\left|\left\|u^{*}\right\|_{L^{2}(\Omega)}-\left\|u_{m}\right\|_{L^{2}(\Omega)}\right| \\
& \leq\left(2+\eta\left(B_{u^{*}}, m\right)\right) \eta\left(B_{u^{*}}, m\right) \leq 5 \eta\left(B_{u^{*}}, m\right)/2 .
\end{aligned}
\end{equation*}
Since  $a^{2}-b^{2}=(a-b)^{2}+2 b(a-b) ,$ it follows from  $0\leq V\leq V_{\max} $  that
\begin{equation*} \label{ineq: bound for difference |mathcalEV(u)-mathcalEV(u*)|}
\begin{aligned}
\left|\mathcal{E}_{V}(u)-\mathcal{E}_{V}\left(u^{*}\right)\right| 
& \leq \int_{\Omega}\left(\left|\nabla u- \nabla u^{*}\right|^{2} + 2\left|\nabla u^{*}\right| \left| \nabla u-\nabla u^{*}\right|\right) d x \\
& \quad + \int_{\Omega}\left(V\left|u-u^{*}\right|^{2}+2 V\left|u^{*}\right| \left|u-u^{*}\right| \right) d x \\
& \leq \max \left\{1, V_{\max}\right\} \left\|u-u^{*}\right\|_{H^{1}(\Omega)}^{2}
+ 2\left\| \nabla u^{*}\right\|_{L^{2}(\Omega)} \left\| \nabla u-\nabla u^{*}\right\|_{L^{2}(\Omega)}   \\
& \quad  + 2\left\| \sqrt{V}u^{*}\right\|_{L^{2}(\Omega)}  \left\| \sqrt{V}(u-u^{*})\right\|_{L^{2}(\Omega)} \\
& \leq \max \left\{1, V_{\max }\right\} \left\|u-u^{*}\right\|_{H^{1}(\Omega)}^{2} + 2\left(\mathcal{E}_{V}\left(u^{*}\right) \mathcal{E}_{V}\left(u-u^{*}\right)\right)^{1/2} \\
& \leq \max \left\{1, V_{\max }\right\} \left\|u-u^{*}\right\|_{H^{1}(\Omega)}^{2} + 2 \sqrt{\lambda_{k} \max \left\{1, V_{\max }\right\}} \left\|u-u^{*}\right\|_{H^{1}(\Omega)} ,
\end{aligned}
\end{equation*}
\iffalse
\begin{equation} \label{ineq: bound for difference |mathcalEV(u)-mathcalEV(u*)|}
\begin{aligned}
\left|\mathcal{E}_{V}(u)-\mathcal{E}_{V}\left(u^{*}\right)\right| & \leq \int_{\Omega}\left(\left|\nabla u- \nabla u^{*}\right|^{2} + 2\left|\nabla u^{*}\right| \left| \nabla u-\nabla u^{*}\right|\right) d x \\
& \quad + \int_{\Omega}\left(V\left|u-u^{*}\right|^{2}+2 V\left|u^{*}\right| \left|u-u^{*}\right| \right) d x \\
& \leq \max \left\{1, V_{\max}\right\} \left\|u-u^{*}\right\|_{H^{1}(\Omega)}^{2} \\
& \quad + 2\left(\int_{\Omega}\left| \nabla u^{*}\right|^{2} d x\right)^{1/2}   \left(\int_{\Omega}\left|\nabla u-\nabla u^{*}\right|^{2} d x\right)^{1/2} \\
& \quad  + 2\left(\int_{\Omega} V\left|u^{*}\right|^{2} d x\right)^{1/2}\left(\int_{\Omega} V\left|u-u^{*}\right|^{2} d x\right)^{1/2} \\
& \leq \max \left\{1, V_{\max }\right\} \left\|u-u^{*}\right\|_{H^{1}(\Omega)}^{2} + 2\left(\mathcal{E}_{V}\left(u^{*}\right) \mathcal{E}_{V}\left(u-u^{*}\right)\right)^{1/2} \\
& \leq \max \left\{1, V_{\max }\right\} \left\|u-u^{*}\right\|_{H^{1}(\Omega)}^{2} + 2 \sqrt{\lambda_{k} \max \left\{1, V_{\max }\right\}} \left\|u-u^{*}\right\|_{H^{1}(\Omega)} ,
\end{aligned}
\end{equation}
\fi
where we have used $\mathcal{E}_{V}\left(u\right) = \left\| \nabla u\right\|^{2}_{L^{2}(\Omega)} + \left\|\sqrt{V} u\right\|^{2}_{L^{2}(\Omega)}$  and  $\mathcal{E}_{V}\left(u^{*}\right) = \lambda_{k}$. Hence,
\begin{equation*} \label{|mathcalE_V(u_m) - mathcalE_V(u*)|}
\begin{aligned}
\left|\mathcal{E}_{V}\left(u_{m}\right)-\mathcal{E}_{V}\left(u^{*}\right)\right| & \leq \left(3\max \left\{1, V_{\max }\right\}/2  + \lambda_{k} \right) \eta\left(B_{u^{*}}, m\right) .
\end{aligned}
\end{equation*}
Moreover, since $\{\psi_{j}\}_{j=1}^{k-1}$ are normalized orthogonal eigenfunctions,
\begin{equation*} \label{bound for |mathcalE_P(u_m) - mathcalE_P(u*)|}
\begin{aligned}
\left|\mathcal{E}_{P}\left(u_{m}\right)-\mathcal{E}_{P}\left(u^{*}\right)\right|
%& = \beta \left| \sum_{j=1}^{k-1} \left(\langle u_{m}, \psi_{j}\rangle^{2}-\langle u^{*}, \psi_{j}\rangle^{2}\right) \right| \\
%& = \beta \left| \sum_{j=1}^{k-1} \langle u_{m}+u^{*}, \psi_{j}\rangle \langle u_{m}-u^{*}, \psi_{j}\rangle \right| \\
& \leq \beta \left(\sum_{j=1}^{k-1} \langle u_{m}+u^{*}, \psi_{j}\rangle^{2}\right)^{1/2}\left(\sum_{j=1}^{k-1} \langle u_{m}-u^{*}, \psi_{j}\rangle^{2}\right)^{1/2}\\
& \leq \beta \left(\left\|u^{*}\right\|_{L^{2}(\Omega)}+\left\|u_{m}\right\|_{L^{2}(\Omega)}\right) \left\|u^{*}-u_{m}\right\|_{L^{2}(\Omega)}\\
&  \leq 5 \beta \eta\left(B_{u^{*}}, m\right)/2 .
\end{aligned}
\end{equation*}
Substituting all the above estimates into (\ref{tildeum-u* energy excess decomp}), we obtain Theorem \ref{Thm: bounding the approximation error Lk(u in F)-lambdak}.

% Substituting all the above estimates (\ref{lower and upper bound for L2 norm of um}), (\ref{bound for |mathcalE_2(u_m) - mathcalE_2(u*)|}), (\ref{|mathcalE_V(u_m) - mathcalE_V(u*)|}) and (\ref{bound for |mathcalE_P(u_m) - mathcalE_P(u*)|}) into (\ref{tildeum-u* energy excess decomp}), we completes the proof.
\end{proof}

\section{Missing proof in section \ref{section: Statistical error}}

\subsection{Bounding the covering numbers}  \label{appendix section: Bounding the covering numbers}

The following proposition gives an upper bound for the covering number  $\mathcal{N}\left(\delta, \Theta, \rho_{\Theta}\right)$.
\begin{proposition}{\citep[Proposition 5.1]{DRMlu2021priori}}\label{Proposition: covering number of parameter space Theta}
Consider the metric space  $\left(\Theta, \rho_{\Theta}\right)$  with  $\rho_{\Theta}$  defined in (\ref{metric rho defined on parameter space Theta}). Then for any  $\delta>0$, the covering number  $\mathcal{N}\left(\delta, \Theta, \rho_{\Theta}\right)$  satisfies that
$$\mathcal{N}\left(\delta, \Theta, \rho_{\Theta}\right) \leq \frac{2 C}{\delta} \left(\frac{3 \Gamma}{\delta}\right)^{m}  \left(\frac{3 W}{\delta}\right)^{d m}  \left(\frac{3 T}{\delta}\right)^{m} .$$
\end{proposition}

%The next lemma provides an upper bound for  $\mathcal{N}\left(\delta, \mathcal{G}_{m}^{1}/M_{1},\|\cdot\|_{L^{2}(Q)}\right)$.
\begin{proof}[Proof of Proposition \ref{proposition: covering number for Gm1, Gm2 and Gm3}] \label{Appendix proof of Proposition: covering number for Gm1, Gm2 and Gm3}
\textbf{Bounding  $\mathcal{N}\left(\delta, \mathcal{G}_{m}^{1}/M_{1},\|\cdot\|_{L^{2}(Q)}\right)$.} 
For  $\theta, \theta^{\prime} \in \Theta $, by adding and subtracting terms, we may bound $\left|v_{\theta}(x)-v_{\theta^{\prime}}(x)\right|$ by
\begin{equation} \label{decompose utheta-uthetaprime}
\begin{aligned}
& \left|c-c^{\prime}\right|+\left|\sum_{i=1}^{m}\gamma_{i} \phi\left(w_{i} \cdot x-t_{i}\right)-\sum_{i=1}^{m} \gamma_{i}^{\prime} \phi\left(w_{i}^{\prime} \cdot x-t_{i}^{\prime}\right)\right| \\
& \leq\left|c-c^{\prime}\right|+\sum_{i=1}^{m}\left|\gamma_{i}-\gamma_{i}^{\prime}\right| \left|\phi\left(w_{i} \cdot x-t_{i}\right)\right| + \sum_{i=1}^{m}\left|\gamma_{i}^{\prime}\right|\left|\phi\left(w_{i} \cdot x-t_{i}\right)-\phi\left(w_{i}^{\prime} \cdot x-t_{i}^{\prime}\right)\right| .
\end{aligned}
\end{equation}
Since  $\phi$  satisfies Assumption \ref{assumption for activation function}, 
$\left|\phi\left(w_{i} \cdot x-t_{i}\right)-\phi\left(w_{i}^{\prime} \cdot x-t_{i}^{\prime}\right)\right| \leq L\left(\left|w_{i}-w_{i}^{\prime}\right|_{1}+\left|t_{i}-t_{i}^{\prime}\right|\right) .
$
Therefore, it follows from (\ref{decompose utheta-uthetaprime}) that
\begin{equation} \label{utheta(x)-uthetaprime(x) bounded by rho(theta, thetaprime)}
\begin{aligned}
\left|v_{\theta}(x)-v_{\theta^{\prime}}(x)\right| & \leq\left|c-c^{\prime}\right|+\phi_{\max }\left|\gamma-\gamma^{\prime}\right|_{1} + L \Gamma\left(\max _{i}\left|w_{i}-w_{i}^{\prime}\right|_{1}+\left|t-t^{\prime}\right|_{\infty}\right) \\
& \leq\left(1+\phi_{\max }+2 L \Gamma\right) \rho_{\Theta}\left(\theta, \theta^{\prime}\right).
\end{aligned}
\end{equation}
Consequently, using \ref{maximum of Linf norm for vtheta} and \ref{utheta(x)-uthetaprime(x) bounded by rho(theta, thetaprime)}, we obtain
\begin{equation*} \label{}
\begin{aligned}
\left\| \varphi^{2} v_{\theta}^{2} - \varphi^{2} v_{\theta^{\prime}}^{2}\right\|_{L^{2}(Q)} & \leq \left\|\varphi^{2}\right\|_{L^{2}(Q)} \left(\left\|v_{\theta}\right\|_{*} + \left\|v_{\theta}^{\prime}\right\|_{*} \right) \left\|v_{\theta} - v_{\theta}^{\prime}\right\|_{*}  \\
& \leq 2\left(C+\Gamma \phi_{\max }\right)\left(1+\phi_{\max }+2 L \Gamma\right)  \rho_{\Theta}\left(\theta, \theta^{\prime}\right) /d^{2}. %\\
%& = \Lambda_{1} \rho_{\Theta}\left(\theta, \theta^{\prime}\right) ,
\end{aligned}
\end{equation*}
Thus, for any $g_{\theta}, g_{\theta^{\prime}} \in \mathcal{G}_{m}^{1},$ $\left\| g_{\theta}/M_{1} - g_{\theta^{\prime}}/M_{1}\right\|_{L^{2}(Q)} \leq \Lambda_{1} \rho_{\Theta}\left(\theta, \theta^{\prime}\right) / M_{1} $ and then
%Dividing both sides of the above inequality by $M_{1}$, we have  
$$\mathcal{N}\left(\delta, \mathcal{G}_{m}^{1}/M_{1},\|\cdot\|_{L^{2}(Q)}\right) \leq \mathcal{N}\left(M_{1} \delta/\Lambda_{1}, \Theta, \rho_{\Theta}\right) 
\leq \mathcal{M}\left(\delta, \Lambda_{1}/M_{1}, m, d\right),$$ 
where the second inequality follows from Proposition \ref{Proposition: covering number of parameter space Theta} with  $\delta$  replaced by  $M_{1} \delta/\Lambda_{1}$.

\textbf{Bounding  $\mathcal{N}\left(\delta, \mathcal{G}_{m}^{2}/M_{2},\|\cdot\|_{L^{2}(Q)}\right)$.} 
As in (\ref{decompose utheta-uthetaprime}), by adding and subtracting terms, %we have
\begin{equation*} \label{nabla utheta(x)-uthetaprime(x) bounded by rho(theta, thetaprime)}
\begin{aligned}
\left|\nabla v_{\theta}(x) \! - \! \nabla v_{\theta^{\prime}}(x)\right| & \leq \sum_{i=1}^{m}\left|\gamma_{i}-\gamma_{i}^{\prime}\right|\left|w_{i}\right|_{1}\left|\phi^{\prime}\left(w_{i}  \! \cdot \!  x \! + \! t_{i}\right)\right| +\sum_{i=1}^{m}\left|\gamma_{i}^{\prime}\right|\left|w_{i}-w_{i}^{\prime}\right|_{1}\left|\phi^{\prime}\left(w_{i}  \! \cdot \!  x \! + \! t_{i}\right)\right| \\
& \quad +\sum_{i=1}^{m}\left|\gamma_{i}^{\prime}\right|\left|w_{i}^{\prime}\right|_{1}\left|\phi^{\prime}\left(w_{i}  \! \cdot \!  x \! + \! t_{i}\right)-\phi^{\prime}\left(w_{i}^{\prime}  \! \cdot x \! + \! t_{i}^{\prime}\right)\right| \\
& \leq W \phi_{\max }^{\prime}\left|\gamma-\gamma^{\prime}\right|_{1}+\Gamma \phi_{\max }^{\prime} \max _{i}\left|w_{i}-w_{i}^{\prime}\right|_{1}  + \Gamma W L^{\prime}\cdot  \\
& \quad \left(\max _{i}\left|w_{i}-w_{i}^{\prime}\right|_{1} + \left|t-t^{\prime}\right|_{\infty}\right) \\
& \leq\left((W+\Gamma) \phi_{\max }^{\prime}+2 \Gamma W L^{\prime}\right) \rho_{\Theta}\left(\theta, \theta^{\prime}\right) .
\end{aligned}
\end{equation*}
Combining (\ref{utheta(x)-uthetaprime(x) bounded by rho(theta, thetaprime)}) and the above estimate, we obtain
\begin{equation*} \label{}
\begin{aligned}
\left\|\left|\nabla \!  \left(\varphi v_{\theta}\right) \!- \!  \nabla  \! \left(\varphi v_{\theta^{\prime}} \right) \right|\right\|_{L^{2}(Q)} & \leq \left\|\varphi\right\|_{L^{2}(Q)} \left\|\nabla v_{\theta}-\nabla v_{\theta^{\prime}}\right\|_{*} + \left\| \nabla\varphi \right\|_{L^{2}(Q)} \left\|v_{\theta}-v_{\theta^{\prime}}\right\|_{*} \\
& \leq \left[ \left((W \! + \! \Gamma) \phi_{\max }^{\prime}  +  2 \Gamma W L^{\prime}\right)  /   d + \pi\left(1 \! + \! \phi_{\max } \! + \! 2 L \Gamma\right) \right]\rho_{\Theta}\left(\theta, \theta^{\prime}\right) .
\end{aligned}
\end{equation*}
Using the fact
$\max_{\theta \in \Theta} \left\|\left|\nabla \left(\varphi v_{\theta}\right)\right|\right\|_{*} \leq \Gamma W \phi_{\max }^{\prime} /d + \pi\left(C+\Gamma \phi_{\max }\right) ,$ we have
\begin{equation*} \label{}
\begin{aligned}
\left\| \left|\nabla\left(\varphi v_{\theta}\right)\right|^{2}-\left|\nabla\left(\varphi v_{\theta^{\prime}}\right)\right|^{2} \right\|_{L^{2}(Q)}
& \leq \left\| \left|\nabla\left(\varphi v_{\theta}\right)\right|+\left|\nabla\left(\varphi v_{\theta^{\prime}}\right)\right| \right\|_{*} \left\|\nabla \left(\varphi v_{\theta}\right) - \nabla \left(\varphi v_{\theta^{\prime}}\right) \right\|_{L^{2}(Q)} \\
& \leq \Lambda_{21} \rho_{\Theta}\left(\theta, \theta^{\prime}\right),
\end{aligned}
\end{equation*}
where  $\Lambda_{21} := 2\left[\Gamma W \phi_{\max }^{\prime}/d + \pi \left(C + \Gamma\phi_{\max }\right)\right] \left[\big((W+\Gamma) \phi_{\max }^{\prime}+2 \Gamma W L^{\prime}\big)/d + \pi\left(1+\phi_{\max }+  \right.\right.$ $\left.\left. 2 L \Gamma\right) \right].$
It follows from $0 \leq V \leq V_{\max}$ and (\ref{utheta(x)-uthetaprime(x) bounded by rho(theta, thetaprime)}) that
\begin{equation*} \label{}
\begin{aligned}
\left\| V \varphi^{2} v_{\theta}^{2} - V \varphi^{2} v_{\theta^{\prime}}^{2}\right\|_{L^{2}(Q)} & \leq V_{\max} \left\|\varphi\right\|_{*}^{2} \left(\left\|v_{\theta}\right\|_{*} + \left\|v_{\theta^{\prime}}\right\|_{*} \right) \left\|v_{\theta} - v_{\theta^{\prime}}\right\|_{*}  \\
& \leq 2V_{\max}\left(C+\Gamma \phi_{\max }\right)\left(1+\phi_{\max }+2 L \Gamma\right) \rho_{\Theta}\left(\theta, \theta^{\prime}\right) /d^2 \\
& =: \Lambda_{22} \rho_{\Theta}\left(\theta, \theta^{\prime}\right) .
\end{aligned}
\end{equation*}
Hence, combining the last two estimates, for any $g_{\theta}, g_{\theta^{\prime}} \in \mathcal{G}_{m}^{2},$
\begin{equation*} \label{}
\begin{aligned}
\left\| g_{\theta}-g_{\theta^{\prime}} \right\|_{L^{2}(Q)}  & \leq \left\| \left|\nabla\left(\varphi v_{\theta}\right)\right|^{2}-\left|\nabla\left(\varphi v_{\theta^{\prime}}\right)\right|^{2} \right\|_{L^{2}(Q)} + \left\| V \varphi^{2} v_{\theta}^{2} - V \varphi^{2} v_{\theta^{\prime}}^{2}\right\|_{L^{2}(Q)}   \\
& \leq \left(\Lambda_{21} + \Lambda_{22}\right) \rho_{\Theta}\left(\theta, \theta^{\prime}\right) = \Lambda_{2} \rho_{\Theta}\left(\theta, \theta^{\prime}\right) .
\end{aligned}
\end{equation*}
Dividing both sides of the above inequality by $M_{2}$, we obtain
$$\mathcal{N}\left(\delta, \mathcal{G}_{m}^{2}/M_{2},\|\cdot\|_{L^{2}(Q)}\right) \leq \mathcal{N}\left(M_{2}\delta/\Lambda_{2}, \Theta, \rho_{\Theta}\right) 
\leq \mathcal{M}\left(\delta, \Lambda_{1}/M_{1}, m, d\right), $$ 
where the second inequality follows from Proposition \ref{Proposition: covering number of parameter space Theta} with  $\delta$  replaced by  $M_{2}\delta/\Lambda_{2}$.

\textbf{Bounding  $\mathcal{N}\left(\delta, \mathcal{G}_{m}^{3}/M_{3},\|\cdot\|_{L^{2}(Q)}\right)$.} 
\iffalse
Recall that $\{\psi_{j}\}_{j=1}^{k-1}$ are the first $k-1$ normalized eigenfunctions. 
Let $\|\psi_{j}\|_{L^{\infty}(\Omega)} \leq \mu_{j}$ for each $j$. 
Then, $\left(C+\Gamma \phi_{\max }\right)\psi_{j}/d$ is a measurable envelope
of $\mathcal{G}_{m}^{3, j}$ and $$ \sup _{g \in \mathcal{G}_{m}^{3, j}} \|g\|_{L^{\infty}(\Omega)} \leq \left(C+\Gamma \phi_{\max }\right)\mu_{j}/d.$$ 
\fi
%Similarly, we obtain the following lemma. %upper bound for  $\mathcal{N}\left(\delta, \mathcal{G}_{m}^{3, j}/M_{3,j},\|\cdot\|_{L^{2}(Q)}\right)$.
It follows from (\ref{utheta(x)-uthetaprime(x) bounded by rho(theta, thetaprime)}) that
\begin{equation*} \label{}
\begin{aligned}
\left\|\varphi \psi_{j} v_{\theta} - \varphi \psi_{j} v_{\theta^{\prime}}\right\|_{L^{2}(Q)} & \leq  \left\|\varphi\right\|_{*} \left\|\psi_{j}\right\|_{L^{2}(Q)} \left\|v_{\theta} - v_{\theta^{\prime}}\right\|_{*}  \\
& \leq \left\|\psi_{j}\right\|_{L^{2}(Q)}\left(1+\phi_{\max }+2 L \Gamma\right) \rho_{\Theta}\left(\theta, \theta^{\prime}\right) /d.
\end{aligned}
\end{equation*}
The bound for $\mathcal{N}\left(\delta, \mathcal{G}_{m}^{3}/M_{3},\|\cdot\|_{L^{2}(Q)}\right)$ follows from a similar argument that leads to the bound of $\mathcal{N}\left(\delta, \mathcal{G}_{m}^{2}/M_{2},\|\cdot\|_{L^{2}(Q)}\right)$.
\end{proof}

\subsection{Estimates of the expectation of suprema of empirical processes} \label{appendix section: Estimates of the expectation of suprema of empirical processes}

\begin{proof}[Proof of Lemma \ref{Prop 2.1 from Gin2001OnCO, U=2, centered}] \label{appendix proof of Lemma: Prop 2.1 from Gin2001OnCO, U=2, centered}
By standard symmetrization, 
$$\mathbf{E}\left\| \sum_{i=1}^{n} \left( f\left(X_{i}\right)-\mathcal{P}f\right) \right\|_{\mathscr{F}} \leq 2 \mathbf{E}\left\| \sum_{i=1}^{n} \epsilon_i \left( f\left(X_{i}\right)-\mathcal{P}f\right)\right\|_{\mathscr{F}}.$$
Next, we shall apply Lemma \ref{Proposition 2.1 in cite Gin2001OnCO, original version} to the centered class  $\bar{\mathscr{F}} := \left\{f - \mathcal{P} f: f \in \mathscr{F}\right\} .$
To this end, we verify the assumptions in Lemma \ref{Proposition 2.1 in cite Gin2001OnCO, original version} for $\bar{\mathscr{F}}$.
If functions in $\mathscr{F}$ take values in  $[-1, 1]$, then functions in $\bar{\mathscr{F}}$ take values in  $[-2, 2]$.
If  $F$  is a measurable envelope of $\mathscr{F}$, then $F + \mathcal{P}F$ is a measurable envelope of $\bar{\mathscr{F}}$.
Since for any $f \in \mathscr{F}$ and $0 < \tau < 1$, there exists $f_{i}$ with $\left\| f -f_{i}\right\|_{L^{2}(Q)} \leq  \tau\|F\|_{L^{2}(Q)}$ for all probability measures $Q$,
then 
%$\left|f - f_{i}\right| \leq \tau F$ and
$\left|\mathcal{P}\left(f -f_{i}\right)\right| \leq \mathcal{P}\left|f -f_{i}\right| \leq \tau \mathcal{P} F .$
Hence, 
\begin{equation*}
\begin{aligned}
\left\| \left(f - \mathcal{P} f\right)-\left(f_{i}-\mathcal{P} f_{i}\right)\right\|_{L^{2}(Q)} & \leq \left\|f - f_{i}\right\|_{L^{2}(Q)} + \left|\mathcal{P}\left(f - f_{i}\right)\right| 
 \\
& \leq \tau\left(\|F\|_{L^{2}(Q)}+\mathcal{P} F\right) \leq \sqrt{2} \tau \|F+\mathcal{P} F\|_{L^{2}(Q)},
\end{aligned}
\end{equation*}
which means that $\left\{f_{i} - \mathcal{P}f_{i}\right\}_{i = 1}^{M}$ is a $\sqrt{2}\tau\|F+\mathcal{P} F\|_{L^{2}(Q)}$-net for $\bar{\mathscr{F}}$.
Therefore, $\bar{\mathscr{F}}$ is a VC class and
\begin{equation*} \label{covering number for centered VC class mathcalF}
\begin{aligned} 
\mathcal{N}\left(\tau\|F+\mathcal{P} F\|_{L^{2}(Q)}, \bar{\mathscr{F}}, \|\cdot\|_{L^{2}(Q)}\right) \leq \left(\sqrt{2} A / \tau \right)^{v}
\end{aligned}
\end{equation*}
for all probability measures $Q$ and  $0 < \tau < 1$. 
The proof is completed by applying Lemma \ref{Proposition 2.1 in cite Gin2001OnCO, original version} to $\bar{\mathscr{F}}$.
\end{proof}

\begin{proof}[Proof of Corollary \ref{estimates of En,q(r, s], U = 1}] \label{appendix proof of Corollary: estimates of En,q(r, s], U = 1}
Applying Lemma \ref{Prop 2.1 from Gin2001OnCO, U=2, centered} to $\mathscr{F}\left(\rho_{j}\right)$ with $\rho_{j} = 2^{j/2}r$, we have
\begin{equation*} \label{}
\begin{aligned} 
\mathcal{K}_{n}(\mathscr{F}, r) &  \leq \max_{1 \leq j \leq l} \frac{\mathbf{E}\left\|\mathcal{P}_{n}-\mathcal{P}\right\|_{\mathscr{F}\left(\rho_{j}\right)}}{\rho_{j-1}^{2}}
\leq C \max_{1 \leq j \leq l} \left( \frac{v}{n \rho_{j-1}^{2}} \ln \frac{A}{\rho_{j}} +   \sqrt{\frac{2v}{n\rho_{j-1}^{2}} \ln \frac{A}{\rho_{j}}}\right).
\end{aligned}
\end{equation*}
Notice that the quantity in parenthesis decreases as $j$ increases and the maximum value reaches at $j=1$, which completes the proof.
\end{proof}

\subsection{Bounding \texorpdfstring{$\mathcal{K}_{n}$}{} in the statistical error}
\begin{proof}[Proof of Theorem \ref{Thm:  bounds for all spaces needed in oracle inequality for the generalization error}] \label{Appendix proof of Thm:  bounds for all spaces needed in oracle inequality for the generalization error}
With (\ref{constants in assumption for Softplus activation function}) and (\ref{concrete representation of constants C Gamma W T}), a direct calculation yields %bounds for amplification factors 
\begin{equation} \label{upper bound of Lambda for covering with Softplus activation}
\begin{aligned}
\Lambda_{1} & \leq  19B \left(4+8B\right)/d^2 , \quad \\
\Lambda_{2} & \leq 68B\left(11 + 30B + 17B\sqrt{m}/d\right) + 19 V_{\max} B \left(4+8B\right)/d^2, \\
\Lambda_{3} & \leq \|\psi\|_{L^{2}(Q)}\left(4+8B\right) / d ,
\end{aligned}
\end{equation}
and for $B \geq 1$, %the constant 
\begin{equation} \label{upper bound of mathcalM with Softplus activation}
\begin{aligned} 
\mathcal{M}(\delta, \Lambda, m, d) 
%& =  \frac{2 C \Lambda}{\delta} \left(\frac{3 \Gamma \Lambda}{\delta}\right)^{m} \left(\frac{3 W \Lambda}{\delta}\right)^{d m} \left(\frac{3 T \Lambda}{\delta}\right)^{m} \\
& = 2^{2m+1} 3^{(d+2)m} B^{m+1} \left(\Lambda/\delta\right)^{(d+2)m+1}\\
& \leq \left( 2^{\frac{3}{d+3}} B^{\frac{2}{d+3}} 3 \Lambda/\delta\right)^{(d+2)m+1} .\\
\end{aligned}
\end{equation}

For $\mathcal{G}_{1}/M_{\mathcal{F}}^{2}$, it has a measurable envelope $F \equiv 1$.
By Lemma \ref{covering number for Gm1}, (\ref{upper bound of maximum value of hypothesis class with Softplus activation}), (\ref{upper bound of Lambda for covering with Softplus activation}) and (\ref{upper bound of mathcalM with Softplus activation}),
\begin{equation*} \label{}
\begin{aligned}
\mathcal{N}\left(\tau \|F\|_{L^{2}(Q)}, \mathcal{G}_{1}/M_{\mathcal{F}}^{2},\|\cdot\|_{L^{2}(Q)}\right) & \leq \mathcal{M}\left(\tau, \Lambda_{1}/M_{\mathcal{F}}^{2}, m, d\right) \\
& \leq \left( 2^{\frac{3}{d+3}} B^{\frac{2}{d+3}} 6B \left(4+8B\right)/\left(9.5B^{2}\tau\right)\right)^{(d+2)m+1}\\
& \leq \left( 2^{\frac{3}{d+3} + 3} B^{\frac{2}{d+3}} /\tau\right)^{(d+2)m+1}.
\end{aligned}
\end{equation*}
Thus, we may take $A = 2^{\frac{3}{d+3}+3} B^{\frac{2}{d+3}}$, $v = (d+2)m+1$, $\hat{r} = r/M_{\mathcal{F}}$ and then
\begin{equation*} \label{}
\begin{aligned}
\frac{v}{n \hat{r}^{2}} \ln \frac{A }{\hat{r}} & 
\leq C_{1} \frac{mB^{2}}{n d r^{2}} \ln \frac{ B}{rd} =: I_{1},
\end{aligned}
\end{equation*}
where $C_{1}$ is an absolute constant.
If  $I_{1} \leq 1$,  then (\ref{bound for mathcalK_n(mathcalG_1/M_mathcalF^2)}) follows from applying Corollary \ref{estimates of En,q(r, s], U = 1} to $\mathcal{G}_{1}/M_{\mathcal{F}}^{2}$. 
\iffalse
\begin{equation} \label{Simplifying condition for mathcalG1/MmathcalF2}
C_{1} \frac{mB^{2}}{n d r^{2}} \ln \frac{ B}{rd} \leq 1,
\end{equation} \fi

For $\mathcal{G}_{2}/M_{\mathcal{G}_{2}}$, it has a measurable envelope $F \equiv 1$. By Lemma \ref{covering number for Gm2}, (\ref{upper bound of maximum value of hypothesis class with Softplus activation}), (\ref{upper bound of Lambda for covering with Softplus activation}) and (\ref{upper bound of mathcalM with Softplus activation}),
\begin{equation*} \label{}
\begin{aligned}
\mathcal{N}\left(\tau \|F\|_{L^{2}(Q)}, \mathcal{G}_{2}/M_{\mathcal{G}_{2}},\|\cdot\|_{L^{2}(Q)}\right) &  \leq \mathcal{M}\left(\tau, \Lambda_{2}/M_{\mathcal{G}_{2}}, m, d\right) \\
& \leq \left( 2^{\frac{3}{d+3}} B^{\frac{2}{d+3}} \left(8 + 3\sqrt{m}/d\right)/\tau\right)^{(d+2)m+1}.
\end{aligned}
\end{equation*}
Thus, we may take $A = 2^{\frac{3}{d+3}} B^{\frac{2}{d+3}} \left(8 + 3\sqrt{m}/d\right)$, $v = (d+2)m+1$, $\hat{r} = r\sqrt{\lambda_{1}/M_{\mathcal{G}_{2}}}$ and thus
\begin{equation*} \label{}
\begin{aligned}
\frac{v}{n \hat{r}^{2}} \ln \frac{A}{\hat{r}} & \leq  \tilde{C}_{2} \frac{d m \left(B^{2} + V_{\max } \left(B/d\right)^{2}\right)}{n \lambda_{1} r^{2}} \ln \frac{ B^{\frac{2}{d+3}} \left(1 + \sqrt{m}/d\right) \left(B^{2} + V_{\max } \left(B/d\right)^{2}\right)}{r \sqrt{\lambda_{1}} }\\
& \leq C_{2} \frac{m B^{2} \left(1 + V_{\max } \right)}{n r^{2}} \ln \frac{ B \left(1 + \sqrt{m}/d\right) \left(1 + V_{\max}\right)}{r d } =: I_{2},
\end{aligned}
\end{equation*}
where we have used $\lambda_{1}\geq d\pi^{2}$ in the second inequality.
% and $\tilde{C}_{2}$, $C_{2}$ are absolute constants. 
If  $I_{2} \leq 1$,  (\ref{bound for mathcalK_n(mathcalG_2/M_mathcalG_2)}) follows from applying Corollary \ref{estimates of En,q(r, s], U = 1} to $\mathcal{G}_{2}/M_{\mathcal{G}_{2}}$.
\iffalse
\begin{equation} \label{Simplifying condition for mathcalG2/MmathcalG2}
C_{2} \frac{m B^{2} \left(1 + V_{\max } \right)}{n r^{2}} \ln \frac{ B \left(1 + \sqrt{m}/d\right) \left(1 + V_{\max}\right)}{r d } \leq 1,
\end{equation} 
\fi
Let  $C_{0} = \max \{C_{1}, C_{2}\}$, and then (\ref{Global simplifying condition for statistical error}) ensures both $I_{1} \leq 1$ and $I_{2} \leq 1$.

To bound $\mathcal{K}_{n}(\mathcal{F}_{j}/\left(2\mu_{j} M_{\mathcal{F}}\right) + 1/2, \sqrt{\hat{r}/4 \mu_{j} M_{\mathcal{F}}})$, %$1 \leq j \leq k-1$, 
recall (\ref{def sigma_P^2(f) for mathcalF_j/(2mu_j M_mathcalF) + 1/2}) for the choice of $\sigma_{\mathcal{P}}(f)$,  %$$\sigma_{P}^{2}(f) = \frac{\|u\|_{L^{2}(\Omega)}}{4 \mu_{j} M_{\mathcal{F}}} \geq \operatorname{Var}_{P}(f), \quad \text{ for}\ \ f = \frac{u \psi_{j}}{2 \mu_{j} M_{\mathcal{F}}} +\frac{1}{2} \in \frac{\mathcal{F}_{j}}{2\mu_{j} M_{\mathcal{F}}} + \frac{1}{2}.$$
%For generality and brevity, we may omit the subscripts of  $\mu_{j}$ and $\psi_{j}$.
and note that 
\begin{equation} \label{rewrite mathcalK_n(mathcalF_j) in the proof of statistical bounds}
\begin{aligned}
\mathcal{K}_{n}\left(\frac{\mathcal{F}_{j}}{2\mu_{j} M_{\mathcal{F}}} \! + \! \frac{1}{2},  \sqrt{\frac{r}{4 \mu_{j} M_{\mathcal{F}}}}\right)  %& \leq \max_{1 \leq i \leq l} \frac{1}{\rho_{i-1}^{2}} \mathbf{E} \sup _{\frac{\|u\|_{L^{2}}}{4 \mu_{j} M_{\mathcal{F}}} \in \left(\rho^{2}_{i-1}, \rho^{2}_{i}\right]} \frac{\left|P_{n}(u \psi_{j})-P(u \psi_{j})\right|}{2 \mu_{j} M_{\mathcal{F}}}  \\
& = \frac{1}{2} \max_{1 \leq i \leq l} \frac{1}{\rho_{i-1}^{2}} \mathbf{E} \sup _{ \|u\|_{L^{2}}/(\mu_{j} M_{\mathcal{F}}) \in \left(4\rho^{2}_{i-1}, 4\rho^{2}_{i}\right]} \frac{\left|\mathcal{P}_{n}\left(u \psi_{j}\right)  \! - \!  \mathcal{P}\left(u \psi_{j}\right) \right|}{\mu_{j} M_{\mathcal{F}}} .
\end{aligned}
\end{equation}
To estimate the expectation, we may apply Lemma \ref{Prop 2.1 from Gin2001OnCO, U=2, centered} to the set $$\left\{\left.\frac{u \psi_{j}}{M_{\mathcal{F}} \mu_{j}} \right\rvert\, u \in \mathcal{F}, \frac{\|u\|_{L^{2}}}{ \mu_{j} M_{\mathcal{F}}} \in\left(4\rho^{2}_{i-1}, 4\rho^{2}_{i}\right] \right\},$$
which has a measurable envelope $F_{j} = \psi_{j}/\mu_{j}$. 
Taking $\psi = \psi_{j}$, by Lemma \ref{covering number for Gm3}, (\ref{upper bound of maximum value of hypothesis class with Softplus activation}), (\ref{upper bound of Lambda for covering with Softplus activation}) and (\ref{upper bound of mathcalM with Softplus activation}),
\begin{equation*} \label{}
\begin{aligned}
\mathcal{N}\left(\tau \|F_{j}\|_{L^{2}(Q)}, \mathcal{F}_{j}/(\mu_{j} M_{\mathcal{F}}),\|\cdot\|_{L^{2}(Q)}\right) & \leq \mathcal{M}\left(\tau \|\psi_{j}\|_{L^{2}(Q)}/\mu_{j}, \Lambda_{3,j}/(\mu_{j} M_{\mathcal{F}}), m, d\right)  \\
& \leq \left(8 B^{\frac{2}{d+3}} / \tau\right)^{(d+2)m+1}.
\end{aligned}
\end{equation*}
Taking $A = 8 B^{\frac{2}{d+3}}$, $v = (d+2)m+1$, $\sigma = 2\rho_{i}$ in Lemma \ref{Prop 2.1 from Gin2001OnCO, U=2, centered} yields that for all  $n \in \mathbb{N}$,
\begin{equation*} \label{}
\begin{aligned}
\mathbf{E} \sup _{\|u\|_{L^{2}}/(\mu_{j} M_{\mathcal{F}}) \in\left(4\rho^{2}_{i-1}, 4\rho^{2}_{i}\right]} \frac{\left|\mathcal{P}_{n}\left(u \psi\right) - \mathcal{P}\left(u \psi\right) \right|}{\mu M_{\mathcal{F}}} & \leq  C\left( \frac{dm}{n} \ln \frac{B}{\rho_{j}} +  \rho_{j} \sqrt{\frac{dm}{n} \ln \frac{B}{\rho_{j}}}\right),
\end{aligned}
\end{equation*}
where  $C$ is an absolute constant.
Therefore, substituting the above bound into (\ref{rewrite mathcalK_n(mathcalF_j) in the proof of statistical bounds})
and noting that $\rho_{i} = 2^{i/2} \sqrt{r/(4\mu_{j} M_{\mathcal{F}})}  = 2^{i/2} \sqrt{r d/(38 \mu_{j} B)}$, we obtain (\ref{bound for mathcalK_n(mathcalF_j/2mu_j M_mathcalF + 1/2)}).
\iffalse
\begin{equation*} \label{}
\begin{aligned}
\mathcal{K}_{n}\left(\frac{\mathcal{F}_{j}}{2\mu_{j} M_{\mathcal{F}}} + \frac{1}{2},  \sqrt{\frac{r}{4 \mu_{j} M_{\mathcal{F}}}}\right)  & \leq \max_{1 \leq i \leq l} \frac{1}{\rho_{i-1}^{2}} \mathbf{E} \sup _{\frac{\|u\|_{L^{2}}}{4 \mu_{j} M_{\mathcal{F}}} \leq \rho^{2}_{i}} \frac{\left|\mathcal{P}_{n}(u \psi_{j})-\mathcal{P}(u \psi_{j})\right|}{2 \mu_{j} M_{\mathcal{F}}}  \\
& = \frac{1}{2} \max_{1 \leq i \leq l} \frac{1}{\rho_{i-1}^{2}} \mathbf{E} \sup _{\frac{\|u\|_{L^{2}(\Omega)}}{ \mu_{j} M_{\mathcal{F}}} \leq 4\rho^{2}_{i}} \frac{\left|\mathcal{P}_{n}\left(u \psi_{j}\right) - \mathcal{P}\left(u \psi_{j}\right) \right|}{\mu_{j} M_{\mathcal{F}}} \\
& \leq  C\left[ \frac{m\mu_{j} B}{n  \hat{r}} \ln \left(\frac{B\mu_{j}}{\hat{r} d}\right) +  \sqrt{\frac{m\mu_{j} B}{n  \hat{r}} \ln \left(\frac{B\mu_{j}}{\hat{r} d}\right)}\right],
\end{aligned}
\end{equation*} 
\fi
\end{proof}

\section{Missing proof in section \ref{Section: Solution theory in spectral Barron Spaces}} \label{Appendix Section: Solution theory in spectral Barron Spaces}
%\section{Some useful facts on sine and cosine series} 
\subsection{Some useful facts on sine and cosine series} \label{section: Some useful facts on sine and cosine series}
%Assume that  $u \in L^{1}(\Omega)$  satisfying homogeneous Dirichlet boundary condition admits the sine series expansion
Assume that  $u \in L^{1}(\Omega)$  admits the sine series expansion $u(x)=\sum_{k \in \mathbb{N}_{+}^{d}} \hat{u}(k) \Phi_{k}(x),$ 
\iffalse
\begin{equation*} \label{}
\begin{aligned}
u(x)=\sum_{k \in \mathbb{N}_{+}^{d}} \hat{u}(k) \Phi_{k}(x), 
\end{aligned}
\end{equation*}
\fi
where %$k=\left(k_{1}, k_{2}, \ldots, k_{d}\right) $, 
%$\Phi_{k}(x)=\prod_{i=1}^{d} \sin \left(k_{i} \pi x_{i}\right)$  
%and  
%$\left\{\hat{u}(k)\right\}_{k \in \mathbb{N}_{+}^{d}}$  
$\hat{u}(k)$  are the sine expansion coefficients, i.e.,
\begin{equation} \label{sine expansion coefficients, integral representation}
\begin{aligned}
\hat{u}(k) =\frac{\int_{\Omega} u(x) \Phi_{k}(x) d x}{\int_{\Omega} \Phi_{k}^{2}(x) d x} = 2^{d}\int_{\Omega} u(x) \Phi_{k}(x) d x. %,  \quad \text{for } k \in \mathbb{N}_{+}^{d}.
\end{aligned}
\end{equation}
Let  $\widetilde{\Omega}:=[-1,1]^{d}$  and define the odd extension  $u_{o}$  of a function  $u$  by
\begin{equation*} \label{}
\begin{aligned}
u_{o}(x) = u_{o}\left(x_{1}, x_{2}, \ldots, x_{d}\right) 
= \operatorname{sign}\left(x_{1} x_{2} \cdots x_{d}\right) u\left(\left|x_{1}\right|,\left|x_{2}\right|, \ldots,\left|x_{d}\right|\right),\quad x \in \widetilde{\Omega} ,
\end{aligned}
\end{equation*}
where $\operatorname{sign}\left(y\right) = \mathbf{1}_{\{y > 0\}} - \mathbf{1}_{\{y < 0\}}.$ 
\iffalse
\begin{equation*} \label{}
\operatorname{sign}\left(y\right) = 
\begin{cases}
1,  & \text{ if } y > 0; \\
0,  & \text{ if } y = 0; \\
-1,  & \text{ if } y < 0.
\end{cases}
\end{equation*} 
\fi
Let  $\tilde{u}_{o}(k)$  be the Fourier coefficients of  $u_{o}$, i.e., 
$u_{o}(x) = \sum_{k \in \mathbb{Z}^{d}} \tilde{u}_{o}(k) e^{i \pi k \cdot x},$
where $ \tilde{u}_{o}(k) = 2^{-d} \int_{\widetilde{\Omega}} u_{o}(x) e^{-i \pi k \cdot x} d x .$  
By abuse of notation, we use  $|k|$  to stand for the vector  $\left(\left|k_{1}\right|, \left|k_{2},\right|,  \ldots , \left|k_{d}\right|\right)$ in this section and 
denote by $|k|_{2}$ the Euclid norm of the vector $k$.
Since  $u_{o}$  is real and odd,  $\tilde{u}_{o}(k) = \operatorname{sign}\left(k_{1} k_{2} \cdots k_{d}\right) \tilde{u}_{o}(|k|).$
Particularly, when  $d$  is odd,  $u_{o}(x) = -u_{o}(-x) ,$  $\tilde{u}_{o}(k)  = - \tilde{u}_{o}(-k)$  and 
\begin{equation} \label{d is odd, Fourier sine expan}
\begin{aligned}
u_{o}(x) = \sum_{k \in \mathbb{Z}^{d}} \tilde{u}_{o}(k) i \sin (\pi k \cdot x) \quad \text{with} \quad
i \tilde{u}_{o}(k) = 2^{-d} \int_{\widetilde{\Omega}} u_{o}(x) \sin (\pi k \cdot x) d x.
\end{aligned}
\end{equation}
\iffalse
\begin{equation} \label{d is odd, Fourier sine expan}
\begin{aligned}
u_{o}(x) & = \sum_{k \in \mathbb{Z}^{d}} \tilde{u}_{o}(k) i \sin (\pi k \cdot x), \\
i \tilde{u}_{o}(k) & = 2^{-d} \int_{\widetilde{\Omega}} u_{o}(x) \sin (\pi k \cdot x) d x.
\end{aligned}
\end{equation}
\fi
where $i = \sqrt{-1}.$
When  $d$  is even,  $u_{o}(x)=u_{o}(-x) ,$  $\tilde{u}_{o}(k)  = \tilde{u}_{o}(-k)$  and 
\begin{equation} \label{d is even, Fourier cosine expan}
\begin{aligned}
u_{o}(x) = \sum_{k \in \mathbb{Z}^{d}} \tilde{u}_{o}(k) \cos (\pi k \cdot x) \quad \text{with} \quad 
\tilde{u}_{o}(k) = 2^{-d} \int_{\widetilde{\Omega}} u_{o}(x) \cos (\pi k \cdot x) d x .
\end{aligned}
\end{equation}
\iffalse
\begin{equation} \label{d is even, Fourier cosine expan}
\begin{aligned}
u_{o}(x) & = \sum_{k \in \mathbb{Z}^{d}} \tilde{u}_{o}(k) \cos (\pi k \cdot x) \\
\tilde{u}_{o}(k) & =2^{-d} \int_{\widetilde{\Omega}} u_{o}(x) \cos (\pi k \cdot x) d x .
\end{aligned}
\end{equation}
\fi

We may extend $\{\hat{u}(k)\}$ from a sequence on $\mathbb{N}_{+}^{d}$ to a sequence on $\mathbb{N}_{0}^{d}$ by letting $\hat{u}(k) = 0$ if $k \in \mathbb{N}_{0}^d\setminus \mathbb{N}_{+}^d$. 
The relation between $\tilde{u}_{o}(k)$ and $\hat{u}(k)$ is established in the following lemma. 
\begin{lemma} \label{tildeuk represented by hatuk}
%Let $\hat{u}(k) = 0$ if $k \in \mathbb{N}_{0}^d\setminus \mathbb{N}_{+}^d$.
For every  $k \in \mathbb{Z}^{d}$, it holds 
\begin{equation*} \label{}
\begin{aligned}
& i \tilde{u}_{o}(k) =  2^{-d} (-1)^{\frac{d-1}{2}} \operatorname{sign}\left(k_{1} k_{2} \cdots k_{d}\right) \hat{u}(|k|), & \quad \text{ if $d$ is odd}; \\
& \tilde{u}_{o}(k) = 2^{-d}(-1)^{\frac{d}{2}} \operatorname{sign}\left(k_{1} k_{2} \cdots k_{d}\right) \hat{u}(|k|), & \quad \text{ if $d$ is even}.
\end{aligned}
\end{equation*}
\end{lemma}

\begin{proof}
When  $d$  is odd, by oddness of $u_{o}(x)$, 
\begin{equation*} \label{}
\begin{aligned}
\int_{\widetilde{\Omega}} u_{o}(x) \sin (\pi k \cdot x) d x 
%& = \int_{\widetilde{\Omega}} u_{o}(x) \sin \left(\pi\left(\sum_{i=1}^{d} k_{i} x_{i}\right)\right) d x \\
& =\underbrace{\int_{\widetilde{\Omega}} u_{o}(x) \sin \left(\pi\left(\sum_{i=1}^{d-1} k_{i} x_{i}\right)\right) \cos \left(\pi k_{d} x_{d}\right) d x}_{= 0} \\
& \ \ +\int_{\widetilde{\Omega}} u_{o}(x) \cos \left(\pi\left(\sum_{i=1}^{d-1} k_{i} x_{i}\right)\right) \sin \left(\pi k_{d} x_{d}\right) d x \\
& = \underbrace{\int_{\widetilde{\Omega}} u_{o}(x) \cos \left(\pi\left(\sum_{i=1}^{d-2} k_{i} x_{i}\right)\right) \cos \left(\pi k_{d-1} x_{d-1}\right) \sin \left(\pi k_{d} x_{d}\right) d x}_{= 0} \\
&  \ \ - \int_{\widetilde{\Omega}} u_{o}(x) \sin \left(\pi\left(\sum_{i=1}^{d-2} k_{i} x_{i}\right)\right) \sin \left(\pi k_{d-1} x_{d-1}\right) \sin \left(\pi k_{d} x_{d}\right) d x \\
& =(-1)^{\frac{d-1}{2}} \int_{\widetilde{\Omega}} u_{o}(x) \prod_{i=1}^{d} \sin \left(\pi k_{i} x_{i}\right) d x \\
& =(-1)^{\frac{d-1}{2}} 2^{d} \operatorname{sign}\left(k_{1} k_{2} \cdots k_{d}\right) \int_{\Omega} u_{o}(x) \Phi_{|k|}(x) d x.
\end{aligned}
\end{equation*}
According to (\ref{d is odd, Fourier sine expan}) and (\ref{sine expansion coefficients, integral representation}), for every  $k \in \mathbb{Z}^{d}$, we have
\begin{equation*} \label{}
\begin{aligned}
i \tilde{u}_{o}(k)  = 2^{-d} \int_{\widetilde{\Omega}} u_{o}(x) \sin (\pi k \cdot x) d x   =  2^{-d}(-1)^{\frac{d-1}{2}} \operatorname{sign}\left(k_{1} k_{2} \cdots k_{d}\right) \hat{u}(|k|).
\end{aligned}
\end{equation*}

When  $d$  is even, similarly, by oddness of $u_{o}(x)$, 
\begin{equation*} \label{}
\begin{aligned}
\int_{\widetilde{\Omega}} u_{o}(x) \cos (\pi k \cdot x) d x & = \underbrace{\int_{\widetilde{\Omega}} u_{o}(x) \cos \left(\pi\left(\sum_{i=1}^{d-1} k_{i} x_{i}\right)\right) \cos \left(\pi k_{d} x_{d}\right) d x}_{= 0} \\
& \ \ -\int_{\widetilde{\Omega}} u_{o}(x) \sin \left(\pi\left(\sum_{i=1}^{d-1} k_{i} x_{i}\right)\right) \sin \left(\pi k_{d} x_{d}\right) d x \\
& = -\underbrace{\int_{\widetilde{\Omega}} u_{o}(x) \sin \left(\pi\left(\sum_{i=1}^{d-2} k_{i} x_{i}\right)\right) \cos \left(\pi k_{d-1} x_{d-1}\right) \sin \left(\pi k_{d} x_{d}\right) d x}_{= 0} \\
&  \ \ - \int_{\widetilde{\Omega}} u_{o}(x) \sin \left(\pi\left(\sum_{i=1}^{d-2} k_{i} x_{i}\right)\right) \sin \left(\pi k_{d-1} x_{d-1}\right) \sin \left(\pi k_{d} x_{d}\right) d x \\
& =(-1)^{\frac{d}{2}} \int_{\widetilde{\Omega}} u_{o}(x) \prod_{i=1}^{d} \sin \left(\pi k_{i} x_{i}\right) d x \\
& =(-1)^{\frac{d}{2}} 2^{d} \operatorname{sign}\left(k_{1} k_{2} \cdots k_{d}\right) \int_{\Omega} u_{o}(x) \Phi_{|k|}(x) d x.
\end{aligned}
\end{equation*}
According to (\ref{d is even, Fourier cosine expan}) and (\ref{sine expansion coefficients, integral representation}), for every  $k \in \mathbb{Z}^{d}$, we have
\begin{equation*} \label{}
\begin{aligned}
\tilde{u}_{o}(k) & = 2^{-d} \int_{\widetilde{\Omega}} u_{o}(x) \cos (\pi k \cdot x) d x 
=  2^{-d}(-1)^{\frac{d}{2}} \operatorname{sign}\left(k_{1} k_{2} \cdots k_{d}\right) \hat{u}(|k|).
\end{aligned}
\end{equation*}
The lemma is proved.
\end{proof}

Assume that $V(x) \in L^{1}(\Omega)$  admits the cosine series expansion $V(x) = \sum_{k \in \mathbb{N}_{0}^{d}} \check{V}(k) \Psi_{k}, $
\iffalse
\begin{equation*} \label{}
\begin{aligned}
V(x) = \sum_{k \in \mathbb{N}_{0}^{d}} \check{V}(k) \Psi_{k}, %\left(\prod_{i=1}^{d} \cos \left(k_{i} \pi x_{i}\right)\right),
\end{aligned}
\end{equation*}
\fi
where  $\left\{\check{V}(k)\right\}_{k \in \mathbb{N}_{0}^{d}}$  are the cosine expansion coefficients, i.e.,
\begin{equation*} \label{}
\begin{aligned}
\check{V}(k) = 2^{\sum_{i=1}^{d} \mathbf{1}_{k_{i} \neq 0}} \int_{\Omega} V(x) \left(\prod_{i=1}^{d} \cos \left(k_{i} \pi x_{i}\right)\right) d x .
\end{aligned}
\end{equation*}
Define the even extension  $V_{e}$  of the function  $V$  by
\begin{equation*} \label{}
\begin{aligned}
V_{e}(x) = V_{e}\left(x_{1}, \cdots, x_{d}\right)=V\left(\left|x_{1}\right|, \cdots,\left|x_{d}\right|\right),\quad x \in \widetilde{\Omega} .
\end{aligned}
\end{equation*}
Let  $\tilde{V}_{e}(k)$  be the Fourier coefficients of  $V_{e}$. Since  $V_{e}$  is real and even, 
$V_{e}(x) = \sum_{k \in \mathbb{Z}^{d}} \tilde{V}_{e}(k)\cdot$ $ \cos (\pi k \cdot x),$
where
\begin{equation*}  \label{}
\begin{aligned}
\tilde{V}_{e}(k) = \frac{\int_{\widetilde{\Omega}} V_{e}(x) \cos (\pi k \cdot x) d x}{\int_{\widetilde{\Omega}} \cos ^{2}(\pi k \cdot x) d x} = 2^{-d+\mathbf{1}_{k \neq 0}} \int_{\widetilde{\Omega}} V_{e}(x) \cos (\pi k \cdot x) d x .
\end{aligned}
\end{equation*}

From \citep[Lemma B.1]{DRMlu2021priori}, the relation between $\tilde{V}_{e}(k)$ and $\check{V}(k)$ is clear:
\begin{lemma}{{\citep[Lemma B.1]{DRMlu2021priori}}} \label{tildeVk represented by checkVk}
For every  $k \in \mathbb{Z}^{d} $, $\tilde{V}_{e}(k) = \beta_{k} \check{V}(|k|)$  where  $\beta_{k} = 2^{\mathbf{1}_{k \neq 0}-\sum_{i=1}^{d} \mathbf{1}_{k_{i} \neq 0}} $.
%For every  $k \in \mathbb{Z}^{d} $, it holds that  $\tilde{V}_{e}(k) = \beta_{k} \check{V}(|k|)$  where  $\beta_{k} = 2^{\mathbf{1}_{k \neq 0}-\sum_{i=1}^{d} \mathbf{1}_{k_{i} \neq 0}} $.
\end{lemma}
Let $w(x) = u(x) V(x)$ in $\Omega$ and its odd extension  $w_{o}(x) = u_{o}(x) V_{e}(x)$ in $\widetilde{\Omega}$ with Fourier coefficients $\tilde{w}_{o}(k)$. By the oddness of $w_{o}(x)$, we have
$w_{o}(x) 
%= \sum_{k \in \mathbb{Z}^{d}} \tilde{w}_{o}(k) e^{i k \cdot x}  
= \sum_{k \in \mathbb{Z}^{d}} i \tilde{w}_{o}(k)  \sin (\pi k \cdot x),$
where $\tilde{w}_{o}(k) = \operatorname{sign}\left(k_{1} k_{2} \cdots k_{d}\right) \tilde{w}_{o}(|k|).$
By the properties of Fourier transform,
\begin{equation} \label{convolution of Fourier coefficients}
\begin{aligned}
\tilde{w}_{o}(k) = \sum_{m \in \mathbb{Z}^{d}} \tilde{u}_{o}(m) \tilde{V}_{e}(k-m).
\end{aligned}
\end{equation}
Similar to $u_{o}(x)$, $w_{o}(x)$ admits the sine series expansion $w_{o}(x)=\sum_{k \in \mathbb{N}_{+}^{d}} \hat{w}(k) \Phi_{k}(x)$ on $\Omega$.
\iffalse
\begin{equation*} \label{}
\begin{aligned}
w_{o}(x)=\sum_{k \in \mathbb{N}_{+}^{d}} \hat{w}(k) \Phi_{k}(x),  \quad x \in \Omega.
\end{aligned}
\end{equation*}
\fi

The following proposition gives a representation of $\hat{w}(k)$ with respect to $\hat{u}(k)$ and $\check{V}(k)$. 
\begin{proposition} \label{proposition: hatwk represented by hatuk and checkVk}
Let $\beta_{k} = 2^{\mathbf{1}_{k \neq 0}-\sum_{i=1}^{d} \mathbf{1}_{k_{i} \neq 0}} $. 
Then, for any $k \in \mathbb{N}_{+}^{d},$ 
$$\hat{w}(k) = \widehat{(uV)}(k) = \sum_{m \in \mathbb{Z}^{d}} \operatorname{sign}\left(m_{1} m_{2} \cdots m_{d}\right) \beta_{|m-k|}  \hat{u}(|m|)  \check{V}(|k-m|). $$
%where  $\beta_{|m-k|}=2^{\mathbf{1}_{k \neq m}-\sum_{i=1}^{d} \mathbf{1}_{k_{i} \neq m_{i}}} $.
\end{proposition}

\begin{proof}
Thanks to Lemma \ref{tildeuk represented by hatuk}, Lemma \ref{tildeVk represented by checkVk} and relation (\ref{convolution of Fourier coefficients}), when $d$ is odd, for each $k \in \mathbb{N}_{+}^{d},$
\begin{equation*} \label{}
\begin{aligned}
\hat{w}(k) & = 2^{d} (-1)^{\frac{d-1}{2}}  i \tilde{w}_{o}(k)  =  2^{d}(-1)^{\frac{d-1}{2}} \sum_{m \in \mathbb{Z}^{d}} i \tilde{u}_{o}(m) \tilde{V}_{e}(k-m)\\
& = 2^{d}(-1)^{\frac{d-1}{2}} \sum_{m \in \mathbb{Z}^{d}} 2^{-d}(-1)^{\frac{d-1}{2}} \operatorname{sign}\left(m_{1} m_{2} \cdots m_{d}\right) \hat{u}(|m|) \cdot \beta_{|m-k|} \check{V}(|k-m|) \\
& = \sum_{m \in \mathbb{Z}^{d}} \operatorname{sign}\left(m_{1} m_{2} \cdots m_{d}\right) \beta_{|m-k|}  \hat{u}(|m|)  \check{V}(|k-m|) ,
\end{aligned}
\end{equation*}
where  $\beta_{|m-k|}=2^{\mathbf{1}_{k \neq m}-\sum_{i=1}^{d} \mathbf{1}_{k_{i} \neq m_{i}}} $. Similarly, when $d$ is even, for each $k \in \mathbb{N}_{+}^{d},$
\begin{equation*} \label{}
\begin{aligned}
\hat{w}(k) & = 2^{d} (-1)^{\frac{d}{2}}   \tilde{w}_{o}(k)   =  2^{d}(-1)^{\frac{d}{2}} \sum_{m \in \mathbb{Z}^{d}}  \tilde{u}_{o}(m) \tilde{V}_{e}(k-m)\\
& = 2^{d}(-1)^{\frac{d}{2}} \sum_{m \in \mathbb{Z}^{d}} 2^{-d}(-1)^{\frac{d}{2}} \operatorname{sign}\left(m_{1} m_{2} \cdots m_{d}\right) \hat{u}(|m|) \cdot \beta_{|m-k|} \check{V}(|k-m|) \\
& = \sum_{m \in \mathbb{Z}^{d}} \operatorname{sign}\left(m_{1} m_{2} \cdots m_{d}\right) \beta_{|m-k|}   \hat{u}(|m|)  \check{V}(|k-m|)  ,
\end{aligned}
\end{equation*}
where  $\beta_{|m-k|}=2^{\mathbf{1}_{k \neq m}-\sum_{i=1}^{d} \mathbf{1}_{k_{i} \neq m_{i}}} $.
\end{proof}

\subsection{Boundedness of \texorpdfstring{$\mathcal{H}^{-1}: \mathfrak{B}^{s}(\Omega) \rightarrow \mathfrak{B}^{s+2}(\Omega)$}{}} \label{appendix subsection: Boundness of inverse of Schrödinger operator in spectral Barron spaces}
%\subsection{Compactness of \texorpdfstring{$\mathcal{H}^{-1}$}{} in spectral Barron spaces} \label{appendix subsection: Boundness of inverse of Schrödinger operator in spectral Barron spaces}

\begin{proof}[Proof of Theorem \ref{Thm: compactness of inverse of Schrödinger Operator}] \label{Appendix proof of Thm: compactness of inverse of Schrödinger Operator}
It is clear that there exists a unique solution  $u \in H_{0}^{1}(\Omega)$ such that  
\begin{equation*}\label{stability estimate for static Schrödinger equation}
\begin{aligned}
\|\nabla u\|_{L^{2}(\Omega)}^{2} \leq \|f\|_{L^{2}(\Omega)}\|u\|_{L^{2}(\Omega)} .
\end{aligned} 
\end{equation*}
To show  $u \in \mathfrak{B}^{s+2}(\Omega) $, we firstly derive an operator equation that is equivalent to the static Schrödinger equation (\ref{static Schrödinger equation}). 
Multiplying  $\Phi_{k}$  on both sides of (\ref{static Schrödinger equation}) and then integrating yields 
\begin{equation}\label{operator eq that is equivalent to static Schrödinger eq}
\begin{aligned}
\pi^{2}|k|_{2}^{2} \hat{u}(k) + \widehat{(V u)}(k) = \hat{f}(k), \quad k \in \mathbb{N}_{+}^{d}.
\end{aligned} 
\end{equation}
Using Proposition \ref{proposition: hatwk represented by hatuk and checkVk}, we rewrite (\ref{operator eq that is equivalent to static Schrödinger eq}) as
\begin{equation*}\label{}
\begin{aligned}
\pi^{2}|k|_{2}^{2} \hat{u}(k) + \sum_{m \in \mathbb{Z}^{d}} \operatorname{sign}\left(m_{1} m_{2} \cdots m_{d}\right) \beta_{|m-k|}  \hat{u}(|m|)  \check{V}(|k-m|) = \hat{f}(k), \quad k \in \mathbb{N}_{+}^{d} ,
\end{aligned} 
\end{equation*}
where $\beta_{k} = 2^{\mathbf{1}_{k \neq 0}-\sum_{i=1}^{d} \mathbf{1}_{k_{i} \neq 0}}$.
%Recall that a function  $u \in \mathfrak{B}^{s}(\Omega)$  is equivalent to that  $\hat{u}(k)$  belongs to the weighted  $\ell^{1}$  space  $\ell_{W_{s}}^{1}\left(\mathbb{N}_{+}^{d}\right)$  with the weight  $W_{s}(k)=1+\pi^{s}|k|_{1}^{s} $. We would like to rewrite the above equations as an operator equation on the space  $\ell_{W_{s}}^{1}\left(\mathbb{N}_{+}^{d}\right) $. For doing this, let us define some useful operators. 
Define the operator  $\mathbb{M}: \hat{u} \mapsto \mathbb{M} \hat{u}$  by
$$(\mathbb{M} \hat{u})(k)=\pi^{2}|k|_{2}^{2} \hat{u}(k),  \quad k \in \mathbb{N}_{+}^{d} .$$
We may extend $\hat{u}(k)$ as 0 when $k \in \mathbb{N}_{0}^{d} \setminus \mathbb{N}_{+}^{d}$. 
Define the operator  $\mathbb{V}: \hat{u} \mapsto \mathbb{V} \hat{u}$  by   
$$(\mathbb{V} \hat{u})(k) = \sum_{m \in \mathbb{Z}^{d}} \operatorname{sign}\left(m_{1} m_{2} \cdots m_{d}\right) \beta_{|m-k|}  \hat{u}(|m|)  \check{V}(|k-m|),\quad k \in \mathbb{N}_{+}^{d}  .$$
We rewrite (\ref{operator eq that is equivalent to static Schrödinger eq}) as
\begin{equation}  \label{operator eq that is equivalent to static Schrödinger eq, abbreviated version}
\begin{aligned}
(\mathbb{M}+\mathbb{V}) \hat{u} = \hat{f}.
\end{aligned} 
\end{equation}
Since  the diagonal operator  $\mathbb{M}$  is invertible, the operator equation (\ref{operator eq that is equivalent to static Schrödinger eq, abbreviated version}) is equivalent to
\begin{equation}\label{operator eq that is equivalent to static Schrödinger eq,  Fredholm format}
\begin{aligned}
\left(\mathbb{I}+\mathbb{M}^{-1} \mathbb{V}\right) \hat{u} = \mathbb{M}^{-1} \hat{f} .
\end{aligned} 
\end{equation}

Next, we claim that equation (\ref{operator eq that is equivalent to static Schrödinger eq,  Fredholm format}) has a unique solution  $\hat{u} \in \ell_{W_{s}}^{1}\left(\mathbb{N}_{+}^{d}\right)$  and there exists a constant  $C_{1}$  depending on  $V$  and  $d$  such that
\begin{equation} \label{uellWs1 bounded by fellWs1}
\begin{aligned}
\|\hat{u}\|_{\ell_{W_{s}}^{1}\left(\mathbb{N}_{+}^{d}\right)} \leq C_{1}(V, d) \|\hat{f}\|_{\ell_{W_{s}}^{1}\left(\mathbb{N}_{+}^{d}\right)}.
\end{aligned} 
\end{equation}
%To see this, observe that owing to the compactness of  $\mathbb{M}^{-1} \mathbb{V}$  as shown in Lemma \ref{lemma: operator mathbbV is bounded} the operator  $\mathbb{I}+\mathbb{M}^{-1} \mathbb{V}$  is a Fredholm operator on  $\ell_{W_{s}}^{1}\left(\mathbb{N}_{+}^{d}\right) $. 
It follows from the compactness of  $\mathbb{M}^{-1} \mathbb{V}$  as shown in Lemma \ref{lemma: operator mathbbV is bounded} that $\mathbb{I}+\mathbb{M}^{-1} \mathbb{V}$  is a Fredholm operator on  $\ell_{W_{s}}^{1}\left(\mathbb{N}_{+}^{d}\right) $. 
By the celebrated Fredholm alternative theorem, the operator  $\mathbb{I}+\mathbb{M}^{-1} \mathbb{V}$  has a bounded inverse  $\left(\mathbb{I}+\mathbb{M}^{-1} \mathbb{V}\right)^{-1}$  if and only if  $\left(\mathbb{I}+\mathbb{M}^{-1} \mathbb{V}\right) \hat{u} = 0$  has a trivial solution. 
%Therefore to obtain the bound (\ref{uellWs1 bounded by fellWs1}), it suffices to show that  $\left(\mathbb{I}+\mathbb{M}^{-1} \mathbb{V}\right) \hat{u}=0$  implies  $\hat{u}=0 $. 
By the equivalence between equation (\ref{static Schrödinger equation}) and (\ref{operator eq that is equivalent to static Schrödinger eq,  Fredholm format}), we only need to show that the only solution of (\ref{static Schrödinger equation}) is zero when $f = 0$, which is a direct consequence of (\ref{stability estimate for static Schrödinger equation}) and the Poincaré's inequality.

Since  $\hat{u} \in \ell_{W_{s}}^{1}\left(\mathbb{N}_{+}^{d}\right) $,   it follows from (\ref{operator eq that is equivalent to static Schrödinger eq, abbreviated version}) and the boundedness of  $\mathbb{V}$  on  $\ell_{W_{s}}^{1}\left(\mathbb{N}_{+}^{d}\right.  )$ proved in Lemma \ref{lemma: operator mathbbV is bounded} that
\begin{equation} \label{norm of mathbbMhatu bounded by norm of hatf}
\begin{aligned}
\|\mathbb{M} \hat{u}\|_{\ell_{W_{s}}^{1}\left(\mathbb{N}_{+}^{d}\right)} & \leq\|\mathbb{V} \hat{u}\|_{\ell_{W_{s}}^{1}\left(\mathbb{N}_{+}^{d}\right)}+\|\hat{f}\|_{\ell_{W_{s}}^{1}\left(\mathbb{N}_{+}^{d}\right)} \\
& \leq C_{2}(V, d) \|\hat{u}\|_{\ell_{W_{s}}^{1}\left(\mathbb{N}_{+}^{d}\right)}+\|\hat{f}\|_{\ell_{W_{s}}^{1}\left(\mathbb{N}_{+}^{d}\right)} \\
& \leq C_{3}(V, d) \|\hat{f}\|_{\ell_{W_{s}}^{1}\left(\mathbb{N}_{+}^{d}\right)} ,
\end{aligned}
\end{equation}
where we have used (\ref{uellWs1 bounded by fellWs1}). 
Then, the estimate (\ref{norm of mathbbMhatu bounded by norm of hatf}) implies 
%Then, it follows from (\ref{norm of mathbbMhatu bounded by norm of hatf}) that 
\begin{equation*} \label{}
\begin{aligned}
\|u\|_{\mathcal{B}^{s+2}(\Omega)} & =\sum_{k \in \mathbb{N}_{+}^{d}}\left(1+\pi^{s+2}|k|_{1}^{s+2}\right)\left|\hat{u}(k)\right| 
= \sum_{k \in \mathbb{N}_{+}^{d}} \frac{1+\pi^{s+2}|k|_{1}^{s+2}}{\pi^{2}|k|_{2}^{2}} \cdot \pi^{2}|k|_{2}^{2}\left|\hat{u}(k)\right| \\
& \leq \left(\pi^{-2}+d\right)\|\mathbb{M} \hat{u}\|_{\ell_{W_{s}}^{1}\left(\mathbb{N}_{+}^{d}\right)} 
\leq C(d, V)\|\hat{f}\|_{\ell_{W_{s}}^{1}\left(\mathbb{N}_{+}^{d}\right)},
\end{aligned}
\end{equation*}
which completes the proof.
\end{proof}

Proceeding along the same line as in  \citep[Lemma 7.2]{DRMlu2021priori}, we obtain
\begin{lemma} \label{compactness of a multiplication operator}
Suppose that  $\mathbb{T}$  is a multiplication operator on  $\ell_{W_{s}}^{1}\left(\mathbb{N}_{+}^{d}\right)$  defined by for each $a =   \left(a(k)\right)_{k \in \mathbb{N}_{+}^{d}}$  that  $(\mathbb{T} a)_{k} = \lambda_{k} a_{k}$  with  $\lambda_{k} \rightarrow 0$  as  $|k|_{2} \rightarrow \infty $. Then  $\mathbb{T}: \ell_{W_{s}}^{1}\left(\mathbb{N}_{+}^{d}\right) \rightarrow \ell_{W_{s}}^{1}\left(\mathbb{N}_{+}^{d}\right)$  is compact.
\end{lemma}

The following lemma shows that the operator $\mathbb{V}$ is bounded on $\ell_{W_{s}}^{1}\left(\mathbb{N}_{+}^{d}\right)$.
\begin{lemma} \label{lemma: operator mathbbV is bounded}
Assume that  $V \in \mathfrak{C}^{s}(\Omega)$. Then the operator $\mathbb{V}$ is bounded on $\ell_{W_{s}}^{1}\left(\mathbb{N}_{+}^{d}\right)$ and the operator  $\mathbb{M}^{-1} \mathbb{V}$  is compact on  $\ell_{W_{s}}^{1}\left(\mathbb{N}_{+}^{d}\right)$.
\end{lemma}

\begin{proof}[Proof of Lemma \ref{lemma: operator mathbbV is bounded}]
Since  $\mathbb{M}^{-1}$  is a multiplication operator on  $\ell_{W_{s}}^{1}\left(\mathbb{N}_{+}^{d}\right)$  with the diagonal entries converging to zero, it follows from Lemma \ref{compactness of a multiplication operator} that  $\mathbb{M}^{-1}$  is compact on  $\ell_{W_{s}}^{1}\left(\mathbb{N}_{+}^{d}\right)$. 
Therefore to show the compactness of  $\mathbb{M}^{-1} \mathbb{V}$, it is sufficient to show that the operator  $\mathbb{V}$  is bounded on  $\ell_{W_{s}}^{1}\left(\mathbb{N}_{+}^{d}\right) $. 
%Recall that by definition, $\beta_{k}=2^{\mathbf{1}_{k \neq 0}-\sum_{i=1}^{d} \mathbf{1}_{k_{i} \neq 0}} \in [2^{1-d}, 1]$.  In addition, since $V \in \mathfrak{C}^{s}(\Omega)$, 
Since $\beta_{k}=2^{\mathbf{1}_{k \neq 0}-\sum_{i=1}^{d} \mathbf{1}_{k_{i} \neq 0}} \in [2^{1-d}, 1]$ and $V \in \mathfrak{C}^{s}(\Omega)$, 
using Proposition \ref{proposition: hatwk represented by hatuk and checkVk}, one has that for any $\hat{u} \in \ell_{W_{s}}^{1}\left(\mathbb{N}_{+}^{d}\right)$ with $\hat{u}(k) = 0$ when $k \in \mathbb{N}_{0}^{d} \setminus \mathbb{N}_{+}^{d}$,
\begin{equation*} \label{}
\begin{aligned}
\|\mathbb{V} \hat{u}\|_{\ell_{W_{s}}^{1}\left(\mathbb{N}_{+}^{d}\right)} & = \sum_{k \in \mathbb{N}_{+}^{d}} \left(1+\pi^{s}|k|_{1}^{s}\right) \left|\sum_{m \in \mathbb{Z}^{d}} \operatorname{sign}\left(m_{1} m_{2} \cdots m_{d}\right) \beta_{|m-k|}  \hat{u}(|m|)  \check{V}(|k-m|)\right| \\
& \leq \sum_{m \in \mathbb{Z}^{d}} \sum_{k \in \mathbb{N}_{+}^{d}} \left(1+\pi^{s} \max\left(2^{s-1}, 1\right)\left(|m-k|_{1}^{s}+|m|_{1}^{s}\right)\right) \left|\hat{u}(|m|)\right| \left| \check{V}(|k-m|)\right| \\
& \leq 2^{d} \max\left(2^{s-1}, 1\right)  \left( \|\hat{u}\|_{\ell^{1}\left(\mathbb{N}_{+}^{d}\right)}\|\check{V}\|_{\ell_{W_{s}}^{1}\left(\mathbb{N}_{0}^{d}\right)}+\|\hat{u}\|_{\ell_{W_{s}}^{1}\left(\mathbb{N}_{+}^{d}\right)}\|\check{V}\|_{\ell^{1}\left(\mathbb{N}_{0}^{d}\right)} \right)\\
& \leq 2^{d+1} \max\left(2^{s-1}, 1\right) \|V\|_{\mathfrak{C}^{s}(\Omega)} \|\hat{u}\|_{\ell_{W_{s}}^{1}\left(\mathbb{N}_{+}^{d}\right)},
\end{aligned} 
\end{equation*}
where we have used the elementary inequality  $|a+b|^{s} \leq \max\left(2^{s-1}, 1\right) \left(|a|^{s}+|b|^{s}\right)$ in the first inequality and the fact   $\sum_{m \in \mathbb{Z}^{d}}\left|\hat{u}(|m|)\right| \leq   2^{d}\|\hat{u}\|_{\ell^{1}\left(\mathbb{N}_{+}^{d}\right)} \leq 2^{d}\|\hat{u}\|_{\ell_{W_{s}}^{1}\left(\mathbb{N}_{+}^{d}\right)} $ in the second inequality.
\end{proof}

\begin{proof}[Proof of Corollary \ref{corollary: inverse of Schrödinger operator is compact}]
According to Theorem \ref{Thm: compactness of inverse of Schrödinger Operator}, we have proved that the operator $\mathcal{S}: \mathfrak{B}^{s}(\Omega) \rightarrow \mathfrak{B}^{s+2}(\Omega)$  is bounded. 

Note that the inclusion  $\mathcal{J}: \mathfrak{B}^{s+2}(\Omega) \hookrightarrow \mathfrak{B}^{s}(\Omega)$  is compact. In fact, by definition, the space  $\mathfrak{B}^{s}(\Omega)$  may be viewed as a weighted  $\ell^{1}$  space  $\ell_{W_{s}}^{1}\left(\mathbb{N}_{0}^{d}\right)$  of the sine coefficients defined on the lattice  $\mathbb{N}_{+}^{d}$  with the weight  $W_{s}(k)=\left(1+\pi^{s}|k|_{1}^{s}\right)$. Therefore the inclusion satisfies 
$$\|\mathcal{J} u\|_{\mathfrak{B}^{s}(\Omega)}=\sum_{k \in \mathbb{N}_{+}^{d}} W_{s}(k)|\hat{u}(k)|=\sum_{k \in \mathbb{N}_{+}^{d}} \frac{W_{s}(k)}{W_{s+2}(k)} W_{s+2}(k)|\hat{u}(k)|.$$
Since  $\frac{W_{s}(k)}{W_{s+2}(k)} \rightarrow 0$  as  $|k|_{2} \rightarrow \infty$, by a similar argument as the proof of Lemma \ref{compactness of a multiplication operator}, one can conclude that  $\mathcal{J}$  is compact from  $\ell_{W_{s+2}}^{1}\left(\mathbb{N}_{+}^{d}\right)$  to  $\ell_{W_{s}}^{1}\left(\mathbb{N}_{+}^{d}\right)$  and hence from  $\mathcal{B}^{s+2}(\Omega)$  to  $\mathcal{B}^{s}(\Omega)$. 
Corollary \ref{corollary: inverse of Schrödinger operator is compact} is then a direct consequence of the boundness of $\mathcal{S}: \mathfrak{B}^{s}(\Omega) \rightarrow \mathfrak{B}^{s+2}(\Omega)$ and the compactness of the inclusion  $\mathcal{J}$  from  $\mathfrak{B}^{s+2}(\Omega)$  to  $\mathfrak{B}^{s}(\Omega)$.
\end{proof}

\section{About the penalty method} \label{appendix section: About the penalty method}
\iffalse
Consider the penalty method which employs the population loss
\begin{equation*}
\begin{aligned}
\mathscr{L}_{k}(u) & := L_{k}(u) + \gamma\left(\mathcal{E}_{2}(u)-1\right)^{2} =\frac{\mathcal{E}_{V}(u)+\mathcal{E}_{P}(u)}{\mathcal{E}_{2}(u)}+\gamma\left(\mathcal{E}_{2}(u)-1\right)^{2}
\end{aligned}
\end{equation*}
and the corresponding empirical loss
\begin{equation*}
\begin{aligned}
\mathscr{L}_{k, n}(u) & := L_{k, n}(u)+\gamma\left(\mathcal{E}_{n, 2}(u)-1\right)^{2}
=\frac{\mathcal{E}_{n, V}(u)+\mathcal{E}_{n, P}(u)}{\mathcal{E}_{n, 2}(u)}+\gamma\left(\mathcal{E}_{n, 2}(u)-1\right)^{2},
\end{aligned}
\end{equation*}
where $\gamma>0$ is a penalty parameter. 
Denote by $\mathscr{u}_{n}$ a minimizer of $\mathscr{L}_{k, n}(u)$ within $\mathcal{F}$, i.e.,
$$\mathscr{u}_{n} = \mathop{\arg\min}_{u \in \mathcal{F}} \mathcal{L}_{k, n}(u).$$
\fi

Firstly, we prove that when 
$\gamma$ is chosen properly large, $\|\mathscr{u}_{n}\|_{L^{2}(\Omega)} \geq 1/2$ with high probability. 
To this end, we decompose $\mathcal{E}_{2}(\mathscr{u}_{n})-1$ as 
\begin{equation*}
\begin{aligned}
\mathcal{E}_{2}(\mathscr{u}_{n})-1 = \mathcal{E}_{n, 2}(\mathscr{u}_{n})-1 - \big(\mathcal{E}_{n, 2}(\mathscr{u}_{n})-\mathcal{E}_{2}(\mathscr{u}_{n})\big) =: \mathcal{E}_{n, 2}(\mathscr{u}_{n})-1 - R_{1}
\end{aligned}
\end{equation*}
and for any $u_{\mathcal{F}} \in \mathcal{F}$,
\begin{equation} \label{error decomposition for gamma(mathcalE_n, 2(mathscru_n)-1)^2 in penalty method}
\begin{aligned}
\gamma\left(\mathcal{E}_{n, 2}\left(\mathscr{u}_{n}\right)-1\right)^{2} & \leq \mathscr{L}_{k, n}\left(\mathscr{u}_{n}\right) \leq \mathscr{L}_{k, n}\left(u_{\mathcal{F}}\right) \\
& = \Big(L_{k, n}(u_{\mathcal{F}})-L_{k}(u_{\mathcal{F}})\Big) + \Big(L_{k}(u_{\mathcal{F}})-\lambda_{k}\Big) + \lambda_{k} \\
& \quad + \gamma \left[ \Big(\mathcal{E}_{n,2}(u_{\mathcal{F}})-\mathcal{E}_{2}(u_{\mathcal{F}})\Big) + \Big(\mathcal{E}_{2}(u_{\mathcal{F}})-1\Big) \right]^{2} \\
& =: R_{2} + R_{3} + \lambda_{k} + \gamma \left( R_{4} + R_{5} \right)^{2},
\end{aligned}
\end{equation}
where the second inequality follows from the fact that $\mathscr{u}_{n}$ is a minimizer of $\mathscr{L}_{k, n}(u)$. Thus,
\begin{equation}  \label{bound for |mathcalE_2(mathscru_n)-1|}
\begin{aligned}
\left|\mathcal{E}_{2}(\mathscr{u}_{n})-1\right| \leq  \left|R_{1}\right| + \left[ \frac{\lambda_{k}}{\gamma} + \frac{ R_{2}+R_{3}}{\gamma}+\left(R_{4}+R_{5}\right)^{2} \right]^{1/2}.
\end{aligned}
\end{equation}
Note that  $R_{1}$ may be regarded as the statistical error term,   $R_{2}$, $R_{4}$ may be regarded as the Monte Carlo error terms, and  $R_{3}$, $R_{5}$ may be regarded as the approximation error terms. 

\textbf{Bounding $R_{1}$.} 
To control the statistical error term $R_{1}$, we employ the well-known tool of Rademacher complexity. We recall the definition firstly. %as follows.

\begin{definition}
Define for a set of random variables  $\left\{Z_{j}\right\}_{j=1}^{n}$  independently distributed according to  $P$  and a function class  $\mathcal{G}$  the empirical Rademacher complexity %random variable
$$\widehat{\mathscr{R}}_{n}(\mathcal{G}):=\mathbf{E}_{\sigma}\left[\sup _{g \in \mathcal{G}}\left|\frac{1}{n} \sum_{j=1}^{n} \sigma_{j} g\left(Z_{j}\right)\right| \mid Z_{1}, \cdots, Z_{n}\right],$$
where the expectation  $\mathbf{E}_{\sigma}$  is taken with respect to the independent uniform Bernoulli sequence  $\left\{\sigma_{j}\right\}_{j=1}^{n}$  with  $\sigma_{j} \in\{ \pm 1\}$. The Rademacher complexity of  $\mathcal{G}$  is defined by  $$\mathscr{R}_{n}(\mathcal{G}) = \mathbf{E}_{P^{n}}\left[\widehat{\mathscr{R}}_{n}(\mathcal{G})\right].$$
\end{definition}

We introduce the following generalization bound via the Rademacher complexity.
\begin{lemma}{{\citep[Theorem 4.10]{HDSwainwright_2019}}}
\label{A uniform law of large numbers via Rademacher complexity} 
Let  $\mathcal{G}$  be a class of integrable real valued functions such that  $\sup _{g \in \mathcal{G}} \|g\|_{L^{\infty}(\Omega)} \leq M_{\mathcal{G}}$. Let  $Z_{1}, Z_{2}, \cdots, Z_{n}$  be $i.i.d.$ random samples from some distribution $P$ over $\Omega$. Then for any positive integer $n \geq 1$ and any scalar $\delta \geq 0$, with probability at least  $1 - \delta$,
$$\sup _{g \in \mathcal{G}}\left|\frac{1}{n} \sum_{i=1}^{n} g\left(Z_{i}\right)-\mathbf{E} g(Z)\right| \leq 2 \mathscr{R}_{n}(\mathcal{G}) + M_{\mathcal{G}} \sqrt{\frac{2 \ln (1 / \delta)}{n}} .$$
\end{lemma}
Recall the function class $\mathcal{G}_{1}$ defined in (\ref{function classes defined for bounding generalization error}) 
%We assume that the set  $\mathcal{F}$  satisfies  
and $\sup _{u \in \mathcal{F}}\|u\|_{L^{\infty}(\Omega)} \leq M_{\mathcal{F}}$  so that  
$\sup _{g \in \mathcal{G}_{1}} $ $ \|g\|_{L^{\infty}(\Omega)} \leq   M_{\mathcal{F}}^{2}$. 
Define for  $n \in \mathbb{N}$  and  $\delta>0$  the constant
\begin{equation*} \label{xi4(n, delta)}
\xi_{4}(n, \delta) := 2 \mathscr{R}_{n}\left(\mathcal{G}_{1}\right) +  M_{\mathcal{F}}^{2}  \sqrt{\frac{2 \ln (1 / \delta)}{n}},
\end{equation*}
%where  $\mathscr{R}_{n}\left(\mathcal{G}_{1}\right)$  denotes the Rademacher complexity of the set  $\mathcal{G}_{1}$. 
and the event $A_{4}(n, \delta) := \left\{\left|\mathcal{E}_{n, 2}\left(\mathscr{u}_{n}\right)-\mathcal{E}_{2}\left(\mathscr{u}_{n}\right)\right| \leq \xi_{4}(n, \delta)\right\} .$
\iffalse
\begin{equation*} 
A_{4}(n, \delta) := \left\{\left|\mathcal{E}_{n, 2}\left(\mathscr{u}_{n}\right)-\mathcal{E}_{2}\left(\mathscr{u}_{n}\right)\right| \leq \xi_{4}(n, \delta)\right\} .
\end{equation*} 
\fi
Then applying Lemma \ref{A uniform law of large numbers via Rademacher complexity} to  $\mathcal{G}_{1}$  we have 
\begin{equation} \label{bound A4(n, delta)}
\mathbf{P}\left[A_{4}(n, \delta)\right] \geq 1-\delta.
\end{equation}

Next, the celebrated Dudley's theorem will be used to bound the Rademacher complexity in terms of the metric entropy.
%Dudley's Entropy Integral Bound,
\begin{theorem}{\cite{Dudley1967290, Liao2020NotesOR}} \label{Dudley's Entropy Integral Bound for empirical Rademacher complexity}
Let  $\mathscr{F}$  be a class of real functions,  $\left\{Z_{i}\right\}_{i = 1}^{n}$  be random $i.i.d.$ samples and the empirical measure $P_{n} = n^{-1} \sum_{i=1}^{n} \delta_{Z_{i}} $.  Assuming
$$\sup _{f \in \mathscr{F}} \|f\|_{L^{2}(P_{n})} := \sup _{f \in \mathscr{F}}\left(\frac{1}{n} \sum_{i=1}^{n} f^{2}\left(Z_{i}\right)\right)^{1/2} \leq c,$$
then 
$$\widehat{\mathscr{R}}_{n}(\mathscr{F}) \leq \inf _{\epsilon \in[0, c / 2]}\left(4 \epsilon + \frac{12}{\sqrt{n}} \int_{\epsilon}^{c / 2} \sqrt{\ln \mathcal{N}\left(\delta, \mathscr{F}, \|\cdot\|_{L^{2}(P_{n})}\right)}  d \delta\right).$$    
\end{theorem} 
By Lemma \ref{covering number for Gm1}, (\ref{upper bound of maximum value of hypothesis class with Softplus activation}), (\ref{upper bound of Lambda for covering with Softplus activation}) and (\ref{upper bound of mathcalM with Softplus activation}), the covering number
\begin{equation*} \label{}
\begin{aligned}
\mathcal{N}\left(\delta, \mathcal{G}_{1},\|\cdot\|_{L^{2}(P_{n})}\right) %& \leq \mathcal{M}\left(\delta, \Lambda_{1}, m, d\right) \\
& \leq \left( 2^{\frac{3}{d+3}} B^{\frac{2}{d+3}} 3\cdot19B(4+8B)/(d^{2}\delta)\right)^{(d+2)m+1}\\
& \leq \Big( 1368 B^{2.5} /(d^{2}\delta)\Big)^{(d+2)m+1}.
\end{aligned}
\end{equation*}
Thus, it follows from Theorem \ref{Dudley's Entropy Integral Bound for empirical Rademacher complexity} and (\ref{upper bound of maximum value of hypothesis class with Softplus activation}) that there exist absolute constants $C$ such that
\begin{equation} \label{bound Rademacher complexity of mathcalG_1}
\begin{aligned}
\widehat{\mathscr{R}}_{n}(\mathcal{G}_{1}) %& \leq \frac{12}{\sqrt{n}} \int_{0}^{M_{\mathcal{F}}^{2} / 2} \sqrt{\ln \mathcal{N}\left(\delta, \mathcal{G}_{1}, \|\cdot\|_{L^{2}(P_{n})}\right)}  d \delta  \\
& \leq 12\sqrt{\frac{(d+2)m+1}{n}} \int_{0}^{M_{\mathcal{F}}^{2} / 2} \sqrt{\ln \left(1368 B^{2.5} /(d^{2}\delta) \right)}  d \delta  \\
& \leq C\sqrt{\frac{dm}{n}} \left( M_{\mathcal{F}}^{2}\sqrt{1 + \ln \left(B/d\right)} + \int_{0}^{M_{\mathcal{F}}^{2} / 2} \sqrt{\left(\ln (1/\delta) \right)_{+}}  d \delta \right)  \\
& \leq  \frac{C B}{d}\left(\frac{B}{d}+1\right) \sqrt{ \frac{d \left(1+\ln B\right) m}{n}},
\end{aligned}
\end{equation}
where in the last inequality we have used %the fact
\begin{equation*} 
\begin{aligned} 
& \int_{0}^{M_{\mathcal{F}}^{2} / 2} \sqrt{\left(\ln (1/\delta) \right)_{+}}  d \delta \leq \int_{0}^{\min(1,M_{\mathcal{F}}^{2} / 2)} \sqrt{1 / \delta} d \delta \leq \min(2,\sqrt{2} M_{\mathcal{F}}).
\end{aligned}
\end{equation*}
By (\ref{bound Rademacher complexity of mathcalG_1}) and (\ref{upper bound of maximum value of hypothesis class with Softplus activation}), on the event $A_{4}(n, \delta)$,
\begin{equation} \label{bound R1}
\begin{aligned}
R_{1} \leq \xi_{4}(n, \delta) %& \leq C \left(B/d\right)\left(B/d+1\right) \sqrt{ \frac{d \left(1+\ln B\right) m}{n}} + C \left(B/d\right)^{2}  \sqrt{\frac{ \ln (1 / \delta)}{n}}\\
& \leq \frac{C B}{d}\left(\frac{B}{d}+1\right) \sqrt{ \frac{d \left(1+\ln B\right) m +  \ln (1 / \delta)}{n}} .
\end{aligned}
\end{equation}

\textbf{Bounding $R_{3}$ and $R_{5}$.} 
Let $\mathcal{F} = \varphi \mathcal{F}_{\mathrm{SP}_{\tau}, m}(B)$  with $B = \left(1+\frac{2 d}{\pi s}\right) \|u^{*}\|_{\mathfrak{B}^{s}(\Omega)}$  and  $\tau=9\sqrt{m}$.
By Theorem \ref{Thm: u H1 approximation by varphi Softplus networks}, there exists $u_{\mathcal{F}} \in \varphi \mathcal{F}_{\mathrm{SP}_{\tau}, m}(B)$ such that 
$\left\|u^{*}-u_{\mathcal{F}}\right\|_{H^{1}(\Omega)} \leq 64B/\sqrt{m} \leq 1/2.$
By Theorem \ref{Thm: bounding the approximation error Lk(u in F)-lambdak}, 
\begin{equation} \label{bound R3}
\begin{aligned}
R_{3} = L_{k}\left(u_{\mathcal{F}}\right)-\lambda_{k} \leq 64 \left(3\max \left\{1, V_{\max }\right\}  + 7\lambda_{k} + 5 \beta\right) B /\sqrt{m}.
\end{aligned}
\end{equation}
Similar as in (\ref{lower and upper bound for L2 norm of um}) and (\ref{bound for |mathcalE_2(u_m) - mathcalE_2(u*)|}),
we have $1/2 \leq \|u_{\mathcal{F}}\|_{L^{2}(\Omega)} \leq 3/2$ and
\iffalse
\begin{equation} \label{lower and upper bound for |umathcalF|L2(Omega)}
\begin{aligned}
1/2 \leq \|u^{*}\|_{L^{2}(\Omega)} - \left\|u^{*}-u_{\mathcal{F}}\right\|_{H^{1}(\Omega)} \leq \|u_{\mathcal{F}}\|_{L^{2}(\Omega)} \leq \|u^{*}\|_{L^{2}(\Omega)} + \left\|u^{*}-u_{\mathcal{F}}\right\|_{H^{1}(\Omega)}  \leq 3/2
\end{aligned}
\end{equation}
\fi
\begin{equation} \label{bound R5}
\begin{aligned}
R_{5} & \leq \left(\|u_{\mathcal{F}}\|_{L^{2}(\Omega)} + \|u^{*}\|_{L^{2}(\Omega)}\right) \left|\|u_{\mathcal{F}}\|_{L^{2}(\Omega)} - \|u^{*}\|_{L^{2}(\Omega)}\right| 
\leq 160B/\sqrt{m}.
\end{aligned}
\end{equation}

\textbf{Bounding $R_{2}$ and $R_{4}$.}
As a preparation, we introduce Hoeffding's inequality to control the Monte Carlo error terms. 
\begin{lemma}[Hoeffding's inequality] 
{\citep[Theorem 2.2.6]{HDPVershynin_2018}}
\label{Lemma: Hoeffding's inequality}
Let  $Z_{1}$, $Z_{2}$, $\cdots$, $Z_{n}$  be independent random variables. Assume that  $Z_{i} \in\left[m_{i}, M_{i}\right]$  for every  $i$. Then, for any  $t>0$, 
$$\mathbf{P}\left(\sum_{i=1}^{n}\left(Z_{i}-\mathbf{E} Z_{i}\right) \geq t\right) \leq \exp \left(-\frac{2 t^{2}}{\sum_{i=1}^{n}\left(M_{i}-m_{i}\right)^{2}}\right).$$
In particular, if $Z_{1}, Z_{2}, \cdots, Z_{n}$ are identically distributed with $|Z_{i}| \leq M$, then for any  $t>0$,
$$\mathbf{P}\left(\left|\frac{\sum_{i=1}^{n} Z_{i}}{n}-\mathbf{E} Z_{1}\right| \geq t\right) \leq 2 \exp \left(-\frac{n t^{2}}{2M^{2}}\right).$$
\end{lemma}

Consider $u_{\mathcal{F}} \in \varphi \mathcal{F}_{\mathrm{SP}_{\tau}, m}(B)$ given by Theorem \ref{Thm: u H1 approximation by varphi Softplus networks}. 
To bound  $R_{4}$, we define the constant $\xi_{5}(n, \delta) := M_{\mathcal{F}}^{2}  \sqrt{\ln (2 / \delta)/(2 n)}$
\iffalse
\begin{equation*} \label{xi5(n, delta)}
\xi_{5}(n, \delta) := M_{\mathcal{F}}^{2}  \sqrt{\frac{\ln (2 / \delta)}{2 n}},
\end{equation*}
\fi
and the event $A_{5}(n, \delta) := \left\{\left|\mathcal{E}_{n, 2}\left(u_{\mathcal{F}}\right)-\mathcal{E}_{2}\left(u_{\mathcal{F}}\right)\right| \leq \xi_{5}(n, \delta)\right\}.$
\iffalse
\begin{equation*} 
A_{5}(n, \delta) := \left\{\left|\mathcal{E}_{n, 2}\left(u_{\mathcal{F}}\right)-\mathcal{E}_{2}\left(u_{\mathcal{F}}\right)\right| \leq \xi_{5}(n, \delta)\right\}.
\end{equation*}
\fi
On the event $A_{5}(n, \delta),$ 
\begin{equation} \label{bound R4}
R_{4} \leq \xi_{5}(n, \delta).
\end{equation} 
Since $\|u_{\mathcal{F}}^{2}\|_{L^{\infty}(\Omega)} \leq   M_{\mathcal{F}}^{2}$, applying Lemma \ref{Lemma: Hoeffding's inequality} yields that $\mathbf{P}\left[A_{5}(n, \delta)\right] \geq 1-\delta.$ 
\iffalse
\begin{equation} \label{bound A5(n, delta)}
\mathbf{P}\left[A_{5}(n, \delta)\right] \geq 1-\delta.
\end{equation}
\fi
To bound  $R_{2}$, we decompose
\begin{equation*}
\begin{aligned}
R_{2} %& = L_{k, n}\left(u_{\mathcal{F}}\right)-L_{k}\left(u_{\mathcal{F}}\right) \\
& \leq \frac{\left|\mathcal{E}_{n, V}\left(u_{\mathcal{F}}\right) \! - \! \mathcal{E}_{V}\left(u_{\mathcal{F}}\right)\right|}{\mathcal{E}_{n, 2}\left(u_{\mathcal{F}}\right)} + \frac{\mathcal{E}_{V}\left(u_{\mathcal{F}}\right) \! + \! \mathcal{E}_{P}\left(u_{\mathcal{F}}\right)}{\mathcal{E}_{2}\left(u_{\mathcal{F}}\right) \mathcal{E}_{n, 2}\left(u_{\mathcal{F}}\right)}\left|\mathcal{E}_{2}\left(u_{\mathcal{F}}\right) \! - \! \mathcal{E}_{n, 2}\left(u_{\mathcal{F}}\right)\right| + \frac{\left|\mathcal{E}_{n, P}\left(u_{\mathcal{F}}\right) \! - \! \mathcal{E}_{P}\left(u_{\mathcal{F}}\right)\right|}{\mathcal{E}_{n, 2}\left(u_{\mathcal{F}}\right)}\\
& =: R_{21}+R_{22}+R_{23}.
\end{aligned}
\end{equation*}
Recall the function classes $\mathcal{G}_{2}$, $\mathcal{F}_{j}$ %, $1\leq j \leq k-1$,
\iffalse
\begin{equation*} \label{}
\begin{aligned}
& \mathcal{G}_{2} := \left\{g: g=|\nabla u|^{2}+V|u|^{2} \text { where } u \in \mathcal{F}\right\}, \\
& \mathcal{F}_{j} := \left\{g: g = u \psi_{j}  \text { where } u \in \mathcal{F}\right\}, \quad \text{for }  j = 1, 2, \ldots, k-1.
\end{aligned}
\end{equation*}  \fi
defined in \ref{function classes defined for bounding generalization error} and that
we assume $\sup _{g \in \mathcal{G}_{2}} \|g\|_{L^{\infty}(\Omega)} \leq M_{\mathcal{G}_{2}}$. Since $\{\psi_{j}\}_{j = 1}^{k-1}$ are normalized orthogonal eigenfunctions and $\|u_{\mathcal{F}}\|_{L^{2}(\Omega)} \leq 3/2$,
\begin{equation*} 
\mathcal{E}_{P}\left(u_{\mathcal{F}}\right) = \beta \sum_{j=1}^{k-1}\left\langle u_{\mathcal{F}}, \psi_{j}\right\rangle^{2} \leq \beta \|u_{\mathcal{F}}\|_{L^{2}(\Omega)}^{2} \leq 9\beta/4 .
\end{equation*}
By (\ref{ineq: bound for difference |mathcalEV(u)-mathcalEV(u*)|}) and $\mathcal{E}_{V}\left(u^{*}\right) = \lambda_{k}$,
\begin{equation*} \label{}
\begin{aligned}
\mathcal{E}_{V}(u_{\mathcal{F}})
& \leq \max \left\{1, V_{\max }\right\} \left\|u-u^{*}\right\|_{H^{1}(\Omega)}^{2} + 2 \sqrt{\lambda_{k} \max \left\{1, V_{\max }\right\}} \left\|u-u^{*}\right\|_{H^{1}(\Omega)} + \mathcal{E}_{V}\left(u^{*}\right) \\
& \leq \left(\max \left\{1, \sqrt{V_{\max }}\right\}/2 + \sqrt{\lambda_{k}}\right)^{2}.
\end{aligned}
\end{equation*}
Therefore if  $\xi_{5}(n, \delta) < 1/2$, on the event $A_{5}(n, \delta),$
\begin{equation} \label{bound R22}
R_{22} \leq \frac{\left(\left(\max \left\{1, \sqrt{V_{\max }}\right\} + 2\sqrt{\lambda_{k}}\right)^{2}+9\beta \right) \xi_{5}(n, \delta)}{1-2\xi_{5}(n, \delta)}.
\end{equation}

To bound  $R_{21}$, we define the constant $\xi_{6}(n, \delta) :=   M_{\mathcal{G}_{2}}  \sqrt{\ln (2 / \delta) / (2n)}$
\iffalse
\begin{equation*} \label{xi6(n, delta)}
\xi_{6}(n, \delta) :=   M_{\mathcal{G}_{2}}  \sqrt{\frac{ \ln (2 / \delta)}{2n}},
\end{equation*}
\fi
and the event %$A_{6}(n, \delta) := \left\{ \left| \mathcal{E}_{n, V}(u_{\mathcal{F}}) - \mathcal{E}_{V}(u_{\mathcal{F}}) \right| \leq \xi_{6}(n, \delta) \right\} .$
\begin{equation*} 
A_{6}(n, \delta):=\left\{ \left|\mathcal{E}_{n, V}(u_{\mathcal{F}})-\mathcal{E}_{V}(u_{\mathcal{F}})\right| \leq \xi_{6}(n, \delta)\right\} .
\end{equation*}
Then applying Lemma \ref{Lemma: Hoeffding's inequality} leads to $\mathbf{P}\left[A_{6}(n, \delta)\right] \geq 1-\delta .$  
\iffalse
\begin{equation} \label{bound A6(n, delta)}
\mathbf{P}\left[A_{6}(n, \delta)\right] \geq 1-\delta .
\end{equation}
\fi
Hence, if  $\xi_{5}(n, \delta)<1/2$, within event $A_{5}(n, \delta) \cap A_{6}(n, \delta)$,
\begin{equation} \label{bound R21}
R_{21} \leq \frac{2\xi_{6}(n, \delta)}{1-2\xi_{5}(n, \delta)}.
\end{equation}

To bound  $R_{23}$, recall that $\bar{\mu}_{k} = \max_{1 \leq j \leq k-1} \|\psi_{j}\|_{L^{\infty}(\Omega)}$.
We define the constant %$\xi_{7}(n, \delta) := \bar{\mu}_{k} M_{\mathcal{F}}  \sqrt{\ln (2k / \delta) / (2 n)}$
\begin{equation*} \label{xi7(n, delta)}
\xi_{7}(n, \delta) := \bar{\mu}_{k} M_{\mathcal{F}}  \sqrt{\frac{\ln (2k / \delta)}{2 n}},
\end{equation*}
and the events
\begin{equation*} 
\begin{gathered}
A_{7, j}(n, \delta) := \left\{ \left| \frac{1}{n} \sum_{i=1}^{n} u_{\mathcal{F}}\left(X_{i}\right) \psi_{j}\left(X_{i}\right) - \left\langle u_{\mathcal{F}}, \psi_{j}\right\rangle \right| \leq \xi_{7}(n, \delta)\right\}, \quad \text{for each } 1 \leq j \leq k-1.  %,\\
%A_{7}(n, \delta) := \bigcap_{j=1}^{k-1} A_{7, j}(n, \delta).    
\end{gathered}
\end{equation*}
Let $A_{7}(n, \delta) := \bigcap_{j=1}^{k-1} A_{7, j}(n, \delta)$.
Then by Lemma \ref{Lemma: Hoeffding's inequality}, $\mathbf{P}\left[A_{7, j}(n, \delta)\right] \geq 1-\delta/k$ %for each $j$
and so $\mathbf{P}\left[A_{7}(n, \delta)\right] \geqslant 1-\delta.$
\iffalse
\begin{equation} \label{bound A7(n, delta)}
\mathbf{P}\left[A_{7}(n, \delta)\right] \geqslant 1-\frac{\delta}{k}\left(k-1\right) \geqslant 1-\delta.
\end{equation}
\fi
On event $A_{7}(n, \delta)$, it follows from the fact $a^{2} - b^{2} = (a-b)^{2} + 2b(a-b)$ that
\begin{equation*} \label{}
\begin{aligned}
\beta^{-1} \left|\mathcal{E}_{n, P}\left(u_{\mathcal{F}}\right)-\mathcal{E}_{P}\left(u_{\mathcal{F}}\right)\right| & \leq \sum_{j=1}^{k-1} \left[ \xi_{7}^{2}(n, \delta) + 2 \left| \left\langle u_{\mathcal{F}}, \psi_{j}\right\rangle \right|\xi_{7}(n, \delta) \right] \\ 
& \leq  2 \left(\sum_{j=1}^{k-1}\left\langle u_{\mathcal{F}}, \psi_{j}\right\rangle^{2}\right)^{1/2}\left(\sum_{j=1}^{k-1} \xi_{7}^{2}(n, \delta)\right)^{1/2} + (k-1) \xi_{7}^{2}(n, \delta)\\
& \leq 3 \sqrt{k} \xi_{7}(n, \delta) + k \xi_{7}^{2}(n, \delta),
\end{aligned}
\end{equation*}
where we have used $\|u_{\mathcal{F}}\|_{L^{2}(\Omega)} \leq 3/2$ in the last inequality. Hence, on the event $A_{5}(n, \delta) \cap A_{7}(n, \delta)$,
\begin{equation} \label{bound R23}
R_{23} \leq \frac{2\beta \sqrt{k} \xi_{7}(n, \delta)\left(3 + \sqrt{k} \xi_{7}(n, \delta) \right) }{1-2\xi_{5}(n, \delta)}.
\end{equation}
Thus, it may be concluded from (\ref{bound R22}), (\ref{bound R21}) and (\ref{bound R23}) that if  $\xi_{5}(n, \delta)<1/2$, within event $\bigcap_{i = 5}^{7} A_{i}(n, \delta) $,
\begin{equation*} \label{bound R2, former}
\begin{aligned}
R_{2} \leq \frac{\left(\left(\max \! \left\{1, \!  \sqrt{V_{\max }}\right\}  \! + \!  2\sqrt{\lambda_{k}}\right)^{2} \! + \! 9\beta \right) \xi_{5}(n, \delta) + 2\xi_{6}(n, \delta) + 2\beta \sqrt{k} \xi_{7}(n, \delta) \! \left(3  \! + \!  \sqrt{k} \xi_{7}(n, \delta) \right) }{1-2\xi_{5}(n, \delta)}.
\end{aligned}
\end{equation*}
By (\ref{upper bound of maximum value of hypothesis class with Softplus activation}) and direct calculations, there exist certain absolute constants $C$ such that
\begin{equation*} \label{}
\begin{aligned}
& \xi_{5}(n, \delta) \leq C \left(\frac{B}{d}\right)^{2}  \sqrt{\frac{\ln (1 / \delta)}{n}} , \quad  
\xi_{6}(n, \delta) \leq C B^{2} \left(1 + V_{\max }\right)   \sqrt{\frac{ \ln (1 / \delta)}{n}}, \\ 
& \xi_{7}(n, \delta) \leq  \frac{C \bar{\mu}_{k} B}{d}  \sqrt{\frac{\ln (k / \delta)}{ n}},
\end{aligned}
\end{equation*}
and then if $\xi_{5}(n, \delta) < 1/4$, 
\begin{equation} \label{bound R2}
\begin{aligned}
R_{2} %& \leq C \left(V_{\max } +\lambda_{k}+\beta \right) \left(B/d\right)^{2}\sqrt{\frac{\ln (1 / \delta)}{n}} + C B^{2} \left(1 + V_{\max }\right)\sqrt{\frac{\ln (1 / \delta)}{n}} \\
%& \quad + C\beta \left(\bar{\mu}_{k} B/d\right)  \sqrt{\frac{k\ln (k / \delta)}{n}} \left(1 + \left(\bar{\mu}_{k} B/d\right)  \sqrt{\frac{k\ln (k / \delta)}{n}} \right) \\
& \leq C \left(V_{\max } +\lambda_{k}+\beta \right) B^{2}\sqrt{\frac{\ln (1 / \delta)}{n}} + C\beta \left(\bar{\mu}_{k} B/d\right)  \sqrt{\frac{k\ln (k / \delta)}{n}}  .
\end{aligned}
\end{equation}
Note that the bound for $\xi_{5}$ is smaller than that for $R_{1}$ and $R_{2}$ up to an absolute constant. The bound for $R_{5}$ is smaller than that for $R_{3}$ up to an absolute constant.
By the estimates (\ref{bound R1}), (\ref{bound R3}), (\ref{bound R5}),   (\ref{bound R4}) and (\ref{bound R2}) with the choice $\gamma \geq 4 \lambda_{k}$,  %we conclude that 
(\ref{ineqs: conditions for theorem: solutions obtained by penalty method are away from 0})
\iffalse
\begin{equation*} \label{}
\begin{aligned}
& C \left(B/d\right)\left(B/d+1\right) \sqrt{ \frac{d \left(1+\ln B\right) m +  \ln (1 / \delta)}{n}} \leq 1  \\
%& C \left(B/d\right)^{2}  \sqrt{\frac{\ln (1 / \delta)}{ n}} \leq 1 \\
& C \left(1 + V_{\max }\right) B^{2}\sqrt{\frac{\ln (1 / \delta)}{n}} + C \left(\bar{\mu}_{k} B/d\right)  \sqrt{\frac{k\ln (k / \delta)}{n}} \leq 1 \\
& C \left(1 + V_{\max }\right) B /\sqrt{m} \leq  1 ,
\end{aligned}
\end{equation*}
\fi
may guarantee that $\xi_{5}(n, \delta) < 1/4$ and $\left| R_{1}\right|$, $R_{2}/\gamma$, $R_{3}/\gamma$, $\left(R_{4}+R_{5}\right)^{2}$ are all bounded by $1/16$ on event $\bigcap_{i = 4}^{7} A_{i}(n, \delta) $. 
Then, for $0<\delta<1/4$, %on the event $\bigcap_{i = 4}^{7} A_{i}(n, \delta) $, $\mathcal{E}_{2}(\mathscr{u}_{n}) \geq 1/2$ and so 
it follows from (\ref{bound for |mathcalE_2(mathscru_n)-1|}), (\ref{bound A4(n, delta)}) amd $\mathbf{P}\left[A_{i}(n, \delta)\right] \geq 1-\delta$ for $i = 5, 6, 7$ that %, (\ref{bound A5(n, delta)}), (\ref{bound A6(n, delta)}), (\ref{bound A7(n, delta)}) that
\begin{equation} \label{bound cap A4567(n, delta)}
\mathbf{P}\left(\mathcal{E}_{2}(\mathscr{u}_{n}) \geq 1/2 \right) \geq \mathbf{P}\left(\bigcap_{i = 4}^{7} A_{i}(n, \delta) \right) \geq 1 - 4\delta,
\end{equation}
which completes the proof of Theorem \ref{theorem: solutions obtained by penalty method are away from 0}.

Next, we analyze the generalization error of the penalty method.
To this end, we decompose
\begin{equation} \label{decomposition of generalization error L_k(mathscru_n)-lambda_k in penalty method}
\begin{aligned}
L_{k}(\mathscr{u}_{n})-\lambda_{k} & \leq 
%\mathscr{L}_{k}(\mathscr{u}_{n})-\lambda_{k} =
\big[\mathscr{L}_{k}(\mathscr{u}_{n})-\mathscr{L}_{k, n}(\mathscr{u}_{n})\big] + \mathscr{L}_{k, n}(\mathscr{u}_{n}) - \lambda_{k} \\
& =: R_{6} + R_{2} + R_{3} + \gamma \left( R_{4} + R_{5} \right)^{2},
\end{aligned}
\end{equation}
where we have used (\ref{error decomposition for gamma(mathcalE_n, 2(mathscru_n)-1)^2 in penalty method}).
Further, we decompose $R_{6}$ as follows
\begin{equation} \label{decomposition of R_6  in penalty method}
\begin{aligned}
R_{6} & = L_{k}(\mathscr{u}_{n}) - L_{k, n}(\mathscr{u}_{n}) 
+\gamma\left(\mathcal{E}_{2}(\mathscr{u}_{n}) - 1\right)^{2}-\gamma\left(\mathcal{E}_{n, 2}(\mathscr{u}_{n})-1\right)^{2} \\
& \leq L_{k}(\mathscr{u}_{n}) - L_{k, n}(\mathscr{u}_{n}) + \gamma\left[\left(\mathcal{E}_{2}(\mathscr{u}_{n})-\mathcal{E}_{n, 2}(\mathscr{u}_{n})\right)^{2} + 2\left|\mathcal{E}_{2}(\mathscr{u}_{n})-1\right| \left|\mathcal{E}_{2}(\mathscr{u}_{n})-\mathcal{E}_{n, 2}(\mathscr{u}_{n})\right|\right] \\
& =: R_{61} + \gamma\left[R_{1}^{2} + 2\left|\mathcal{E}_{2}(\mathscr{u}_{n})-1\right| \left|R_{1}\right|\right].
\end{aligned}
\end{equation}
Notice that under the assumptions of Theorem \ref{theorem: solutions obtained by penalty method are away from 0}, within event $\bigcap_{i = 4}^{7} A_{i}(n, \delta) $, 
$\mathcal{E}_{2}(\mathscr{u}_{n}) \geq 1/2$.
Thus, the analysis in \S \ref{section: Oracle inequality for the generalization error} is fully applicable to $\mathscr{u}_{n}$ and $u_{\mathcal{F}}$. 
%Here, without loss of generality, we may take $\hat{r} = 0.45$.
Proceeding along the same line as in \S \ref{section: Oracle inequality for the generalization error} to obtain (\ref{bound T1}), we have that if  $\xi_{1}(n, r, \delta)<1$, within event $\bigcap_{i=1}^{3} A_{i}(n, r, \delta)$,
\begin{equation} \label{bound R61}
\begin{aligned}
R_{61} 
%& \leq \frac{\xi_{1}+\xi_{2}}{1-\xi_{1}} \cdot \frac{\mathcal{E}_{V}\left(u_{n}\right)}{\mathcal{E}_{2}\left(u_{n}\right)} + \frac{\xi_{1}}{1 - \xi_{1}} \frac{\mathcal{E}_{P}\left(u_{n}\right)}{\mathcal{E}_{2}\left(u_{n}\right)} + \frac{\beta}{1-\xi_{1}}\left(\frac{k}{4} \xi_{3}^{2}+\sqrt{k} \xi_{3}\right)\\
& \leq %\frac{\xi_{1}+\xi_{2}}{1-\xi_{1}} L_{k}(\mathscr{u}_{n}) + \frac{\beta}{1-\xi_{1}}\left(\frac{k}{4} \xi_{3}^{2}+\sqrt{k} \xi_{3}\right)\\ & =
\frac{\xi_{1}+\xi_{2}}{1-\xi_{1}} \left(L_{k}(\mathscr{u}_{n})-\lambda_{k}\right) + \lambda_{k}\frac{\xi_{1}+\xi_{2}}{1-\xi_{1}} +\frac{\beta}{1-\xi_{1}}\left(\frac{k}{4} \xi_{3}^{2}+\sqrt{k} \xi_{3}\right) .
\end{aligned}
\end{equation}
Notice that $R_{2}$ coincides with $T_{2}$. 
Thus, by (\ref{bound T2}),
\begin{equation*}
\begin{aligned}
R_{2} & \leq  \frac{\xi_{1}+\xi_{2}}{1-\xi_{1}} \left(L_{k}(u_{\mathcal{F}})-\lambda_{k}\right) + \lambda_{k}\frac{\xi_{1}+\xi_{2}}{1-\xi_{1}} +\frac{\beta}{1-\xi_{1}}\left(\frac{k}{4} \xi_{3}^{2}+\sqrt{k} \xi_{3}\right) .
\end{aligned}
\end{equation*}
Combining this with (\ref{decomposition of generalization error L_k(mathscru_n)-lambda_k in penalty method}), (\ref{decomposition of R_6  in penalty method}) and (\ref{bound R61}) yields that if $2\xi_{1}+\xi_{2}\leq 1/2$,
\begin{equation*}
\begin{aligned}
L_{k}(\mathscr{u}_{n})-\lambda_{k} %& \leq R_{61} + \gamma\left[ \left( R_{4} + R_{5} \right)^{2} + R_{1}^{2} +  \left|R_{1}\right|\right] + R_{2} + R_{3}\\
& \leq \frac{\xi_{1}+\xi_{2}}{1-\xi_{1}} \left(L_{k}(\mathscr{u}_{n})-\lambda_{k}\right) + 2\lambda_{k}\frac{\xi_{1}+\xi_{2}}{1-\xi_{1}} +\frac{2\beta}{1-\xi_{1}}\left(\frac{k}{4} \xi_{3}^{2}+\sqrt{k} \xi_{3}\right)  \\
& \quad + \frac{\xi_{1}+\xi_{2}}{1-\xi_{1}} \left(L_{k}(u_{\mathcal{F}})-\lambda_{k}\right) + \gamma\left[ \left( R_{4} + R_{5} \right)^{2} + R_{1}^{2} +  \left|R_{1}\right|\right] + R_{3},
\end{aligned}
\end{equation*}
and so
\begin{equation} \label{ineq: generalization error bound for L_k(mathscru_n)-lambda_k to be substituted}
\begin{aligned}
L_{k}(\mathscr{u}_{n})-\lambda_{k} 
& \leq  4\lambda_{k}\left(\xi_{1}+\xi_{2}\right) +\beta\left(k \xi_{3}^{2}+ 4 \sqrt{k} \xi_{3}\right)  + 2 \left(L_{k}(u_{\mathcal{F}})-\lambda_{k}\right)  \\
& \quad + 2 \gamma\left( 2 R_{4}^{2} + 2 R_{5}^{2} + R_{1}^{2} +  \left|R_{1}\right|\right) + 2 R_{3} .%\\
%& \quad +  C\gamma B^{2}\left(\left(d^{-2}+B^{-2}\right) \sqrt{ \frac{d \left(1+\ln B\right) m +  \ln (1 / \delta)}{n}}  + \frac{1}{m}\right) .
\end{aligned}
\end{equation}
\iffalse
\begin{equation} \label{ineq: generalization error bound for L_k(mathscru_n)-lambda_k to be substituted}
\begin{aligned}
L_{k}(\mathscr{u}_{n})-\lambda_{k} 
& \leq  4\lambda_{k}\left(\xi_{1}+\xi_{2}\right) +\beta\left(k \xi_{3}^{2}+ 4 \sqrt{k} \xi_{3}\right)  + 2 \left(L_{k}(u_{\mathcal{F}})-\lambda_{k}\right)  \\
& \quad + 2 \gamma\left[ \left( R_{4} + R_{5} \right)^{2} + R_{1}^{2} +  \left|R_{1}\right|\right] + 2 R_{3} .
\end{aligned}
\end{equation}
\fi
Note that under the assumptions of Theorem \ref{theorem: solutions obtained by penalty method are away from 0}, $\xi_{5}(n, \delta)$, $\left|R_{1}\right|$,  $R_{4}+R_{5}$ are all bounded by $1/4$. 
%Under the assumptions of Theorem \ref{Main generalization theorem}, $2\xi_{1}+\xi_{2}\leq 1/2$.
When (\ref{ineq: assumption in Main generalization theorem to ensure xi1+xi2<1/2}) and (\ref{Simplifying condition for xi3 term}) with $r = 0.49$ hold,  $2\xi_{1}+\xi_{2}\leq 1/2$.
Here, without loss of generality, we may take $r = 0.49$ and $\delta\in (0, 1/7)$.
Substituting the bounds (\ref{bound R1}), (\ref{bound R3}), (\ref{bound R4}), (\ref{bound R5}) derived for $R_{1}$, $R_{3}$, $R_{4}$, $R_{5}$ earlier in this section, and the bounds (\ref{statistical error bound for xi1, xi2 and xi3}) for $\{\xi_{i}\}_{i = 1}^{3}$ into (\ref{ineq: generalization error bound for L_k(mathscru_n)-lambda_k to be substituted}),  we obtain the error bound (\ref{ineq: error bound in corollary: Main generalization theorem for the penalty method}) on the event $A_{in} = \left(\bigcap_{i=1}^{3} A_{i}(n, r, \delta)\right) \bigcap \left(\bigcap_{i = 4}^{7} A_{i}(n, \delta)\right)$ with
$\mathbf{P}\left(A_{in}\right) \geq 1-7\delta$ followed from (\ref{bound cap A4567(n, delta)}) and (\ref{bound probability of intersection of A(n, r, delta)}),
which completes the proof of Corollary \ref{corollary: Main generalization theorem for the penalty method}.

\section{Analysis of the error accumulation} \label{appendix section: About the cumulative error}
%\section{About the cumulative error} \label{appendix section: About the cumulative error}
In this section, we aim to prove Theorem \ref{Main generalization theorem when cumulative error is added} and Proposition \ref{proposition: Quadratic growth of cumulative error}.
Recall the loss $L_{k}(u)$ defined in (\ref{population loss}) and here we take $\psi_{j}$ to be the normalization of the orthogonal projection of $\mathfrak{u}_{\theta j}$ to subspace $U_{j}$, i.e.,  $\psi_{j} = P_{j} \mathfrak{u}_{\theta j} / \|P_{j} \mathfrak{u}_{\theta j}\|_{L^{2}(\Omega)}.$
Since $\mathfrak{u}_{\theta k}$ is a minimizer of $\widetilde{L}_{k, n}(u)$ within $\mathcal{F}_{>r}$,
$\widetilde{L}_{k, n}(\mathfrak{u}_{n}) - \widetilde{L}_{k, n}(u_{\mathcal{F}}) \leq 0$ for any $u_{\mathcal{F}} \in \mathcal{F}_{>r}$.
Similar as in (\ref{ineq: decomposition of L_k(u_n)-lambda_k}), we decompose the generalization error 
\iffalse
\begin{equation}
\begin{aligned}
L_{k}(\mathfrak{u}_{\theta k}) - \lambda_{k} = & \left(L_{k}(\mathfrak{u}_{\theta k})-\widetilde{L}_{k}(\mathfrak{u}_{\theta k})\right) + \left(\widetilde{L}_{k}(\mathfrak{u}_{\theta k})-\widetilde{L}_{k, n}(\mathfrak{u}_{\theta k})\right) + \left(\widetilde{L}_{k, n}(\mathfrak{u}_{\theta k})-\tilde{L}_{k, n}(u_{\mathcal{F}})\right) \\
& + \left(\widetilde{L}_{k, n}(u_{\mathcal{F}})-\widetilde{L}_{k}(u_{\mathcal{F}})\right) + \left(\widetilde{L}_{k}(u_{\mathcal{F}})-L_{k}(u_{\mathcal{F}})\right)+\left(L_{k}(u_{\mathcal{F}})-\lambda_{k}\right) \\
\leq & \left(L_{k}(\mathfrak{u}_{\theta k})-\widetilde{L}_{k}(\mathfrak{u}_{\theta k})\right) + \left(\widetilde{L}_{k}(\mathfrak{u}_{\theta k})-\widetilde{L}_{k, n}(\mathfrak{u}_{\theta k})\right) \\
&  + \left(\widetilde{L}_{k, n}(u_{\mathcal{F}})-\widetilde{L}_{k}(u_{\mathcal{F}})\right) + \left(\widetilde{L}_{k}(u_{\mathcal{F}})-L_{k}(u_{\mathcal{F}})\right)+\left(L_{k}(u_{\mathcal{F}})-\lambda_{k}\right) \\
=: & \ S_{1}+S_{2}+S_{3}+S_{4}+S_{5} ,
\end{aligned}
\end{equation}
\fi
\begin{equation} \label{decomposition of L_k(mathfraku_theta k) - lambda_k in cumulative error}
\begin{aligned}
L_{k}(\mathfrak{u}_{\theta k}) - \lambda_{k} %= & \left(L_{k}(\mathfrak{u}_{\theta k})-\widetilde{L}_{k}(\mathfrak{u}_{\theta k})\right) + \left(\widetilde{L}_{k}(\mathfrak{u}_{\theta k})-\widetilde{L}_{k, n}(\mathfrak{u}_{\theta k})\right) + \left(\widetilde{L}_{k, n}(\mathfrak{u}_{\theta k})-\tilde{L}_{k, n}(u_{\mathcal{F}})\right) \\
%& + \left(\widetilde{L}_{k, n}(u_{\mathcal{F}})-\widetilde{L}_{k}(u_{\mathcal{F}})\right) + \left(\widetilde{L}_{k}(u_{\mathcal{F}})-L_{k}(u_{\mathcal{F}})\right)+\left(L_{k}(u_{\mathcal{F}})-\lambda_{k}\right) \\
\leq & \left(L_{k}(\mathfrak{u}_{\theta k})-\widetilde{L}_{k}(\mathfrak{u}_{\theta k})\right) + \left(\widetilde{L}_{k}(\mathfrak{u}_{\theta k})-\widetilde{L}_{k, n}(\mathfrak{u}_{\theta k})\right)  
+ \left(\widetilde{L}_{k, n}(u_{\mathcal{F}})-\widetilde{L}_{k}(u_{\mathcal{F}})\right) \\
&   + \left(\widetilde{L}_{k}(u_{\mathcal{F}})-L_{k}(u_{\mathcal{F}})\right) + \Big(L_{k}(u_{\mathcal{F}})-\lambda_{k}\Big) \\
=: & \ S_{1}+S_{2}+S_{3}+S_{4}+S_{5} ,
\end{aligned}
\end{equation}
for any $u_{\mathcal{F}} \in \mathcal{F}_{>r}$.
%where $u_{\mathcal{F}}$ is an arbitrary function in $\mathcal{F}_{>r}$.
Note that  $S_{1}$, $S_{4}$  are the cumulative errors due to the use of the eigenfunctions approximated by neural networks in the penalty term,  
$S_{2}$ is the statistical error, $S_{3}$ is the Monte Carlo error and  $S_{5}$  is the approximation error.

Firstly, we give a uniform bound on the cumulative error $\widetilde{L}_{k}(u) - L_{k}(u)$.
\begin{proposition} \label{proposition: uniform bound on the cumulative error widetildeLk(u) - Lk(u)}
Assume that for each $j\in \mathbb{N}_{+}$, $\psi_{j}$ is the normalization of the orthogonal projection of $u_{\theta j}$ to subspace $U_{j}$. Set $\beta = \beta_{k}$ in $L_{k}(u)$.
Then for any $u \in L^{2}(\Omega)$,
\begin{equation*} 
\begin{aligned}
\left| \widetilde{L}_{k}(u) - L_{k}(u) \right| & \leq 2\beta_{k} \sum_{j=1}^{k-1} \sqrt{\frac{L_{j}(u_{\theta j})-\lambda_{j}}{\min \left\{\beta_{j}+\lambda_{1}-\lambda_{j}, \lambda_{j^{\prime}}-\lambda_{j}\right\}} }.
\end{aligned}
\end{equation*}
\end{proposition}
\begin{proof}
%Recall the ideal loss function $L_{k}(u)$ for computing the  $k$-th eigenvalue problem defined in (\ref{the first definition of the loss function Lk(u)}).
\iffalse
\begin{equation*}
\begin{aligned} \label{}
L_{k}(u) = \frac{\mathcal{E}_{V}(u)}{\mathcal{E}_{2}(u)} + \beta_{k} \sum_{j=1}^{k-1} \frac{\langle u, \psi_{j}\rangle^{2} }{\mathcal{E}_{2}(u)}.
\end{aligned}
\end{equation*}
\fi
%where $\psi_{j}$ is a normalized eigenfunction lying in the subspace $U_{j}$.
Note that the only difference between $\widetilde{L}_{k}(u)$ and $L_{k}(u)$ is the penalty term. Let 
$\bar{u} = u / \|u\|_{L^{2}(\Omega)}$ and $\bar{u}_{\theta j} = u_{\theta j} / \|u_{\theta j}\|_{L^{2}(\Omega)}.$
%Here, we abbreviate $\|\cdot\|_{L^{2}(\Omega)}$ to $\|\cdot\|_{L^{2}}$. 
Then, for any $u \in L^{2}(\Omega)$, by triangle inequality and Cauchy's inequality, we have
\iffalse
\begin{equation*} 
\begin{aligned}
\frac{1}{\beta_{k}}\left| \widetilde{L}_{k}(u) - L_{k}(u) \right| & = \left|\sum_{j=1}^{k-1} \frac{\langle u, u_{\theta j}\rangle^{2} }{\|u\|_{L^{2}}^{2} \|u_{\theta j}\|_{L^{2}}^{2}} - \sum_{j=1}^{k-1} \frac{\langle u, \psi_{j}\rangle^{2} }{\|u\|_{L^{2}}^{2}} \right| 
\leq \sum_{j=1}^{k-1}\left| \langle \bar{u}, \bar{u}_{\theta j}\rangle^{2}  -  \langle \bar{u}, \psi_{j}\rangle^{2}  \right| \\
& = \sum_{j=1}^{k-1}\left|\langle \bar{u}, \bar{u}_{\theta j} + \psi_{j}\rangle\right| \left|\langle \bar{u}, \bar{u}_{\theta j} - \psi_{j}\rangle\right| \leq \sum_{j=1}^{k-1} \|\bar{u}\|_{L^{2}}^{2} \|\bar{u}_{\theta j} + \psi_{j}\|_{L^{2}} \|\bar{u}_{\theta j} - \psi_{j}\|_{L^{2}}
\end{aligned}
\end{equation*}
where we have used Cauchy's inequality in the second inequality. 
\fi
\begin{equation} \label{ineq in pf of proposition: uniform bound on the cumulative error widetildeLk(u) - Lk(u) (1)}
\begin{aligned}
\frac{1}{\beta_{k}}\left| \widetilde{L}_{k}(u) - L_{k}(u) \right| &
\leq  \sum_{j=1}^{k-1}\left|\langle \bar{u}, \bar{u}_{\theta j} + \psi_{j}\rangle\right| \left|\langle \bar{u}, \bar{u}_{\theta j} - \psi_{j}\rangle\right| \\
& \leq \sum_{j=1}^{k-1} \|\bar{u}\|_{L^{2}}^{2} \|\bar{u}_{\theta j} + \psi_{j}\|_{L^{2}} \|\bar{u}_{\theta j} - \psi_{j}\|_{L^{2}}.
\end{aligned}
\end{equation}
Since $\bar{u}_{\theta j}$ and $\psi_{j}$ are both normalized, 
\begin{equation} \label{ineq in pf of proposition: uniform bound on the cumulative error widetildeLk(u) - Lk(u) (2)}
\begin{aligned}
\|\bar{u}_{\theta j} + \psi_{j}\|_{L^{2}} \|\bar{u}_{\theta j} - \psi_{j}\|_{L^{2}} & = 2 \sqrt{1 - \langle \bar{u}_{\theta j}, \psi_{j}\rangle^{2}} .
\end{aligned}
\end{equation}
By the choice of $\psi_{j}$ and the stability estimates stated in Proposition \ref{Generalized proposition 2.1},
\begin{equation} \label{ineq in pf of proposition: uniform bound on the cumulative error widetildeLk(u) - Lk(u) (3)}
\begin{aligned}
\sqrt{1 - \langle \bar{u}_{\theta j}, \psi_{j}\rangle^{2}} & = \frac{\|P_{j}^{\perp} u_{\theta j}\|_{L^{2}}}{\|u_{\theta j}\|_{L^{2}}} \leq \sqrt{\frac{L_{j}(u_{\theta j})-\lambda_{j}}{\min \left\{\beta_{j}+\lambda_{1}-\lambda_{j}, \lambda_{j^{\prime}}-\lambda_{j}\right\}} }.
\end{aligned}
\end{equation}
The proof completes by combining (\ref{ineq in pf of proposition: uniform bound on the cumulative error widetildeLk(u) - Lk(u) (1)}), (\ref{ineq in pf of proposition: uniform bound on the cumulative error widetildeLk(u) - Lk(u) (2)}) and (\ref{ineq in pf of proposition: uniform bound on the cumulative error widetildeLk(u) - Lk(u) (3)}).
\iffalse
Hence, we conclude that for any $u \in L^{2}(\Omega)$,
\begin{equation*} 
\begin{aligned}
\left| \widetilde{L}_{k}(u) - L_{k}(u) \right| & \leq 2\beta_{k} \sum_{j=1}^{k-1} \sqrt{\frac{L_{j}(u_{\theta j})-\lambda_{j}}{\min \left\{\beta_{j}+\lambda_{1}-\lambda_{j}, \lambda_{j^{\prime}}-\lambda_{j}\right\}} }.
\end{aligned}
\end{equation*}
\fi
\end{proof}

Recall the constants $\{\xi_{i}(n, r, \delta)\}_{i=1}^{3}$ and the events $\{A_{i}(n, r, \delta)\}_{i=1}^{3}$ defined in \S \ref{section: Oracle inequality for the generalization error}. For statistical errors, we decompose
\begin{equation*}
\begin{aligned}
S_{2} & \leq \left|\frac{\mathcal{E}_{n,V}(\mathfrak{u}_{\theta k})}{\mathcal{E}_{n, 2}(\mathfrak{u}_{\theta k})}-\frac{\mathcal{E}_{V}(\mathfrak{u}_{\theta k})}{\mathcal{E}_{2}(\mathfrak{u}_{\theta k})}\right|  +  \beta_{k} \sum_{j=1}^{k-1}\left|\frac{\mathscr{P}_{n, j}(\mathfrak{u}_{\theta k})}{\mathcal{E}_{n, 2}(\mathfrak{u}_{\theta k}) \mathcal{E}_{n, 2}(u_{\theta j})} - \frac{\mathscr{P}_{j}(\mathfrak{u}_{\theta k})}{\mathcal{E}_{2}(\mathfrak{u}_{\theta k}) \mathcal{E}_{2}(u_{\theta j})}\right| \\
& =: S_{21} + S_{22}.
\end{aligned}
\end{equation*}
Proceeding along the same line as to obtain \ref{bound T11}, we get that if $\xi_{1}(n, r, \delta) < 1$, on the event $\bigcap_{i=1}^{2} A_{i}(n, r, \delta)$, 
\begin{equation} \label{bound S21}
\begin{aligned}
S_{21} \leq  \frac{\xi_{1}(n, r, \delta)+\xi_{2}(n, r, \delta)}{1-\xi_{1}(n, r, \delta)} \cdot \frac{\mathcal{E}_{V}\left(\mathfrak{u}_{\theta k}\right)}{\mathcal{E}_{2}\left(\mathfrak{u}_{\theta k}\right)}.
\end{aligned}
\end{equation}
\iffalse
\begin{equation*}
\begin{aligned}
& \quad \mathscr{P}_{n, j}(u) \mathcal{E}_{2}(u) \mathcal{E}_{2}(u_{\theta j}) - \mathscr{P}_{j}(u) \mathcal{E}_{n, 2}(u) \mathcal{E}_{n, 2}(u_{\theta j}) \\
&  = \mathcal{E}_{2}(u) \mathcal{E}_{2}(u_{\theta j}) \left(\mathscr{P}_{n, j}(u) - \mathscr{P}_{j}(u)\right) + \left[ \mathcal{E}_{2}(u) \mathcal{E}_{2}(u_{\theta j}) - \mathcal{E}_{n, 2}(u) \mathcal{E}_{n, 2}(u_{\theta j})\right] \mathscr{P}_{j}(u) 
\end{aligned}
\end{equation*}
\fi
For $S_{22}$, notice that
\begin{equation*}
\begin{aligned}
\frac{S_{22}}{\beta_{k}} & \leq \frac{\left|\mathscr{P}_{n, j}(\mathfrak{u}_{\theta k}) - \mathscr{P}_{j}(\mathfrak{u}_{\theta k})\right|}{\mathcal{E}_{2}(\mathfrak{u}_{\theta k}) \mathcal{E}_{2}(u_{\theta j})} \frac{\mathcal{E}_{2}(\mathfrak{u}_{\theta k}) \mathcal{E}_{2}(u_{\theta j})}{\mathcal{E}_{n, 2}(\mathfrak{u}_{\theta k}) \mathcal{E}_{n, 2}(u_{\theta j})}  +  \frac{\mathscr{P}_{j}(\mathfrak{u}_{\theta k})}{\mathcal{E}_{2}(\mathfrak{u}_{\theta k}) \mathcal{E}_{2}(u_{\theta j})} \left| \frac{\mathcal{E}_{2}(\mathfrak{u}_{\theta k}) \mathcal{E}_{2}(u_{\theta j})}{\mathcal{E}_{n, 2}(\mathfrak{u}_{\theta k}) \mathcal{E}_{n, 2}(u_{\theta j})}-1\right| .
\end{aligned}
\end{equation*}
Recall that to bound the statistical error, we only use the normalization property and $L^{\infty}(\Omega)$ boundedness of $\psi_{j}$.
%We may bound the above terms similarly by only replacing  $\psi_{j}$ with $\bar{u}_{\theta j}$ and accordingly $\|\psi_{j}\|_{L^{\infty}(\Omega)}$ with $\|\bar{u}_{\theta j}\|_{L^{\infty}(\Omega)}$.
We may define $\xi_{\theta 3}(n, r, \delta)$, $A_{\theta 3}(n, r, \delta)$ in the same way as $\xi_{3}(n, r, \delta)$, $A_{3}(n, r, \delta)$ by only replacing  $\psi_{j}$ with $\bar{u}_{\theta j}$ and get bounds for them by accordingly replacing $\|\psi_{j}\|_{L^{\infty}(\Omega)}$ with $\|\bar{u}_{\theta j}\|_{L^{\infty}(\Omega)}$.
Then, similar as in (\ref{ineq: bound |mathcalE_n, P(u_n) -mathcalE_P(u_n)|/beta mathcalE_2(u_n), references for similar formulas}), we obtain
\begin{equation*}
\begin{aligned}
\sum_{j=1}^{k-1}\frac{\left|\mathscr{P}_{n, j}(\mathfrak{u}_{\theta k}) - \mathscr{P}_{j}(\mathfrak{u}_{\theta k})\right|}{\mathcal{E}_{2}(\mathfrak{u}_{\theta k}) \mathcal{E}_{2}(u_{\theta j})} %& = \sum_{j=1}^{k-1}\left[\frac{\left|P_{n}\left(\mathfrak{u}_{\theta k} \bar{u}_{\theta j}\right)-P\left(\mathfrak{u}_{\theta k} \bar{u}_{\theta j}\right)\right|^{2}}{\left\|\mathfrak{u}_{\theta k}\right\|_{L^{2}(\Omega)}^{2}}+\frac{2 \left|\left\langle \mathfrak{u}_{\theta k}, \bar{u}_{\theta j}\right\rangle\right|}{\left\|\mathfrak{u}_{\theta k}\right\|_{L^{2}(\Omega)}} \frac{\left| P_{n}\left(\mathfrak{u}_{\theta k} \bar{u}_{\theta j}\right)-P\left(\mathfrak{u}_{\theta k} \bar{u}_{\theta j}\right)\right|}{\left\|\mathfrak{u}_{\theta k}\right\|_{L^{2}(\Omega)}}\right] \\
& \leqslant \frac{k}{4} \xi_{\theta 3}(n, r, \delta)^{2} + \sqrt{k} \xi_{\theta 3}(n, r, \delta) .
\end{aligned}
\end{equation*}
If $\xi_{1}(n, r, \delta) < 1$, on the event $A_{1}(n, r, \delta)$, 
\begin{equation*}
\begin{aligned}
\left(1+\xi_{1}\right)^{-2} < \frac{\mathcal{E}_{2}(\mathfrak{u}_{\theta k}) \mathcal{E}_{2}(u_{\theta j})}{\mathcal{E}_{n, 2}(\mathfrak{u}_{\theta k}) \mathcal{E}_{n, 2}(u_{\theta j})} < \left(1-\xi_{1}\right)^{-2} .
\end{aligned}
\end{equation*}
Hence,
\begin{equation*}
\begin{aligned}
\left| \frac{\mathcal{E}_{2}(u) \mathcal{E}_{2}(u_{\theta j})}{\mathcal{E}_{n, 2}(u) \mathcal{E}_{n, 2}(u_{\theta j})}-1\right| < \frac{\xi_{1} \left(2-\xi_{1}\right)}{\left(1-\xi_{1}\right)^{2}}  .
\end{aligned}
\end{equation*}
Thus, on event $A_{1}(n, r, \delta) \cap A_{\theta 3}(n, r, \delta)$, if  $\xi_{1}(n, r, \delta) < 1$,
\begin{equation} \label{bound S22}
S_{22} \leq \frac{\beta_{k}}{\left(1-\xi_{1}\right)^{2}}\left(\frac{k}{4} \xi_{\theta 3}^{2} + \sqrt{k} \xi_{\theta 3} \right) 
+ \frac{\xi_{1} \left(2-\xi_{1}\right)}{\left(1-\xi_{1}\right)^{2}} 
\sum_{j=1}^{k-1} \frac{\beta_{k} \mathscr{P}_{j}(\mathfrak{u}_{\theta k})}{\mathcal{E}_{2}(\mathfrak{u}_{\theta k}) \mathcal{E}_{2}(u_{\theta j})} .
\end{equation}
We conclude from (\ref{bound S21}) and (\ref{bound S22}) that on event $A_{1}(n, r, \delta) \cap A_{2}(n, r, \delta) \cap A_{\theta 3}(n, r, \delta)$, if  $\xi_{1}(n, r, \delta) < 1$,
\begin{equation*} \label{bound S2}
\begin{aligned}
S_{2} \leq  \frac{2\xi_{1} + \xi_{2}}{\left(1-\xi_{1}\right)^{2}}  L_{k}(\mathfrak{u}_{\theta k}) + \frac{\beta_{k}}{\left(1-\xi_{1}\right)^{2}}\left(\frac{k}{4} \xi_{\theta 3}^{2} + \sqrt{k} \xi_{\theta 3} \right) .
\end{aligned}
\end{equation*}
Similarly, since $\|u_{\mathcal{F}}\|_{L^{2}(\Omega)} > r$, we obbtain
\begin{equation*} \label{bound S3}
\begin{aligned}
S_{3} \leq  \frac{2\xi_{1} + \xi_{2}}{\left(1-\xi_{1}\right)^{2}}  L_{k}(u_{\mathcal{F}}) + \frac{\beta_{k}}{\left(1-\xi_{1}\right)^{2}}\left(\frac{k}{4} \xi_{\theta 3}^{2} + \sqrt{k} \xi_{\theta 3} \right) .
\end{aligned}
\end{equation*}
Then, from the above two estimates and (\ref{decomposition of L_k(mathfraku_theta k) - lambda_k in cumulative error}), we obtain
\begin{equation*}
\begin{aligned}
L_{k}(\mathfrak{u}_{\theta k}) - \lambda_{k} & \leq  \frac{2\xi_{1} + \xi_{2}}{\left(1-\xi_{1}\right)^{2}} \left(L_{k}(\mathfrak{u}_{\theta k})-\lambda_{k}\right) + \frac{2\xi_{1} + \xi_{2}}{\left(1-\xi_{1}\right)^{2}}  \left(L_{k}(u_{\mathcal{F}})-\lambda_{k}\right) + 2\lambda_{k} \frac{2\xi_{1} + \xi_{2}}{\left(1-\xi_{1}\right)^{2}}  \\
& \quad + \frac{2\beta_{k}}{\left(1-\xi_{1}\right)^{2}}\left(\frac{k}{4} \xi_{\theta 3}^{2} + \sqrt{k} \xi_{\theta 3} \right)+S_{1}+S_{4}+S_{5} ,
\end{aligned}
\end{equation*}
and if $4 \xi_{1} + \xi_{2} \leq 1/2$,
\begin{equation} \label{last decomposition of L_k(mathfraku_theta k) - lambda_k in cumulative error}
\begin{aligned}
L_{k}(\mathfrak{u}_{\theta k}) - \lambda_{k} %& \leq   \frac{2\xi_{1} + \xi_{2}}{1- 4 \xi_{1} - \xi_{2} }  S_{5} + 2\lambda_{k} \frac{2\xi_{1} + \xi_{2}}{1- 4 \xi_{1} - \xi_{2} }  + \frac{2\beta_{k}}{1- 4 \xi_{1} - \xi_{2} }\left(\frac{k}{4} \xi_{\theta 3}^{2} + \sqrt{k} \xi_{\theta 3} \right) \\
%& \quad + \frac{\left(1-\xi_{1}\right)^{2}}{1- 4 \xi_{1} - \xi_{2} }\left(S_{1}+S_{4}+S_{5}\right) \\
& \leq  4\lambda_{k} \left(2\xi_{1} + \xi_{2}\right)  + \beta_{k}\left(k \xi_{\theta 3}^{2} + 4\sqrt{k} \xi_{\theta 3} \right) + 2 S_{1} + 2 S_{4} + 3 S_{5},
\end{aligned}
\end{equation}
Let $u_{\mathcal{F}}$ be given by Theorem \ref{Thm: u H1 approximation by varphi Softplus networks}. Then, the proof of Theorem \ref{Main generalization theorem when cumulative error is added} completes by substituting the bounds (\ref{statistical error bound for xi1, xi2 and xi3}) for $\{\xi_{i}\}_{i = 1}^{3}$, the bound for $S_{1}$, $S_{4}$ given by Proposition \ref{proposition: uniform bound on the cumulative error widetildeLk(u) - Lk(u)} and the bound for $S_{5}$ given by Theorem \ref{Thm: bounding the approximation error Lk(u in F)-lambdak} into (\ref{last decomposition of L_k(mathfraku_theta k) - lambda_k in cumulative error}).

\begin{proof}[Proof of Proposition \ref{proposition: Quadratic growth of cumulative error}]
Let $d_{0}=0$ and
$d_{k}=\sum_{j=1}^{k} \sqrt{\big( L_{j}\left(u_{\theta j}\right)-\lambda_{j} \big) / \beta_{j}}  .$
Since
\begin{equation*}
\begin{aligned}
\frac{L_{k}\left(u_{\theta k}\right)-\lambda_{k}}{\beta_{k}} \leq \frac{\Delta_{k}}{\beta_{k}} + 8 \sum_{j=1}^{k-1} \sqrt{\frac{\beta_{j}}{\min \left(\beta_{j}+\lambda_{1}-\lambda_{j}, \lambda_{j^{\prime}}-\lambda_{j}\right)} \cdot \frac{L_{j}\left(u_{\theta j}\right)-\lambda_{j}}{\beta_{j}}} ,
\end{aligned}
\end{equation*}
we have 
\begin{equation} \label{Pf of prop: Quadratic growth of cumulative error, recursive relation}
\left(d_{k} - d_{k-1} \right)^{2} = \frac{L_{k}\left(u_{\theta k}\right)-\lambda_{k}}{\beta_{k}} \leq \tau_{k} + 2 \rho_{k-1} d_{k-1}
\end{equation}
and then $d_{k} \leq d_{k-1}+\sqrt{\tau_{k} + 2 \rho_{k-1} d_{k-1}}.$
We claim that $d_{k} \leq \rho_{k-1} k^{2} /2 + \sqrt{\tau_{k}} k$  for all  $k \in \mathbb{N}$, which may be proved by induction. 
When $k = 0$, the bound is trivial. Assume that the claim holds for $k$. For $k+1$, since $\rho_{k-1} \leq \rho_{k}$ and $\tau_{k} \leq \tau_{k+1}$,
\begin{equation*}
\begin{aligned}
d_{k+1} & \leq \rho_{k-1} k^{2}/2 + \sqrt{\tau_{k}} k + \sqrt{\tau_{k+1} + 2 \rho_{k} \left(\rho_{k-1} k^{2}/2 + \sqrt{\tau_{k}} k\right)} \\
& = \rho_{k} k^{2}/2 + \sqrt{\tau_{k}} k + \left(\rho_{k} k + \sqrt{\tau_{k+1}} \right) \\
& \leq  \rho_{k} (k+1)^{2} /2 + \sqrt{\tau_{k+1}} (k+1),
\end{aligned}
\end{equation*}
which completes the proof of the claim. Then, the estimate (\ref{eq: conclusion of prop: Quadratic growth of cumulative error, recursive relation}) follows from the claim and (\ref{Pf of prop: Quadratic growth of cumulative error, recursive relation}) directly.
\end{proof}

\section{Simple properties of eigenvalues and eigenfunctions} \label{appendix section: Simple properties of eigenvalues and eigenfunctions}
In this part, we characterize the asymptotic distribution of the eigenvalues %of the operator $\mathcal{H}$ 
and estimate the maximum norm of the eigenfunctions.
%\subsection{Asymptotic distribution of eigenvalues}
\begin{lemma} \label{lemma: Asymptotic distribution of eigenvalues}
Assume that $\mathcal{H}$ satisfies Assumption \ref{Assumption: V is bounded above and below}. Then the $k$-th eigenvalue $\lambda_{k}$ satisfies 
$$\tilde{c} d k^{2/d} + V_{\min} \leq \lambda_{k} \leq c d k^{2/d} + V_{\max},$$ 
where $\tilde{c}, c$ are suitable absolute constants.
\end{lemma}

\begin{proof}
First,  when $V\equiv0$, we consider the eigenvalue problem of the Laplacian.
\iffalse
\begin{equation*}
\begin{cases}
-\Delta u = \nu u, &  x \in \Omega \\
 u = 0, &  x \in \partial\Omega. 
\end{cases}
\end{equation*}
\fi
By Weyl's formula\cite{Weyl1912DasAV, Weyl1912ÜberDA}, the $k$-th eigenvalue $\nu_{k}$ satisfies
$$c_{1}C(d)^{2/d}k^{2/d} \leq \nu_{k} \leq c_{2}C(d)^{2/d}k^{2/d},$$
where $c_{1}, c_{2}$ are suitable absolute constants and 
$C(d) = d 2^{d-1} \pi^{d / 2} \Gamma\left(d/2\right)$
with the Euler gamma function 
%; explicitly
\begin{equation*}
\Gamma\left(\frac{d}{2}\right)=\left\{\begin{array}{ll}
(p-1) ! , & d = 2 p, \\
2^{-2 p}(2 p) ! \pi^{1 / 2} / p ! , & d = 2 p+1,
\end{array}\right.
\end{equation*}
where  $p$  is an integer. Thanks to the fact    $\lim_{d \to \infty} d^{2/d} = 1$  and  the Stirling formula, 
\begin{equation*}
c_{3} d \leq \pi d^{2/d} \Gamma\left(d/2\right)^{2/d} \leq C(d)^{2/d} \leq 4 \pi d^{2/d} \Gamma\left(d/2\right)^{2/d} \leq c_{4} d,
\end{equation*}
where $c_{3}$, $ c_{4}$ are absolute constants.  
Thus, $\tilde{c} d k^{2/d} \leq \nu_{k} \leq c d k^{2/d}.$

Second, by the minimax principle
\begin{equation*}
\lambda_{k} = \min_{\operatorname{dim} E = k} \max_{u \in E \backslash\{0\}} \frac{\int_{\Omega} \left( |\nabla u|^{2}+V u^{2} \right) d x}{\int_{\Omega}   u^{2}  d x},
\end{equation*}
where  the minimum is taken over all  $k$-dimensional subspace $E \subset H_{0}^{1}(\Omega)$.
Since $V_{\min} \leq V(x) \leq V_{\max}, $ %$V$ satisfies Assumption \ref{Assumption: V is bounded above and below}, 
\iffalse
$$\int_{\Omega} \left( |\nabla u|^{2}+V_{\min} u^{2} \right) d x \leq \int_{\Omega} \left( |\nabla u|^{2}+V u^{2} \right) d x \leq \int_{\Omega} \left( |\nabla u|^{2}+V_{\max} u^{2} \right) d x,$$
which implies 
\fi
\begin{equation*}
\min_{\operatorname{dim} E = k} \max_{u \in E \backslash\{0\}} \frac{\int_{\Omega}  |\nabla u|^{2}  d x}{\int_{\Omega}   u^{2}  d x} + V_{\min} \leq \lambda_{k} \leq \min_{\operatorname{dim} E = k} \max_{u \in E \backslash\{0\}} \frac{\int_{\Omega}  |\nabla u|^{2}  d x}{\int_{\Omega}   u^{2}  d x} + V_{\max},
\end{equation*}
%where by the minimax principle $$\min_{\operatorname{dim} E = k} \max_{u \in E \backslash\{0\}} \frac{\int_{\Omega}  |\nabla u|^{2}  d x}{\int_{\Omega}   u^{2}  d x} = \nu_{k}. $$
which implies $\nu_{k} + V_{\min} \leq \lambda_{k} \leq \nu_{k} + V_{\max}$ by the minimax principle. 
Substituting the upper and lower bounds for $\nu_{k}$ completes the proof.
%we obtain that $\tilde{c} d k^{2/d} + V_{\min} \leq \lambda_{k} \leq c d k^{2/d} + V_{\max}.$
\end{proof}

\begin{lemma} \label{lemma: Linfty bound for eigenfunctions w.r.t. eigenvalues}
Assume that $\mathcal{H}$ satisfies Assumption \ref{Assumption: V is bounded above and below}. 
Then the normalized eigenfunctions of $\mathcal{H}$ satisfy 
\begin{equation} \label{ineq of lemma: Linfty bound for eigenfunctions w.r.t. eigenvalues}
\|\psi_{k}\|_{L^{\infty}(\Omega)} \leq \left(c k^{2/d} +  \frac{e\left(V_{\max}- V_{\min}\right)}{\pi d}\right)^{d / 4}, 
\end{equation}
where $c$ is an absolute constant.
\end{lemma}

\begin{proof}
Combining Example 2.1.9 and Lemma 2.1.2 in \cite{davies_1989}, for any Schrödinger operator  $\mathcal{H}_{1} = -\Delta + V_{1}$  where  $0 \leq V_{1} \in L_{\mathrm{loc}}^{1}(\Omega)$, the kernel $K(t, x, y)$ of $\mathrm{e}^{-\mathcal{H}_{1} t}$ satisfies
$$0  \leq K(t, x, y) \leq (4 \pi t)^{-d / 2} \mathrm{e}^{-(x-y)^{2} / 4 t} \leq (4 \pi t)^{-d / 2},$$
which implies that $\mathrm{e}^{-\mathcal{H}_{1} t}$ is a symmetric Markov semigroup on  $L^{2}(\Omega)$ and $\mathrm{e}^{-\mathcal{H}_{1} t} : L^{2}(\Omega) \to L^{\infty}(\Omega)$ is a bounded operator with norm
$\|\mathrm{e}^{-\mathcal{H}_{1} t}\|_{\infty,2} \leq (4 \pi t)^{-d / 4}$ for all  $0<t<\infty$.
Suppose that $\phi$ is a normalized eigenfunction of 
$\mathcal{H}_{1}$ corresponding to the eigenvalue $\lambda$. Then, 
\begin{equation} \label{general Linf bounds for eigenfunctions}
\|\phi\|_{L^{\infty}(\Omega)} = \mathrm{e}^{\lambda t}\|\mathrm{e}^{-\mathcal{H}_{1} t}\phi\|_{L^{\infty}(\Omega)} \leq (4 \pi t)^{-d / 4}\mathrm{e}^{\lambda t}.
\end{equation}
Let $V_{1} = V - V_{\min}$. 
By Assumption \ref{Assumption: V is bounded above and below}, $0 \leq V_{1} \leq V_{\max} - V_{\min}$ and $\psi_{k}$ is the $k$-th normalized eigenfunction of 
$\mathcal{H}_{1}$ corresponding to the eigenvalue $\lambda_{k} - V_{\min}$.
Taking $t = d/4\lambda,$ $\phi = \psi_{k}$ and $\lambda = \lambda_{k} - V_{\min}$ in (\ref{general Linf bounds for eigenfunctions}), we have
$$\|\psi_{k}\|_{L^{\infty}(\Omega)} \leq \left(\frac{e\left(\lambda_{k}- V_{\min}\right)}{\pi d}\right)^{d / 4}, \quad \text{for all } k\geq 1.$$
Substituting the upper bound for $\lambda_{k}$ in Lemma \ref{lemma: Asymptotic distribution of eigenvalues} completes the proof.
\end{proof}

\vskip 0.2in
\bibliography{main}

\end{document}